\documentclass[10pt]{amsart}
\usepackage[usenames, dvipsnames]{xcolor}
\usepackage{amsmath, latexsym, amsfonts, amssymb, amsthm,MnSymbol,stmaryrd}
\usepackage{graphicx,hyperref,dsfont,epsfig,caption,wrapfig,subfig,tikz}
\usetikzlibrary{patterns}
\usepackage{epigraph}
\usepackage{etoolbox}
\makeatletter
\patchcmd{\epigraph}{\@epitext{#1}}{\itshape\@epitext{#1}}{}{}
\makeatother
\usepackage{datetime}
\usepackage{movie15}
\hypersetup{
    linktoc=page,
    linkcolor=red,          
    citecolor=blue,        
    filecolor=blue,      
    urlcolor=cyan,
    colorlinks=true           
}

\newcommand{\nocontentsline}[3]{}
\newcommand{\tocless}[2]{\bgroup\let\addcontentsline=\nocontentsline#1{#2}\egroup}
\numberwithin{equation}{section}

\newtheorem{theorem}{Theorem}[section]
\newtheorem{lemma}[theorem]{Lemma}
\newtheorem{proposition}[theorem]{Proposition}
\newtheorem{corollary}[theorem]{Corollary}

\newtheorem{remark}[theorem]{Remark}
\newtheorem{definition}[theorem]{Definition}

\theoremstyle{definition}

\renewcommand{\tilde}{\widetilde}          
\DeclareMathSymbol{\leqslant}{\mathalpha}{AMSa}{"36} 
\DeclareMathSymbol{\geqslant}{\mathalpha}{AMSa}{"3E} 
\DeclareMathSymbol{\eset}{\mathalpha}{AMSb}{"3F}     
\renewcommand{\leq}{\;\leqslant\;}                   
\renewcommand{\geq}{\;\geqslant\;}                   
\newcommand{\dd}{\text{\rm d}}             

\newcommand{\C}{\mathbb{C}}

\newcommand{\D}{\mathbb{D}}
\newcommand{\R}{\mathbb{R}}
\newcommand{\Z}{\mathbb{Z}}
\newcommand{\N}{\mathbb{N}}
\newcommand{\Q}{\mathbb{Q}}

\newcommand{\E}{\mathds{E}}
\renewcommand{\P}{\mathds{P}}
\newcommand{\ind}{\mathds{1}}
\newcommand{\cjd}{\rangle}
\newcommand{\cjg}{\langle}

\newcommand{\hf}{\frac{_1}{^2}}
\def\M{\mathbf{M}}

\def\l{\mathbf{l}}
\def\k{\mathbf{k}}

\newcommand{\pl}{\partial}
\newcommand{\bbar}{\overline}
\newcommand{\mc}{\mathcal}
\newcommand{\la}{\lambda}
\newcommand{\til}{\widetilde}

\usepackage{xparse}

\def\eps{\varepsilon}

\def\T{\mathbb{T}}

\def\bi{\begin{itemize}}
\def\ei{\end{itemize}}

\def\bnum{\begin{enumerate}}
\def\enum{\end{enumerate}}
\def\<#1{\langle #1 \rangle}

\def\M{\mathbf{M}}

\newcommand{\caA}{{\mathcal A}}
\newcommand{\caB}{{\mathcal B}}

\newcommand{\caD}{{\mathcal D}}
\newcommand{\caE}{{\mathcal E}}
\newcommand{\caF}{{\mathcal F}}
\newcommand{\caG}{{\mathcal G}}
\newcommand{\caH}{{\mathcal H}}

\newcommand{\caN}{{\mathcal N}}
\newcommand{\caO}{{\mathcal O}}
\newcommand{\caP}{{\mathcal P}}

\newcommand{\caS}{{\mathcal S}}
\newcommand{\caT}{{\mathcal T}}

\newcommand{\caZ}{{\mathcal Z}}
\def\indic{\operatorname{1\hskip-2.75pt\relax l}}

\author{Colin Guillarmou}
\address{Universit\'e Paris-Saclay, CNRS,  Laboratoire de math\'ematiques d'Orsay, 91405, Orsay, France.}
\email{colin.guillarmou@math.u-psud.fr}

\author{Antti Kupiainen}
\address{University of Helsinki, Department of Mathematics and Statistics}
\email{antti.kupiainen@helsinki.fi}

\author{R\'emi Rhodes}
\address{Aix-Marseille University, Institut de Math\'ematiques de Marseille (I2M), and Institut Universitaire de France (IUF)}
\email{remi.rhodes@univ-amu.fr}

\author{Vincent Vargas}
\address{LAGA Universit\'e Sorbonne Paris Nord, CNRS, UMR 7539, F-93430, Villetaneuse, France}
\email{vargas@math.univ-paris13.fr}

\title{Conformal bootstrap in Liouville Theory}    

\begin{document}
 
 \begin{abstract}
The conformal bootstrap hypothesis is a powerful idea in theoretical physics which has led to spectacular predictions in the context of critical phenomena. It postulates an explicit expression for the correlation functions  of a conformal field theory in terms of its 3-point correlation functions. In this paper we give the first mathematical proof of the conformal bootstrap hypothesis in the context of Liouville theory, a 2-dimensional conformal field theory studied since the eighties in theoretical physics and constructed recently by F. David and the three last authors using probability theory. The proof is based on a probabilistic  construction of the Virasoro algebra highest weight modules through spectral analysis of an associated self adjoint operator akin to harmonic analysis on non compact Lie groups but in an infinite dimensional setup. 
 
  \end{abstract}
 
 \maketitle
 
 \epigraph{``We developed a general  approach  to CFTs, something like complex analysis in the quantum domain. It worked very well in the various problems of statistical mechanics but the Liouville theory remained unsolved."}{--- \textup{Alexander Polyakov}, From Strings to Quarks (2008)}
 
\tableofcontents

\section{Introduction}\label{sec:intro}
\subsection{Overview}\label{sec:over}
 
There are essentially two approaches to Quantum Field Theory (QFT) in the physics literature.  In the first approach the quantum fields are (generalized) functions   $\hat\caO(t,x)$ on the space-time $\R^{d+1}$  ($d=3$ for the Standard Model of physics) taking values in operators acting in a Hilbert space $\caH$ of physical states. Matrix elements of products of quantum fields 
at different points $\langle \psi|\prod_{k=1}^N \hat\caO_k(t_k,x_k)|\psi\rangle$ where $\psi\in \caH$ (the ``vacuum" state) are (generalized) functions  on $\R^{N(d+1)}$.  Physical principles (positivity of energy) imply that these matrix elements should have an analytic continuation to the {\it Euclidean} domain where $t_k\in i\R$ and be given there as correlation functions of {\it random fields} $\caO(y)$ defined on $y\in \R^{n}$ where $n=d+1$. 
 In the second approach,  based on the so-called path integral approach due   to Feynman, these correlation functions are formally given as integrals over  a  space of generalized functions  on  $\R^{n}$, called fields, with the formal integration measure given explicitly  as a Gibbs measure with potential a functional of  the fields.  The Euclidean formulation serves also as a setup for the theory of second order phase transitions in statistical mechanics systems where now $n\leq 3$. In this case one expects  the correlation functions to possess an additional symmetry under the conformal transformations of $\R^{n}$ and the QFT is now a Conformal Field Theory (CFT).
 
In practice most of the information on QFT obtained by physicists has been perturbative, namely given in terms of a formal power series expansion in parameters perturbing a Gaussian measure (and pictorially described by Feynman diagrams). In CFT however there is another, nonperturbative,  approach going under the name of {\it Conformal bootstrap}. In this approach one postulates a set of special {\it primary} fields $\caO_\alpha(y)$ (or operators  $\caO_\alpha(t,x)$ in the Hilbert space formulation) whose  correlation functions transform as conformal tensors under the action of the conformal group. Furthermore one postulates a rule called the {\it operator product expansion} allowing to expand the product of two primary fields inside a correlation function (or a product of two quantum fields in the  Hilbert space formulation)  as a sum running over a subset of primary fields called the {\it spectrum} with explicit coefficients depending on the three point correlation functions, the so called  {\it structure constants} of the CFT (see Section \ref{dozzboot}). Hence  the correlation functions of a CFTs are determined in the bootstrap approach if one knows  its spectrum and the structure constants.  

In the case of two dimensional conformal field theories ($d=1$ or $n=2$ above)  the conformal symmetry constrains the possible CFTs particularly strongly and Belavin, Polyakov and Zamolodchikov \cite{BPZ} (BPZ from now on) showed the power of the bootstrap hypothesis by producing explicit expressions for the correlation functions of a large family of CFTs of interest to statistical physics among which the CFT that is believed to coincide with the scaling limit of  critical Ising model.  In a nutshell,  {BPZ} argued that 
 one could parametrize  CFTs by a unique parameter $c$  called the central charge and they found the correlation functions for certain rational values of $c$ where the number of primary fields is finite (the minimal models).
During the last decade the bootstrap approach has also led to spectacular predictions of critical exponents in the    three dimensional case \cite{PoRyVi}. The reader may consult \cite{Gaw} for some background on 2d CFTs.

 Giving a rigorous mathematical meaning to these two approaches and relating them has been a huge challenge for mathematicians. On the axiomatic level the transition from the operator theory on Hilbert space to the Euclidean probabilistic theory was understood early on and for the converse the crucial concept of {\it reflection positivity} was isolated \cite{OS1,OS2}. Reflection positivity is a property of the probability law underlying the random fields that allows for a construction of a canonical Hilbert space where operators representing the symmetries of the theory act. Reflection positivity is one of the crucial inputs in the present paper. 
 
 However on a more concrete level of explicit examples of QFTs mathematical progress has been slower.
  The (Euclidean) path integral approach was addressed by constructive field theory in dimensions $d+1\leq 4$ using probabilistic methods but detailed information has been restricted to the cases that are small perturbations of a Gaussian measure. In particular the 2d CFTs have been beyond this approach so far. A different probabilistic approach to conformal invariance has been developed during the past twenty years following the introduction by Schramm \cite{Sch} of random curves called Schramm-Loewner evolution (SLE). This approach, centred around the geometric description of critical models of statistical physics, has led to exact statements on the interfaces of percolation or the critical Ising model; following the introduction of SLE and the work of Smirnov, probabilists also managed to justify and construct the CFT correlation functions of the scaling limit of the 2d Ising model  
 \cite{CS,CHI15} (see also the review \cite{Pel} for the construction of CFT correlations via SLE observables).

 Making a mathematical  theory of the BPZ approach  triggered in the 80's and 90's intense research in the field of  
 Vertex Operator Algebras (VOA for short)  introduced by Borcherds \cite{borcherds} and Frenkel-Lepowsky-Meurman \cite{FLM89}  (see also the book \cite{Hu} and the article \cite{HuKo} for more recent developments on this formalism).  Even if the theory of VOA was quite successful to rigorously formalize numerous CFTs, the approach suffers certain limitations at the moment. First, correlations are defined as formal power series (convergence issues are not tackled in the first place and are often difficult); second, many fundamental CFTs have still not been formalized within this approach, among which the CFTs with uncountable collections of primary fields and in particular Liouville conformal field theory (LCFT in short)  studied in this paper. Moreover, the theory of VOA, which is based on axiomatically implementing the operator product expansion point of view of physics, does not elucidate the link to the the path integral approach or to the models of statistical physics at critical temperature (if any).

 In their seminal work, BPZ were  in fact motivated by the quest to compute the correlation functions in LCFT, which had been introduced a  few years before by Polyakov under the form of a path integral in his approach to bosonic string theory \cite{Pol}. Although BPZ failed to carry out the bootstrap program for LCFT\footnote{See Polyakov's citation \cite{Pol1} at the beginning of this paper.}, this   was successfully implemented later  in the physics literature by Dorn, Otto, Zamolodchikov and Zamolodchikov \cite{DoOt,ZaZa}. Since then, LCFT has appeared in the physics literature in a wide range of fields including random planar maps (discrete quantum gravity, see the review \cite{kos}) and the supersymmetric Yang-Mills theory (via the AGT correspondence \cite{AGT}). Recently, there has been a large effort in probability theory  to make sense of Polyakov's path integral formulation of LCFT within the framework of random conformal geometry and the scaling limit of random planar maps: see \cite{LeGall,Mier,MS1,MS2,MS3,DDDF,DFGPS,MG} for the construction of a random metric space describing (at least at the conjectural level) the limit of random planar maps and \cite{DMS,NS} for exact results on their link with LCFT\footnote{It is beyond the scope of this introduction to state and comment all the exciting results that have been obtained recently in this flourishing field of probability theory.}. In particular, the three last authors of the present paper  in collaboration with F. David \cite{DKRV} have constructed   
the path integral formulation of LCFT on the Riemann sphere using probability theory. This was extended to higher genus surfaces in \cite{DRV,GRV}. In this paper, we will be concerned with LCFT on the Riemann sphere.
 
LCFT depends on two parameters $\gamma \in (0,2)$ and $\mu>0$\footnote{The case $\mu=0$ is different and corresponds to Gaussian Free Field theory; for the study of a related model, see Kang-Makarov \cite{KM}.}. 
In this paper we prove that for all $\gamma \in (0,2)$ and $\mu>0$ the probabilistic construction of LCFT satisfies the hypothesis of the bootstrap approach envisioned in \cite{BPZ}. In particular we determine the spectrum of LCFT and prove the bootstrap formula  for the 4-point function in terms of the structure constants that were identified by the last three authors in \cite{dozz}. 
 LCFT for $\gamma\in (0,2)$ is a highly nontrivial CFT with an uncountable family of primary fields and a nontrivial OPE 
 and we believe the proof of conformal bootstrap in this case provides the first  nontrivial  test case where a mathematical justification for this beautiful idea from physics has been achieved. Let us emphasize that determining  the spectrum and the structure constants of LCFT  is the cornerstone of the deep link between LCFT and representation theory via exact formulae for correlation functions. Indeed they  lead to exact formulae for n point correlation functions on the sphere  and pave the way towards understanding   bootstrap formulae for higher genus surfaces \cite{GKRV}.

In a nutshell, our proof uses reflection positivity to construct a representation of the semigroup of  
dilations in $\C$ 
on an explicit Hilbert space associated to LCFT. Its generator,  
called the  \emph{Hamiltonian} of LCFT, is a   self-adjoint  unbounded operator. It has the form of a Schr\"odinger operator acting in the $L^2$-space of an infinite dimensional Gaussian measure   with a non trivial potential which  is a positive function for $\gamma \in (0,\sqrt{2})$ and a measure   for $\gamma \in [\sqrt{2},2)$.  
We perform the spectral analysis of this operator  to determine the spectrum of LCFT \footnote{The importance of understanding the spectral analysis of the Hamiltonian of LCFT was stressed by Teschner in \cite{Tesc1} which was an inspiration for us.} and show that the bootstrap formula for the 4 point correlations functions  can be seen as a Plancherel formula with respect to its spectral decomposition. We obtain exact expressions for the generalized eigenfunctions of the Hamiltonian by deriving   conformal Ward identities for LCFT correlation functions  which reflect the underlying  symmetry algebra of LCFT (i.e. the Virasoro algebra).

\vskip 3mm

\subsection{Probabilistic approach of Liouville CFT}  Let us start with the physicists formulation of LCFT. It is a theory of a random (generalized) function $\phi:\C\to \R$ called the Liouville field. One is interested in averages of functionals $F$ of  $\phi$ formally given by the path integral
\begin{equation}\label{phydef}
\langle F \rangle_{\gamma,\mu}  := \int F(\phi)e^{-  S_L(\phi) } \text{\rm D} \phi, 
\end{equation}
where $S_L$ is the  Liouville action functional
$$S_L(\phi):= \frac{1}{\pi}\int_{\mathbb{C}}\big(  |\partial_z \phi (z) |^2  + \pi \mu e^{\gamma \phi(z)  } \big)\, \dd z$$
and where $ \dd z$ denotes the Lebesgue measure on $\C$. It
depends on two parameters $\gamma \in (0,2)$ and $\mu>0$ (called cosmological constant). The notation  $ \text{\rm D} \phi$ refers to a formal ``Lebesgue measure" on the space of functions $\phi: \C \to \R $ obeying the asymptotic  $$\phi(z) \underset{|z| \to \infty}{\sim} -2 Q  \ln |z|$$ with $Q=\frac{2}{\gamma}+\frac{\gamma}{2}$. This asymptotic is  formally required by conformal invariance whereby the field $\phi$ should be thought to be defined on the Riemann sphere $\hat \C= \C \cup \lbrace \infty \rbrace$.
The physically  interesting expectations in LCFT are the n-point correlation functions 
\begin{equation}\label{phyvertex}
\langle \prod_{i=1}^n V_{\alpha_i}(z_i)\rangle_{\gamma, \mu}
\end{equation} 
of the exponentials  of the Liouville field $V_\alpha(z)=e^{\alpha\phi(z)}$ called ``vertex operators" in physics.
In \eqref{phyvertex} the points $z_i \in \C$ are distinct and  $\alpha_i \in \C$.

The recent work \cite{DKRV} gave a rigorous mathematical meaning to the correlation functions \eqref{phyvertex} via probability theory as we now describe. As usual the quadratic part of the action functional can be interpreted in terms of a Gaussian Free Field (GFF).
We consider the  GFF $X$ on $ \C$ with the following covariance kernel
\begin{equation}\label{chatcov}
\E [ X(z)X(z')] =\ln\frac{1}{|z-z'|}+\ln|z|_++\ln|z'|_+:=G(z,z')
\end{equation}
with $|x|_+=\max(|x|,1)$. 
This GFF can be defined as a random generalized function on a suitable probability space  $( \Omega, \Sigma, \P)$ (with expectation $\E[.]$): see Section \ref{sub:gff}. As is readily checked with the covariance \eqref{chatcov} $ X(z)\stackrel{law}= X(1/z)$ so that the GFF is naturally defined on the Riemann sphere $\hat\C=\C\cup\{\infty\}$. Furthermore it defines  $\P$ almost surely an element in  the space of distributions $\caD'(\hat\C)$.

The second ingredient needed for making sense of the path integral  is the following Gaussian multiplicative chaos measure (GMC, originally introduced by Kahane \cite{cf:Kah})
\begin{equation}\label{GMCintro}
M_\gamma (\dd z):=  \underset{\epsilon \to 0} {\lim}  \; \; e^{\gamma X_\epsilon(z)-\frac{\gamma^2}{2} \E[X_\epsilon(z)^2]} \frac{\dd z}{|z|_+^4}
\end{equation}
where $X_\epsilon= X \ast \theta_\epsilon$ is the mollification of $X$ with an  approximation   $(\theta_\epsilon)_{\epsilon>0}$ of the Dirac mass $\delta_0$; indeed, one can show that the limit \eqref{GMCintro} exists in probability in the space of Radon measures on $\hat\C$ and that the limit does not depend on the mollifier $\theta_\eps$: see \cite{RoV, review, Ber} for example.   The condition $\gamma \in (0,2)$ stems from the fact that the random measure $M_\gamma$ is different from zero if and only if $\gamma \in (0,2)$. 
With these definitions the rigorous definition of the Liouville field $\phi$ is
\begin{align}\label{liouvillefield}
\phi(z)= c+X(z)-2Q \ln |z|_+
\end{align}
  and the expectation \eqref{phydef} for $F$ continuous and non negative on $\caD'(\hat\C)$ is defined as 
 \footnote{In the recent paper \cite{dozz}, the authors used a convention where the RHS is multiplied by $2$.}
\begin{equation}\label{FL1}
\langle F(\phi) \rangle_{\gamma,\mu}
 =  \:  \int_\R   e^{ -2Qc}\E[ F(c+X-2Q \ln |.|_+)
  e^{ -\mu e^{\gamma c}M_\gamma(\C)}]\,\dd c.
\end{equation}  
The variable $c\in\R$ 
stems from the fact that in the path integral \eqref{phydef} one wants to include also the constant functions on $\C$ which are not captured by the GFF. Indeed, it is needed to 
 ensure conformal invariance of LCFT \cite{DKRV}.  We remark also that the expectation $\langle\cdot\rangle_{\gamma,\mu}$ is {\it not} a probability measure as $\langle 1\rangle_{\gamma,\mu}=\infty$  \cite{DKRV}. 

Following \cite{DKRV}, the $n$-point correlations \eqref{phyvertex} can be defined for \emph{real} valued $\alpha_i$ via the following limit 
\begin{equation}\label{deflimintro}
\langle \prod_{i=1}^n V_{\alpha_i}(z_i)\rangle_{\gamma, \mu}:= \underset{\epsilon \to 0} {\lim}  \:  \langle \prod_{i=1}^n V_{\alpha_i,\epsilon}(z_i)\rangle_{\gamma, \mu} 
\end{equation}
where $z_1,\dots,z_n\in \C$ are distinct, 
\begin{equation}\label{Vregul}
V_{\alpha, \epsilon}(z)= |z|_+^{-4\Delta_\alpha}e^{\alpha c}e^{\alpha X_\epsilon  (z)-\frac{\alpha^2}{2}\E [X_\epsilon  (z)^2]} 
\end{equation}  
 and $\Delta_\alpha$ is called the \emph{conformal weight} of $V_\alpha$
 \begin{align}\label{deltaalphadef}
\Delta_\alpha=\frac{\alpha}{2}(Q-\frac{\alpha}{2}), \ \ \ \alpha\in\C.
\end{align} 
The limit \eqref{deflimintro} exists and is non trivial if and only if the following bounds hold 
\begin{equation}\label{Seibergintro}
 \sum_{i=1}^n \alpha_i >2Q, \quad \quad \quad  \alpha_i <Q, \; \; \forall i=1,\dots,n  \quad \quad \quad \quad  ({\bf Seiberg \: bounds}).
\end{equation}
One of the main results of  \cite{DKRV} is that the limit \eqref{deflimintro} admits the following representation in terms of the moments of GMC
\begin{equation}\label{probarepresentation}
\langle \prod_{i=1}^n V_{\alpha_i}(z_i)\rangle_{\gamma, \mu}=   \gamma^{-1} \left (  \prod_{1 \leq j<j' \leq n}  \frac{1}{|z_j-z_{j'}|^{\alpha_j \alpha_{j'}}} \right ) \mu^{-s} \Gamma(s) \E[ Z^{-s}  ]
\end{equation}
where $s=\frac{\sum_{i=1}^n \alpha_i-2Q}{\gamma}$, $\Gamma$ is the standard Gamma function and (recall that $|x|_+=\max(|x|,1)$)
\begin{equation*}
Z= \int_{\mathbb{C}}  \left (  \prod_{i=1}^n  \frac{|x|_+^{\gamma \alpha_i}}{|x- z_i|^{\gamma \alpha_i}} \right )    M_\gamma (\dd x).
\end{equation*}
We stress that the formula \eqref{probarepresentation} is valid for correlations with $n \geq 3$, which can be seen at the level of  the Seiberg bounds \eqref{Seibergintro}\footnote{In fact, one can extend the probabilistic construction \eqref{probarepresentation} a bit beyond the Seiberg bounds but the extended bounds  also imply $n \geq 3$. We will not discuss these extended bounds in this paper.}. 
 Also it was proved  in \cite[Th 3.5]{DKRV} that these correlation functions 
are {\it conformally covariant}. More precisely, if $z_1, \cdots, z_n$ are $n$ distinct points in $\C$  then for a M\"obius map $\psi(z)= \frac{az+b}{cz+d}$ (with $a,b,c,d \in \C$ and $ad-bc=1$) 
 \begin{equation}\label{KPZformula}
\langle \prod_{i=1}^n V_{\alpha_i}(\psi(z_i))    \rangle_{\gamma, \mu}=  \prod_{i=1}^n |\psi'(z_i)|^{-2 \Delta_{\alpha_i}}     \langle \prod_{k=1}^n V_{\alpha_i}(z_i)     \rangle_{\gamma, \mu}.
\end{equation}  
Because of relation \eqref{KPZformula}, the vertex operators are \emph{primary fields} in the language of CFT. The M\"obius covariance implies in particular that the three point functions are determined up to a constant, called the {\it structure constant}, which we write as
\begin{align}\label{struconst}
 \langle  V_{\alpha_1}(0)  V_{\alpha_2}(1)V_{\alpha_3}(\infty)   \rangle_{\gamma, \mu}:=\lim_{|u|\to \infty} |u|^{4 \Delta_{\alpha_3}} \langle  V_{\alpha_1}(0)  V_{\alpha_2}(1)V_{\alpha_3}(u)   \rangle_{\gamma, \mu}.
\end{align}
Similarly the four point function can be reduced to a function of one complex variable defined by
\begin{align}\label{4pointinfty} 
 \langle  V_{\alpha_1}(0)  V_{\alpha_2}(z) V_{\alpha_3}(1) V_{\alpha_4}(\infty)  \rangle_{\gamma,\mu}  :=& \lim_{|u|\to \infty} |u|^{4 \Delta_{\alpha_4}}\langle  V_{\alpha_1}(0)  V_{\alpha_2}(z) V_{\alpha_3}(1) V_{\alpha_4}(u)  \rangle_{\gamma,\mu}.
\end{align}

Let us now turn to the conformal bootstrap approach to LCFT.

\subsection{DOZZ formula and Conformal Bootstrap.} \label{dozzboot}

In theoretical physics conformal field theory is a quantum field theory with conformal group symmetry. In particular one then postulates existence of primary fields whose correlation functions transform covariantly under conformal maps (M\"obius transformations in the two dimensional case). The 
basic physical axiom of conformal field theory apart from this covariance    is the  {\it operator product expansion} (OPE for short). Denoting the primary fields still as $V_\alpha(x)$ for $x\in\R^n$ then,
in very loose terms and cutting many corners, the OPE is the identity 
\begin{align}\label{OPE}
V_\alpha(x)V_{\alpha'}(x')=\sum_{\beta\in\caS}C^{\beta}_{\alpha\alpha'}(x,x',\nabla_x)V_\beta(x)
\end{align}
 where $\caS$  labels a special set of primary fields  called the {\it spectrum} of the CFT  and the $C^{\beta}_{\alpha,\alpha'}$ are differential operators {\it completely determined } by the conformal weights of the fields $V_\alpha, V_{\alpha'},V_\beta$ and linear in the structure constants $C_{\alpha\alpha'\beta}$ (defined in general analogously to \eqref{struconst}). The identity \eqref{OPE} is assumed to hold once inserted in the correlation functions. Obviously a repeated application of the OPE allows to express an $n$-point function of the CFT as a sum of products of the structure constants.  Hence to  ``solve a CFT" one needs to find its structure constants and spectrum. To find these, the following approach where the 4-point function\footnote{In $d>2$ the four point function is a function of two complex variables.}  \eqref{4pointinfty}
plays a fundamental role has been extremely fruitful. Applying the OPE \eqref{OPE} in \eqref{4pointinfty}  to $\alpha_1$ and  $\alpha_2$  leads to  a quadratic expression in the structure constants.  Applying the OPE instead to  $\alpha_2$ and  $\alpha_3$ yields another quadratic expression and equating the two produces a quadratic equation for the structure constants. Varying $\alpha_1,\dots,\alpha_4$ results in a set of quadratic equations labeled by 4-tuples of $\alpha_i\in\caS$ (for LCFT see  \eqref{crossingsymmetry1}).
 These 4-point bootstrap equations pose strong constraints on the spectrum and the structure constants and have been used to great effect e.g. in the case of the 3-dimensional Ising model \cite{Ry1,Ry2}. In two dimensions their study in the case when one of the fields $V_{\alpha_i}$ is a so-called degenerate field led in \cite{BPZ} to the discovery of the {\it minimal models} (2d Ising model among them) and their spectra and structure constants.

The motivation for the BPZ paper \cite{BPZ} was to solve the bootstrap equations for LCFT  but in this they were  unsuccessful \cite{Pol1}. The spectrum of LCFT was conjectured  in \cite{ct, bct, gn} to be {\it continuous} $\caS=Q+i\R_+$ and an explicit formula (see Appendix \ref{dozz}) for the structure constants, the  DOZZ formula,  was postulated by Dorn, Otto, Zamolodchikov and Zamolodchikov \cite{DoOt,ZaZa}. A derivation based on the degenerate 4-point bootstrap equations was subsequently given by Teschner \cite{Tesc} and further evidence for the formula was given in \cite{Tesc1}, \cite{Tesc2}. The DOZZ expression  which we denote by $C_{\gamma,\mu}^{{\rm DOZZ}} (\alpha_1,\alpha_2,\alpha_3 )$  is analytic in the variables $\alpha_1,\alpha_2,\alpha_3 \in \C$ (with a countable number of  poles) and one is led to expect that it coincides with the probabilistic expression \eqref{struconst}  on the domain of validity of the probabilistic construction, i.e. for real $\alpha_1,\alpha_2,\alpha_3$ satisfying the Seiberg bounds \eqref{Seibergintro}. This is indeed the case: in a  recent work \cite{dozz}, the last three authors  proved that the probabilistically constructed $3$-point correlation functions satisfy the DOZZ formula:
\begin{equation*}
 \langle  V_{\alpha_1}(0)  V_{\alpha_2}(z) V_{\alpha_3}(\infty)  \rangle_{\gamma,\mu} =\frac{1}{2}C_{\gamma,\mu}^{{\rm DOZZ}} (\alpha_1,\alpha_2,\alpha_3 )\footnote{The $\frac{1}{2}$ factor here is a general global constant which can be absorbed in the definition of the probablistic construction: see footnote associated to \eqref{FL1}.}.
 \end{equation*} 

Given that the spectrum is $\caS=Q+i\R_+$, the formula \eqref{OPE} (where the sum becomes an integral) leads formally to the following bootstrap conjecture for the 4 point correlation functions
\begin{align}
& \cjg  V_{\alpha_1}(0)  V_{\alpha_2}(z)  V_{\alpha_3}(1) V_{\alpha_4}(\infty)\cjd^{{\rm Boot}}_{\gamma,\mu}   \nonumber \\
& =  \frac{1}{8 \pi}\int_{0}^\infty  C_{\gamma,\mu}^{{\rm DOZZ}}(\alpha_1, \alpha_2, Q- iP  )  C_{\gamma,\mu}^{{\rm DOZZ}}(Q+ iP, \alpha_3, \alpha_4  ) |z|^{2(\Delta_{Q+iP}-\Delta_{\alpha_1}-\Delta_{\alpha_2})}  |  \mathcal{F}_P (z) |^2 {\rm d}P \label{4pointidentity}
\end{align}
where $\mathcal{F}_P$ are holomorphic functions in $z$ called  (spherical) \emph{conformal blocks}. The conformal blocks are universal in the sense that they only depend on the conformal weights $\Delta_{\alpha_i}$ and $\Delta_{Q+iP}$  
and    the central charge $c_L=1+6Q^2$ of LCFT,  i.e. 
$  \mathcal{F}_P (z)= \mathcal{F}(c_L,\Delta_{\alpha_1},\Delta_{\alpha_2},\Delta_{\alpha_3},\Delta_{\alpha_4},\Delta_{Q+iP},z)$; see   
Subsection \ref{sub:defblock} for the exact definition.

The second hypothesis on the bootstrap approach i.e. the fact that the OPE may be applied in the two ways explained above goes under the name of {\it crossing symmetry}. More specifically, it is the conjecture that  the following  identity holds for real $z \in (0,1)$
\begin{align}
& \int_{\R^+} C_{\gamma,\mu}^{ \mathrm{DOZZ}}( \alpha_1,\alpha_2, Q-iP  )  C_{\gamma,\mu}^{ \mathrm{DOZZ}}( \alpha_3,\alpha_4, Q+iP  ) |z|^{2(\Delta_{Q+iP}-\Delta_{\alpha_1}-\Delta_{\alpha_2})} |\mathcal{F}_P (z)|^2 dP  \nonumber \\
& = \int_{\R^+}  C_{\gamma,\mu}^{ \mathrm{DOZZ}}( \alpha_3,\alpha_2, Q-iP  )  C_{\gamma,\mu}^{ \mathrm{DOZZ}}( \alpha_1,\alpha_4, Q+iP  ) |1-z|^{2(\Delta_{Q+iP}-\Delta_{\alpha_3}-\Delta_{\alpha_2})} |\tilde{\mathcal{F}}_P (1-z)|^2 dP \label{crossingsymmetry1}
\end{align}
where $\tilde{\mathcal{F}}_P $ is obtained from $\mathcal{F}_P$ by flipping the parameter $\alpha_1$ with $\alpha_3$. As explained above this identity is a very strong constraint in LCFT.

\subsection{Main result on conformal bootstrap}

 In this paper, we justify rigorously the bootstrap approach to LCFT by constructing the spectral representation  of LCFT; as an output, we prove the bootstrap and crossing  formulas described in the previous subsection. 

To understand our approach to the spectrum of LCFT it is useful to draw the analogy with the harmonic analysis on a Lie group $G$ and in particular the Plancherel identity on $L^2(G)$. For a compact $G$, $L^2(G)$  is decomposed into a direct sum of irreducible (highest weight) representations of $G$ whereas in the non-compact case also a continuous family of representations (a direct integral) appears. This decomposition is related to the spectral decomposition of a self-adjoint operator (the Laplacian) acting on the Hilbert space $L^2(G)$.   
In 2d CFT,  Osterwalder-Schrader's method of reflection positivity  provides a canonical Hilbert space where   the symmetry 
algebra of 2d CFT, the \emph{Virasoro algebra}, acts.  The role of the Laplacian  is played by a self adjoint operator, a special element in the Virasoro algebra called the Hamiltonian of the CFT. In the case of ``compact CFTs" (examples being the minimal models of BPZ) the spectrum of this operator is discrete whereas in case of ``non-compact CFTs" the spectrum is continuous.  Our proof is based on finding the spectral resolution of  this operator using scattering theory and representation theory of the Virasoro algebra, leading to a Plancherel type identity as a rigorous version of the OPE.

The main result of this paper is the following theorem proving that the conformal bootstrap formula \eqref{4pointidentity} holds  
for the probabilistic construction of the $4$-point function: 
\begin{theorem}\label{bootstraptheoremintro}
Let  $\gamma \in (0,2)$ and $\alpha_i<Q$ for all $i  \in \llbracket 1,4\rrbracket$. Then the following identity holds for $\alpha_1+\alpha_2 >Q$ and $\alpha_3+\alpha_4>Q$
\begin{equation}
\label{bootstrapidentityintro}
 \langle  V_{\alpha_1}(0)  V_{\alpha_2}(z) V_{\alpha_3}(1) V_{\alpha_4}(\infty)  \rangle_{\gamma,\mu} =\cjg V_{\alpha_1}(0)  V_{\alpha_2}(z)  V_{\alpha_3}(1) V_{\alpha_4}(\infty) \cjd^{{\rm Boot}}_{\gamma,\mu}.
 \end{equation}
 \end{theorem}
 The conditions $\alpha_1+\alpha_2>Q$ and  $\alpha_3+\alpha_4>Q$ are essential. Indeed, if $\alpha_1+\alpha_2<Q$ the analytic continuation of \eqref{4pointidentity} from $\alpha_i \in \mathcal{S}$ requires adding an extra term (cf. the discussion on the so-called discrete terms in \cite{ZaZa}).  
The second  main input to the bootstrap hypothesis, namely the crossing symmetry conjecture \eqref{crossingsymmetry1},
 follows directly from our work since one has by conformal covariance \eqref{KPZformula} of the probabilistic construction of the correlations 
\[
 \langle  V_{\alpha_1}(0)  V_{\alpha_2}(z) V_{\alpha_3}(1) V_{\alpha_4}(\infty)  \rangle_{\gamma,\mu}= \langle  V_{\alpha_3}(0)  V_{\alpha_2}(1-z) V_{\alpha_1}(1) V_{\alpha_4}(\infty)  \rangle_{\gamma,\mu}
\]
whereby we get the following  corollary: 
\begin{corollary}
The bootstrap construction of LCFT satisfies crossing symmetry for $\gamma \in (0,2)$.
\end{corollary}
This result seems to be very hard to prove directly, however 
let us also mention that Teschner has given strong arguments in \cite{Tesc1} in that direction. 

 \begin{remark}
We have stated the bootstrap conjecture as the statements \eqref{4pointidentity} and \eqref{crossingsymmetry1} since these are the crucial relations following from the OPE axiom \eqref{OPE}  used by physicists to study CFTs. However, as explained above  the OPE axiom \eqref{OPE} leads also to a recursive computation of the $n$-point correlation functions for all $n$. 
Likewise, the approach in this paper can be extended  to $n>4$ by a $n-3$-fold application of the Plancherel identity. Hence the spectral analysis of the LCFT Hamiltonian is the crucial result of this paper. In order to keep the length of this paper reasonable, we will discuss these generalisations  elsewhere, in particular for the case of LCFT on the complex torus where the n-point formuli are mathematically quite appealing  \cite{GKRV}. 
\end{remark}

\subsection{Conformal blocks and relations with the AGT conjecture.} 
 Let us mention that it is not at all obvious that the bootstrap definition of the four point correlation function, i.e. the right hand side of  \eqref{bootstrapidentityintro}, exists for real $\alpha_i$ satisfying the condition $ \alpha_i<Q$ for all $i  \in  \llbracket 1,4\rrbracket$ along with $\alpha_1+\alpha_2 >Q$ and $\alpha_3+\alpha_4>Q$. Indeed, first the conformal blocks are defined via a series expansion 
 \begin{equation}\label{blocksintro}
 \mathcal{F}_P(z)= \sum_{n=0}^\infty \beta_n z^n
 \end{equation}
 where the coefficients $\beta_n$, which have a strong representation theoretic content, are non explicit 
 : see \eqref{expressionbeta} for the exact definition of $\beta_n$. Hence, it is not obvious that the series \eqref{blocksintro} converges for $|z|<1$. Second, it is not clear that the integral in $P\in \R^+$ of expression \eqref{4pointidentity} is convergent. As a matter of fact, in the course of the proof of Theorem \ref{bootstraptheoremintro}, we establish both that the radius of convergence\footnote{We acknowledge here an argument that was given to us by Slava Rychkov in private communication.} of \eqref{blocksintro} is $1$ for almost all $P$ and that the integral \eqref{4pointidentity} makes sense. 
 To the best of our knowledge, the proof of the convergence of the conformal blocks is new and we expect that the result holds for all $P$, although we do not need such a strong statement for our purpose. Let us mention here the recent work 
 \cite{GRSS} which establishes a probabilistic formula involving moments of a GMC type variable for the conformal blocks  of LCFT on the complex torus 
  thereby proving the existence of the torus blocks for all values of the relevant parameters. 
 Convergence of conformal blocks defining series is also topical in physics, see \cite{rych1,rych2,rych3}.

The AGT correspondence \cite{AGT} between $4d$ supersymmetric Yang-Mills theory and the bootstrap construction of LCFT 
conjectures  that $\mathcal{F}_P(z)$ coincides with special cases of Nekrasov's partition function \cite{Ne04}. 
In particular this leads to  an explicit formula for $\beta_n$ in \eqref{blocksintro}. However, even admitting this conjecture, it remains difficult to show that the radius of convergence in \eqref{blocksintro}  is $1$: 
see for instance 
\cite{FLM18}. 
The AGT conjecture has  been proved  as an identity between formal power series in the case of the torus in 
 \cite{Ne} following the works 
 \cite{MO, SV} 
 but this does not address the issue of convergence. 
 See also \cite{FL, AFLT} 
 for arguments in the physics literature which support the AGT conjecture on the torus or the Riemann sphere.  \\

\textbf{Acknowledgements.} C. Guillarmou acknowledges that this project has received funding from the European Research Council (ERC) under the European Union’s Horizon 2020 research and innovation programme, grant agreement No 725967. A. Kupiainen is supported by the Academy of Finland and ERC Advanced Grant 741487. R. Rhodes is partially supported by the Institut Universitaire de France (IUF). 
The authors wish to thank Zhen-Qing Chen, Naotaka Kajino for discussions on Dirichlet forms, Ctirad  Klimcik and Yi-Zhi Huang for explaining the links with Vertex Operator Algebras, Slava Rychkov for fruitful discussions on the convergence of conformal blocks, Alex Strohmaier, Tanya Christiansen and Jan Derezinski for discussions on the scattering part and Baptiste Cercl\'e for comments on  earlier versions of this manuscript.

\section{Outline of the proof}\label{sec:outline}
 In this section  we give  an informal summary of the proof of the bootstrap formula with pointers to precise definitions and statements.

\subsection{Reflection positivity}\label{outline:OS}

The LCFT expectation \eqref{FL1} is an expectation in a positive measure but it has another positivity property called reflection positivity (or Osterwalder-Schrader positivity) that allows us to express the correlation functions of LCFT in terms of a scalar product in the {\it physical Hilbert space} of LCFT. This task is carried out in Section \ref{sec:ospos} and outlined now. The construction of the Hilbert space is based on an involution acting on observables $F$. For this,  we consider   the reflection at the unit circle  $ \theta: \hat\C\to\hat\C$ 
\begin{align}
 \theta(z)=1/\bar{z}
\label{thetadef}
\end{align}  
 which maps the unit disk $\D$ to its complement $\D^c$. We promote it to an operator $\Theta$  acting on observables $F \mapsto \Theta F$ by
 \begin{align}\label{Thetadef}
 (\Theta F) (\phi):=F(\phi\circ \theta-2Q\ln|\cdot|).
\end{align}
This allows us to define a sesquilinear form acting on a set  $\mathcal{F}_\D$ of observables $F$ that depend only on the restriction $\phi|_\D$ of $\phi$ to the unit disk (i.e. $F(\phi)=F(\phi|_\D)$) by
\begin{equation}\label{osform:intro}
(F,G)_\D:=\langle \Theta F(\phi) \overline{G(\phi)}\rangle_{\gamma,\mu},\quad  F,G\in \mathcal{F}_\D.
\end{equation}
Note that $ \Theta F$ is an observable depending  on the restriction $\phi|_{\D^c}$ of $\phi$ to the complementary disc  $\mathbb{D}^c$ so that the scalar product of two observables on the disk $\mathbb{D}$ is given by the LCFT expectation of their product when one of them is reflected to $\mathbb{D}^c$. Reflection positivity is the statement that the sesquilinear form \eqref{osform:intro} is nonnegative:
\begin{equation*}
(F,F)_\D\geq 0
\end{equation*}
see Proposition \ref{OS1}.
 The canonical Hilbert space $\caH_\D$ of LCFT is  defined as the completion of $ \mathcal{F}_\D$, quotiented out by the null set   $\caN_0=\{F\in \caF_\D \,|\, (F,F)_{\D}=0\}$, with respect to the sesquilinear form \eqref{osform:intro}. This space can be realized in more concrete terms as follows.
 
\subsection{Hilbert space of LCFT}\label{outline:hilb}
 The space $\caH_\D$ can be realized as an $L^2$ space on  a set of fields on the equatorial circle $\T=\partial\D$ using the domain Markov property of the GFF. To do this let $\varphi=X|_\T$ be the restriction of the GFF to the unit circle. $\varphi$ can be realized (see Section \ref{sub:gff})  as a (real valued) random  Fourier series
\begin{equation}\label{GFFcircle0}
\varphi(\theta)=\sum_{n\not=0}\varphi_ne^{in\theta} 
\end{equation}
with $\varphi_{n}=\frac{1}{2\sqrt{n}}(x_n+iy_n)$ for $n>0$ where $x_n,y_n$ are i.i.d. unit Gaussians.
$\varphi$ can be understood as   a random element in a Sobolev space  $W^{s}(\T)$ with $s<0$  (see \eqref{outline:ws}) and 
as the coordinate function in the probability space $\Omega_\T=(\R^2)^{\N^\ast}$ equipped with a cylinder set sigma algebra and a Gaussian probability measure $\P$ (see \eqref{Pdefin}). The GFF $X$ \eqref{chatcov} can now be decomposed (Section \ref{sub:gff}) as an independent sum
  \begin{equation}\label{decomposeGFF}
 X\stackrel{{\rm law}}= P\varphi+ X_\D+X_{\D^c}
 \end{equation} 
  where $P\varphi$ is the harmonic extension of $\varphi$ to $\hat\C$  and $X_\D,X_{\D^c}$ are two independent GFFs on $\D$  and $\D^c$  with Dirichlet boundary conditions. In Proposition \ref{OS1} we show that there is a unitary map $U:\caH_\D\to  L^2(\R\times \Omega_\T)$ given by 
   \begin{align}\label{udeff0}
 (UF)(c,\varphi)=
 e^{-Qc}\E_\varphi[ F(c+X)e^{-\mu e^{\gamma c}M_\gamma(\D)}], 
\end{align}
where $\E_\varphi$ is expectation over $X_\D$ in the decomposition $X|_\D=X_\D+P\varphi$. Hence   $\caH_\D$ can by identified with $L^2(\R\times \Omega_\T)$. The operator $U$ then allows us to write the 4 point function in terms of the scalar product $\langle\cdot|\cdot\rangle_2$  in $L^2(\R\times \Omega_\T)$, namely, we have for $|z|<1$ 
\begin{equation}\label{4point_via_U}
\langle  V_{\alpha_1}(0)  V_{\alpha_2}(z) V_{\alpha_3}(1) V_{\alpha_4}(\infty)  \rangle_{\gamma,\mu}=\big\langle U \big(V_{\alpha_1}(0)V_{\alpha_2}(z)\big) \,\big|\, U \big(V_{\alpha_4}(0)V_{\alpha_3}(1)\big)  \big\rangle_{2}.
\end{equation}
The bootstrap formula is then obtained by 
expanding this scalar product along the eigenstates of a self adjoint operator $\mathbf{H}$ on $L^2(\R\times \Omega_\T)$, the LCFT Hamiltonian to which we now turn.


\subsection{Hamiltonian of Liouville theory}

 For $q\in \C$ with $|q|\leq 1$,   the dilation map $z\in \C\to s_q(z)=qz$ maps the unit disc to itself and it gives rise to a map $S_q:\mathcal{F}_\D\to \mathcal{F}_\D$ by
 \begin{align}\label{dilationS}
 S_q F(\phi):=F(\phi\circ s_q+Q\ln |q|)
\end{align}
 for $F\in \mathcal{F}_\D$. In Proposition \ref{dilationsemi}  we show that these operators satisfy $S_qS_{q'}=S_{qq'}$ and $q\mapsto S_q$ descends to  a strongly continuous semigroup on $\caH_\D$. Taking $q=e^{-t}$ for $t\geq 0$, we get a contraction semigroup  on $L^2(\R\times \Omega_\T)$ via the unitary map $U$ in \eqref{udeff0}, given by the relation 
\begin{align}\label{lsemig}
e^{-t{\bf H}}=US_{e^{-t}}U^{-1}.
\end{align} 
  The generator of the semi-group is a positive self-adjoint operator $\bf H$ with domain $\mathcal{D}(\mathbf{H})\subset L^2(\R\times \Omega_\T)$,  called  the {\it Hamiltonian of Liouville theory}. 
  
  To get a more concrete representation for $\bf H$ we use  a probabilistic Feynman-Kac representation for $e^{-t{\bf H}}$.  It is based on the well known fact that the GFF $X(e^{-t+i\theta})$ for $t\geq 0$ can be realized as  a continuous Markov process  $$t\mapsto X(e^{-t+i\cdot})=(B_t,\varphi_t(\cdot))\in W^{s}(\T)
  $$ 
  with $s<0$. Here we decomposed $W^{s}(\T)=\R\oplus W^{s}_0(\T)$ where $ W^{s}_0(\T)$ has the $n\neq 0$ Fourier components. In this decomposition $B_t$ is a standard Brownian motion and  the process $\varphi_t$   gives rise to an Ornstein-Uhlenbeck semigroup  on $W^{s}_0(\T)$ (i.e.the  harmonic components evolve as independent OU processes), whose  generator is  a positive self-adjoint operator $ \mathbf{P}$. The operator $ \mathbf{P}$ is essentially an infinite sum of harmonic oscillators (see \eqref{hdefi}   for the exact definition). This leads to the expression for the Hamiltonian  when $\mu=0$  
  \begin{equation}\label{H0def}
\mathbf{H}^0:=\mathbf{H}|_{\mu=0}=  -\frac{1}{2}\pl_c^2 + \frac{1}{2} Q^2+ \mathbf{P} 
\end{equation}
defined on an appropriate domain  where the first two terms come from the Brownian motion, see Section \ref{sec:GFFCFT} for further  details. 

Using the formula \eqref{udeff0} we then deduce 
a Feynman-Kac formula for the LCFT semigroup in  Proposition \ref{prop:FK}) which then
 formally gives $\bf H$ as a "Schr\"odinger operator"
 \begin{equation}\label{Hdef}
\mathbf{H}=  
\mathbf{H}^0+\mu e^{\gamma c}V(\varphi)
\end{equation}
 where the potential is formally given by
 \begin{equation}\label{VVV}
V(\varphi)=\int_0^{2\pi} e^{ \gamma \varphi(\theta)- \frac{\gamma^2}{2}  \E[ \varphi(\theta)^2  ]
} \dd\theta .
\end{equation}
In Section \ref{sub:bilinear}, we give the rigorous definition of \eqref{Hdef} which actually is quite subtle and nonstandard. The operator \eqref{Hdef} is defined as the Friedrichs extension of an associated quadratic form.  For $\gamma\in [0,\sqrt{2})$  $V$ is a well defined random variable defined as the total mass of a GMC measure on $\T$ associated to $\varphi$.  
 For  $\gamma \in [\sqrt{2},2)$ however,  this GMC measure vanishes identically and  $V$ can no longer be considered as a multiplication operator: it has to be understood as a measure on some subspace of $L^2(\Omega_\T)$ but it still gives rise to a positive operator. 

We mention here that the operator \eqref{Hdef} in the absence of the $c$-variable, i.e. the operator $ \mathbf{P} +\mu V$, was first studied in \cite{hoegh} and was shown to be essentially self-adjoint on an appropriate domain provided $\gamma \in (0,1)$, condition under which $V$ is a random variable in $L^2(\Omega_\T)$.

\subsection{Spectral resolution of the Hamiltonian}\label{outline:subspectral}

One of the main mathematical inputs  in our proof comes from stationary {\it scattering theory}, see Section \ref{sec:scattering}. It is based on the observation that the potential $e^{\gamma c}V$ vanishes as $c\to -\infty$ so that $\mathbf{H}$-eigenstates should be reconstructed from their asymptotics at $c\to-\infty$, region over which they should behave like $\mathbf{H}^0$-eigenstates.

The spectral analysis of $\mathbf{H}^0$ is simple. The spectrum $\{0<\la_1<\dots\la_k<\dots\}$ of the operator $ \mathbf{P}$ is given by the natural numbers $\la_k=k$ and each eigenvalue has finite multiplicity, see Section \ref{sec:fock}.  The corresponding eigenspace $\ker (P-\la_k)$ is spanned by a family $(h_{jk})_{j=0,\dots,J(k)}$ of finite products of Hermite polynomials in the Fourier components $(\varphi_n)_{n\in\Z^\ast}$ in \eqref{GFFcircle0}, providing 
an orthonormal basis $\{h_{jk}\}_{k\in\N,j\leq J(k)}$ of $L^2(\Omega_\T)$. The spectrum of $\mathbf{H}^0$ is then readily seen to be absolutely continuous and given by the half-line $[\tfrac{Q^2}{2},+\infty)$
with a complete set of  generalized eigenstates 
\begin{align}\label{h0ev}
\Psi^0_{Q+iP,jk}(c,\varphi)=e^{iPc}h_{jk}(\varphi),  
\end{align}
with $P\in\R$, $k\in\N$, $j\leq J(k)$ and eigenvalue $\frac{1}{2}(Q^2+P^2)+\la_k=2\Delta_{Q+iP}+\la_k$ where $\Delta_{Q+iP}$ is the conformal weight \eqref{deltaalphadef}. 

In physics, the spectrum of $\mathbf{H}$ was first elaborated by Curtright and Thorn \cite{ct} (see also Teschner \cite{Tesc1} for a nice discussion of the scattering picture). 
To explain the intuition, let us consider a toy model where only the $c$-variable enters, namely the operator  $\frac{1}{2}(-\pl_c^2+Q^2)  +\mu e^{\gamma c}$ on $L^2(\R)$. 
It is a Schr\"odinger operator with potential tending to $0$ as $c\to -\infty$ and to $+\infty$ as $c\to +\infty$. Thus for   $c\to -\infty$  the eigenfunctions should tend to eigenfunctions of the $\mu=0$ problem i.e. to a  linear combination of  $e^{\pm iPc}$. Indeed, the operator  
has a complete set of  generalized eigenfunctions $f_P$, $P\in\R^+$, with asymptotics
\begin{equation*}
f_P(c) \sim  
e^{ iPc}+R(P)e^{ -iPc} \ \ \textrm{ as } c\to -\infty
\end{equation*}
and $f_P(c)\to 0$ as $c\to +\infty$. These eigenfunctions  describe scattering of waves from a wall (the exponential potential $e^{\gamma c}$ acts as a wall when $c\to+\infty$) and  $R(P)$ is an explicit coefficient called the reflection coefficient\footnote{This coefficient is a simplified version of the (quantum) reflection coefficient which appears in the proof of the DOZZ formula \cite{KRV,dozz}.}.  

For the operator $\bf H$ we expect the same intuition to hold: an eigenfunction $\Psi(c,\varphi)$ of $\bf H$ should, as $c\to -\infty$, tend to an eigenfunction of ${\bf H}^0$ of the same eigenvalue. 
 Indeed   we prove (see Theorem \ref{spectralmeasure}) which uses a different labelling)  that to each ${\bf H}^0$ eigenvector  $\Psi^0_{Q+iP,jk}$ in \eqref{h0ev} there is a corresponding ${\bf H}$ eigenvector $\Psi_{Q+iP,jk}$ with the same eigenvalue 
 \begin{align}\label{heigenv}
 {\bf H} \Psi_{Q+iP,jk}=(2\Delta_{Q+iP}+\lambda_k)\Psi_{Q+iP,jk},
\end{align}
and these form a complete set for $P\in\R_+$, $k\in\N$, $j\leq J(k)$. Applying the spectral decomposition of ${\bf H}$ to \eqref{4point_via_U} using this basis leads to
\begin{align}\label{4point_via_U1}
& \langle  V_{\alpha_1}(0)  V_{\alpha_2}(z) V_{\alpha_3}(1) V_{\alpha_4}(\infty)  \rangle_{\gamma,\mu}=\\
&\frac{1}{2\pi}
\sum_{k\in \N}\sum_{j=0}^{J(k)}\int_0^\infty \cjg U \big(V_{\alpha_1}(0)V_{\alpha_2}(z)\big)\, |\, \Psi_{Q+iP,jk}\cjd_2 \cjg \Psi_{Q+iP,jk}\, |\, U \big(V_{\alpha_4}(0)V_{\alpha_3}(1)\big)\cjd_2 \dd P.\nonumber
\end{align}
Proving this spectral resolution for $\bf H$ is the technical core of the paper and  involves a considerable amount of work (the whole Section \ref{sec:scattering}). 
The main difficulty comes from the fact that the potential $V$ appearing in ${\bf H}$ acts on the $L^2$-space of an infinite dimensional space $\R\times\Omega_\T$  and moreover $V$ is not even a function for $\gamma\in[\sqrt{2},2)$ as discussed in the previous subsection. This weak regularity and unboundedness of the potential  make the problem quite non-standard.

\subsection{Analytic continuation} 
The spectral resolution \eqref{4point_via_U1} is not yet the bootstrap formula. Indeed, it holds under quite general assumptions on $V$.
To get the bootstrap formula from  \eqref{4point_via_U1}  we need to  connect the scalar products in \eqref{4point_via_U1} to the DOZZ 3-point functions $C_{\gamma,\mu}^{ \mathrm{DOZZ}}( \alpha_1,\alpha_2, Q-iP  )$   and $C_{\gamma,\mu}^{ \mathrm{DOZZ}}( \alpha_3,\alpha_4, Q+iP )$ respectively.  The probabilistic 3-point function $C_{\gamma,\mu}( \alpha_1,\alpha_2, \alpha  )$ is defined only for $\alpha$ real and satisfying the Seiberg bounds $\alpha_1<Q$, $\alpha_2<Q$, $\alpha<Q$, $\alpha_1+\alpha_2+\alpha>2Q$ and $C_{\gamma,\mu}^{ \mathrm{DOZZ}}( \alpha_1,\alpha_2, Q-iP  )$ is an analytic continuation of this probabilistic expression. Our strategy then is to analytically continue the eigenfunction $\Psi_{Q+iP,jk}$ to real values of $Q+iP$ so that we can use the probabilistic construction of the LCFT and in particular its conformal invariance to derive identities that allow to determine the  above scalar products.

In Section, \ref{sec:scattering} we show that the eigenfunctions can be analytically continued to $ \Psi_{\alpha,jk}$ which are analytic in $\alpha$ in a connected region which contains both the spectrum line $\alpha\in Q+i\R_+$ and a half-line $]-\infty,A_k]\subset\R$ for some $A_k<Q$. This analytic continuation  still satisfies the eigenvalue problem \eqref{heigenv} (with $Q+iP$ replaced by $\alpha$) 
but  it has exponential growth in the $c$ variable as $c\to -\infty$  and  it does not contribute to the spectral resolution (compare with eigenfunctions $e^{ac}$ of the operator $-\partial_c^2$ which are associated to the $L^2$ spectrum only for $a\in i\R$).

 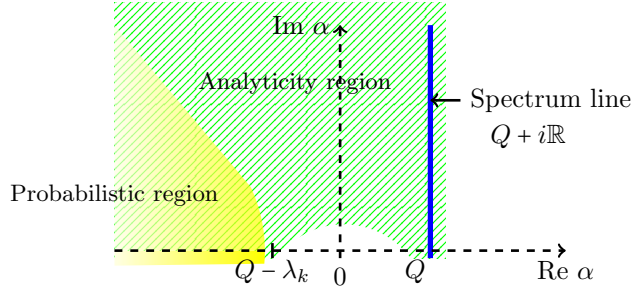
\begin{figure}[h]
 \begin{tikzpicture}[xscale=1,yscale=1]  
\newcommand{\A}{(0,-0.1) rectangle (4.4,3.3)};
\newcommand{\Ci}{(4,-0.2) arc (30:150:1.1) -- (3,-0.2)};
\fill[pattern=north east lines, pattern color=green] \A;
\fill[white] \Ci;
\node[below] at (4,0) {$Q$};
\draw[line width=1pt] (2.1,-0.1) -- (2.1,0.1) ;
\draw [line width=2pt,color=blue](4.2,-0.1) --   (4.2,3);
\draw[line width=1pt,<-] (4.2,2) -- (4.6,2)node[right]{Spectrum line } ;
\node[below] at (5.5,1.8) {$Q+i\R$};
\shade[left color=yellow!10,right color=yellow,opacity=0.7] (2,0) arc (0:25:2) -- (0,3) -- (0,0) -- (2,0);
\shade[left color=yellow!10,right color=yellow] (2,0) arc (0:-5:2) -- (0,-0.2) -- (0,0) -- (2,0);
\draw[style=dashed,line width=1pt,->] (3,-0.1) node[below]{$0$} -- (3,3) node[left]{{\rm Im} $\alpha$};
\draw[style=dashed,line width=1pt,->] (0,0) -- (6,0)node[below]{{\rm Re} $\alpha$};
\node[below] at (0,1) {{\small Probabilistic region}};
\node[below] at (2.4,2.5) {{\small Analyticity region}};
\node[below] at (2.1,0) {$Q-\lambda_{k}$};
\end{tikzpicture}
\caption{Analytic continuation of eigenstates and probabilistic region.}
\label{art:intro}
\end{figure}
 
 The crucial point we show is that the analytically continued $\mathbf{H}$-eigenfunctions $  \Psi_{\alpha,jk}$ can be obtained by {\it intertwining} the $\mathbf{H}^0$-eigenfunctions   for $\alpha$ in some complex neighborhood of the half-line $]-\infty,A_k]$ that we call {\it the probabilistic region} (see Figure \ref{art:intro}):
 \begin{equation}\label{intertwining:intro}
  \Psi_{\alpha,jk}=\lim_{t\to\infty} e^{(2\Delta_\alpha+\lambda_{k})t}e^{-t\mathbf{H}}\Psi^{0}_{\alpha,jk} \quad \quad\quad \text{\textbf{(Intertwining)}}
  \end{equation}
 where $\Psi^{0}_{\alpha,jk}=e^{(\alpha-Q)c}h_{jk}$ are the analytically continued eigenfunctions of $\mathbf{H}^0$ (see Proposition \ref{descconvergence} which has a different labelling). Furthermore, combining \eqref{intertwining:intro} with the 
  Feynman-Kac formula one finds  (Section \ref{sub:HW}) for $\la_k=0$ (then $J(0)=0$, $h_{00}(\varphi)=1$ and we denote $\Psi_{\alpha,0}$ for $\Psi_{\alpha,00}$)  that 
  \begin{align}
    \Psi_{\alpha,0}=U(V_\alpha(0))\label{urelation}
\end{align}
   where $U$ is the unitary map \eqref{udeff0}. Hence the ``non-spectral'' eigenfunctions $ \Psi_{\alpha,0}$ for $\alpha\in\R$ have a probabilistic interpretation in LCFT whereas the ``spectral'' eigenfunctions $\Psi_{Q+ip,0}$ do not have one. In the physical literature the correspondence between local fields and states in the Hilbert space is called the state-operator correspondence and in LCFT this correspondence is broken. The spectral states  $\Psi_{Q+iP,0}$ do not correspond to local fields and they are called {\it macroscopic states} whereas the non-normalizable states   $\Psi_{\alpha,0}$ correspond to the local field $V_\alpha$ and are called  {\it microscopic states}, see \cite{seiberg} for a lucid discussion.

 The relation \eqref{urelation}  leads to
\begin{equation}\label{4point_via_U2}
 \cjg  U(\prod_{i=1}^nV_{\alpha_i}(z_i)) \, |\,    \Psi_{\alpha,0}\cjd_2 
=
\langle  \big ( \prod_{i=1}^nV_{\alpha_i}(z_i) \big ) V_\alpha(\infty)  \rangle_{\gamma,\mu}
\end{equation}  
for $\{\alpha_i\},\alpha$ satisfying the Seiberg bounds (the scalar product in \eqref{4point_via_U2} still makes sense since the vector on the left has sufficient decay in the $c$-variable to counter the exponential  increase of $\Psi_{\alpha,0}$). Note in particular that \eqref{4point_via_U2} is defined only if $n$ is large enough since the Seiberg bound requires $\sum_i\alpha_i>2Q-\alpha>2Q-A_0$. To understand the case $k>0$ we need to discuss the conformal symmetry of LCFT.
 
\subsection{Conformal Ward Identities}   

The  symmetries of conformal field theory are encoded in an infinite dimensional Lie algebra, the Virasoro algebra. In the case of GFF this means that 
the  Hilbert space $L^2(\R\times\Omega_\T)$ carries a representation of two commuting Virasoro algebras with generators $\{{\bf L}^0_n\}_{n\in\Z}$ and  $\{\tilde {\bf L}^0_n\}_{n\in\Z}$ (see Section \ref{repth}). In particular  the
 GFF Hamiltonian is given by ${\bf H}^0={\bf L}^0_0+\tilde{\bf L}^0_0$ and the    ${\bf H}^0$ eigenstates $\Psi_{Q+iP,jk}^0$ for fixed $P$ can be organized to a highest weight representation of these algebras. In concrete terms the highest weight state $\Psi_{Q+iP,0}^0$ satisfies  
 \begin{equation}
\mathbf{L}_0^0 \Psi^{0}_{Q+iP,0} =\widetilde{\mathbf{L}}_0^0\Psi^{0}_{Q+iP,0}=\Delta_{Q+iP}\Psi^{0}_{Q+iP,0} ,\quad \quad \mathbf{L}_n^0\Psi^{0}_{Q+iP,0}=\widetilde{\mathbf{L}}_n^0\Psi^{0}_{Q+iP,0}=0,\ \ \ n>0,
\end{equation}
and  given two non-increasing sequences of  positive  integers $\nu= (\nu_1,\dots,\nu_k) $ and $\tilde{\nu}=(\tilde\nu_1,\dots,\tilde\nu_l) $, $k,l\in \N$ and setting 
$\mathbf{L}_{-\nu}^0=\mathbf{L}_{-\nu_k}^0 \cdots \, \mathbf{L}_{-\nu_1}^0 $ and $\tilde{\mathbf{L}}_{-\tilde \nu}^0=\tilde{\mathbf{L}}_{-\tilde\nu_j}^0 \cdots\, \tilde{\mathbf{L}}_{-\tilde\nu_1}^0$
the states 
\begin{align}\label{psibasis:intro}
\Psi^0_{Q+iP,\nu, \tilde\nu}:=\mathbf{L}_{-\nu}^0\tilde{\mathbf{L}}_{-\tilde\nu}^0  \: \Psi_{Q+iP,0},
\end{align}
are eigenstates of ${\bf H}^0$ of eigenvalue $E=2\Delta_{Q+iP}+\sum_i\nu_i+\sum_j\tilde\nu_j$ and, for fixed $E$, span that eigenspace. Thus at each eigenspace there is a nonsingular matrix $M(Q+iP)$ relating the vectors $\Psi^0_{Q+iP,\nu, \tilde\nu}$ and $\Psi^0_{Q+iP,jk}$ and furthermore we show (Proposition \ref{prop:mainvir0}) that this matrix is analytic in the variable $\alpha=Q+iP$. Setting
\begin{align*}
 \Psi_{\alpha, \nu,\tilde\nu}:=\sum_{k,j} M(\alpha)_{\nu,\tilde\nu;jk} \Psi_{\alpha,jk}
\end{align*}
the vectors $ \Psi_{Q+iP, \nu,\tilde\nu}$ are a complete set of generalized eigenvectors of $\bf H$ and they can be used in the identity \eqref{4point_via_U1} (with appropriate Gram matrices), see Section \ref{sectionproofbootstrap}.  Furthermore 
 $ \Psi_{\alpha, \nu,\tilde\nu}$  provides an analytic continuation of  $ \Psi_{Q+iP, \nu,\tilde\nu}$    
   to the half line $\alpha<Q-A$ for some $A>0$ and it is intertwined there with the corresponding vector $ \Psi^0_{\alpha, \nu,\tilde\nu}$ by the relation \eqref{intertwining:intro}.

The bootstrap formula \eqref{4pointidentity} is a consequence of the following fundamental identity
(see Proposition \ref{Ward})
 \begin{align}\label{Ward3:intro}
\langle \Psi_{Q+iP, \nu,\tilde\nu}\, |\, U(V_{\alpha_1}(0)V_{\alpha_2}(z)) \rangle_2 &=d(\alpha_1,\alpha_2,\nu,\tilde\nu)C^{ \mathrm{DOZZ}}_{\gamma,\mu}( \alpha_1,\alpha_2, Q+iP  ) \bar{z}^{|\nu|} z^{|\tilde{\nu}|}  |z|^{2 (\Delta_{Q+iP}-\Delta_{\alpha_1}-\Delta_{\alpha_2})}   
 \end{align}
where the function $d$ is an explicit function of the parameters that will contribute
to the conformal blocks. To prove \eqref{Ward3:intro} we consider the scalar product $\langle \Psi_{Q+iP, \nu,\tilde\nu}\, |\, U(\prod_{i=1}^nV_{\alpha_i}(z_i))\rangle_2 $ where $z_i\in\D$ and  $\sum_i\alpha_i>Q+A$ and analytically continue it from $\alpha=Q+iP$ ($P \in \R_+$) to $\alpha\in (2Q-\sum_i\alpha_i,Q-A)$. For such $\alpha$ we prove the Conformal Ward Identity (see Proposition \ref{proofward})
 \begin{align}\label{wardintroduction}
 \cjg  U(\prod_{i=1}^nV_{\alpha_i}(z_i)) \, |\,    \Psi_{\alpha,\nu,\tilde\nu}\cjd_2 
&=\caD(\boldsymbol{\alpha},\alpha,\nu,\tilde\nu)   \langle \prod_{i=1}^n V_{\alpha_i}  (z_i)  V_\alpha(\infty)\rangle_{\gamma,\mu}  
 \end{align}
where $\boldsymbol{\alpha}=(\alpha_1, \dots, \alpha_n)$ and $\caD(\boldsymbol{\alpha},\alpha,\nu,\tilde\nu)$ is an explicit partial differential operator in the variables $z_i$.  Using \eqref{4point_via_U2}, continuing back to $\alpha=Q+iP$ and taking $\alpha_i\to 0$ for $i>2$ we then deduce (the complex conjugate of) \eqref{Ward3:intro}.

The proof of \eqref{wardintroduction} occupies the whole  Section \ref{sec:proba}. It is based on  a representation of the states $\Psi^0_{\alpha,\nu,\tilde\nu}$ in terms of $V_\alpha(0)$ and the Stress-Energy-Tensor (SET) in Proposition \ref{def_contour}. Let us briefly explain this for $\nu=n, \tilde\nu=\emptyset$ ie for 
the state ${\bf L}^0_{-n}\Psi^0_{\alpha,0}$. We have $\Psi^0_{\alpha,0}=U_0V_\alpha(0)$ where $U_0$ is the map \ref{udeff0} for $\mu=0$. The SET is
 given in the GFF theory by the field
 \begin{equation}\label{SET:intro}
T(z):=Q\partial_{z}^2X(z)-(\partial_zX(z))^2+\E[(\partial_zX(z))^2],
  \end{equation}
defined through regularization and limit. Then we prove
\begin{align*}
{\bf L}^0_{-n}\Psi^0_{\alpha,0}=\frac{1}{2\pi i}\oint z^{1-n}U_0(T(z)V_\alpha(0))dz
\end{align*}
where the integration contour circles the origin in $\D$. Plugging this identity to 
  the intertwining  relation \eqref{intertwining:intro} and using the Feynman-Kac formula one ends up with a contour integral of the SET insertion in LCFT correlation function. This is analyzed by Gaussian integration by parts and results in the Ward identity.

\subsection{Organization of the paper}  
 
The paper is organized as follows. In Section \ref{sec:ospos}, we will introduce the relevant material on the Gaussian Free Field and explain the construction of the Hilbert space as well as the quantization of dilations; the Liouville Hamiltonian is then defined as the generator of dilations. This section uses the concept of reflection positivity. In Section \ref{sec:GFFCFT}, we study the dynamics induced by dilations in the GFF theory (i.e. $\mu=0$) and recall the basics of representation theory related to the GFF. In Section \ref{sec:LCFTFQ}, we study in more details  the Liouville Hamiltonian: we establish the Feynman-Kac formula for the associated semigroup   and identify its quadratic form, which allows us to use scattering theory to diagonalize the Liouville Hamiltonian in Section \ref{sec:scattering}. In Sections \ref{sec:proba} and \ref{subproofward}, we will prove the Conformal Ward identities for the correlations of LCFT: in a way, this can be seen as an identification of the eigenstates of the Liouville Hamiltonian.  In Section  \ref{sectionproofbootstrap}, we will prove the main result of the paper Theorem \ref{bootstraptheoremintro} using the material proved in the other sections. Finally, in the appendix, we will recall the DOZZ formula and gather auxiliary results (analyticity  of vertex operators).

\subsection{Notations and conventions}\label{notation} 
We gather here the frequently used notations:

\vskip 2mm
\noindent $\langle . \rangle_{\gamma,\mu}$ denotes the LCFT expectation \eqref{FL1}.
\vskip 1mm
\noindent $X,\phi,\varphi$ denote respectively the GFF \eqref{chatcov}, the Liouville field \eqref{liouvillefield} and the GFF on $\T$ \eqref{GFFcircle0}.
\vskip 1mm
\noindent $(\Omega_\T,\P_\T)$ denotes the probability space \eqref{omegat} with measure \eqref{Pdefin}.
\vskip 1mm
\noindent $L^p(\Omega_\T)$ for $p\geq 1$: complex valued functions  $\psi(\varphi)$ with norm  $\|\cdot\|_{L^p(\Omega_\T)}$.

\vskip 1mm \noindent  $L^p(\R \times \Omega_\T)$ for $p\geq 1$: complex valued functions  $\psi(c,\varphi)$ with norm  $\|\cdot\|_{p}$.
\vskip 1mm \noindent $\langle \cdot | \cdot \rangle_{L^2(\Omega_\T)}$: scalar product in $L^2(\Omega_\T)$.
\vskip 1mm \noindent $\langle \cdot | \cdot\rangle_2$:
 scalar product in $L^2(\R \times \Omega_\T)$.
 \vskip 1mm \noindent $e^{\alpha \rho(c)}L^p(\R \times \Omega_\T)$. 
 Weighted $L^p$-space  equipped with the norm $ \|f\|_{e^{\alpha \rho(c)}L^p}:=  \|e^{-\alpha \rho(c)} f\|_p$.
 \vskip 1mm \noindent $C^\infty(D)$ denotes the set of smooth functions on the domain $D$.  
 \vskip 1mm \noindent $C^\infty_c(D)$ denotes the set of smooth functions  with compact support in $D$.
\vskip 1mm \noindent
$(.,.)_\D$ denotes the scalar product \eqref{osform}  associated to reflexion positivity
\vskip 1mm \noindent  $\caH_\D$: Hilbert space  associated to  $(.,.)_\D$.
\vskip 1mm \noindent 
$\langle f,g \rangle _\T:=\int_0^{2 \pi }f(\theta)g(\theta) d\theta $ denotes the scalar product on $L^2(\T)$.\vskip 1mm \noindent $\langle f,g \rangle _\D:=\int_\D f(x)g(x) dx $ denotes the scalar product in   $L^2(\D)$.

\vskip 1mm \noindent All sesquilinear forms are  linear in their first argument, antilinear in their second one.

\vskip 1mm \noindent For integral operators  on some measure space $(M,\mu)$ we use the notation $(Gf)(x)=\int G(x,y)f(y)\mu(dy)$.

\vskip 1mm \noindent GFF quantities as opposed to corresponding LCFT ones will carry a  subscript or superscript $0$: for example the Hamiltonians are  ${\bf H}^0$ (GFF) and  ${\bf H}$ (LCFT).

\section{Reflection positivity} \label{sec:ospos}
In this section we prove the reflection positivity of the LCFT and  explain the isometry $U$  \eqref{udeff0} mapping the LCFT observables to states in the Hilbert space $L^2(\R \times \Omega_\T)$ as well as the semigroup  in \eqref{lsemig}. We start by a discussion of the various GFF's.

\subsection{Gaussian Free Fields}\label{sub:gff}
 We will now define the fields entering the decomposition \eqref{decomposeGFF}.
 
 \subsubsection{GFF on the unit circle $\T$} 
 Given two independent sequences of i.i.d. standard Gaussians $(x_n)_{n\geq 1}$ and $(y_n)_{n\geq 1}$, the GFF on the unit circle is the random  Fourier series
\begin{equation}\label{GFFcircle}
\varphi(\theta)=\sum_{n\not=0}\varphi_ne^{in\theta} 
\end{equation}
where  for $n>0$ 
\begin{align}\label{varphin}
\varphi_n:=\frac{1}{2\sqrt{n}}(x_n+iy_n) ,\ \ \ \varphi_{-n}:=\bbar\varphi_{n}.
\end{align} 
Let $W^s(\T)\subset \C^\Z$ be the set of sequences s.t.
\begin{equation}\label{outline:ws}
\|\varphi\|_{W^s(\T)}^2:=\sum_{n\in\Z}|\varphi_n|^2(|n|+1)^{2s} <\infty.
\end{equation}
One can easily check  that $\E[\|\varphi\|_{W^s(\T)}^2]<\infty$ for any $s<0$ so that the series \eqref{GFFcircle} defines a random element in $W^s(\T)$. Moreover, by a standard computation, one can check that it is a centered Gaussian field  with covariance kernel given by 
\begin{equation}
\E[\varphi({\theta})\varphi({\theta'})]=\ln\frac{1}{|e^{i\theta}-e^{i\theta'}|}.
\end{equation}
We will view $\varphi$ as the coordinate function of the  probability space   
\begin{align}\label{omegat}
  \Omega_\T=(\R^{2})^{\N^*}
\end{align}
 which is equipped with the cylinder sigma-algebra $
 \Sigma_\T=\mathcal{B}^{\otimes \N^*}$, where $\mathcal{B}$  stands for the Borel sigma-algebra on $\R^2$ and the   product measure 
 \begin{align}\label{Pdefin}
 \P_\T:=\bigotimes_{n\geq 1}\frac{1}{2\pi}e^{-\frac{1}{2}(x_n^2+y_n^2)}\dd x_n\dd y_n.
\end{align}
Here $ \P_\T$ is supported on $W^s(\T)$ for any $s<0$ in the sense that $ \P_\T(\varphi\in W^s(\T))=1$.

\subsubsection{Harmonic extension of the  GFF on $\T$} 
 
 The next ingredient we need for the decomposition of the GFF \eqref{decomposeGFF} is the harmonic extension  $P\varphi$ of the  circle GFF  defined on $z\in\D$ by
 \begin{equation}\label{harmonic}
(P\varphi)(z) =\sum_{n\geq 1}(\varphi_n z^n+\bar \varphi_n \bar z^n)
\end{equation}
and on $z\in\D^c$ by $(P\varphi)(1/\bar z)$ so that we have
\begin{align*}
P\varphi= (P\varphi)\circ\theta
\end{align*}
where $\theta$ is the reflection in the unit circle \eqref{thetadef}.
$P\varphi$ is a.s. a smooth field in the complement of the unit circle with covariance kernel given for $z,u\in\D$ 
\begin{align*}
\E[(P\varphi)(z)(P\varphi)(u)]=\hf\sum_{n>0}\frac{1}{n}((z\bar u)^n+(\bar z u)^n)=-\ln|1-z\bar u|
\end{align*}
and  for $z\in\D$, $u\in\D^c$
\begin{align*}
\E[(P\varphi)(z)(P\varphi)(u)]=-\ln|1-z/ u|.
\end{align*}

\subsubsection{Dirichlet GFF on the unit disk} 
 The Dirichlet GFF $X_\D$ on the unit disk  $\D$ is the centered Gaussian distribution (in the sense of Schwartz) with covariance kernel $G_\D$ given by
\begin{equation}\label{dirgreen}
G_\D(x,x'):=\E[X_\D(x)X_\D(x')]=\ln\frac{|1-x\bar x'|}{|x-x'|}.
\end{equation}
Here, $G_\D$ is the Green function of the negative of the Laplacian   $\Delta_\D$  with Dirichlet boundary condition on $\T=\partial\D$ and 
 $X_\D$ can be realized 
 as an expansion in eigenfunctions of  $\Delta_\D$ with Gaussian coefficients. However,  it will be convenient for us to use another  realization based on the following observation. Let, for $n\in\Z$
  \begin{align*}
X_n(t)=\int_0^{2 \pi} e^{-in\theta}X_\D(e^{-t+i\theta})\tfrac{d\theta}{2\pi}.
\end{align*}
 Then we deduce from \eqref{dirgreen}
\begin{align}\label{ipp}
\E[X_n(t)X_m(t')]=\left\{
\begin{array}{ll} \tfrac{1}{2|n|}\delta_{n,-m}(e^{-|t-t'||n|}-e^{-(t+t')|n|})&\, n\neq 0\\
t\wedge t'&\, n=m=0
 \end{array} \right..
\end{align}
Thus $\{X_n\}_{n\geq 0}$ are independent Gaussian processes with $X_{-n}=\bar X_n$ and $X_0$ is Brownian motion. We can and will realize them in a probability space $(\Omega_\D,\Sigma_\D,\P_\D)$ s.t. $X_n(t)$ have
continuous sample paths. Then for fixed $t$
\begin{align}\label{XDdef}
X_\D(e^{-t+i\theta})=\sum_{n\in\Z} X_n(t)e^{in\theta}.
\end{align}
takes values in $W^s(\T)$ for $s<0$ a.s. (defined in \eqref{outline:ws}) and we can take the map $t\in\R^+\mapsto X_\D(e^{-t+i\cdot})\in W^s(\T)$ to be continuous a.s. in $\Omega_\D$. Decompose $W^{s}(\T)=\R\oplus W_0^{s}(\T) $ where $f\in W_0^{s}(\T)$ has zero average: $\int f(\theta)d\theta=0$. Then 
 \begin{align*}
 X_\D(e^{-t+i\cdot})=(B_t, Y_t)
 \end{align*}
 where $B_t$ is Brownian motion and  $Y_t(\theta)=\sum_{n\neq 0} Y_n(t)e^{in\theta}$ is a continuous process in  $W^{s}_0(\T)$, independent of $B_t$.
 The Dirichlet GFF $X_{\D^c}$ on the complement $\D^c$ of $\D$ can be constructed in the same way (in a probability space $(\Omega_{\D^c},\Sigma_{\D^c},\P_{\D^c})$) and we  have the relation in law 
\begin{align}\label{xdirref}
X_{\D^c}\stackrel{{\rm law}}=X_{\D}\circ\theta
\end{align}
 or in other words $X_{\D^c}(e^{t+i\cdot})\stackrel{{\rm law}}=X_{\D}(e^{-t+i\cdot})$, $t\geq 0$. 

\subsubsection{GFF on the Riemann sphere} 

One can check that adding the covariances in the previous subsections we get that the field $X$ defined by \eqref{decomposeGFF} has the covariance
\begin{equation}\label{hatGformula}
\E [ X(x)X(y)] 
=\ln\frac{1}{|x-y|}+\ln|x|_++\ln|y|_+
\end{equation}
which coincides with \eqref{chatcov}.  In the sequel, we suppose that the GFF on the Riemann sphere $X$ is defined on a probability space $(\Omega, \Sigma, \P)$ (with expectation $\E[.]$) where $\Omega= \Omega_\T \times \Omega_\D \times \Omega_{\D^c} $, $\Sigma= \Sigma_\T \otimes \Sigma_\D \otimes \Sigma_{\D^c}$ and $\P$ is a product measure $\P=\P_\T\otimes \P_{\D}\otimes \P_{\D^c}$. At the level of random variables, the GFF decomposes as the sum of three independent variables
 \begin{equation}\label{decomposGFF}
 X= P\varphi+ X_\D+X_{\D^c}
 \end{equation} 
 where $P\varphi$ is the harmonic extension of the GFF restricted to the circle $\varphi=X|_\T$ defined on $ (\Omega_\T, \Sigma_\T,\P_\T)$ and $X_\D,X_{\D^c}$ are two independent GFFs on $\D$  and $\D^c$  with Dirichlet boundary conditions defined respectively on  the probability spaces $ (\Omega_\D, \Sigma_\D,\P_\D)$ and $ (\Omega_{\D^c}, \Sigma_{\D^c}, \P_{\D^c})$\footnote{With a slight abuse of notations, we will assume that these spaces are canonically embedded in the product space $(\Omega,\Sigma)$ and we will identify them with the respective images of the respective embeddings.}. We will write  $\E_\varphi[\cdot ]$ for conditional expectation with respect to the GFF on the circle $\varphi$ (instead of $:=\E[\cdot | \Sigma_\T]$).   We will view $X$ in two ways in what follows:  as a   process $X_t\in W^s(\T)$ ($s<0$)
\begin{align}\label{def:xt}
X_t(\theta)=X_\D(e^{-t+i\theta})1_{t>0}+X_{\D^c}(e^{-t+i\theta})1_{t<0}+(P\varphi)(e^{-|t|+i\theta}).
\end{align}
and as a random element in $\caD'(\hat\C)$.
In the sequel, we will denote for $t\geq 0$: 
\begin{equation}\label{def:phit}
\varphi_t(\theta):=P\varphi(e^{-t+i\theta}) +Y_t(\theta).
\end{equation}

\subsection{Reflection positivity}

 Let  $\D=\{|z|<1\}$ be the unit disk. Recall the definition of the Liouville field \eqref{liouvillefield} which is given on $\D$ by $\phi(z)=c+X(z)=c+X_\D(z)+P\varphi(z)$.  Let  $\caA_\D$ be the sigma-algebra on $\R\times\Omega$ generated by the maps $\phi\mapsto , \langle \phi,g\rangle _\D$ for $g\in C_0^{\infty}(\D)$ and we recall the notation  $\langle \phi,g\rangle _\D=\int_\D g(z)\phi(z)dz$. Let $\caF_\D$ be the set of 
 $\caA_\D$-measurable functions with values in $\R$.
 
 
 For $F,G\in \caF_\D$ such that the following quantities make sense (see below), we define  (recall \eqref{Thetadef})
\begin{equation}\label{osform}
(F,G)_\D:=\langle \Theta F(\phi) \overline{G(\phi)}\rangle_{\gamma,\mu}.
\end{equation}
Reflection positivity is the statement that this bilinear form is non-negative, namely $(F,F)_\D\geq 0$. In what follows, we will study this statement separately for the GFF theory ($\mu=0$) and for LCFT.

\subsubsection{Reflection positivity of the GFF}

Here we assume $\mu=0$. Let $F,G\in \caF_\D$ be nonnegative. The sesquilinear form \eqref{osform} becomes at $\mu=0$   
\begin{align}\label{oszero}
(F,G)_{\D,0}=\int_\R e^{-2Qc}\E  [ (\Theta F)(\phi)\overline{G(\phi)}] dc=\int_\R e^{-2Qc}\E [ F(c+X^{(2)})\overline{G(c+X^{(1)})}] \dd c
\end{align}
where we denoted  $X^{(i)}=X^{(i)}_\D+P\varphi$ with $X^{(1)}_\D= X_\D$ and $X^{(2)}_\D= X_{\D^c} \circ \theta$ which are two independent GFFs in the unit disk.  
Hence by independence of $X_\D^{(i)}$ 
\begin{align}\nonumber
(F,G)_{\D,0}&=\int_\R e^{-2Qc}\E [ \E_\varphi [ F(c+X^{(2)}_\D+P\varphi) ]  \overline{ \E_\varphi [ G(c+X^{(1)}_\D+P\varphi) ] } ]\dd c\nonumber
\\
&=\langle U_0F | U_0G\rangle_2 \label{oszero1}
\end{align}
where   the map $U_0$ is defined by 
 \begin{align}\label{uxerodef}
 (U_0F)(c,\varphi)=e^{-Qc}\E_\varphi [ F(c+X_\D+P\varphi)]
\end{align}
and we recall $\E_\varphi$ denotes expectation over $X_\D$.
Such a map is well defined on nonnegative $F\in  \caF_\D$ and extended to   $\caF_\D^{0,\infty}$, which is defined as the space of  $F\in  \caF_\D$ such that  $U_0|F|<\infty$   $dc\otimes\P_\T$-almost everywhere. Let  $\caF_\D^{0,2}=\{F\in \caF_\D^{0,\infty}\,|\, \|U_0F\|_2<\infty\}$. 

\begin{proposition}\label{OS} 
The sesquilinear form \eqref{oszero} extends to  $\caF_\D^{0,2}$. This extension is non negative
  \begin{align}\label{scalar10} 
\forall F\in\caF_\D^{0,2},\quad\quad (F,F)_{\D,0}\geq 0.
 \end{align}
 Let $\overline{\caF_\D^{0,2}/\caN_0^0}$ be the Hilbert space completion of the pre-Hilbert space $\caF_\D^{0,2}/\caN_0^0$ with $\caN_0^0:=\{F\in \caF_\D^{0,2}\,|\, (F,F)_{\D,0}=0\}$. The map $U_0$  
 in \eqref{uxerodef}  descends to a unitary map $U_0 :\overline{\caF_\D^{0,2}/\caN_0^0}\to L^2(\R \times\Omega_\T)$.

 \end{proposition} 
 
 \proof
By \eqref{oszero1} $U_0$ descends to an isometry on $\caF_\D^{0,2}/\caN_0^0$ so we need to show it is onto.  We take $F$ of the form 
\begin{align}\label{denseset}
F(c+X)=\rho(\langle c+X,g\rangle _\D)e^{\langle c+X,f\rangle _\D-\hf \langle f,G_\D f\rangle_\D}.
\end{align}
with  $\rho\in C_0^\infty(\R)$ and  $g,f\in C_0^{\infty}(\D)$ with  the further conditions that
 $g$ is rotation invariant i.e. $g(re^{i\theta})=g(r)$, and  that $\int_0^{2\pi} f(re^{i\theta})d\theta=0$ for all $r\in [0,1]$.
Then  
 $ \langle c,f \rangle _\D=0$ and $\langle P\varphi,g \rangle_\D=0$ and we get
\begin{align}\nonumber
(U_0 F)(c,\varphi)&=e^{-Qc}e^{ \langle P\varphi,f \rangle_\D}\E [ \rho(c+\langle X_\D,g\rangle_\D)e^{\langle X_\D,f\rangle_\D-\hf \langle f,G_\D f\rangle_\D} ]\\&=e^{-Qc}e^{\langle P\varphi,f\rangle_\D}\E [ \rho(c+\langle X_\D,g\rangle_\D)]\label{denseset1}
\end{align}
where we observed that $\langle X_\D,g\rangle_\D$ and $\langle X_\D,f\rangle_\D$ are independent as their covariance vanishes. Indeed, by rotation invariance of $g$,  the function $\mathcal{O}(r,\theta):= \int_\D g(x)  G_\D(x,r e^{i\theta})\dd x$  does not depend on $\theta$ hence 
\begin{align*}
\E[  \langle X_\D,g\rangle_\D  \langle X_\D,f\rangle_\D  ] & = \int_{\D} \int_{\D} G_\D(x,y) g(x) f(y) \dd y   \\
& = \int_{0}^1 r   \int_0^{2 \pi}  f(re^{i \theta} ) \mathcal{O}(r,\theta) \dd \theta  \dd r   \\
& = \int_{0}^1 r \mathcal{O}(r,0)  \int_0^{2 \pi}  f(re^{i \theta} ) \dd \theta   \dd r=0.
\end{align*}


Let  $h\in C^\infty(\T)$, $f_\epsilon\in C_0^{\infty}(\D)$ and $g_\epsilon$ be given by $g_\epsilon(re^{i\theta})=\epsilon^{-1}\eta(\frac{1-r}{\epsilon})$, 
$f_\epsilon= hg_\epsilon$ 
where $\eta$ is a smooth bump with support on $[1,2]$ and total mass one. Then 
$\lim_{\epsilon\to 0} \langle P\varphi,f_\epsilon \rangle_\D= \langle \varphi,h \rangle_\T$ and  $\lim_{\epsilon\to 0}\E (\langle X_\D,g_\epsilon\rangle_\D^2)=0$
so that 
\begin{align*}
\lim_{\epsilon\to 0}(U_0 F_\epsilon)(c,\varphi)=e^{-Qc}\rho(c)e^{ \langle \varphi,h \rangle_\T}
\end{align*}
where the convergence is in  $L^2(\R \times \Omega_\T)$. 
Thus the functions $e^{-Qc}\rho(c)e^{\langle \varphi,h \rangle _\T}
$ are in the image of $U_0$ for all $\rho\in   C_0^\infty(\R)$ and $h\in C^\infty(\T)$. Since the linear span of these is dense in $L^2(\R \times \Omega_\T)$ the claim follows. \qed

\begin{remark}\label{uoext}
Note that this argument shows that $U_0$ extends from $\caF_\D^{0,2}$ to functionals of form
$F(c+X_{|\T})$ and then  $$(U_0F)(c,\varphi)=e^{-Qc}F(c+\varphi).$$
\end{remark}

\subsubsection{Reflection positivity of LCFT}

Next we want to show reflection positivity for the LCFT expectation \eqref{FL1} with $\mu>0$. The GMC measure $ M_\gamma$ defined in \eqref{GMCintro} can also  be constructed as the martingale limit 
\begin{equation}
M_\gamma(\dd x)=\lim_{N\to\infty}e^{ \gamma X_N(x)-\tfrac{\gamma^2}{2}\E[X_N(x)^2]  }|x|_+^{-4}\, \dd x.
\end{equation}
where in $X_N$ we cut off the series \eqref{XDdef}  and \eqref{harmonic} defining $X_\D^{(i)}$ and $P\varphi$ respectively to finite number of terms $|n|\leq N$. We claim that
\begin{align*}
M_\gamma(\hat\C)=M^{(1)}_\gamma(\D)+ M^{(2)}_\gamma(\D)
\end{align*}
where $M^{(i)}_\gamma$ are the GMC measures of the fields $X^{(i)}=X^{(i)}_\D+P\varphi$, $i=1,2$.  Indeed, we take the limit $N\to\infty$ in
\begin{align*}
\int_{\D^c}e^{ \gamma X_N(x)-\tfrac{\gamma^2}{2}\E[X_N(x)^2]  }|x|^{-4}\,\dd x =\int_{\D^c}e^{ \gamma X^{(2)}_N(\frac{1}{\bar x})-\tfrac{\gamma^2}{2}\E[X^{(2)}_N(\frac{1}{\bar x})^2]  }|x|^{-4}\,\dd x=\int_{\D}e^{ \gamma X^{(2)}_N(x)-\tfrac{\gamma^2}{2}\E[X^{(2)}_N(x)^2]  }\,\dd x.
\end{align*}
Thus, for nonnegative $F,G\in \caF_\D$   
\begin{align*}
(F,G)_\D=\langle \Theta F\overline{ G}\rangle_{\gamma, \mu}=\langle U_0(IF)\, | \, U_0(IG)\rangle_2
\end{align*}
where 
\begin{align*}
I=e^{-\mu e^{\gamma c}M_\gamma(\D)}.
\end{align*}
Let  $ \caF_\D^\infty$  be the space of $F\in  \caF_\D$ such that  $U_0(|F|I)<\infty$   $dc\otimes\P_\T$-almost everywhere. Let  $\caF_\D^{2}=\{F\in \caF_\D^{\infty}\,|\, \|U_0(FI) \|_2<\infty\}$.   From the above considerations, we arrive at:

\begin{proposition}\label{OS1} 
The sesquilinear form \eqref{osform} extends to $\caF_\D^2$, is nonnegative and given by
  \begin{align}\label{scalar1} 
(F,G)_\D=\langle UF \,|\, UG \rangle _{2}
 \end{align}
 for all $F,G\in\caF_\D^2$ where 
 \begin{align}\label{udeff}
 (UF)(c,\varphi)=(U_0(FI))(c,\varphi)=e^{-Qc}\E_\varphi [ F(c+X)e^{-\mu e^{\gamma c}M_\gamma(\D)}], 
\end{align}
$X=X_\D+P\varphi$ and $M_\gamma$ is its GMC measure. Define $\caN_0:=\{F\in \caF_\D^2\,|\, (F,F)_\D=0\}$. Then $U$ descends to a unitary map 
 $$U :\caH_\D\to L^2(\R \times\Omega_\T)$$
 with $\caH_\D:=\overline{\caF_\D^2/\caN_0}$ (the completion with respect to $(.,.)_\D$). 
 \end{proposition} 

\proof
We need to show $U$ is onto. This follows from $U(I^{-1}F)=U_0F$ and the fact that $U_0$ is onto.\qed

\begin{remark}
From Remark \ref{uoext} we conclude  that $U$ extends from $\caF_\D$ to functionals $F(c+X_{|\T})$ for which 
$$(UF)(c,\varphi)=F(c+\varphi)\times (U1)(c+\varphi)
$$
or, in other words, for $f\in L^2(\R \times \Omega_\T)$
\begin{align}\label{Uinverse}
U^{-1}f=(U1)^{-1}f.
\end{align}
\end{remark}

\subsection{Dilation Semigroup}
Recall the action of the dilation map \eqref{dilationS} on $\caF_\D$.
The reason for the $Q \ln |q|$-factor is the  M\"obius invariance 
 property of LCFT \cite{DKRV}
\begin{proposition}\label{int:mobius} Let $\psi:\hat\C\to\hat\C$ be a M\"obius map and  let $F$ be a functional on $\caD'(\hat{\C})$ so that $\langle|F(\phi)|\rangle_{\gamma,\mu}<\infty$. Then 
\begin{align}
\langle F(\phi \circ \psi +Q \ln |\psi'|)\rangle_{\gamma,\mu}=\langle F(\phi)\rangle_{\gamma,\mu}.
\label{moobi}
\end{align}
\end{proposition}
 
We have then:

\begin{proposition}  \label{dilationsemi} The map $S_{q}$ descends to a contraction   $S_{q}:\caH_\D\to \caH_\D$:  
\begin{align}\label{conttra}
\forall F\in \caH_\D,\quad \quad (S_{q}F,S_{q}F)_\D\leq (F,F)_\D.
\end{align}
The adjoint  of $S_q$ is $S_q^\ast=S_{\bar q}$ i.e. for all $F,G\in \caH_\D$
\begin{align}
(S_q F, G)_\D=(F,S_{\bar q}G)_\D.
\label{adjointsq}
\end{align}
Finally the map $q\in\D\mapsto S_{q}$ is strongly continuous and satisfies  the group property
\begin{align}
S_{q}S_{q'}=S_{qq'}
\label{gropupprpo}
\end{align}
so that  $q\in\D\mapsto S_{q}$ is a strongly continuous contraction semigroup.
\end{proposition}

\begin{proof}

Let us start with \eqref{adjointsq}.
 It suffices to consider $F,G\in \caF^2_\D$ real.
 By definition
 \begin{align*}
    (S_q F, G)_\D =\langle  F( \phi\circ\theta\circ s_q+ Q \ln |q| -2Q\ell\circ s_q )   G( \phi) \rangle_{\gamma,\mu}:=\langle  \tilde F( \phi )   G( \phi) \rangle_{\gamma,\mu}
\end{align*}
 where $\ell(z):=\ln|z|$. Applying Proposition \ref{int:mobius} with $\psi=s_{\bar q}$ we get
  \begin{align*}
  \langle  \tilde F( \phi )   G( \phi) \rangle_{\gamma,\mu}= \langle  \tilde F( \phi \circ s_{\bar q}+Q\ln |q|)   G( \phi \circ s_{\bar q}+Q\ln |q|) \rangle_{\gamma,\mu}
\end{align*}
But 
  \begin{align*}
  \tilde F( \phi \circ s_{\bar q}+Q\ln |q|)  =F( \phi\circ s_{\bar q}\circ\theta\circ s_q+2Q\ln|q|-2Q\ell\circ s_q)=F(\phi\circ\theta-2Q\ell)
\end{align*}
and therefore $ \langle  \tilde F( \phi )   G( \phi) \rangle_{\gamma,\mu}= ( F, S_{\bar q}G)_\D$ as claimed.

The group property \eqref{gropupprpo} is obvious.

To prove the contraction, denote for  $F\in \caF_\D$,  the seminorm $\|F\|_\D:=(F,F)_\D^\hf$. Then we have 
 \begin{align*}
 \|S_{q}F\|_\D&=
 (S_{q}F,S_{q}F)_\D^\hf=(F,S_{|q|^{2}}F)_\D^\hf\leq \|F\|^\hf_\D\|S_{{|q|^{2}}}F\|^\hf_\D.
 \end{align*}
 Iterating this inequality we obtain
  \begin{align*}
 \|S_{q}F\|_\D\leq  \|F\|_\D^{1-2^{-k}}\|S_{|q|^{2k}}F\|_\D^{{2^{-k}}}.
 \end{align*}
Recall that
 \begin{align*}
 (G,G)_\D=
 \langle U_0(IG) | U_0(IG) \rangle _{2}=
\int_\R e^{-2Qc}\E[ \E_\varphi[ IG] ^2]dc
\end{align*}
 and then by Cauchy-Schwartz applied to $\E_\varphi[ . ]$
  \begin{align*}
\E [\E_\varphi [ IG ]^2 ]= \E[\E_\varphi[I^{\hf}G I^{\hf}]^{2}]\leq \E [ \E_\varphi [ IG^{2} ]\E_\varphi[ I ] ]
\end{align*}
 so that
  \begin{align*}
 (G,G)_\D\leq \langle U_{0}(IG^{2}) | U_{0}I\rangle_2=\langle U G^{2}| U 1\rangle_2 =\langle G^{2}\rangle_{\gamma,\mu}.
\end{align*}
Hence
 \begin{align*}
 \|S_{q}F\|_\D\leq  \|F\|_\D^{1-2^{-k}}\langle (S_{|q|^{2k}}F)^{2}\rangle_{\gamma,\mu}^{{2^{-k-1}}}= \|F\|_\D^{1-2^{-k}}\langle F^{2}\rangle_{\gamma,\mu}^{{2^{-k-1}}}
 \end{align*}
where we used  again the M\"obius invariance of $\langle \cdot\rangle_{\gamma,\mu}$.  Taking $k\to\infty$ we conclude $ \|S_{q}F\|_\D
 \leq \|F\|_\D$ for 
$F\in \caF_\D$ which satisfy $ \langle F^2\rangle_{\gamma,\mu}<\infty$. Such $F$ form a  dense set in $\caF_\D$. Indeed, let $F\in \caF_\D$ with $\|F\|_\D<\infty$ and let $F_R=F1_{|F|<R}$. Then $ \langle F_R^2\rangle_{\gamma,\mu}<\infty$ and
\begin{align*}
\|F-F_R\|_\D^2=\|F1_{|F|\geq R}\|_\D^2\leq \|F\|_\D^2\|1_{|F|\geq R}\|_\D^2
\end{align*}
and $\|1_{|F|\geq R}\|_\D^2=\langle 1_{F\geq R}\Theta 1_{|F|\geq R}\rangle\leq \frac{1}{R^2}\langle |F|\theta| F|\rangle\to 0$ as $R\to\infty$.

Hence \eqref{conttra} holds for all $F\in\caF_\D$ with $\|F\|_\D<\infty$. This implies $S_q$ maps the null space $\caN_0$ to $\caN_0$ and thus  $S_q$ extends to $\caH_\D$ so that \eqref{conttra} holds.

Finally to prove strong continuity, by the semigroup property it suffices to prove it at $q=1$ and by the contractive property we need to prove it only on a dense set.  Since
\begin{align*}
 \|S_{q}F-F\|_\D^{2}= \|S_{q}F\|_\D^{2}+ \|F\|_\D^{2}- (S_{q}F,F)_\D-(F,S_{q}F)_\D\leq 2\|F\|_\D^{2}- (S_{q}F,F)_\D-(F,S_{q}F)_\D
\end{align*}
it suffices to prove $(S_{q}F,F)_\D\to (F,F)_\D$ as $q\to 1$ on a dense set of $F$.
Take  $F=GI^{-1}$  so that $UF=U_0G$. Then 
\begin{align*}
(F,S_{q}F)_\D
=\int_\R e^{-2Qc}\E( \Theta GG
e^{-\mu e^{\gamma c}M_{\gamma}(\D\setminus|q|\D)})dc
\end{align*}
which converges as $q\to 1$ to $(F,F)_\D$ (use $\P (M_{\gamma}(\D\setminus|q|\D)>\epsilon)\to 0$ as $q\to 1$).  
\end{proof} 
In particular we can form two one-parameter (semi) groups from $S_q$. Taking $q=e^{-t}$ we define $T_{t}=S_{e^{-t}}$. Then $T_{t+s}=T_{t}T_{s}$ so $T_{t}$ is a strongly continuous contraction semigroup on the Hilbert space $\caH_\D$. Hence by the Hille-Yosida theorem
\begin{align}\label{hstar}
US_{e^{-t}}U^{-1}=e^{-t{\bf H_\ast}}
\end{align}
where the
 generator ${\bf H_\ast}$  (in the case $\mu=0$, we will write  ${\bf H}_\ast^0$) is a positive self-adjoint operator with domain $\caD({\bf H}_\ast)$ consisting of $\psi\in L^2(\R \times\Omega_\T)$ such that $\lim_{t\to 0}\frac{1}{t}(e^{-t{\bf H_\ast}}-1)\psi$ exists in $L^2(\R \times\Omega_\T)$. The operator $\bf{H_\ast}$ is the {\it Hamiltonian} of LCFT. Taking $q=e^{i\alpha}$ we get that $\alpha\mapsto S_{e^{i\alpha}}$ is a strongly continuous unitary group so that by Stone's theorem
$$
US_{e^{i\alpha}}U^{-1}=e^{i\alpha \Pi_\ast}
$$
where $\Pi_\ast$ is the self adjoint {\it momentum} operator of LCFT.  As we will have no use for $ \Pi_\ast$ in this paper we will concentrate on ${\bf H_\ast}$ from now on. Let us emphasize here that it is defined in the full range $\gamma \in (0,2)$. One of our next tasks will be to show that for $\gamma \in (0,2)$: ${\bf H_\ast}={\bf H}$, where the Hamiltonian ${\bf H}$ will be defined as the Friedrichs extension of \eqref{Hdef}.

\section{Gaussian Free Field: dynamics and CFT aspects} \label{sec:GFFCFT}
 
\subsection{Fock space and harmonic oscillators}\label{sec:fock}

The Hilbert space $L^2(\Omega_\T,\P_\T)$ (denoted from now on by  $L^2(\Omega_\T)$) has the structure of Fock space. Let $\caP\subset L^2(\Omega_\T)$ (resp. $ \mathcal{S}\subset L^2(\Omega_\T)$) be the linear span of the functions of the form $F(x_1,y_1, \cdots, x_N,y_N)$ for some $N \geq 1$ where $F$ is a polynomial on $\R^{2N}$ (resp.   $F\in C^\infty((\R^2)^N)$ with at most polynomial growth at infinity for $F$ and its derivatives). Obviously $\mathcal{P}\subset\mathcal{S}$ and they are both dense in $L^2(\Omega_\T)$.

On $\mathcal{S}$ we define the annihilation and  creation operators
\begin{align}\label{crea}
\mathbf{X}_n&=\partial_{x_n},\ \ \ \mathbf{X}_n^\ast=-\partial_{x_n}+x_n,\\
\mathbf{Y}_n&=\partial_{y_n},\ \ \ \mathbf{Y}_n^\ast=-\partial_{y_n}+y_n.
\end{align}
They are formally adjoint of each other   (see e.g. \cite[VIII.11]{rs1}  for more about the closure of these operators, which we will not need here)  and form a representation of the algebra of canonical commutation relations on $\mathcal{S}$:
\begin{align}\label{ccr}
[\mathbf{X}_n,\mathbf{X}_m^\ast]=\delta_{nm}=[\mathbf{Y}_n,\mathbf{Y}_m^\ast]
\end{align} 
with other commutators vanishing. The operator $ \mathbf{P}$ is then given on  $\mathcal{S}$ as
 \begin{align}\label{hdefi}
 \mathbf{P}=\sum_{n=1}^\infty n(\mathbf{X}_n^\ast \mathbf{X}_n+\mathbf{Y}_n^\ast \mathbf{Y}_n)
\end{align}
(only finite number of terms in the sum contributes when acting on $\mathcal{S})$ and extends uniquely to an unbounded 
self-adjoint positive operator on $L^2(\Omega_\T)$: this follows from the fact that we can find a complete system of eigenfunctions in $\caP$, as described now.  Let $\mathcal{N}$ be the set of  non-negative integer valued sequences with only a finite number of non null integers, namely ${\bf k}=(k_1,k_2,\dots)\in \mathcal{N}$ iff ${\bf k}\in \N^{\N_+}$ and $k_n=0$ for all $n$ large enough. For   $\bf k,\bf l  \in\mathcal{N}$ define the polynomials (here $1\in  L^2(\Omega_\T)$  is the constant function)
\begin{align}\label{fbasishermite}
\hat{\psi}_{{\bf k}{\bf l}}=\prod_n ( \mathbf{X}_n^\ast)^{k_n}( \mathbf{Y}_n^\ast)^{l_n}1 \in \caP.
\end{align}
Equivalently, $\hat{\psi}_{{\bf k}{\bf l}}= \prod_n {\rm He}_{k_n}(x_n) {\rm He}_{l_n}(y_n)$ where $({\rm He}_k)_{k \geq 0}$ are the standard Hermite polynomials. Then, using \eqref{ccr}, one checks that these are eigenstates of $\bf P$:
 \begin{align}\label{fbasis2}
 \mathbf{P}\hat{\psi}_{{\bf k}{\bf l}}=(|{\bf k}|+|{\bf l}|)
 \hat{\psi}_{{\bf k}{\bf l}}=\lambda_{{\bf k}{\bf l}}\hat{\psi}_{{\bf k}{\bf l}}.
\end{align}
where we use the notations
\begin{equation}\label{firstlength}
|{\bf k}|:=\sum_{n=1}^\infty nk_n,\quad \lambda_{{\bf k}{\bf l}}:=|{\bf k}|+|{\bf l}|
\end{equation}
 for ${\bf k},{\bf l}\in\caN$.
It is also well known that the family $\{\psi_{{\bf k}{\bf l}}=\hat{\psi}_{{\bf k}{\bf l}}/\|\hat{\psi}_{{\bf k}{\bf l}}\|_{L^2(\Omega_\T)}\}$ (where $\|\cdot \|_{L^2(\Omega_\T)}$ is the standard norm in $L^2(\Omega_\T)$) forms an orthonormal basis of $L^2(\Omega_\T)$. Finally we claim

\begin{proposition}\label{semigrouppt}
The operator $\mathbf{P}$ generates a strongly continuous semigroup of self-adjoint contractions $(e^{-t\mathbf{P}})_{t\geq 0}$ on $L^2(\Omega_\T) $  with  probabilistic representation, for $t\geq 0$, 
$$\forall f\in L^2(\Omega_\T),\quad e^{-t\mathbf{P}}f=\E_\varphi[f(\varphi_t)] $$
with $(\varphi_t)_{t\geq 0}$ the process defined by \eqref{def:phit}.
\end{proposition}

\begin{proof}
The fact that $\mathbf{P}$ generates a  strongly continuous semigroup of  self-adjoint contractions results from the fact that $\mathbf{P}$ is self-adjoint and nonnegative.  Since $(\psi_{{\bf k}{\bf l}})_{{\bf k},{\bf l}}$ form an orthonormal basis of $L^2(\Omega_\T)$, it suffices to study the semigroup on this basis. Obviously $ e^{-t\mathbf{P}}\psi_{{\bf k}{\bf l}}=e^{-\lambda_{{\bf k}{\bf l}}t}\psi_{{\bf k}{\bf l}}$. Furthermore the decomposition \eqref{def:xt} together with the covariance structure \eqref{ipp} (and recalling the decomposition \eqref{GFFcircle}+\eqref{varphin} of the field $\varphi$)  entails that   the law of $\varphi_t$  (see  \eqref{def:phit}) conditionally on $\varphi$ is given by
$$\varphi_t(\theta)= \sum_{n>0}\frac{x_n(t)+iy_n(t)}{2\sqrt{n}}e^{in\theta}+\sum_{n<0}\frac{x_{-n}(t)-iy_{-n}(t)}{2\sqrt{-n}}e^{in\theta}$$
where $x_n(t),y_n(t)$ are independent Ornstein-Uhlenbeck processes. In particular  for each fixed $t$, there are two independent sequences of independent standard Gaussians $(\bar{x}_n)_n$ and $(\bar{y}_n)_n$ such that,  for $n\geq 1$
$$x_n(t)\stackrel{\text{law cond. on }\varphi}{=}e^{-nt}x_n+\sqrt{1-e^{-2tn} }\bar{x}_n, \quad\quad y_n(t)\stackrel{\text{law cond. on }\varphi}{=}e^{-nt}y_n+\sqrt{1-e^{-2tn} }\bar{y}_n.$$
Finally, we recall the following elementary result:  given $Y$  a standard Gaussian random variable, the standard Hermite polynomials $({\rm He}_k)_{k \geq 0}$ on $\R$, $x\in\R$ and $u,v\geq 0$ such that $u^2+v^2=1$ then
\begin{equation}\label{hermiteOU}
\E[{\rm He}_k(ux+vY)]=u^k{\rm He}_k(x).
\end{equation}
Using this lemma and our description of the law of $X_t$, it is then plain to deduce that
$$e^{-t\mathbf{P}}\psi_{{\bf k}{\bf l}}=\E_\varphi[\psi_{{\bf k}{\bf l}}(\varphi_t)] =e^{-\lambda_{{\bf k}{\bf l}}t}\psi_{{\bf k}{\bf l}}.$$
Hence our claim.
\end{proof}

\begin{remark}\label{eigP}
List the eigenvalues $\la= {|\bf k}|+{|\bf l}|$ of $\bf P$ in increasing order $\lambda_1<\lambda_2<\dots$ and let $P_i$ be the corresponding spectral projectors. Since each $\lambda_i$ is of finite multiplicity and $\lambda_i\to\infty$ as $i\to\infty$ the semigroup $e^{-t\bf P}=\sum_ie^{-t\bf \lambda_i}P_i$ and the resolvent $(z-{\bf P})^{-1}=\sum_i(z-\lambda_i)^{-1}P_i$ are compact if $t>0$ and $\Im z\neq 0$ since they are norm convergent limits of finite rank operators.
\end{remark}

\subsection{Quadratic form}

Introduce the bilinear form (with associated quadratic form still denoted by $\mathcal{Q}_0$)  
\begin{equation}\label{defQ0}
\forall u,v\in \mathcal{C},\quad \mathcal{Q}_0(u,v):=\tfrac{1}{2}\E\int_{\R} \Big( \pl_c u \pl_c \bar{v}+Q^2u\bar{v}+ 2( \mathbf{P} u)\bar{v} \Big)\dd c
\end{equation}
with
\begin{equation}\label{core}
\mathcal{C}=\mathrm{Span}\{ \psi(c)F\,|\,\psi\in C_c^\infty(\R)\text{ and }F\in\mathcal{S} \}.
\end{equation}
We claim
\begin{proposition}\label{FQ0:GFF}
The quadratic form \eqref{defQ0} is closable (and we still denote its closure by $\mathcal{Q}_0$ with domain $\mathcal{D}(\mathcal{Q}_0)$) and lower semibounded: $\mathcal{Q}(u)\geq Q^2\|u\|_2^2/2$.
It determines uniquely a self-adjoint operator $\mathbf{H}^0 $, called the \emph{Friedrichs extension},  with domain denoted by $\mc{D}(\mathbf{H}^0)$ such that:
$$\mc{D}(\mathbf{H}^0 )=\{u\in \mathcal{D}(\mathcal{Q}_0)\, |\, \exists C>0,\forall v\in \mathcal{D}(\mathcal{Q}_0),\,\,\, |\mc{Q}_0(u,v)|\leq C\|v\|_2\}$$
and for $u\in \mc{D}(\mathbf{H}^0)$, $\mathbf{H}^0  u$ is the unique element in $L^2(\R\times \Omega_\T)$  satisfying
$$\mc{Q}_0(u,v)=\langle \mathbf{H}^0  u|v\rangle_2 .$$
\end{proposition}

\begin{proof} Recall that the closability of the quadratic form means that its completion with respect to the $ \mathcal{Q}_0$-norm embeds continuously and injectively in $L^2(\R\times \Omega_\T)$.
Its completion is the vector space consisting of equivalence classes of Cauchy sequences of 
$\mc{C}$ for the $\mc{Q}_0$-norm  under the equivalence relation 
$u\sim v$ iff $\mc{Q}_0(u_n-v_n)\to 0$ as $n\to \infty$. This space is a Hilbert space. Let us show that it   embeds injectively and continuously in $L^2(\R\times \Omega_\T)$ by the map $j: [u]\mapsto \lim_{n\to \infty}u_n$. Indeed, $u_n$ is Cauchy for $L^2(\R\times \Omega_\T)$ since $\|u_n-u_m\|^2_{2}\leq 2Q^{-2}\mc{Q}_0(u_n-v_n)$, it thus converges in $L^2(\R\times \Omega_\T)$. Moreover $\|\lim_{n}u_n\|_{2}^2\leq 2Q^{-2}\lim_{n}\mc{Q}_0(u_n)$ 
thus $j$ is bounded. Finally if $j([u])=0$, then for $(u_n)_n$ a representative Cauchy sequence of $[u]$, we have $u_n\to 0$ in $L^2(\R\times \Omega_\T)$ and using
\[ \frac{1}{2}\|\pl_c (u_n-u_m)\|_{2}^2+
\|{\bf P}^{1/2}(u_n-u_m)\|_{2}^2  \leq \mc{Q}_0(u_n-u_m,u_n-u_m),\] 
 one has the convergence in $L^2(\R\times \Omega_\T)$ of $\pl_cu_n\to v$ and ${\bf P}^{1/2}u_n\to w$ for some $v,w \in L^2(\R\times \Omega_\T)$. For each $\varphi \in \mc{C}$, we have as $n\to \infty$
\[ \cjg \pl_cu_n,\varphi\cjd_{2}=\cjg u_n,-\pl_c\varphi\cjd\to 0, \quad 
\cjg {\bf P}^{1/2}u_n, \varphi\cjd_2=\cjg u_n,{\bf P}^{1/2}\varphi\cjd_2\to 0,
\]
  thus $v=w=z=0$ by density of $\mc{C}$ in $L^2(\R\times \Omega_\T)$. This implies that $\mc{Q}_0(u_n)\to 0$ and thus $j$ is injective.
 
 Let us now consider the closure  $\mathcal{Q}_0$ with domain $\mathcal{D}(\mathcal{Q}_0)$. Obviously it is closed  and lower semi-bounded $\mathcal{Q}_0(u)\geq Q^2\|u\|_2^2/2$ so that the construction of the Friedrichs extension then follows from \cite[Theorem 8.15]{rs1}. 
\end{proof}

If we let $\mc{D}(\mc{Q}_0)'$ be the dual to $\mc{D}(\mc{Q}_0)$ (i.e. the space of bounded conjugate linear functionals on $\mc{D}(\mc{Q}_0)$), the injection $L^2(\R\times \Omega_\T)\subset \mc{D}(\mc{Q}_0)'$ is continuous and the operator ${\bf H}^0$ can be extended as a bounded isomorphism 
\[{\bf H}^0:\mc{D}(\mc{Q}_0)\to \mc{D}(\mc{Q}_0)'.\] 
We alsohave $\mc{D}({\bf H}^0)=\{ u\in\mc{D}(\mc{Q}_0)\,|\, {\bf H}^0u\in L^2(\R\times \Omega_\T)\}$ and $({\bf H}^0)^{-1}:L^2(\R\times \Omega_\T)\to \mc{D}({\bf H}^0)$ is bounded. Furthermore, by the spectral theorem, it generates a strongly continuous contraction semigroup of self-adjoint operators $(e^{-t \mathbf{H}^0 } )_{t\geq 0}$ on $L^2(\R\times\Omega_\T)$.

\subsection{Dynamics of the GFF}

The goal of this subsection is to prove the relation  ${\bf H}_\ast^0={\bf H}^0$,  i.e. we want to show
\begin{proposition} 
 For all $f\in L^2(\R \times\Omega_\T)$ and all $t\geq 0$
\begin{align}\label{u00identity}
U_0S_{e^{-t}}U_0^{-1}f=e^{-t{\bf H}^0}f=e^{-\frac{Q^2t}{2}}\E_\varphi[ f (c+B_t, \varphi_t)]
\end{align}
\end{proposition} 
\begin{proof}
Recalling \eqref{XDdef}, we have the independent sum
\begin{align*}
X_\D(e^{-t+i\theta})=B_{t}+Y_t(\theta)
\end{align*}
where $B_t$ is a Brownian motion and $Y_t$  has zero average on the circle.  We then have 
\begin{align*}  
(U_0S_{e^{-t}}U_0^{-1}f)(c,\varphi)=&e^{-Qc}\E_\varphi [e^{Q(c+B_t-Qt)}f(c+B_t-Qt,P\varphi(e^{-t+i\cdot})+Y_t(\cdot))]\\
 =&e^{-\frac{Q^2}{2}t}\E_\varphi [ f(c+B_t,P\varphi(e^{-t+i\cdot})+Y_t(\cdot))]
\end{align*}
where we have used the Girsanov transform to obtain the last equality.
Since $B_t$ and $Y_t$ are independent conditionally on $\varphi$, this last quantity is also equal to  $e^{-t(\frac{Q^2}{2}-\frac{\partial^2_c}{2})}e^{-t\mathbf{P}}f$ by using Proposition \ref{semigrouppt}.
Furthermore, for $f\in \mathcal{C}$, it is plain to see that the mapping $t\mapsto e^{-t(\frac{Q^2}{2}-\frac{\partial^2_c}{2})}e^{-t\mathbf{P}}f$ solves the Cauchy problem $\partial_tu=-\mathbf{H}^0u$ with $u(0)=f$. Hence $e^{-t(\frac{Q^2}{2}-\frac{\partial^2_c}{2})}e^{-t\mathbf{P}}f=e^{-t\mathbf{H}^0}f$.
\end{proof}

Finally, we have the simple:

\begin{proposition}\label{l2alpha} The following properties hold:
\begin{enumerate}
\item The measure $\dd c\times\P_\T$ is invariant for $e^{\frac{Q^2t}{2}}e^{-t{\bf H}^{0}}$.
\item $e^{-t{\bf H}^{0}}$ extends to a  continuous semigroup on $ L^p(\R \times \Omega_\T)$ for all $p\in [1,+\infty]$ with norm $e^{ -\frac{Q^2}{2}t}$ and it is strongly continuous for $p\in [1,+\infty)$.
\item
$e^{-t{\bf H}^{0}}$ extends to a strongly continuous semigroup on $e^{-\alpha c} L^2(\R \times \Omega_\T)$ for all $\alpha\in\R$ with norm $e^{(\frac{\alpha^2}{2}-\frac{Q^2}{2})t}$.
\end{enumerate}
\end{proposition}
\begin{proof} 1) This is a consequence of \eqref{u00identity}: indeed the processes $B$ and $Y$ are independent and describe two dynamics for which the measures $\dd c$ and $\P$ are respectively invariant.
2) follows from \eqref{u00identity}, Jensen's inequality and the fact that $\dd c\otimes \P_\T$ is invariant for ${\bf H}^{0}$. 3)
The map $K:f\mapsto e^{ -\alpha c}f$ is unitary from  $L^2(\R \times \Omega_\T )\to e^{- \alpha  c} L^2(\R \times \Omega_\T)$. We have $Ke^{-t{\bf H}^{0}}K^{-1}=e^{t(\frac{\alpha^2}{2}- \alpha \partial_c)}e^{-t{\bf H}^{0}}$ which implies the claim.
\end{proof}

\begin{remark}\label{discreteH0}
Using the decomposition $L^2( \Omega_\T)=\bigoplus_{{\bf k,l}}\ker 
({\bf P}-\lambda_{{\bf k,l}})$, the operator ${\bf H}^0$ is unitarily equivalent to the 
direct sum $\bigoplus_{{\bf k,l}} (-\frac{1}{2}\pl_c^2+\frac{Q^2}{2} +\lambda_{{\bf k,l}})$, 
each of these operators being a shifted Laplacian on the real line $\R$. Consequently (using Fourier transform in $c$), 
${\bf H}^0$ has no $L^2$-eigenvalue, its spectrum is absolutely continuous and the family $(e^{iPc}\psi_{\bf k,l})_{P,{\bf k,l}}$ form a complete family of generalized eigenstates diagonalizing ${\bf H}^0$. 
\end{remark}

\subsection{Diagonalization of the free Hamiltonian using the Virasoro algebra}\label{repth}

We start by explaining the diagonalization of the free (i.e. non interacting) Hamiltonian $\mathbf{H}^0$ which corresponds to the case $\mu=0$ in \eqref{Hdef}. As explained in Remark \ref{discreteH0}, that can be done directly by using the orthonormal basis of Hermite polynomials $\psi_{{\bf kl}}$ of $L^2(\Omega_\T)$ combined  with the Plancherel formula for the Fourier transform on the real line: for each $u_1,u_2\in L^2(\R\times \Omega_\T)$, one has  
\begin{equation}\label{<F,G>usingH0}
\cjg u_1\,|\, u_2\cjd_2 =\frac{1}{2\pi} \sum_{{\bf k},{\bf l}\in \mc{N}}\int_\R  \cjg u_1\,|\, e^{iPc}\psi_{{\bf kl}}\cjd_2  \cjg  e^{iPc}\psi_{{\bf kl}}\,| \,u_2\cjd_2 \,\dd P.
\end{equation}
It will be useful however to use another basis for $L^2(\Omega_\T)$ which respects its  underlying complex analytic structure; this new basis, made up of $\mathbf{H}^0$-eigenstates, will be generated by the action on $L^2(\R\times \Omega_\T)$  of two commuting unitary representations of the Virasoro algebra (as motivated in the end of Subsection \ref{outline:subspectral}).  We follow below the Segal-Sugawara construction for the Fock representation of the Heisenberg algebra. Let us emphasize that the material we introduce  here is standard;  to keep the paper self-contained, we recall the main properties of the construction and just give sketches of the proofs  (see for instance \cite{gordon,kac} for more details). 

 \subsubsection{Fock representation of the Heisenberg algebra}
 
 We will work on the vector space 
 \begin{equation}\label{smoothexpgrowth}
 \mathcal{C}_\infty:=\mathrm{Span}\{ \psi(c)F\,|\,\psi\in C^\infty(\R)\text{ and }F\in\mathcal{S} \}.
 \end{equation}
(not to be confused with $\mathcal{C}$ which is a subset of $\mathcal{C}_\infty$) and use the  complex coordinates \eqref{varphin}, i.e. we denote for $n>0$
\begin{align*}
\partial_n:=\frac{\partial}{\partial\varphi_{n}}= \sqrt{n} (\partial_{x_n}-i \partial_{y_n}) \quad \text{ and }\quad \partial_{-n}:=\frac{\partial}{\partial\varphi_{-n}}= \sqrt{n} (\partial_{x_n}+i \partial_{y_n}).
\end{align*}
We define 
on $ \mathcal{C}_\infty$  the following operators for $n>0$: 
\begin{align*}
 \mathbf{A}_n&= \tfrac{i}{2}\partial_{n},\ \ \  \mathbf{A}_{-n}=\tfrac{i}{2}(\partial_{-n}-2n\varphi_{n})\\
\widetilde{\mathbf{A}}_n&= \tfrac{i}{2}\partial_{-n},\ \ \ \widetilde{\mathbf{A}}_{-n}=\tfrac{i}{2}(\partial_{n}-2n\varphi_{-n})\\
\mathbf{A}_0&=\widetilde{\mathbf{A}}_0=\tfrac{i}{2}(\partial_c+Q).
\end{align*}
Their restrictions to $\mathcal{C}$ are closable operators in  $L^2(\R\times \Omega_\T)$  satisfying (on their closed extension)
\begin{align}
 \mathbf{A}_n^\ast= \mathbf{A}_{-n},\ \ \ \widetilde{\mathbf{A}}_n^\ast=\widetilde{\mathbf{A}}_{-n}.
\label{adjoan}
\end{align}
 Furthermore 
 $ \mathbf{A}_n1=0$ and $\widetilde{ \mathbf{A}}_n1=0$ for  $n>0$. It is easy to see that the space $ \mathcal{C}_\infty$ is stable by the operators $\mathbf{A}_n,\widetilde{\mathbf{A}}_n$  and we have the commutation relations  on $ \mathcal{C}_\infty$
\begin{align}
[\mathbf{A}_n,
 \mathbf{A}_{m}]=\frac{_n}{^2}\delta_{n,-m}=[\widetilde{\mathbf{A}}_n,
\widetilde{\mathbf{A}}_{m}],\ \ 
[ \mathbf{A}_n,\widetilde{ \mathbf{A}}_m]=0.\label{com  mu}
\end{align}
Thus $ \mathbf{A}_n$ and $\widetilde{ \mathbf{A}}_n$  ($n>0$) are annihilation operators   and $  \mathbf{A}_{-n},\widetilde{ \mathbf{A}}_{-n}$ creation operators.  By identifying canonically  $\mathcal{S}$ (defined in the beginning of subsection \ref{sec:fock}) as a subspace of $ \mathcal{C}_\infty$, it is plain to check that   $\mathcal{S}$ is stable under the action of these operators. As before, let ${\bf k}, {\bf l} \in \mathcal{N}$ and define
the polynomials 
\begin{align}\label{fbasis}
\hat \pi_{{\bf k}{\bf l}}=
\prod_{n>0}  \mathbf{A}_{-n}^{k_n} \tilde{\mathbf{A}}_{-n}^{l_n}1.
\end{align}
Then $\hat \pi_{{\bf k}{\bf l}}$ and $\hat \pi_{{\bf k'}{\bf l'}}$ 
are orthogonal  if ${\bf k}\neq {\bf k'}$  or ${\bf l}\neq {\bf l'}$ and the $\hat \pi_{{\bf k}{\bf l}}$'s with $|{\bf k}|+|{\bf l}|=N$ 
span the eigenspace of ${\bf P}$ in $L^2(\Omega_\T)$ with eigenvalue $N$\footnote{Explicitly: $\hat \pi_{{\bf k}{\bf l}}=\prod_{n>0} (-in)^{k_n+l_n}  \varphi_{n}^{k_n}   \varphi_{-n}^{l_n}   +P(\varphi_{n},\varphi_{-n})$ where $P$ is a polynomial in $\varphi_{n},\varphi_{-n}$ spanned by monomials of the form $\prod_{n>0}   \varphi_{n}^{k'_n}   \varphi_{-n}^{l'_n} $ with $k'_n\leq k_n$, $l'_n,\leq l_n$  and  $\sum_{n>0} k'_n+l'_n<\sum_{n>0} k_n+l_n$. }. We denote $ \pi_{{\bf k}{\bf l}}:= \hat \pi_{{\bf k}{\bf l}}/ \|\hat \pi_{{\bf k}{\bf l}}\|_{L^2(\Omega_\T)} $ the normalized eigenvectors.


%

 \subsubsection{Segal-Sugawara construction}
Now we use the Fock representation of the Heisenberg algebra to construct the Virasoro representation.
We define the {\it normal ordered product} on $ \mathcal{C}_\infty$ by
 $:\!\mathbf{A}_n\mathbf{A}_m\!\!:\,=\mathbf{A}_n\mathbf{A}_m$ if $m>0$ and $\mathbf{A}_m\mathbf{A}_n$ if $n>0$ (i.e. annihilation operators are on the right) and then for all $n \in \Z$
\begin{align}\label{virassoro}
\mathbf{L}_n^0:=-i(n+1)Q\mathbf{A}_n+\sum_{m\in\Z}:\mathbf{A}_{n-m}\mathbf{A}_m:
\end{align}
\begin{align}\label{virassorotilde}
\widetilde{\mathbf{L}}_n^0:=-i(n+1)Q\widetilde{\mathbf{A}}_n+\sum_{m\in\Z}:\widetilde{\mathbf{A}}_{n-m}\widetilde{\mathbf{A}}_m:\,\,.
\end{align}
These operators are well defined on $ \mathcal{C}_\infty$ (since only a finite number of terms contribute) and their restrictions to $\mathcal{C}$ are closable operators satisfying (on their closed extensions) 
\begin{align}
(\mathbf{L}_n^0)^\ast=\mathbf{L}^0_{-n},\ \ \ (\widetilde{\mathbf{L}}_n^0)^\ast=\widetilde{\mathbf{L}}^0_{-n}.
\label{adjo}
\end{align}
Furthermore the vector space $ \mathcal{C}_\infty$ is stable under  $\mathbf{L}_n^0$ and $\widetilde{\mathbf{L}}_n^0$ for all $n \in \Z$; on $ \mathcal{C}_\infty$ the $\mathbf{L}_n^0$ satisfy the commutation relations of the {\it Virasoro Algebra} (see \cite[Prop 2.3]{kac}): 
\begin{align}\label{virasoro}
[\mathbf{L}_n^0,\mathbf{L}_m^0]=(n-m)\mathbf{L}_{n+m}^0+\frac{c_L}{12}(n^3-n)\delta_{n,-m}
\end{align}
where the central charge is
\begin{align*}
c_L=1+6Q^2.
\end{align*}
These commutation relations can be checked by using the fact that, on $ \mathcal{C}_\infty$, only finitely many terms contribute in \eqref{virasoro} and using the commutation relation \eqref{com mu}.
 $\widetilde{\mathbf{L}}_n^0$ satisfy the same commutation relations \eqref{virasoro} and commute with the $ \mathbf{L}_n^0$'s. Note also that 
 \begin{align}\label{L0def}
\mathbf{L}_{0}^0&=\tfrac{1}{4}(-\pl_c^2+Q^2)+2\sum_{n>0}\mathbf{A}_{-n}\mathbf{A}_n\\
\widetilde{\mathbf{L}}_{0}^0&=\tfrac{1}{4}(-\pl_c^2+Q^2)+2\sum_{n>0}\widetilde{\mathbf{A}}_{-n}\widetilde{\mathbf{A}}_n,
\label{L0def1}
\end{align}
so that one can easily check that the $\mu=0$ Hamiltonian $\mathbf{H}^0:=  -\frac{1}{2}\pl_c^2 + \frac{1}{2} Q^2+ \mathbf{P}$ has the following decomposition when restricted on $\mathcal{C}_\infty$  
$$\mathbf{H}^0=\mathbf{L}_{0}^0+
\widetilde{\mathbf{L}}_{0}^0.
$$ 
\begin{remark} In the terminology of representation theory, we have a {\it unitary representation} of two commuting  Virasoro Algebras on $L^2(\R \times \Omega_\T)$ (unitary in the sense that \eqref{adjo} holds) and this representation is reducible as we will see below by constructing stable sub-representations.
\end{remark}

\subsubsection{Diagonalizing $\mathbf{H}^0$ using the Virasoro representation }\label{sectionDiagvirasoro}
Now we explain how to construct the generalized eigenstates of the free Hamiltonian $\mathbf{H}^0$ using the families of operators $(\mathbf{L}_n^0)_n$ and $(\widetilde{\mathbf{L}}_n^0)_n$. 
Recall that, for $\alpha\in \C$, we have defined the function  
\begin{align}\label{psialphadef}
\Psi^0_\alpha(c,\varphi):=e^{(\alpha-Q)c}\in \mathcal{C}_\infty.
\end{align}
For $\alpha\in \C$, these are generalized eigenstates of ${\bf H}^0$: they never belong to $L^2(\R \times \Omega_\T)$ but rather to some weighted spaces $e^{\beta |c|}L^2(\R\times \Omega_\T)$ for $\beta>|{\rm Re}(\alpha)-Q|$, hence their name ``generalized eigenstates". We have
\begin{equation}\label{L0psialpha}
\begin{split}
\mathbf{L}_0^0\Psi^0_\alpha&=\widetilde{\mathbf{L}}_0^0\Psi^0_\alpha=\Delta_{\alpha}
\Psi^0_\alpha\\
\mathbf{L}_n^0\Psi^0_\alpha&=\widetilde{\mathbf{L}}_n^0\Psi^0_\alpha=0,\ \ \ n>0,
\end{split}
\end{equation}
where  $\Delta_\alpha$ is the conformal weight \eqref{deltaalphadef}.   In the language of representation theory (or in the CFT terminology),  $\Psi^0_\alpha$ is called {\it highest weight state} with highest weight $\Delta_{\alpha}$ for both algebras. Before defining the so-called descendants of $\Psi^0_\alpha$, we introduce the following definition:  
%
%
  
\begin{definition}\label{young}
A sequence of integers    $\nu= (\nu_i)_{i \geq 0}$ is called a Young diagram if the  mapping $i\mapsto\nu_i $ is non-increasing and if $\nu_i=0$ for $i$ sufficiently large. 
We denote by $\mathcal{T}$ the set of all Young diagrams. We will sometimes write $\nu=(\nu_i)_{i \in\llbracket 1,k\rrbracket}$ where $k$ is the last integer i such that $\nu_i>0$ and denote by $|\nu|:= \sum_{i \geq 1} \nu_i$ the length of the Young diagram\footnote{This length should not be confused with the length \eqref{firstlength} of a sequence of integers.}.  We set $\mc{T}_j:=\{\nu\in \mc{T}\, |\, |\nu|=j\}$ the set of Young diagrams of length $j$.
\end{definition}

Given two Young diagrams $\nu= (\nu_i)_{i \in [1,k]}$ and $\tilde{\nu}= (\tilde{\nu}_i)_{i \in [1,j]}$ we denote  
\begin{equation*}
\mathbf{L}_{-\nu}^0=\mathbf{L}_{-\nu_k}^0 \cdots \, \mathbf{L}_{-\nu_1}^0, \quad \quad \quad   \tilde{\mathbf{L}}_{-\tilde \nu}^0=\tilde{\mathbf{L}}_{-\tilde\nu_j}^0 \cdots\, \tilde{\mathbf{L}}_{-\tilde\nu_1}^0
\end{equation*}
and define
\begin{align}\label{psibasis}
\Psi^0_{\alpha,\nu, \tilde\nu}=\mathbf{L}_{-\nu}^0\tilde{\mathbf{L}}_{-\tilde\nu}^0  \: \Psi^0_\alpha,
\end{align}
with the convention that $\Psi^0_{\alpha,\emptyset, \emptyset}=\Psi^0_{\alpha}$.
The vectors $\Psi^0_{\alpha,\nu, \tilde\nu}$ are called the \emph{descendants} states of $\Psi^0_\alpha$. We gather in the following proposition their main properties

 \begin{proposition}\label{prop:mainvir0} The following holds:\\
1) For each pair of Young diagrams $\nu,\tilde{\nu}\in \mathcal{T}$,  the  descendant state $\Psi^0_{\alpha,\nu, \tilde\nu}$ can be written as   
\begin{align}\label{psibasis1}
\Psi^0_{\alpha,\nu, \tilde\nu}=\mathcal{Q}_{\alpha,\nu,\tilde \nu}\Psi^0_\alpha
\end{align}
where $\mathcal{Q}_{\alpha,\nu,\tilde\nu}\in \mathcal{P} $   is a polynomial.\\
2)  for all $\alpha \in \C$
\begin{equation*}
\mathbf{L}_0^0\Psi^0_{\alpha,\nu, \tilde\nu} = (\Delta_{\alpha}
+|\nu|)\Psi^0_{\alpha,\nu ,\tilde\nu},\ \ \  \tilde{\mathbf{L}}_0^0\Psi^0_{\alpha,\nu, \tilde\nu} = (\Delta_{\alpha}+|\tilde\nu|)\Psi^0_{\alpha,\nu ,\tilde\nu}
\end{equation*}
and thus since $\mathbf{H}^0=\mathbf{L}_0^0+\tilde{\mathbf{L}}_0^0$
\begin{equation*}
\mathbf{H}^0\Psi^0_{\alpha,\nu,\tilde{\nu}} = (2 \Delta_\alpha+|\nu|+|\tilde{\nu}| )\Psi^0_{\alpha,\nu,\tilde{\nu}}  .
\end{equation*}
3) {\bf Completeness}: the inner products of the descendant states obey
\begin{equation}\label{scapo}  
 \langle \mathcal{Q}_{2Q-\bar\alpha,\nu,\tilde\nu} | \mathcal{Q}_{\alpha,\nu',\tilde\nu'} \rangle_{L^2(\Omega_\T)}=\delta_{|\nu| ,|\nu'|}\delta_{|\tilde\nu| ,|\tilde\nu'|}F_{\alpha}(\nu,\nu')F_{\alpha}(\tilde\nu,\tilde\nu')
\end{equation} 
where each coefficient   $F_{\alpha}(\nu,\nu')$  is a  polynomial in  ${\alpha}$, called the {\it Schapovalov form}.  The functions $(\mathcal{Q}_{\alpha,\nu,\tilde \nu})_{\nu,\tilde\nu\in\mathcal{T}}$ are linearly independent for $$\alpha\nin \{{\alpha_{r,s}} \mid \,\,r,s\in \N^\ast ,rs \leq \max(|\nu|, |\tilde \nu |)\}\quad \quad\text{with }\quad {\alpha_{r,s}}=Q-r\frac{\gamma}{2}-s\frac{2}{\gamma}.$$
4)  the following holds
 \begin{equation}\label{defqalpha}
\begin{gathered}
\mathcal{Q}_{\alpha,\nu,\tilde{\nu}}: =   \sum_{\k,\l, |{\bf k}|+|{\bf l}|=N}M^{N}_{\alpha,\k\l,\nu\tilde\nu} \psi_{\k\l},
\end{gathered}
\end{equation}  
for some coefficients $M^{N}_{\alpha,\k\l,\nu\tilde\nu}$ polynomial in $\alpha\in \C$.\\
5) {\bf Spectral decomposition}: if $u_1,u_1\in L^2(\R\times\Omega_\T)$ then
\begin{align}\label{fcomplete}
\langle u_1\, |\,  u_2\rangle_{2}=\frac{1}{2 \pi}\sum_{\nu,\tilde\nu,\nu',\tilde\nu'\in \mc{T}}\int_\R \langle u_1\,|\,
\Psi^0_{Q+iP,\nu',\tilde{\nu}'} \rangle_{2} \langle \Psi^0_{Q+iP,\nu,\tilde{\nu}}\, |\, u_2\rangle _{2}F^{-1}_{P+iQ}(\nu,\nu')F^{-1}_{P+iQ}(\tilde\nu,\tilde\nu')\, \dd P.
\end{align}

   \end{proposition}

\begin{proof}
The decomposition \eqref{psibasis1} can be obtained from the definition \eqref{psibasis} by using   that $[{\bf A}_n,e^{(\alpha-Q)c}]=[\tilde{{\bf A}}_n,e^{(\alpha-Q)c}]=0$ for all $n\not=0$, $[\pl_c,e^{(\alpha-Q)c}]=e^{(\alpha-Q)c}(\alpha-Q){\rm Id}$ and that finitely many applications of ${\bf A}_{n}$ and $\tilde{\bf A}_n$ to ${\bf 1}$ is a polynomial in $\mc{P}$.  
Actually, one can even show  that  
for $\nu$ a Young diagram
\begin{align*}
\mathbf{L}_{-\nu}^0 \Psi^0_\alpha=(\mathbf{L}^{0,\alpha}_{-\nu} 1)\Psi^0_\alpha
\end{align*}
where, for a Young diagram $\nu=( \nu_1, \cdots, \nu_k )$, we define $\mathbf{L}^{0,\alpha}_{\nu}:=\mathbf{L}^{0,\alpha}_{\nu_1}\dots \mathbf{L}^{0,\alpha}_{\nu_k}$ and the operators $\mathbf{L}^{0,\alpha}_n$, $n\in\Z$ act in $L^2(\Omega_\T)$ and are given by the expression \eqref{virassoro} where we replace ${\bf A}_0$ by $
\frac{i}{2}(\alpha+Q)$ i.e.
\begin{equation*}
\mathbf{L}^{0,\alpha}_n:= \left\{
\begin{array}{ll} i(\alpha-Q-nQ)\mathbf{A}_n+\sum_{m\neq n,0} \mathbf{A}_{n-m}\mathbf{A}_m
& \, n\neq 0\\
\frac{\alpha}{2}(Q-\frac{\alpha}{2})+2\sum_{m>0} \mathbf{A}_{-m}\mathbf{A}_m
&\, n=0
 \end{array} \right..
\end{equation*}
Hence on their closed extensions we have
\begin{equation}\label{adjoint}
(\mathbf{L}^{0,\alpha}_n)^\ast=\mathbf{L}^{0,2Q-\bar\alpha}_{-n}
\end{equation}
and $\mathbf{L}^{0,\alpha}_n$ satisfies \eqref{virasoro}.  The reader should notice here that the order in which the $\nu_j$ appear in previous definition ensures that $(\mathbf{L}^{0,2Q-\bar\alpha}_{\nu})^\ast=\mathbf{L}^{0,\alpha}_{-\nu}$. The operators $(\tilde{\mathbf{L}}^{0,\alpha}_{n})_n$ are defined in a similar fashion and commute with $(\mathbf{L}^{0,\alpha}_{n})_n$. Therefore $ \mathcal{Q}_{\alpha,\nu,\tilde\nu} =\mathbf{L}^{0,\alpha}_{-\nu}\tilde{\mathbf{L}}^{0,\alpha}_{-\tilde\nu}1$. Then 2) results from   \eqref{L0psialpha} and the commutation relations \eqref{virasoro} with $n=0$.

Since $\mathbf{L}^{0,\alpha}_\nu$ and $\tilde {\mathbf{L}}^{0,\alpha}_{\tilde{\nu}}$ commute we have
\begin{align}\label{QQ}
 \langle \mathcal{Q}_{2Q-\bar\alpha,\nu,\tilde\nu}\, |\, \mathcal{Q}_{\alpha,\nu',\tilde\nu'} \rangle_{L^2(\Omega_\T)}=\langle 1\,|\, \tilde{\mathbf{L}}^{0,\alpha}_{\tilde\nu}\tilde{\mathbf{L}}^{0,\alpha}_{-\tilde\nu'}\mathbf{L}^{0,\alpha}_{\nu}\mathbf{L}^{0,\alpha}_{-\nu'}1\rangle_{L^2(\Omega_\T)}.
\end{align}
Let us now compute the right-hand side. By Lemma \ref{LemAppendixVir}, one has for arbitrary $t_1, \dots, t_k \in \Z$ such that $t_1+ \cdots +t_k>0$  
\begin{equation}\label{sumpositive}
\mathbf{L}^{0,\alpha}_{t_1} \cdots \mathbf{L}_{t_k}^{0,\alpha}1=0, \quad \tilde{\mathbf{L}}^{0,\alpha}_{t_1} \cdots \tilde{\mathbf{L}}_{t_k}^{0,\alpha}1=0.
\end{equation}
 Therefore, if $|\nu| > |\nu'|$ or  $|\tilde{\nu}| > |\tilde{\nu}'|$ then we get that \eqref{QQ} is equal to $0$ by using \eqref{sumpositive}. The case  $|\nu| < |\nu'|$ or  $|\tilde{\nu}| < |\tilde{\nu}'|$ can be dealt similarly and yields $0$ also. Hence in what follows we suppose that $|\nu| = |\nu'|$ and $|\tilde \nu| = |\tilde{\nu}'|$.

Lemma \ref{LemAppendixVir1} establishes that for  $|\nu|=|\nu'|$ 
 \begin{equation}\label{decompL0}
 \mathbf{L}^{0,\alpha}_{\nu}{\bf L}^{0,\alpha}_{-\nu'}1=\sum_{k\geq 0}a_k(\mathbf{L}^{0,\alpha}_0)^k1
  \end{equation}
 where the coefficients $a_k$ are   determined by the algebra \eqref{virasoro} and are independent of $\alpha$. Since $(\mathbf{L}^{0,\alpha}_0)^k1=\Delta_{\alpha}^k$ and $\Delta_\alpha=\frac{\alpha}{2}(Q-\frac{\alpha}{2})$ we conclude 
 $$ 
 \mathbf{L}^{0,\alpha}_{\nu}{\bf L}^{0,\alpha}_{-\nu'}1=F_{\alpha}(\nu,\nu')
 $$ 
 where $F_{\alpha}(\nu,\nu')$  is a  polynomial in  $\Delta_\alpha$ and thus  in $\alpha$. 
 Repeating the argument for $\tilde{\bf L}^{0,\alpha}_{\tilde\nu}\tilde{\bf L}^{0,\alpha}_{-\tilde\nu'} 1$ yields \eqref{scapo}.
 
The determinant of the matrix  $(F_{\alpha}(\nu,\nu'))_{|\nu|=|\nu'|=N}$  is given by  the Kac determinant formula (see Feigin-Fuchs \cite{FF})
\begin{align}\label{}
\det (F_{\alpha}(\nu,\nu'))_{|\nu|=|\nu'|=N}=\kappa_N\prod_{r,s=1; \: rs \leq N}^N(\Delta_\alpha-\Delta_{\alpha_{r,s}})^{p(N-rs)}
\end{align}
where $\kappa_N$ does not depend   on $\alpha$ or $c_L$,  $p(M)$ is the number of Young Diagrams of length $M$ and
\begin{align*}
{\alpha_{r,s}}=Q-r\frac{\gamma}{2}-s\frac{2}{\gamma}.
\end{align*}

Now we turn to 4). The polynomials $\mathcal{Q}_{\alpha,\nu,\tilde\nu}$ with 
$|\nu|+|\tilde\nu|=N$ belong to  the space spanned by the basis  $\psi_{\bf kl}$  \eqref{fbasis2} (or equivalently the basis $\pi_{\bf kl}$ \eqref{fbasis}) above with $ 
|{\bf k}|+|{\bf l}|=N$. Let us also stress that item 3) entails that the space spanned by $\mathcal{Q}_{\alpha,\nu,\tilde\nu}$ with 
$|\nu|+|\tilde\nu|=N$ is exactly the same as $\psi_{\bf kl}$ with $ 
|{\bf k}|+|{\bf l}|=N$ as soon as $\alpha \not \in Q-\frac{\gamma}{2}\N^*-\frac{2}{\gamma}\N^*$. Hence the existence of a decomposition of the type \eqref{defqalpha}. The fact that the  coefficients of this decomposition are polynomials in $\alpha$ can be more easily seen in the basis $\pi_{\bf kl}$:  in that case, the coefficients are given by the scalar products $  \langle \mathcal{Q}_{\alpha,\nu,\tilde\nu} \,|\, \pi_{\k\l} \rangle_{L^2(\Omega_\T)}$ and one can use the expression \eqref{fbasis} together with the commutation relations
$$[{\bf A}_{n},{\bf L}^{0,\alpha}_m]=n{\bf A}^\alpha_{n+m}-\frac{i}{2}n(n+1)Q\delta_{n,-m}$$
where ${\bf A}^\alpha_0=\frac{i }{2}(\alpha+Q)$ and ${\bf A}^\alpha_n={\bf A}_n$ for $n\neq 0$ and   ${\bf L}^{0,\alpha}_01=\Delta_{\alpha} $. 4) follows.

Now, we specialise the above considerations to the case $\alpha=Q+iP$ with $P \in \R$. In this case, one has
\begin{equation}\label{scapointro}
 \langle \mathcal{Q}_{Q+iP,\nu,\tilde\nu} \,|\, \mathcal{Q}_{Q+iP,\nu',\tilde\nu'} \rangle _{L^2(\Omega_\T)}=\delta_{|\nu| ,|\nu'|}\delta_{|\tilde\nu| ,|\tilde\nu'|}F_{Q+iP}(\nu,\nu')F_{Q+iP}(\tilde\nu,\tilde\nu')
\end{equation} 
where the matrices $(F_{Q+iP}(\nu,\nu'))_{\nu,\nu'\in\mc{T}_j}$ are positive definite (hence invertible) for each $j\in\N$ and they depend  on ${P+iQ}$: more precisely they are polynomials in the weight $\Delta_{Q+iP}=\tfrac{1}{4}(P^2+Q^2)$ and in the central charge $c_L$. 
 Item 5) then follows from the representation \eqref{<F,G>usingH0} and decomposing the basis $\psi_{{\bf kl}}$ in terms of the new basis $\mc{Q}_{Q+iP,\nu,\nu'}$. \end{proof}


%


\section{Liouville CFT: dynamics and quadratic form} \label{sec:LCFTFQ}

The goal of this  section is to construct explicitly the quadratic form  associated to  ${\bf H}_\ast$ and to obtain a probabilistic representation of the semigroup. This probabilistic representation, i.e. a Feynman-Kac type formula, easily follows from the definition of ${\bf H}_\ast$, this is done in subsection \ref{Dynamics of LCFT} just below. The construction of the quadratic form is more subtle, especially in the case  $\gamma\in (\sqrt{2},2)$. Indeed we will construct the Friedrichs extension of the operator \eqref{Hdef} (restricted to appropriate domain) and we call it $\mathbf{H}$. Then we will show that ${\bf H}_\ast={\bf H}$. The main reason why the difficulty to construct   the quadratic form increases with $\gamma$ is due to the interpretation of the term $V$ in \eqref{Hdef}. For $\gamma\in (0,\sqrt{2})$ only, it makes sense as a GMC random variable. Indeed,  define the regularized field for $k\geq 0$
 \begin{equation}\label{GFFcirclecutoff}
\varphi^{(k)}(\theta)=\sum_{|n|\leq k}\varphi_ne^{in\theta} 
\end{equation}
 which is a.s. a smooth function. 
 Then the GMC random variable $V$ can be defined as the following  limit  
\begin{equation}\label{Vdefi}
V:=\lim_{k\to\infty}V^{(k )},\quad \quad V^{(k)}:=\int_0^{2\pi}e^{ \gamma \varphi^{(k)}(\theta)- \frac{\gamma^2}{2}  \E[ \varphi^{(k)}(\theta)^2  ]}  \dd\theta
\end{equation}
where the above limit exists $\P_\T$-almost surely and is non trivial for $\gamma\in (0,\sqrt{2})$, see \cite{cf:Kah,review,Ber} for instance on the topic, in which case $V\in L^{p}(\Omega_\T)$ for all $p<\frac{2}{\gamma^2}$. For $\gamma\in [\sqrt{2},2)$, the limit $V$ vanishes, in which case we will rather make sense of the multiplication operator $V$ as a measure  (singular with respect to $\P_\T$):  this case is more problematic, see   Subsection \ref {sub:bilinear}.

\subsection{Feynman-Kac formula}\label{Dynamics of LCFT}
  
 We consider the semigroup $e^{-t{\bf H}_\ast}$ and  start by writing some form of Feynman-Kac  formula for this semigroup.
\begin{proposition}{\bf (Feynman-Kac formula)}\label{prop:FK}
For $f\in L^2(\R\times\Omega_\T)$ we have
\begin{equation}\label{FKgeneral}
e^{-t\mathbf{H}_\star}f=e^{-\frac{Q^2t}{2}}\E_{\varphi}\big[ f(c+B_t,\varphi_t) e^{-\mu e^{\gamma c}\int_{\D_t^c} |z|^{-\gamma Q}M_\gamma (\dd z)}\big]
\end{equation}
where   $(c+B_t,\varphi_t)$ is the process on $W^{s}(\T)$ defined in subsection \ref{sub:gff}  and $\D_t^c:=\{z\in\D\, |\,  |z|>e^{-t}\}$.\\
In particular, for $\gamma\in (0,\sqrt{2})$
\begin{equation}\label{fkformula}
e^{-t\mathbf{H}_\star}f=e^{-\frac{Q^2t}{2}}\E_{\varphi}\big[ f(c+B_t,\varphi_t)e^{-\mu\int_0^t e^{\gamma (c +B_s) }V(\varphi_s)\dd s}\big].
\end{equation}
where 
$$V(\varphi_s):=\int_0^{2\pi}e^{ \gamma \varphi_s(\theta)- \frac{\gamma^2}{2}  \E[ \varphi_s(\theta)^2  ]}  \dd\theta.$$
\end{proposition}

\begin{proof}
Let   $f\in L^2(\R\times\Omega_\T)$.  By definition we have $e^{-t\mathbf{H}_\star}F=US_{e^{-t}}U^{-1}f$, where,  from \eqref{Uinverse}, we have $U^{-1}f=(U1)^{-1} f$ with
\begin{align*}
(U1)(c,\varphi)=(U_0 e^{-\mu e^{\gamma c}M_\gamma(\D)})(c,\varphi).
\end{align*}
 Now we claim that we have for $t>0$
\begin{equation}\label{Vmarkov}
\E[e^{-\mu  e^{\gamma c} M_\gamma ( \D_t)}|\mathcal{F}_{\D^c_t}  ]= S_{e^{-t}}(e^{Qc}U1)
\end{equation}
where $\D_t=\{z\in\D;|z|<e^{-t}\}$. Indeed, by the domain Markov property of the GFF and conditionally  on  $\mathcal{F}_{\D^c_t}$, the law of $X$ inside $\D_t$ is given by the independent sum
$$X\stackrel{law}{=}P(s_{e^{-t}}X)(e^t\cdot )+X_\D(e^t\cdot).$$
As a consequence and using a change of variables by dilation, we get the relation conditionally on $\mathcal{F}_{\D^c_t} $
$$  M_\gamma(\D_t)\stackrel{law}{=}S_{e^{-t}}M_\gamma(\D),$$
which gives \eqref{Vmarkov}. Next, using \eqref{Vmarkov}, we get
 \begin{align*}
 e^{-t\mathbf{H}_\star}f=&US_{e^{-t}}U^{-1}f\\
 =&e^{-Qc}\E_\varphi\big[S_{e^{-t}}(U^{-1}f)e^{-\mu e^{\gamma c}M_\gamma ( \D)}\big]\\
  =&e^{-Qc}\E_\varphi\big[\E\big[S_{e^{-t}}(U^{-1}f)e^{-\mu  e^{\gamma c} M_\gamma ( \D)}|\mathcal{F}_{\D^c_{t}}\big]\big]\\
    =&e^{-Qc}\E_\varphi\big[ S_{e^{-t}}(U^{-1}f)e^{-\mu  e^{\gamma c} M_\gamma ( \D_t^c)}\E\big[e^{-\mu  e^{\gamma c}M_\gamma ( \D_t)}|\mathcal{F}_{\D^c_{t}}\big]\big]\\
      =&e^{-Qc}\E_\varphi\big[ S_{e^{-t}}(e^{Qc}f ) e^{-\mu  e^{\gamma c}M_\gamma ( \D^c_t)} \big].
 \end{align*}
We  complete the argument by applying the Cameron-Martin formula to $S_{e^{-t}}(e^{Qc})$ to get \eqref{FKgeneral}.

 In the case when $\gamma\in (0,\sqrt{2})$, we write the chaos measure in terms of the process $Z_t=(c+B_t,\varphi_t)$. The function $V:W^{s}(\T)\to \R^+$ given by 
$$V(\varphi)=\lim_{k\to\infty}\int_\T e^{\gamma \varphi^{(k)}(\theta)-\frac{\gamma^2}{2}\E[ \varphi^{(k)}(\theta)^2]}\dd\theta
$$
 is measurable   and we get that conditionally on $\varphi=\varphi_0$ (by making the change of variables   $dx=r dr d\theta=e^{-2s}dsd\theta$ with $r=|x|=e^{-s}$ and $-\frac{\gamma^2}{2} \mathbb{E}[B_s^2]-2s=-\gamma Qs$)
\[
e^{\gamma c}M_\gamma(\D_t^c)\stackrel{law}=\int_0^t e^{\gamma (c+B_s-Qs)}V(\varphi_s)\dd s.\qedhere
\]
\end{proof}
 As a consequence of this formula and similarly to $\bf H^0$ we have 
\begin{proposition}\label{l2alphamu} The following properties hold:
\begin{enumerate}
\item $e^{-t{\bf H}_\ast}$ extends to a  continuous semigroup on $ L^p(\R \times \Omega_\T)$ for all $p\in [1,+\infty]$ with norm $e^{ -\frac{Q^2}{2}t}$ and it is strongly continuous for $p\in [1,+\infty)$.
\item $e^{-t{\bf H}_\ast}$ extends to a strongly continuous semigroup on $e^{-\alpha c} L^2(\R \times \Omega_\T)$ for all $\alpha\in\R$ with norm $e^{(\frac{\alpha^2}{2}-\frac{Q^2}{2})t}$.
\end{enumerate}
\end{proposition}
\begin{proof}   Using in turn \eqref{fkformula} and $V\geq 0$, we see that $\|e^{-t\mathbf{H}_\star}f\|_p\leq \|e^{-t\mathbf{H}^0}|f|\|_p$ so that the claim 1) follows from Proposition \ref{l2alpha}. The same argument works for 2). 
\end{proof}

\subsection{Quadratic forms and the Friedrichs extension of ${\bf H}$}\label{sub:bilinear}
 Here we construct the quadratic forms associated to \eqref{Hdef} and   the Friedrichs extension of ${\bf H}$ in the case $\gamma\in (0,2)$.

Recall that the underlying measure on the space  $L^2(\R\times \Omega_\T)$ is  $\dd c\times \P_\T$, with   $\Omega_\T=(\R^{2})^{\N^*}$, $  \Sigma_\T=\mathcal{B}^{\otimes \N^*}$ (where $\mathcal{B}$  stands for the Borel sigma-algebra on $\R^2$) and the   probability measure $ \P_\T$ defined by  \eqref{Pdefin}. Also, recall the GFF on the unit circle $\varphi:\Omega \to W^{s}(\T)$ (with $s<0$)  defined by \eqref{GFFcircle}. Finally recall that $\mathcal{S}$ is the set of smooth functions depending on finitely many coordinates, i.e. of the form
$F(x_1,y_1, \dots,x_n,y_n)$  with $n\geq 1$ and $F\in C^\infty((\R^2)^n)$, with at most polynomial growth at infinity for $F$ and its derivatives.   $\mathcal{S}$ is dense in $L^2(\Omega_\T)$.

We will construct the quadratic form associated to $\mathbf{H}$  as a limit (in a suitable sense) of regularized quadratic forms associated with the regularized potential $V^{(k)}$.

 For $k\geq 1$ and $\mu> 0$, we introduce the bilinear form (with associated quadratic form still denoted by $\mathcal{Q}_k$)    
\begin{equation}\label{defQn}
\mathcal{Q}^{(k)}(u,v):=\tfrac{1}{2}\E\int_{\R} \Big( \pl_c u \pl_c \bar{v}+Q^2u\bar{v}+ 2( \mathbf{P} u)\bar{v}+2\mu e^{\gamma c}V^{(k)}u\bar{v}\Big)\dd c.
\end{equation}
Here   $u,v$ belong to the domain $\mathcal{D}(\mathcal{Q}^{(k)})$ of the quadratic form, namely the completion  for the $\mathcal{Q}^{(k)}$-norm in $L^2(\R\times \Omega_\T)$  of the space $\mathcal{C}$ defined by  \eqref{core}. Also, it is clear that $ \mathcal{D}(\mathcal{Q}^{(k)})$ embeds continuously and injectively in $L^2(\R\times \Omega_\T)$ (same argument as in the proof of Prop \ref{FQ0:GFF}). Also, since $V^{(k)}\geq 0$ it is clearly lower semi-bounded $\mathcal{Q}^{(k)}(u)\geq Q^2\|u\|_2^2/2$ so that the construction of the Friedrichs extension then follows from \cite[Theorem 8.15]{rs1}. It   determines uniquely a self-adjoint operator $\mathbf{H}^{(k)}$, called the \emph{Friedrichs extension},  with domain denoted by $\mc{D}(\mathbf{H}^{(k)})$ such that:
$$\mc{D}(\mathbf{H}^{(k)} )=\{u\in \mathcal{D}(\mathcal{Q}^{(k)}); \exists C>0,\forall v\in \mathcal{D}(\mathcal{Q}^{(k)}),\,\,\, \mc{Q}^{(k)}(u,v)\leq C\|v\|_2\}$$
and for $u\in \mc{D}(\mathbf{H}^{(k)})$, $\mathbf{H}^{(k)}  u$ is the unique element in $L^2(\R\times \Omega_\T)$  satisfying
$$\mc{Q}^{(k)}(u,v)=\langle \mathbf{H}^{(k)}  u|v\rangle_2 .$$
We also denote by  $\mathbf{R}_{\lambda}^{(k)}$ the associated resolvent family.  
The following result is rather standard but we found no reference corresponding exactly to our context so that we give a short proof.
\begin{proposition}\label{prop:fkhk}
The strongly continuous contraction semigroup $(e^{-t\mathbf{H}^{(k)}})_{t\geq 0}$ of self-adjoint operators on $L^2(\R\times\Omega_\T)$   obeys the Feynman-Kac formula 
\begin{equation}\label{fkformulaforhk}
e^{-t\mathbf{H}^{(k)}}f=e^{-\frac{Q^2t}{2}}\E_{\varphi}\big[ f(c+B_t,\varphi_t)e^{-\mu\int_0^t e^{\gamma (c +B_s) }V^{(k)}(\varphi_s)\dd s}\big],
\end{equation}
where we have, after decomposing the field $\varphi_s $ along its harmonics i.e. $\varphi_s(\theta):=\sum_{n\not=0}\varphi_{s,n}e^{in\theta}$,
   $$V^{(k)}(\varphi)=\int_0^{2\pi}e^{ \gamma \varphi^{(k)}_s(\theta)- \frac{\gamma^2}{2}  \E[ \varphi^{(k)}_s(\theta)^2  ]}  \dd\theta\quad\text{ with }\varphi^{(k)}_s(\theta):=\sum_{|n|\leq k,n\not=0}\varphi_{s,n}e^{in\theta}.$$  
\end{proposition}
 
\begin{proof}
We use Kato's strong Trotter product formula (see \cite[Theorem S.21  page 379]{rs1}) applied to the self-adjoint operators  ${\bf H}_0$ and $e^{\gamma c}V^{(k)}$: since  the domain  $\mathcal{D}(\mathcal{Q}^{(k)})$ of the quadratic form $\mathcal{Q}^{(k)}$  is dense in $L^2(\R\times\Omega_\T)$ and satisfies $\mathcal{D}(\mathcal{Q}^{(k)})=\mathcal{D}(\mathcal{Q}_0)\cap \mathcal{D}(\mathcal{Q}_{V^{(k)}})$, where  $\mathcal{D}(\mathcal{Q}_{V^{(k)}})=\{f\mid \|e^{\gamma c/2}(V^{(k)})^{1/2}f\|_2<+\infty\}$ is the domain of the quadratic form associated to the operator of multiplication by $e^{\gamma c}V^{(k)}$,  we have the identity  
$$\lim_{n\to\infty} (e^{-\frac{t}{n}{\bf H}_0}e^{-\frac{t}{n}\mu e^{\gamma c}V^{(k)}})^n =e^{-t{\bf H}^{(k)}}$$
where the limit is understood in the strong sense (i.e. convergence in $L^2(\R\times\Omega_\T)$  when this relation is applied to   $f\in L^2(\R\times\Omega_\T)$). Now we compute the limit in the left-hand side. For  $f\in L^2(\R\times\Omega_\T)$  we have
\begin{equation}\label{trotter}
(e^{-\frac{t}{n}{\bf H}_0}e^{-\frac{t}{n}\mu e^{\gamma c}V^{(k)}})^n f=e^{-Q^2t/2}\E_{\varphi}\big[ f(c+B_t,\varphi_t)e^{-\mu R_t^{n,(k)}}\big]
\end{equation}
with $R_t^{n,(k)}$ the Riemann sum
$$R_t^{n,(k)}:=\frac{t}{n}\sum_{j=1}^n e^{\gamma (c +B_{jt/n}) }V^{(k)}(\varphi_{jt/n}) .$$
The right-hand side  of  \eqref{trotter} converges in $L^2(\R\times\Omega_\T)$ towards the same expression with $R_t^{n,(k)}$ replaced by $\int_0^t e^{\gamma (c +B_s) }V^{(k)}(\varphi_{s})\dd s$: indeed this can be established by using Jensen and the fact that, almost surely, the Riemann sum $R_t^{n,(k)}$ converges almost surely (for all fixed $c$) towards the integral $\int_0^t e^{\gamma (c +B_s) }V^{(k)}(\varphi_{s})\dd s$ since  the process $s\mapsto e^{\gamma (c +B_{s}) }V^{(k)}(\varphi_{s}) $ is continuous.  This provides the Feynman-Kac representation as claimed.
\end{proof}

Now our main goal is to construct  a quadratic form corresponding to the limit $k\to\infty$ of the quadratic forms $(\mathcal{Q}^{(k)},\mathcal{D}(\mathcal{Q}^{(k)}))$. For $u,v\in \mathcal{C}$, we define 
\begin{equation}\label{defQ2}
\mathcal{Q}(u,v):=\lim_{k\to\infty}\mathcal{Q}^{(k)}(u,v).
\end{equation}
Of course, when $\gamma\in (0,\sqrt{2})$, existence of the limit is trivial as $(V^{(k)})$ converges towards $V$ in $L^p(\Omega_\T)$ for $p<2/\gamma^2$ and $u\bar{v}\in C^\infty_c(\R;L^q(\Omega_\T))$ for any $q>1$. But treating the case $\gamma\in (\sqrt{2},2)$ as well requires another argument: the existence of the limit is guaranteed by the Girsanov transform, namely that for $u=u(x_1,y_1\dots,x_n,y_n)$, $v=v(x_1,y_1\dots,x_n,y_n)$ and $k\geq n$ the term involving $V^{(k)}$ in $\mathcal{Q}^{(k)}$ can be rewritten as 
\begin{equation}\label{girsuk}
\E\int_{\R} e^{\gamma c}V^{(k)}u\bar{v} \dd c=\E\int_{\R} \int_0^{2\pi}e^{\gamma c} u_{\rm shift}\bar{v}_{\rm shift} \dd c\dd \theta
\end{equation}
where the function $u_{\rm shift}$ (and similarly for $v$) is defined by
$$u_{\rm shift}(\theta,c, x_1,y_1\dots,x_n,y_n):=u \Big(c,x_1+\gamma \cos(\theta),y_1-\gamma \sin(\theta)\dots,x_n+\frac{\gamma}{\sqrt{n}}\cos(n\theta),y_n-\frac{\gamma}{\sqrt{n}}\sin(n\theta)\Big).$$
Hence the term $\E\int_{\R} e^{\gamma c}V^{(k)}u\bar{v} \dd c$ does not depend on $k\geq n$.
We also denote by $u_{\rm shift(N)}$ the same as above except that the first $N$ variables $(x_1,y_1,\dots,x_N,y_N)$ are shifted and other variables remain unchanged.

Also, we denote by $\mathbf{R}_{*,\lambda}$ the resolvent family   associated with  the Feynman-Kac semigroup $e^{-t\mathbf{H}_*}$. Let   $\mathcal{Q}_*$ denote the quadratic form associated to $\mathbf{H}_*$ with domain $\{u\in L^2; \lim_{t\to0}\langle u, \frac{u-e^{-t\mathbf{H}_*}u}{t}\rangle_2<\infty\}$. Using \cite[Section 1.4]{sznitman}, $\mathcal{Q}_*$ is closed. The following lemma is   fundamental: though the   limit \eqref{defQ2} makes sense for any value of $\gamma\in \R$, the fact that it can be related to the quadratic form $ \mathcal{Q}_*$ deeply relies on GMC theory, hence on the fact that $\gamma\in (0,2)$.

\begin{lemma}
For $u,v\in\mathcal{C}$, we have  $ \mathcal{Q}_*(u,v)= \mathcal{Q}(u,v)$. In particular, $ \mathcal{Q}$ is closable.
\end{lemma}

\begin{proof}
Let $u,v\in\mathcal{C}$, say depending on the first $n$-th harmonics. From \eqref{FKgeneral}, we have to show that, as $t\to 0$, 
$$D_t:=e^{-\frac{Q^2t}{2}}\int\E\Big[   u(c+B_t,\varphi_t) v(c,\varphi)e^{-\mu e^{\gamma c}\int_{\D_t^c} |z|^{-\gamma Q}M_\gamma (\dd z)}\Big] \,\dd c=\langle u,v\rangle_2-t \mathcal{Q}(u,v)+o(t).$$
For notational simplicity and only in this proof, let us denote $V_t:=\int_{\D_t^c} |z|^{-\gamma Q}M_\gamma (\dd z)$. Then
\begin{align*}
D_t-\langle u,v\rangle_2=& \langle e^{-t\mathbf{H}_0}u-u,v\rangle_2+(1+o(1))\int\E\Big[   u(c+B_t,\varphi_t) v(c,\varphi)(e^{-\mu e^{\gamma c}V_t}-1)\Big] \,\dd c\\
=&-t\mathcal{Q}_0(u,v)+o(t)+(1+o(1))\int\E\Big[   (u(c+B_t,\varphi_t) -u(c,\varphi))v(c,\varphi)(e^{-\mu e^{\gamma c}V_t}-1)\Big] \,\dd c\\
&-(1+o(1))\int\E\Big[   u(c,\varphi)v(c,\varphi)(1-e^{-\mu e^{\gamma c}V_t})\Big] \,dc \\
=&:-t\mathcal{Q}_0(u,v)+o(t)+(1+o(1))D^1_t-(1+o(1))D^2_t.
\end{align*}
Let us show that 
\begin{equation}\label{D2t}
D_t^2=t(1+o(1))\mu\int e^{\gamma c}\int_0^{2\pi} \E\Big[   u_{\rm shift}(c,\varphi)v_{\rm shift}(c,\varphi)\Big] \,\dd c\dd\theta .
\end{equation} For this and even if it means decomposing $u,v$ into their positive/negative parts, we may assume that $u,v$ are non-negative. First, we use the inequality $1-e^{-x}\leq x$ for $x\geq 0$ to get
\begin{align*}
D_t^2\leq & \int e^{\gamma c}\E\Big[   u(c,\varphi)v(c,\varphi)\mu e^{\gamma c}V_t\Big] \,\d c\\
=& \mu\int e^{\gamma c}\int_0^{2\pi}\int_{e^{-t}}^1\E\Big[  e^{\gamma P\varphi(re^{i\theta})-\frac{\gamma^2}{2}\E[P\varphi(re^{i\theta})^2]} u(c,\varphi)v(c,\varphi)\Big] \,\dd c r^{1-\gamma Q}\,\dd r\dd\theta \\
=& \mu\int e^{\gamma c}\int_0^{2\pi}\int_{e^{-t}}^1\E\Big[   u^r_{\rm shift}(c,\varphi)v^r_{\rm shift}(c,\varphi)\Big] \,\dd c r^{1-\gamma Q}\,\dd r\dd\theta 
\end{align*}
where we have used the Girsanov transform in the last line and the function $ u^r_{\rm shift}$ is defined by 
$$u_{\rm shift}(\theta,c, x_1,y_1\dots,x_n,y_n):=u \big(c,x_1+a_1,y_1+b_1,\dots,x_n+a_n,y_n+b_n\big),$$
with $a_k:=  \frac{\gamma}{\sqrt{k}}r^k\cos(k\theta)$ and $b_k:= -\frac{\gamma}{\sqrt{k}}r^k\sin(k\theta)$. 
Because $u,v\in \mathcal{C}$, it is then plain to deduce that $D^2_t\leq t(1+o(1))\mu\int e^{\gamma c}\int_0^{2\pi} \E\Big[   u_{\rm shift}(c,\varphi)v_{\rm shift}(c,\varphi)\Big] \,\dd c\dd\theta $.
Second, we use the inequality $1-e^{-x}\geq x e^{-x}$ for $x\geq 0$  to get, using again the Girsanov transform,
\begin{align*}
D_t^2\geq & \int e^{\gamma c}\E\Big[   u(c,\varphi)v(c,\varphi)\mu e^{\gamma c}V_te^{-\mu e^{\gamma c}V_t}\Big] \,\d c\\
=& \mu\int e^{\gamma c}\int_0^{2\pi}\int_{e^{-t}}^1\E\Big[   u^r_{\rm shift}(c,\varphi)v^r_{\rm shift}(c,\varphi,re^{})e^{-\mu e^{\gamma c}V^{\rm shift}_t}\Big] \,\dd c r^{1-\gamma Q}\,\dd r\dd\theta 
\end{align*}
with $V^{\rm shift}_t(re^{i\theta}):=\int_{\D_t^c} |z|^{-\gamma Q}|z-re^{i\theta}|^{-\gamma^2}M_\gamma (\dd z)$. It remains to get rid of the exponential term in the last expectation.  So we split the above expectation in two parts by writing $e^{-\mu e^{\gamma c} V^{\rm shift}_t} =1+(e^{-\mu e^{\gamma c} V^{\rm shift}_t}  -1)$. The first part produces $ t(1+o(1))\mu\int e^{\gamma c}\int_0^{2\pi} \E\Big[   u_{\rm shift}(c,\varphi)v_{\rm shift}(c,\varphi)\Big] \,\dd c\dd\theta $ similarly as above. For the second part corresponding to $(e^{-\mu e^{\gamma c} V^{\rm shift}_t}  -1)$, we want to show that it is neglectable. For this,  we write
\begin{align}
\mu\int e^{\gamma c}&\int_0^{2\pi} \int_{e^{-t}}^1\E\Big[   u^r_{\rm shift}(c,\varphi)v^r_{\rm shift}(c,\varphi)(1-e^{-\mu e^{\gamma c}V^{\rm shift}_t})\Big] \,\dd c r^{1-\gamma Q}\,\dd r\dd\theta\label{fuck} \\
\leq &\mu\int e^{\gamma c}\int_0^{2\pi}\int_{e^{-t}}^1\E\Big[   u^r_{\rm shift}(c,\varphi)v^r_{\rm shift}(c,\varphi))^p\Big]^{1/p}\E\Big[  |1-e^{-\mu e^{\gamma c}V^{\rm shift}_t(re^{i\theta})}   |^q      \Big]^{1/q} \,\dd c r^{1-\gamma Q}\,\dd r\dd\theta\nonumber
\end{align}
where we have used  H\"{o}lder's inequality in the last line for any fixed conjugate exponents $p,q$. The first expectation   $c\mapsto \E\Big[   u^r_{\rm shift}(c,\varphi)v^r_{\rm shift}(c,\varphi))^p\Big]^{1/p}$  is bounded uniformly with respect to $r,c$ and has fixed compact support in $c$ (because $v$ has so), say $[-A,A]$ for some $A>0$. Now we claim that the family of functions $(c,r)\mapsto f_t(c,r):=\E\Big[  |1-e^{-\mu e^{\gamma c}V^{\rm shift}_t(re^{i\theta})}   |^q      \Big]^{1/q}$ (this quantity does not depend on $\theta$ by invariance in law under rotation) converges uniformly towards $0$ as $t\to 0$, from which one can easily deduce that the quantity \eqref{fuck} is $o(t)$. Since $f_t$ is increasing in $c$ hence monotonic for $c$ in $[-A,A]$) it is enough to prove this claim for a fixed $c$.  Next, recall \cite[Lemma 3.10]{DKRV} the condition of finiteness for integrals of GMC with singularities: for any $\alpha\in (0,Q)$, $p>0$ and $z_0\in\D$ 
\begin{equation}\label{conising}
\E\Big[\Big(\int_\D \frac{1}{|z-z_0|^{\gamma\alpha}}M_\gamma(\dd z)\Big)^p\Big]<+\infty \Leftrightarrow p<\tfrac{2}{\gamma}(Q-\alpha).
\end{equation} 
Taking $\alpha=\gamma$ and $z_0=re^{i\theta}$, the finiteness of the expectation above  implies that for fixed $r\in [\tfrac{1}{2},1]$ (and fixed $c$) $\lim_{t\to 0}f_t(c,r)=0$ by using dominated convergence. Now we claim that for fixed $t$ and $c$, the mapping $r\mapsto f_t(c,r)$ is continuous. To see this,  for each $\delta>0$  let us introduce the mapping $$r\in  [\tfrac{1}{2},1] \mapsto f_{t,\delta}(c,r):=\E\Big[  |1-e^{-\mu e^{\gamma c}V^{\rm shift}_{t,\delta}(re^{i\theta})}   |^q      \Big]^{1/q}$$ where $V^{\rm shift}_{t,\delta}$ is the potential with regularized   singularity at scale $\delta$
$$V^{\rm shift}_{t,\delta}(re^{i\theta}):=\int_{\D_t^c} |z|^{-\gamma Q}(|z-re^{i\theta}|\vee \delta)^{-\gamma^2}M_\gamma (\dd z).$$
Obviously, $f_{t,\delta}(c,r)$ is a continuous function of the variable $r$. We want to show that the family $(f_{t,\delta})_\delta$ converges uniformly towards $f_{t}$ as $\delta\to 0$. By using the triangular inequality
\begin{align*}
|f_{t,\delta}(c,r)-f_{t}(c,r)|\leq &\E\Big[|e^{-\mu e^{\gamma c}V^{\rm shift}_{t,\delta}(re^{i\theta})}  -e^{-\mu e^{\gamma c}V^{\rm shift}_{t}(re^{i\theta})}  |^q\Big]^{1/q}\\
\leq &\mu e^{\alpha\gamma c}\E\Big[|  V^{\rm shift}_{t}(re^{i\theta}) -V^{\rm shift}_{t,\delta}(re^{i\theta})   |^{\alpha q}\Big]^{1/q}.
\end{align*}
where we have used the inequality $|e^{-x}-e^{-y}|\leq |x-y|^{\alpha}$ for any arbitrary $\alpha\in (0,1]$ and $x,y\geq 0$. We fix $\alpha$ such that $\alpha q<\tfrac{2}{\gamma}(Q-\gamma)$ to make sure that the expectation is finite using the criterion \eqref{conising}. By invariance in law of $M_\gamma$ under translation, this quantity does not depend on $r$  and is less than  (in fact this is true when the singularity is at positive distance from the boundary: when the singularity approaches the boundary, the above quantity is strictly less than the bound below)
\begin{equation}\label{singsmallball}
C\E\Big[\Big(\int_{|z|\leq \delta }  |z|^{-\gamma^2}M_\gamma (\dd z)\Big)^{\alpha q}\Big]^{1/q}.
\end{equation}
By multifractal scaling (see the proof of  \cite[Lemma 3.10]{DKRV}), \eqref{singsmallball} is  equal  $C\delta ^{(\gamma (Q-\gamma)\alpha-\frac{\gamma^2}{2}\alpha^2q}$ and one can check that the exponent is positive under the condition $\alpha q<\tfrac{2}{\gamma}(Q-\gamma)$. This establishes the uniform convergence of $(f_{t,\delta})_\delta$   towards $f_{t}$ as $\delta\to 0$ over $r\in [\tfrac{1}{2},1]$. In conclusion, for fixed $c$, the family $(f_t(c,\cdot))_{t>0}$ is a family of continuous functions that decrease pointwise towards $0$ as $t\to 0$. Hence the convergence is uniform by the Dini theorem. So we have proved \eqref{D2t}. 

Similar arguments can be used to show that $D^1_t=o(t)$. Indeed, we use again the inequality $1-e^{-x}\leq x$ to get the bound
$$|D^1_t|\leq \mu\int e^{\gamma c}\E\Big[   |u(c+B_t,\varphi_t)-u(c,\varphi)|| v(c,\varphi)| V_t \Big] \,\dd c$$
and then the Girsanov transform and H\"{o}lder to obtain
\begin{align*}
|D^1_t|\leq &\mu\int_0^{2\pi}\int_{e^{-t}}^1\int e^{\gamma c}\E\Big[   |u^{\rm shift}(t,c+B_t,\theta,\varphi_t)-u^{\rm shift}(0,c,\theta,\varphi)|| v^{\rm shift}(0,c,\theta,\varphi)|   \Big] \,\dd \theta\dd r\dd c\\
\leq &\mu\int_0^{2\pi}\int_{e^{-t}}^1\int e^{\gamma c}\E\Big[   |u^{\rm shift}(t,c+B_t,\theta,\varphi_t)-u^{\rm shift}(0,c,\theta,\varphi)|^p\Big]^{1/p}\E[| v^{\rm shift}(0,c,\theta,\varphi)|^q]^{1/q} \,\dd \theta\dd r\dd c
\end{align*}
with 
$$u^{\rm shift}(t,c,\theta,x_1,y_1\dots,x_n,y_n):=u \Big(c,x_1+a_1,y_1+b_1,\dots,x_n+a_n,y_n+b_n\Big)$$
and
$$a_k:=\frac{\gamma k^{1/2}}{\pi}\int_0^{2\pi}\ln\frac{1}{|e^{-t+i\theta'}-e^{-r+i\theta}|}\cos(k\theta')\,\dd \theta',\quad b_k:=\frac{\gamma k^{1/2}}{\pi}\int_0^{2\pi}\ln\frac{1}{|e^{-t+i\theta'}-e^{-r+i\theta}|}\sin(k\theta)\,\dd \theta. $$
Again, the second expectation has compact support in $c$ whereas the second satisfies $\sup_{r\leq t} \E\Big[   |u^{\rm shift}(t,c+B_t,\theta,\varphi_t)-u^{\rm shift}(0,c,\theta,\varphi)|^p\Big]^{1/p}\to 0$ as $t\to 0$ since $u\in \mathcal{C}$  and $B_t$ and the first $n$-th harmonics of the field $(\varphi_t)_t$  are  continuous. One can then easily conclude.
\end{proof}

 From the above lemma, the quadratic form $(\mathcal{C},\mathcal{Q} )$  is closable.  Let $ \mathcal{D}(\mathcal{Q})$ be the completion of $ \mathcal{C}$ for the $\mathcal{Q}$-norm. The completion is the vector space consisting of equivalence classes of Cauchy sequences of 
$\mc{C}$ for the norm $\|u\|_{\mc{Q}}:=\sqrt{\mc{Q}(u,u)}$ under the equivalence relation 
$u\sim v$ iff $\|u_n-v_n\|_{\mc{Q}}\to 0$ as $n\to \infty$. This space is a Hilbert space. It can be identified with the closure of the space   of quadruples $\{(u,\partial_cu,\mathbf{P}^{1/2}u,u_{\rm shift});u\in \mathcal{C}\}$ in $(L^2)^3\times L^2(\R\times\Omega_\T\times [0,2\pi],e^{\gamma c}\dd c\otimes \P_\T\otimes \dd \theta)$. $ \mathcal{D}(\mathcal{Q})$ embeds injectively in $L^2$ as it is closed. Now we want to show that  
\begin{proposition}
For $\gamma\in (0,2)$, the quadratic form $( \mathcal{D}(\mathcal{Q}),\mathcal{Q})$  defines an operator $\mathbf{H}$ (the Friedrichs extension), which satisfies  $\mathbf{H}_*=\mathbf{H}$.
\end{proposition}

\begin{proof}
Take $F\in \mathcal{C}$ and consider $u:=\mathbf{R}_{*,\lambda}F$ (resp. $u_k:=\mathbf{R}_{\lambda}^{(k)}F$) for $\lambda>0$. We will repeatedly use below the fact that for $k$ large enough $u_{k,{\rm shift}}=u_{k,{\rm shift}(k)}$ since $F\in \mathcal{C}$. Indeed, this follows from the Feynman-Kac formula (by Prop. \ref{prop:fkhk})
\begin{equation}\label{fkres}
u_k=\int_0^{\infty}e^{-\lambda t}e^{-t\mathbf{H}^{(k)}}f\,\dd t=\int_0^{\infty }e^{-(\lambda+\frac{Q^2}{2})t}\E_{\varphi}\big[ f(c+B_t,\varphi_t)e^{-\mu\int_0^t e^{\gamma (c +B_s) }V^{(k)}(\varphi_s)\dd s}\big]\,\dd t.
\end{equation}
With this expression, one can see that if $f$ depends on the first $N$ harmonics of the field $\varphi$ then for $k\geq N$ we have $u_{k,{\rm shift}}=u_{k,{\rm shift}(k)}$.

The first step of the proof is to observe that
$$u_k\to u \quad \text{as }k\to\infty \quad\text{in }L^2(\R\times\Omega_\T).$$  This follows from the Feynman-Kac representation \eqref{fkres}+\eqref{FKgeneral}  and  standard GMC theory that ensures  that $\int_0^t e^{\gamma (c +B_s) }V^{(k)}(\varphi_s)\dd s$ converges almost surely towards $e^{\gamma c}\int_{\D_t^c} |z|^{-\gamma Q}M_\gamma (\dd z)$ for  $\gamma\in (0,2)$.
Furthermore, since
$$\lambda\|u_k\|_2^2+\mathcal{Q}^{(k)}(u_k)=\langle F,u_k\rangle_2$$
we deduce using Cauchy-Schwartz in the r.h.s. that 
\begin{equation}\label{tight}
\sup_k\mathcal{Q}^{(k)}(u_k)<+\infty.
\end{equation}
This entails that the sequences $(\partial_cu_k)_k$ and $(\mathbf{P}^{1/2}u_k)_k$ weakly converge up to subsequences in $L^2(\R\times\Omega_\T)$ and, by \eqref{girsuk}, that $(u_{k,{\rm shift}})_k$ weakly converges in $e^{-\gamma c/2}L^2([0,2\pi]\times\R\times\Omega_\T)$. Strong convergence of $(u_k)_k$ towards $u$ and weak convergence  of $(\partial_cu_k)_k$ and $(\mathbf{P}^{1/2}u_k)_k$ implies that their respective weak limits must be $\partial_cu$ and $\mathbf{P}^{1/2}u$. The resolvent equation associated to $u_k$ reads for $v\in\mathcal{C}$
\begin{equation}
\lambda\langle u_k,v\rangle_2+\mathcal{Q}^{(k)}(u_k,v)=\langle F,v\rangle_2.
\end{equation}
Denote by $z$ a possible weak limit of $(u_{k,{\rm shift}})_k$   in $e^{-\gamma c/2}L^2([0,2\pi]\times\R\times\Omega_\T)$.
Passing to the limit in $k\to\infty$ (up to appropriate subsequence) produces
\begin{equation}\label{eqlimitQ}
\tfrac{1}{2}\E\int_{\R} \Big( \pl_c u \pl_c \bar{v}+(Q^2+2\lambda)u\bar{v}+ 2( \mathbf{P}^{1/2} u)\overline{\mathbf{P}^{1/2}v}+2\mu \int_0^{2\pi}e^{\gamma c}z \bar{v}_{\rm shift}\dd\theta\Big)\dd c=\langle F,v\rangle_2.
\end{equation}
Taking $v=u_k$ and passing to the limit as $k\to\infty$, we get
\begin{equation}
\tfrac{1}{2}\E\int_{\R} \Big( |\pl_c u|^2+(Q^2+2\lambda)|u|^2+ 2| \mathbf{P}^{1/2} u|^2+2\mu \int_0^{2\pi}e^{\gamma c}|z|^2\dd\theta\Big)\dd c=\langle F,u\rangle_2.
\end{equation}
By weak limit we have
\begin{align*}
\tfrac{1}{2}\E\int_{\R} &\Big( |\pl_c u|^2+(Q^2+2\lambda)|u|^2+ 2| \mathbf{P}^{1/2} u|^2+2\mu \int_0^{2\pi}e^{\gamma c}|z|^2\dd\theta\Big)\dd c\\
\leq &
\liminf_k\tfrac{1}{2}\E\int_{\R} \Big( |\pl_c u_k|^2+(Q^2+2\lambda)|u_k|^2+ 2| \mathbf{P}^{1/2} u_k|^2+2\mu \int_0^{2\pi}e^{\gamma c}| u_{k,{\rm shift}}|^2\dd\theta\Big)\dd c .
\end{align*}
Also
\begin{align*}
\tfrac{1}{2}\E\int_{\R} \Big( |\pl_c u_k|^2+(Q^2+2\lambda)|u_k|^2+ 2| \mathbf{P}^{1/2} u_k|^2+2\mu \int_0^{2\pi}e^{\gamma c}| u_{k,{\rm shift}}|^2\dd\theta\Big)\dd c =\langle F,u_k\rangle_2\to \langle F,u\rangle_2
\end{align*}
as $k\to\infty$. This shows that the weak convergence of the sequence  $(u_k,\partial_cu_k,\mathbf{P}^{1/2} u_k,u_{k,{\rm shift}})_k $ actually holds in the strong sense.  Also, it is easy to check from \eqref{eqlimitQ} that the limit is unique. This implies the convergence of $(u_k)_k$ in $\mathcal{Q}$-norm toward $u$, which thus  belongs to $ \mathcal{D}(\mathcal{Q})$.   Hence $ \mathbf{R}_{*,\lambda}$ maps $\mathcal{C}$ into $\mathcal{D}(\mathcal{Q})\subset L^2(\R\times\Omega_\T)$ and it coincides on $\mathcal{C}$ with the resolvent associated to $\mathcal{Q}$. Since $\mathcal{C}$ is dense in $L^2(\R\times\Omega_\T)$, this shows that both resolvent families coincide, hence their semigroups, quadratic forms and generators  too.
\end{proof}

Again, we stress that $\mc{D}({\bf H})=\{ u\in\mc{D}(\mc{Q})\,|\, {\bf H}u\in L^2(\R\times \Omega_\T)\}$ and ${\bf H}^{-1}:L^2(\R\times \Omega_\T)\to \mc{D}({\bf H})$ is bounded. Furthermore, by the spectral theorem, $\mathbf{H}$ generates a strongly continuous contraction semigroup of self-adjoint operators $(e^{-t \mathbf{H} } )_{t\geq 0}$ on $L^2(\R\times\Omega_\T)$.

If we let $\mc{D}(\mc{Q})'$ be the dual to $\mc{D}(\mc{Q})$ (i.e. the space of bounded conjugate linear functionals on $\mc{D}(\mc{Q})$), the injection $L^2(\R\times \Omega_\T)\subset \mc{D}(\mc{Q})'$ is continuous and the operator ${\bf H}$ can be extended as a bounded isomorphism 
\[{\bf H}:\mc{D}(\mc{Q})\to \mc{D}(\mc{Q})'.\]

\begin{remark}
By adapting some argument in \cite{rs1,rs2} one can prove that ${\bf H}$ is essentially self-adjoint on $\mathcal{D}({\bf H}^0)\cap \mathcal{D}(e^{\gamma c}V)$   for $\gamma\in (0,1)$, condition that ensures that the potential $V$ is in $L^2(\Omega_\T)$.
\end{remark}

\section{Scattering of the Liouville Hamiltonian}\label{sec:scattering}

In this section,   we develop the scattering theory for the operator $\mathbf{H}$ on $L^2(\R\times \Omega_\T)$ with underlying measure $dc\otimes \P_\T$ (where  $\mathbf{H}$ is the generator of the dilation semigroup studied above). This operator has continuous spectrum and cannot be diagonalized with a complete set of $L^2(\R\times \Omega_\T)$-eigenfunctions. 
We will rather use a stationary approach for this operator, in a way similar to what has been done in geometric scattering theory for manifolds with cylindrical ends in \cite{Gui, Mel}. The goal is to obtain a spectral resolution for $\mathbf{H}$ in terms of generalized  eigenfunctions, which  will be shown to be analytic in the spectral parameter. In other words, 
we search to write the spectral measure of $\mathbf{H}$ using these generalized eigenfunctions, which are similar to plane waves $(e^{i\la \omega.x})_{\la\in \R,\omega\in S^{n-1}}$ in Euclidean scattering for the Laplacian $\Delta_x$ on $\R^n$. In our case, the generalized eigenfunctions will be functions in weighted spaces of the form $e^{-\beta c_-}L^2(\R\times \Omega_\T)$ for $\beta>0$ with particular asymptotic expansions at $c=-\infty$. Let us explain briefly the simplest one, corresponding to the functions $\Psi_{\alpha}:=\Psi_{\alpha,{\bf 0},{\bf 0}}=U(V_\alpha(0))$ defined in the Introduction and represented probabilistically by \eqref{defvalpha} using the unitary map \eqref{udeff} when $\alpha<Q$ is real. 
For $\alpha \in (Q-\gamma/2,Q)$, they will be the only eigenfunctions of ${\bf H}$ in $e^{-\beta c_-}L^2(\R\times \Omega_\T)$ for $\beta>Q-\alpha$ satisfying 
\[ ({\bf H}-2\Delta_\alpha)\Psi_\alpha=0, \quad \Psi_{\alpha}=e^{(\alpha-Q)c}+ u, \quad \textrm{ with } u\in L^2(\R\times \Omega_\T).\]
Similarly, we will construct (in Section \ref{sub:holomorphic}) some eigenfunctions $\Psi_{\alpha,{\bf k},{\bf l}}$ of ${\bf H}$ with eigenvalue $2\Delta_\alpha+\la_{{\bf k}{\bf l}}$, with $\Psi_{\alpha,{\bf k},{\bf l}}|_{\R^+}\in L^2$ and asymptotic (for all $\beta\in(0,\gamma/2)$)
\[\Psi_{\alpha,{\bf k},{\bf l}}|_{\R^-}=e^{(\alpha-Q)c}\psi_{{\bf k}{\bf l}}+ u, \quad \textrm{ with }u\in e^{(\alpha-Q+\beta)c}L^2(\R^-\times \Omega_\T),\]
where we recall that $\la_{{\bf kl}}\in \N$ are the eigenvalues of $P$ defined in \eqref{firstlength} and the corresponding eigenfunctions $\psi_{{\bf kl}}$ of $P$ appear in \eqref{fbasishermite}. 
One way to construct the $\Psi_{\alpha}$, and similarly for $\Psi_{\alpha,{\bf k},{\bf l}}$, is to take the limit (see Proposition \ref{Pellproba})
\[ \Psi_\alpha= \lim_{t\to \infty}e^{-t{\bf H}}(e^{2t\Delta_\alpha}e^{(\alpha-Q)c}),\] 
where we observe that $e^{-t{\bf H}^0}e^{(\alpha-Q)c}=e^{-2t\Delta_\alpha}e^{(\alpha-Q)c}$ so that, formally speaking, $\Psi_\alpha$ is the limit of the intertwining $e^{-t{\bf H}}e^{t{\bf H}^0}(e^{(\alpha-Q})c)$ as $t\to +\infty$.
An alternative expression is to write them as 
\begin{equation}\label{Psialphawithresolvent} 
\Psi_\alpha= e^{(\alpha-Q)c}\chi(c)- ({\bf H}-2\Delta_\alpha)^{-1}({\bf H}-2\Delta_\alpha)(e^{(\alpha-Q)c}\chi(c))
\end{equation}
where $\chi\in C^\infty(\R)$ equal to $1$ near $-\infty$ and $0$ near $+\infty$ (see \eqref{definPellproba}  and Lemma \ref{firstPoisson}); here we notice that 
 $e^{(\alpha-Q)c}\chi(c)$ is not $L^2(\R\times \Omega_\T)$ but one can check that $({\bf H}-2\Delta_\alpha)(e^{(\alpha-Q)c}\chi(c))\in L^2$ if $\gamma > Q-\alpha$ (and $\gamma<1$) so that we can apply the resolvent ${\bf R}(\alpha):=({\bf H}-2\Delta_\alpha)^{-1}$ to it, and  $({\bf H}-2\Delta_\alpha)^{-1}({\bf H}-2\Delta_\alpha)(e^{(\alpha-Q)c}\chi(c))\in L^2$ is not equal to $(e^{(\alpha-Q)c}\chi(c))$.
Our goal will be to extend analytically these $\Psi_\alpha$ (and actually the whole family of generalized eigenstates denoted by $\Psi_{\alpha,\k,\l}$) to ${\rm Re}(\alpha)\leq Q$ and in particular to the line $\alpha\in Q+i\R$ corresponding to the spectrum of ${\bf H}$. To perform this, we see that we need to extend analytically the resolvent operator ${\bf R}(\alpha)$ to ${\rm Re}(\alpha)\leq Q$, which will be the main part of this section. In fact, we shall show that ${\bf R}(\alpha)$ extends analytically on an open set of a Riemann surface covering the complex plane, containing the real half-plane ${\rm Re}(\alpha)\leq Q$. We note that the functions $\Psi_{\alpha,\k,\l}$ will be expressed as the elements in the range of some Poisson operator denoted $\mc{P}(\alpha)$, mapping (some subspaces of) $L^2(\Omega_\T)$ to weighted spaces $e^{-\beta c_-}L^2(\R\times \Omega_T)$ for some $\beta>0$ depending on $\alpha$. The results proved here hold in some cases for geometric scattering in finite dimension \cite{Gui, Mel}, but we are not aware of some results of this type in quantum field theory where the base space is infinite dimensional. 
The main difficulty will be to deal with the fact that the perturbation $V$ is quite singular (even for $\gamma<\sqrt{2}$ where $V$ is a potential, it is not bounded) and the fact that the eigenfunctions of ${\bf P}$ (Hermite polynomials) have $L^p( \Omega_\T)$ norms which blow up 
very fast in terms of their eigenvalues.

In this section, we shall start by describing the resolvent of ${\bf H}$ in the \emph{probabilistic region} $\{{\rm Re}((\alpha-Q)^2)>\beta^2\}$ acting on weighted spaces $e^{\beta c_-}L^2$ for 
$\beta\in\R$ and deduce the construction of the $\Psi_{\alpha,\k,\l}$ in this region. Next, 
we will show that the resolvent ${\bf R}(\alpha)$ admits an analytic extension in a neighborhood of $\{{\rm Re}(\alpha)\leq Q\}$ (for $\alpha$ on some Riemann suface $\Sigma$). We shall use these results to prove the analytic continuation of the $\Psi_{\alpha,\k,\l}$ to ${\rm Re}(\alpha)\leq Q$ and we shall finally construct the scattering operator ${\bf S}(\alpha)$ in Section \ref{section_scattering} and write the spectral decomposition of ${\bf H}$ in terms of the $\Psi_{\alpha,\k,\l}$ in Theorem \ref{spectralmeasure} (written in terms of Poisson operator in this section).\\

In what follows, we will mostly consider the $L^2$ (or $L^p$) spaces on $\Omega_\T$ or on 
$\R\times \Omega_\T$ respectively equipped with the measure $\P_\T$ or $dc\otimes \P_\T$, which we will denote by $L^2(\Omega_\T)$ or $L^2(\R\times \Omega_\T)$ for short. When the space is omitted, i.e. we simply write $L^2$, this means that we consider $L^2(\R\times \Omega_\T)$: this will relieve notations in some latter part of the paper.
Recall that we denote by $\langle \cdot\,|\,\cdot\rangle_{2}$ the standard scalar product associated to $L^2(\R\times \Omega_\T, dc \otimes \P_\T)$ and $\|\cdot\|_2$ the associated norm; in general our scalar products will always be complex linear in the left component and anti-linear in the right component.
Given two normed vector space $E$ and $F$, the space of continuous linear mappings from $E$ into $F$ will be denoted by $\mathcal{L}(E,F)$ and when $E=F$ we will simply write $\mathcal{L}(E)$. The corresponding operator norms will be denoted by $\|\cdot\|_{\mathcal{L}(E,F)}$ or $\|\cdot\|_{\mathcal{L}(E)}$.

\subsection{The operators ${\bf H},{\bf H}^0, e^{\gamma c}V$}
The operator ${\bf H}$ is made up of several pieces.  The first piece is the operator $\mathbf{P}$ defined in \eqref{hdefi}, which is a self-adjoint non-negative unbounded operator on $L^2(\Omega_\T)$. It has discrete spectrum $(\la_{\k\l})_{\k,\l\in \mc{N}}$, but to simplify the indexing we shall order them in increasing order (without counting multiplicity) and denote them by 
\[ \sigma(\mathbf{P})=\{ \lambda_j \,|\, j\in\N, \lambda_{j}<\lambda_{j+1}\}.\] 
We denote by 
\begin{equation}
 E_k:=\{F\in L^2(\Omega_\T)\,|\, 1_{[0,\la_k]}(\mathbf{P})F=F\}=\bigoplus_{j\leq k}\ker(\mathbf{P}-\la_j)
\end{equation}
the sum of eigenspaces with eigenvalues less or equal to $\lambda_k$ and $\Pi_{k}: L^2(\Omega_\T)\to E_k$ the orthogonal projection. We also will use the important fact 
\begin{equation}\label{LpEk}
E_k \subset L^p(\Omega_\T) , \quad \forall p<\infty.
\end{equation}
The quadratic form associated to ${\bf H}$ is the form $\mc{Q}$ defined in Subsection  \ref{sub:bilinear}, with domain $\mc{D}(\mc{Q})$. 
We will consider the self-adjoint extension associated of ${\bf H}$ obtained using the quadratic form $\mc{Q}$. The operator ${\bf H}^0:\mc{D}({\bf H}^0)\to L^2$ corresponds to the case $\mu=0$,
its quadratic form has domain $\mc{D}(\mc{Q}_0)$ containing $\mc{D}(\mc{Q})$. 
We recall that 
\[ {\bf H}: \mc{D}(\mc{Q})\to \mc{D}'(\mc{Q}), \quad {\bf H}^0: \mc{D}(\mc{Q})\to \mc{D}'(\mc{Q})\]
if $\mc{D}'(\mc{Q})$ is the dual 
of $\mc{D}(\mc{Q})$.
We then define $e^{\gamma c}V:\mc{D}(\mc{Q})\to \mc{D}'(\mc{Q})$ by the equation 
\begin{equation}\label{definitionofH}
{\bf H}={\bf H}^0+e^{\gamma c}V= -\frac{1}{2} \pl_c^2+\frac{1}{2} Q^2+{\bf P}+e^{\gamma c}V. 
\end{equation} 
We will define the associated quadratic form 
\[\mc{Q}_{e^{\gamma c}V} :=\mc{Q}-\mc{Q}_0\]
defined on $\mc{D}(\mc{Q})$ and remark that 
\[ \mc{Q}_{e^{\gamma c}V}(u,u)=\int_{\R}e^{\gamma c}\mc{Q}_V(u,u)dc \]
for some well-defined quadratic form $\mc{Q}_V$, defined on a domain $\mc{D}(\mc{Q}_V)\subset L^2(\Omega_\T)$ containing $E_k$ for each $k\geq 0$.  

For our study of the  resolvent of ${\bf H}$, we need to introduce the orthogonal projection 
\[\Pi_k:L^2(\Omega_\T)\to E_k,\] 
 extended trivially in the $c$ variable as $\Pi_k:L^2(\R\times \Omega_\T)\to L^2(\R;E_k)$. 
\begin{lemma}\label{chiPikV}
Let $\chi\in L^\infty(\R)\cap C^\infty(\R)$ with support in $(-\infty,A)$ for some $A\in\R$ and $\chi'\in L^\infty(\R)$ . Then, for all $\beta\geq -\gamma/2$ and $\beta'\in\R$, the following operators are bounded
\[ \chi \Pi_k: \mc{D}(\mc{Q})\to \mc{D}(\mc{Q}) , \quad \chi \Pi_k: \mc{D}'(\mc{Q})\to \mc{D}'(\mc{Q})\]
\[
 \chi e^{\gamma c}V\Pi_k : e^{\beta'c}L^2(\R\times \Omega_\T)\to e^{(\beta'-\beta) c}\mc{D}'(\mc{Q}), 
\]
\[\chi e^{\gamma c}\Pi_k V: e^{(\beta-\beta') c}\mc{D}(\mc{Q})\to e^{-\beta' c}L^2(\R;E_k),\]
\[ \chi \Pi_k e^{\gamma c}V\Pi_k: e^{\beta'c}L^2(\R\times \Omega_\T)\to e^{(\beta'-\beta) c}L^2(\R;E_k).\]
If $\beta> -\gamma/2$ then one also has
\[\chi e^{\gamma c}V\Pi_k : e^{\beta'c}L^\infty(\R\times \Omega_\T)\to e^{(\beta'-\beta) c}\mc{D}'(\mc{Q}).\]
\end{lemma}
\begin{proof}
For the first operator $\chi \Pi_k$, it suffices to check that for $u\in L^2((-\infty,A);E_k)$, $\mc{Q}_{e^{\gamma c}V}(u)\leq C_k\|u\|^2_{L^2}$. Since there is $C_k>0$ depending on $k$ such that for all $F\in E_k$, 
$\mc{Q}_{V}(F)\leq C_k\|F\|_{L^2(\Omega_\T)}^2$, we have 
\[ \mc{Q}_{e^{\gamma c}V}(u)\leq C_k \int_{-\infty}^A e^{\gamma c}\|u\|^2_{L^2(\Omega_\T)}dc\leq C_{k,A}\|u\|_{L^2(\R\times \Omega_\T)}^2.\] 
The extension of $\chi \Pi_k$ to $\mc{D}'(\mc{Q})$ follows by duality, using that $\Pi_k^*=\Pi_k$ on $L^2$. 
Next, we analyze $\chi e^{\gamma c}V\Pi_k$. It suffices to deal with the case $\beta'=0$ since $e^{\beta'c}$ commutes with $e^{\gamma c}V$. 
Using that for $F\in L^2(\R;E_k)$ and $\chi\in L^\infty(\R)$ with support in $\R^-$, we have by Cauchy-Schwarz that for all $u\in \mc{D}(\mc{Q})$
\[ \begin{split}
|\cjg \chi(c) e^{(\gamma+\beta) c}V F,u\cjd|\leq 
\int_{-\infty}^A \chi(c)e^{(\gamma +\beta) c} |\mc{Q}_V(F,u)| dc & \leq \|\chi\|_{L^\infty}
 \Big(\int_{-\infty}^A e^{\gamma c}\mc{Q}_V(u)dc\Big)^{\frac{1}{2}}
 \Big( \int_{-\infty}^Ae^{(\gamma+2\beta) c}\mc{Q}_V(F)dc\Big)^\frac{1}{2}\\
 \leq  C_{k,A}\|\chi\|_{L^\infty}\|F\|_{L^2}\mc{Q}(u)^{1/2} 
\end{split}\]
for some constant $C_{k,A}>0$, provided $\gamma+2\beta\geq 0$. The same argument also holds for $F\in L^\infty(\R;E_k)$.
Finally, $\chi \Pi_k e^{\gamma c}V=\chi e^{\gamma c} \Pi_k V$ makes sense as a map $e^{\beta c}\mc{D}(\mc{Q})\to e^{\beta c}\mc{D}'(\mc{Q})$ using the boundedness 
$\Pi_k: \mc{D}'(\mc{Q})\to \mc{D}'(\mc{Q})$. To prove that it actually maps to $L^2$ if $\gamma+2\beta>0$, we note that for all $u'\in \mc{C}$
\[ |\cjg \chi \Pi_k e^{\gamma c}Vu,u'\cjd|=|\cjg \chi e^{\gamma c}Vu,\Pi_k u'\cjd|=\Big| \int_{-\infty}^A
\chi(c)e^{\gamma c}\mc{Q}_{V}(u,\Pi_k(u'))dc\Big|\leq C_{k,A}\|\chi\|_{L^\infty}\mc{Q}(e^{-\beta c}u)^{1/2}\|u'\|_{L^2}\]
where the last bound is obtained as above by Cauchy-Schwarz and the bounds on $\mc{Q}_V$ on $E_k$.
For the last term $\chi \Pi_k V\Pi_k$, we take $u\in L^2(\R\times \Omega)$ and note that for $u'\in \mc{D}(\mc{Q})$
\[\begin{split} 
|\cjg e^{\beta c}\Pi_k V\Pi_k u,u'\cjd|=& |\cjg \chi e^{\beta c}\Pi_k V\Pi_k u, \Pi_k u'\cjd| 
\leq \Big| \int_{-\infty}^A
\chi(c)e^{(\gamma+\beta) c}\mc{Q}_{V}(\Pi_ku),\Pi_ku')dc\Big|\\
\leq &  \Big| \int_{-\infty}^A
\chi(c)e^{(\gamma+\beta) c}\mc{Q}_V(\Pi_k u)^{\frac{1}{2}}\mc{Q}_V(\Pi_k u')^{\frac{1}{2}}dc\Big|\leq C_{k,A}\|u\|_{L^2}\|u'\|_{L^2}
\end{split}\]
showing that $\|e^{\beta c}\Pi_k V\Pi_k u\|_2\leq C_{k,A}\|u\|_2$.
\end{proof}
First, we show a useful result for the spectral decomposition. 
\begin{lemma}\label{embedded}
The operator ${\bf H}$ does not have non-zero eigenvectors $u\in \mc{D}({\bf H})$. If $\gamma\in (0,1)$, the spectrum of $\mathbf{H}$ is given by $\sigma(\mathbf{H})=[\tfrac{Q^2}{2},\infty)$ and consists of essential spectrum.
\end{lemma}
\begin{proof}  In the case $\gamma\in (0,1)$, the space $\mc{C}$ is included in $\mc{D}({\bf H})$. 
It is then easy to check that $\sigma(\mathbf{H})=[\tfrac{Q^2}{2},\infty)$ consists only of essential spectrum by using Weyl sequences $(e^{ipc}\chi(2^{-n}c)/\omega_n)_n\in\N$ where $\chi \in C_c^\infty(\R)$ have support in $[-\frac{3}{2},-1]$ and equal to $1$ on some interval, with $\omega_n=\|\chi( 2^{-n}\cdot)\|_{L^2(\R)}$.  

Let $u\in \mc{D}(\mathbf{H})$ such that $\mathbf{H}u=\la u$ with $\la \in [\tfrac{Q^2}{2},\infty)$. Then $  u\in \mc{D}(\mc{Q})$ (hence $\partial_cu\in L^2$), and it satisfies $\mathcal{Q}(u,v)=\langle \lambda u\, |\,v\rangle_2$ for all $v\in \mc{D}(\mc{Q})$.
Now we claim
\begin{lemma}\label{cderiv}
Assume we are given $f\in L^2$ such that $\partial_cf\in L^2$. Consider $u\in \mc{D}(\mc{Q})$ such that 
\begin{equation}\label{dirF}
\mathcal{Q}(u,v)=\langle   f\, |\, v\rangle_2,\quad \forall v\in \mc{D}(\mc{Q}).
\end{equation}
Then $\partial_cu\in \mc{D}(\mc{Q})$ and 
\begin{equation}\label{dirFbis}
\mathcal{Q}(\partial_cu,v)=\langle  \partial_c f\, |\, v\rangle_2-  \gamma   \int e^{\gamma c}\int_0^{2\pi}\E[ u_{\rm shift}v_{\rm shift}]\dd \theta\dd c,\quad \forall v\in \mc{D}(\mc{Q}).
\end{equation}
\end{lemma}
We postpone the proof of this lemma  and conclude first. Consider next \eqref{dirF} with $f=\lambda u$ and choose $v=\partial_c u\in \mc{D}(\mc{Q})$ to obtain $\mathcal{Q}(u,\partial_cu)=\langle   \lambda u\,|\, \partial_cu\rangle_2=0$. Also, choosing $v=u$  in \eqref{dirFbis} we obtain $\mathcal{Q}(\partial_cu,u)=\langle  \lambda \partial_c u\,|\, u\rangle_2-\gamma\int e^{\gamma c} \int_0^{2\pi}\E|u_{\rm shift}|^2\dd\theta \dd c$. These relations imply $\gamma\int e^{\gamma c} \int_0^{2\pi}\E|u_{\rm shift}|^2\dd\theta \dd c=0$. In the case when $\gamma\in (0,\sqrt{2})$ then $V$ exists as a fairly defined function and this relation translates into $\|e^{\gamma c/2}V^{1/2}u\|_2=0$. Hence $u=0$ as $V>0$ almost surely. In the general situation $\gamma\in (0,2)$, the argument is as follows: the relation  $\gamma\int e^{\gamma c} \int_0^{2\pi}\E|u_{\rm shift}|^2\dd\theta \dd c=0$ implies that $\mathcal{Q}(u,v)=\mathcal{Q}_0(u,v)$ for all $v\in \mathcal{C}$. Therefore, $u\in \mathcal{D}(\mathbf{H}_0)$ and $u$ is an eigenfunction $\mathbf{H}_0$, hence $u=0$ (cf. Remark \ref{discreteH0}). 
\end{proof}

{\it Proof of Lemma \ref{cderiv}.} For $h>0$, introduce the translation operator $T_h:L^2\to L^2$ by $T_hv:=v(c+h,\cdot)$ and the discrete derivative operator $D_h: L^2\to L^2$ by $D_hv:=(T_hv-v)/h$. Note that $T_h$ maps $ \mc{D}(\mc{Q})$ into itself, that $\|D_hv\|_2\leq \|\pl_cv\|_2 $, $D_hu_{\rm shift}=(D_hu)_{\rm shift}$ and we have the discrete IPP $\langle  D_hu \,|\,v\rangle_2=-\langle   u \,|\,D_{-h}v\rangle_2$ for all $u,v\in L^2$.  Now we can replace $v$ by 
$D_{-h}v$ in \eqref{dirF} and use discrete IPP to obtain
\begin{align}\label{derivFD}
\mathcal{Q}(D_hu,v)=&\langle  D_h f,v\rangle_2   -\int\int_0^{2\pi}\E[(T_he^{\gamma c}-e^{\gamma c})D_hu_{\rm shift}v_{\rm shift}]\dd \theta\dd c  \\
      & -\int\int_0^{2\pi}\E[D_h(e^{\gamma c} )u_{\rm shift}v_{\rm shift}]\dd \theta\dd c,\quad \forall v\in \mc{D}(\mc{Q}).
\end{align}
Next we choose $v=D_hu$ and, using  the inequality $ab \leq \tfrac{\epsilon}{2}a^2+ \tfrac{1}{2\epsilon}b^2 $ for arbitrary $\epsilon>0$, we obtain the a priori estimate  (for some constant $C>0$ depending only on $\gamma$)
\begin{equation}\label{aprioriestQDh}
\begin{split} 
\mathcal{Q}(D_hu,D_hu) \leq & \frac{1}{2}(\|D_hf\|_2^2+\|D_hu\|_2^2)-(e^{\gamma h}-1)\int\int_0^{2\pi}\E( e^{\gamma c}|D_hu_{\rm shift}|^2)\dd \theta\dd c 
\\
&+\frac{1}{2}\Big|\frac{e^{\gamma h}-1}{h}\Big|\Big(\epsilon^{-1}\int\int_0^{2\pi}\E( e^{\gamma c}|u_{\rm shift}|^2)\dd \theta\dd c +\epsilon\int\int_0^{2\pi}\E( e^{\gamma c}|D_hu_{\rm shift}|^2)\dd \theta\dd c \Big)\\
\leq &( C+\epsilon^{-1}) \big(\|\partial_c f\|_2^2+ \mathcal{Q}(u,u)\big)+C(\eps+h)\mathcal{Q}(D_hu,D_hu)
\end{split}\end{equation}
for all $h>0$ small, and therefore the last term can be absorded in the left hand side if $\eps,h>0$ are small enough.
Then, writing \eqref{derivFD} for $h$ and $h'$, subtracting and then choosing $v=D_hu-D_{h'}u$, we find 
\[ \mathcal{Q}(D_hu-D_{h'}u,D_hu-D_{h'}u) \leq C \big(\|D_hf-D_{-h'} f\|_2^2+ |h-h'|\mc{Q}(u,u)+ |h-h'|\mc{Q}(D_{(h-h')}u,D_{(h-h')}u)\big).\]
Using \eqref{aprioriestQDh} with $h$ replaced by $h-h'$ to bound the last term, we obtain that the sequence $(D_hu)_h$ is Cauchy for the $ \mathcal{Q}$-norm. Hence the limit  $\partial_cu$ belongs to $\mathcal{D}( \mathcal{Q})$.
\qed

\begin{remark}When $\gamma\in [1,2)$ the spectrum is also $[Q^2/2,\infty)$ and is made of essential spectrum, and this will follow from our analysis of the resolvent: in fact we will show for $\gamma \in (0,2)$ that the spectrum of ${\bf H}$ is absolutely continuous.
\end{remark}
  
\subsection{Resolvent of $\mathbf{H}$}
To describe the spectral measure of ${\bf H}$ and construct its generalized eigenfunctions, 
the main step is to understand the resolvent of $\mathbf{H}$ as a function of the spectral parameter, in particular when the spectral parameter approach the spectrum. 
Due to the fact that the spectrum of ${\bf H}$ starts at $Q^2/2$, it is convenient to use the spectral parameter $2\Delta_\alpha$ where $\Delta_{\alpha}= \frac{\alpha}{2}(Q-\frac{\alpha}{2})$ and $\alpha \in \C$. That way we have with $\alpha=Q+ip$
\[ {\bf H}-2\Delta_\alpha= -\tfrac{1}{2}\pl_c^2 + {\bf P}+e^{\gamma c}V-\tfrac{1}{2}p^2\]
where $p\in\R$ plays the role of a frequency: in particular $2\Delta_\alpha\in [Q^2/2,\infty)$ if and only if $p\in \R$. The half-plane $\{\alpha\in\C \,|\,{\rm Re}(\alpha)<Q\}$ is mapped by $\Delta_\alpha$ to the resolvent set $\C\setminus [Q^2/2,\infty)$ of ${\bf H}$ and will be called the \emph{physical sheet}. By the spectral theorem 
\[ \mathbf{R}(\alpha)=(\mathbf{H}-2\Delta_{\alpha})^{-1}: L^2(\R\times \Omega_\T)\to \mc{D}(\mathbf{H})\]
is bounded if ${\rm Re}(\alpha)<Q$. Our goal is to extend this resolvent up to the line ${\rm Re}(\alpha)=Q$ analytically, and we will actually do it in an even larger region. The price to pay 
is that $\mathbf{R}(\alpha)$ will not be bounded on $L^2$ but rather on certain weighted $L^2$ spaces, where the weights are $e^{\beta c}$ in the region $c\leq -1$, with $\beta\in \R$ tuned with respect to $\alpha$.

\subsubsection{Resolvent and propagator on weighted spaces in the probabilistic region.}
Our first task is to understand the resolvent on weighted spaces in a subregion of ${\rm Re}(\alpha)<Q$, that we call the \emph{probabilistic region} due to the fact that the resolvent can be written in terms of the semigroup $e^{-t{\bf H}}$.
 
Let $\rho:\R\to \R$ be a smooth non-decreasing function satisfying 
\[\rho(c)=c+a \textrm{ for }c\leq -1,\quad  \rho(c)=0 \textrm{ for }c\geq 0, \quad 0\leq \rho'\leq 1\]
for some $a\in \R$. We have for $\beta\geq 0$ the inclusion of weighted spaces
\[e^{\beta \rho(c)}L^2(\R\times \Omega_\T)\subset L^2(\R\times \Omega_\T)\subset e^{-\beta \rho(c)}L^2(\R\times \Omega_\T).\]
The weighted spaces $e^{\beta \rho(c)}L^2(\R\times \Omega_\T)$ are obviously Hilbert spaces with product $\cjg u,v\cjd_{e^{\beta \rho}L^2}:=\cjg e^{-\beta\rho}u,e^{-\beta \rho}v\cjd_2$.

\begin{lemma}\label{resolventweighted}
Let $\beta\in\R$. If ${\rm Re}((\alpha-Q)^2)>\beta^2$ and ${\rm Re}(\alpha)<Q$, the resolvent $\mathbf{R}(\alpha)=(\mathbf{H}-2\Delta_{\alpha})^{-1}$ extends to a bounded operator 
\begin{align*} 
&\mathbf{R}(\alpha): e^{-\beta \rho}L^2(\R\times \Omega_\T)\to e^{-\beta \rho}\mc{D}(\mathbf{H}),\\
&\mathbf{R}(\alpha): e^{-\beta \rho}\mc{D}'(\mc{Q})\to e^{-\beta \rho}\mc{D}(\mc{Q})
\end{align*}
which is analytic in $\alpha$ in this region. The operator $\mathbf{H}: e^{-\beta \rho}L^2\to e^{-\beta \rho}L^2$ is closed with domain $e^{-\beta \rho}\mc{D}(\mathbf{H})$, it is a bijective mapping  
$e^{-\beta \rho}\mc{D}(\mathbf{H})\to e^{-\beta \rho}L^2$ with inverse $\mathbf{R}(\alpha)$.
Moreover, for $\alpha\in (-\infty,Q)$ and $0\leq \beta<Q-\alpha$, the resolvent is bounded with norm $\|\mathbf{R}(\alpha)\|_{\mc{L}(e^{-\beta\rho}L^2)}\leq 2((\alpha-Q)^2-\beta^2)^{-1}$ and is equal to the integral
\begin{equation}\label{resolvvspropagator} 
\mathbf{R}(\alpha)=\int_0^{\infty} e^{-t\mathbf{H}+t2 \Delta_{\alpha}}dt
\end{equation}
where $e^{-t\mathbf{H}}$ is the semigroup on $e^{-\beta\rho}L^2$ obtained by Hille-Yosida theorem with norm \begin{equation}\label{normetHweight}
\forall t\geq 0, \quad \|e^{-t\mathbf{H}}\|_{\mc{L}(e^{-\beta\rho}L^2)}\leq e^{-t\frac{Q^2-\beta^2}{2}}.
\end{equation} 
The integral \eqref{resolvvspropagator} converges in $\mc{L}(e^{-\beta\rho}L^2)$ operator norm and $e^{-t\mathbf{H}}:e^{-\beta \rho}L^2\to e^{-\beta \rho}L^2$ extends the semigroup defined in   \eqref{hstar}. Finally, $e^{-t{\bf H}}:e^{\beta \rho}\mc{D}(\mc{Q})\to e^{\beta \rho}\mc{D}(\mc{Q})$ and $e^{-t{\bf H}}:e^{\beta \rho}\mc{D}'(\mc{Q})\to e^{\beta \rho}\mc{D}'(\mc{Q})$ are bounded and for each $\eps>0$ there is some constant $C_\eps>0$ such that for all $t>0$
\begin{equation}\label{normetHweight2}
\|e^{-t{\bf H}}\|_{\mc{L}(e^{\beta \rho}\mc{D}(\mc{Q}))}+\|e^{-t{\bf H}}\|_{\mc{L}(e^{\beta \rho}\mc{D}'(\mc{Q}))}\leq C_\epsilon e^{-t(\frac{Q^2-\beta^2}{2}-\eps)}.
\end{equation}
\end{lemma}

\begin{proof}
Consider the operator for $\beta\in \R$, acting on the space $\mc{C}$ (defined in \eqref{core}) and with value in $e^{\beta \rho}\mc{D}'(\mc{Q})$, 
\[ \mathbf{H}_{\beta}:= e^{\beta \rho(c)}\mathbf{H}e^{-\beta \rho(c)}=\mathbf{H} -\tfrac{\beta^2}{2}(\rho'(c))^2+\tfrac{\beta}{2} \rho''(c)+\beta \rho'(c)\pl_{c}.\]

Let $u\in \mc{C}$, then we have (using integration by parts)
\begin{equation}\label{coercive_estimate} 
\begin{split}
{\rm Re}\cjg \mathbf{H}_\beta u\, |\, u\cjd_2 & =\mc{Q}(u)-\tfrac{\beta^2}{2}\|\rho'u\|_2^2+\tfrac{\beta}{2} \cjg \rho''u\,|\, u\cjd_2
+\beta{\rm Re}(\cjg \rho'\pl_cu\,|\, u\cjd_2)\\
&= \mc{Q}(u)-\tfrac{\beta^2}{2}\|\rho'u\|_2^2\\
&\geq  \mc{Q}(u)-\tfrac{\beta^2}{2}\|u\|^2_2 \geq\tfrac{Q^2-\beta^2}{2}\|u\|^2_2+\tfrac{1}{2}\|\pl_cu\|^2_2+\|\mathbf{P}^{1/2}u\|^2_2+\mc{Q}_{e^{\gamma c}V}(u,u),
\end{split}
\end{equation} 
where $\mc{Q}$ was defined in Section \ref{sub:bilinear}.
Consider the sesquilinear form $\mc{Q}_{\alpha,\beta}(u,v):=\cjg (\mathbf{H}_{\beta}-2\Delta_\alpha)u\,|\,v\cjd_2$ defined on 
$\mc{C}$. We easily see that if $-{\rm Re}(2\Delta_{\alpha})>\tfrac{\beta^2-Q^2}{2}$, then 
\[\{u \in L^2(\R\times\Omega_\T)\,|\, \mc{Q}_{\alpha,\beta}(u,u)<\infty\}=\mc{D}(\mc{Q}).\]
Let $\mc{D}'(\mc{Q})$ be the dual of $\mc{D}(\mc{Q})$ (note that $L^2\subset \mc{D}'(\mc{Q})$).
By Lax-Milgram, if  $-{\rm Re}(2\Delta_{\alpha})>\tfrac{\beta^2-Q^2}{2}$, then for each $f'\in \mc{D}'(\mc{Q})$, there is a unique $u\in \mc{D}(\mc{Q})$ such that 
\begin{equation}\label{LaxMilram} 
\forall v\in \mc{D}(\mc{Q}), \quad \, \mc{Q}_{\alpha,\beta}(u,v) =f'(v), \quad \mc{Q}(u)^{1/2}\leq C'\|f'\|_{\mc{D}'(\mc{Q})}
\end{equation}
for all $v\in \mc{D}(\mc{Q})$, where $C'>0$ depends only on ${\rm Re}(2\Delta_{\alpha})$ and $\beta^2$. This holds in particular for the linear form $f':v\mapsto \cjg f\,|\, v\cjd_{2}$ 
with norm $\|f'\|_{\mc{D}'(\mc{Q})}\leq C\|f\|_2$ for some $C>0$ depending on ${\rm Re}(2\Delta_{\alpha})$ and $\beta^2$.
 We define $\til{\mathbf{R}}(\alpha)(e^{-\beta\rho}f):=e^{-\beta \rho}u$, 
this gives a bounded linear operator 
\begin{equation}\label{boundontildeR} 
\til{\mathbf{R}}(\alpha): e^{-\beta \rho}\mc{D}'(\mc{Q})\to e^{-\beta \rho}\mc{D}(\mc{Q})\subset e^{-\beta \rho}L^2
\end{equation}
inverting the bounded operator $e^{\beta\rho}(\mathbf{H}-2\Delta_{\alpha})e^{-\beta\rho}:\mc{D}(\mc{Q})\to \mc{D}'(\mc{Q})$. Moreover, by \eqref{coercive_estimate}, its weighted $L^2$-norm is bounded by
\begin{equation}\label{boundtildeRalpha}
\|\til{\mathbf{R}}(\alpha)\|_{\mc{L}(e^{-\beta\rho}L^2)}\leq 2({\rm Re}((\alpha-Q)^2-\beta^2))^{-1}=(\tfrac{Q^2-\beta^2}{2}-2{\rm Re}(\Delta_\alpha))^{-1}
\end{equation} 
Using $\mc{D}(\mc{Q})\subset e^{-\beta \rho}\mc{D}(\mc{Q})$ and the uniqueness property above,
this means that for $f\in L^2$, we have $\mathbf{R}(\alpha)f=\til{\mathbf{R}}(\alpha)f$ and thus $\til{\mathbf{R}}(\alpha)$ is a continuous extension of $\mathbf{R}(\alpha)$ to the Hilbert space $e^{-\beta \rho}L^2$. The analyticity in $\alpha$ comes from Lax-Milgram, but can also alternatively be obtained 
by Cauchy formula (for $\eps>0$ small)
\[ \til{\mathbf{R}}(\alpha)f= \frac{1}{2\pi i}\int_{|z-\alpha|=\eps} \frac{\til{\mathbf{R}}(z)f}{z-\alpha}dz\]
which holds for all $f\in \mc{C}$ (since $\til{\mathbf{R}}(\alpha)f=\mathbf{R}(\alpha)f$ for such $f$), and can then be extended to $e^{-\beta \rho}L^2$ by density of $\mc{C}$ in $e^{-\beta \rho}L^2$. The domain 
$\mc{D}(e^{\beta \rho}{\bf H}e^{-\beta \rho})=\{u\in \mc{D}(\mc{Q})\,|\,e^{\beta \rho}{\bf H}e^{-\beta \rho}u\in L^2\}$ of the operator $e^{\beta \rho}\bf{H}e^{-\beta \rho}$ is actually equal to 
$\mc{D}({\bf H})=\{u\in \mc{D}(\mc{Q})\,|\, {\bf H}u\in L^2\}$ since  
\begin{equation}\label{formuleHcommute}
e^{-\beta \rho}\mathbf{H}(e^{\beta \rho}u)=\mathbf{H}u -\tfrac{\beta^2}{2}(\rho'(c))^2u-\tfrac{\beta}{2} \rho''(c)u-\beta \rho'(c)\pl_{c}u
\end{equation} 
with $\rho'\in C_c^\infty(\R)$ (thus $\rho'(c)\pl_cu\in L^2$ for $u\in \mc{D}(\mc{Q})$). The operator ${\bf H}:e^{-\beta\rho}\mc{D}({\bf H})\to e^{-\beta\rho}L^2$ is thus closed.
By Hille-Yosida theorem, there is an associated bounded semigroup $e^{-t\mathbf{H}}$ on $e^{-\beta \rho}L^2$, and by density of $L^2\subset e^{-\beta \rho}L^2$ when $\beta\geq 0$, it is an extension of the $e^{-t\mathbf{H}}$ semigroup on $L^2$ .
Let us check that the resolvent can be written as an integral of the propagator.
For $f\in \mc{C}\subset L^2$, we have 
\[ 
\til{\mathbf{R}}(\alpha)f=\mathbf{R}(\alpha)f=\int_0^{\infty} e^{-t\mathbf{H}+t2\Delta_{\alpha}}f\, dt.
\]
By Hille-Yosida theorem and \eqref{boundontildeR}, we have $\|e^{-t\mathbf{H}}\|_{\mc{L}(e^{-\beta \rho}L^2)}\leq e^{-t\tfrac{Q^2-\beta^2}{2}}$, so that the integral above converges if $Q-\alpha>\beta$ (for $\alpha\in(-\infty,Q)$) as a bounded operator on $e^{-\beta\rho}L^2$, showing the desired claim by density of $\mc{C}$ in $e^{-\beta \rho}L^2$.

We conclude with a $e^{\beta \rho}\mc{D}'(\mc{Q})$ bound for $\tilde{{\bf R}}(\alpha)$ and 
$e^{-t{\bf H}}$. First, we note using \eqref{formuleHcommute} that for $u\in \mc{D}(\mc{Q})$
\[|{\rm Im}(\mc{Q}_{\alpha,\beta}(u,u))|\geq (|{\rm Im}(2\Delta_\alpha)|-\beta^2)\,\|u\|^2_2
-\tfrac{1}{4}\|\pl_cu\|^2_2,\]
which implies
\[|(\mc{Q}_{\alpha,\beta}(u,u)|\geq \frac{1}{\sqrt{2}}\Big(c_{\Delta_\alpha}\|u\|^2+\tfrac{1}{4}\|\pl_cu\|^2_2+\|\mathbf{P}^{1/2}u\|^2_2+\mc{Q}_{e^{\gamma c}V}(u,u)\Big).
\]
\[ c_{z}:= \min\Big(-2{\rm Re}(z)+2|{\rm Im}(z)|+\tfrac{Q^2-3\beta^2}{2}, -2{\rm Re}(z)+\tfrac{Q^2-\beta^2}{2}\Big).\]
This in turn gives that, provided ${\rm Re}(-2\Delta_\alpha)+|{\rm Im}(2\Delta_\alpha)|+\tfrac{Q^2-3\beta^2}{2}>0$,  $\tilde{{\bf R}}(\alpha):e^{-\beta \rho}L^2\to e^{-\beta \rho}L^2$ is also well-defined, analytic in $\alpha$ and bounded, and moreover satisfies for $|\Delta_\alpha|\gg 1$ and ${\rm Re}(\Delta_\alpha)\leq \frac{1}{2}|{\rm Im}(\Delta_\alpha)|$
\[ \|\tilde{{\bf R}}(\alpha)\|_{\mc{L}(e^{-\beta \rho}L^2)}\leq C|\Delta_\alpha|^{-1} \]
where $C>0$ is a uniform constant.
First, for each $\eps>0$, there is $C_\eps>0$ such that for all $f\in e^{-\beta \rho}\mc{D}(\mc{Q})$, $u:={\bf R}(\alpha)f$ 
\[\mc{Q}(u,u)\leq C_{\eps}{\rm Re}(\mc{Q}_{\alpha,\beta}(u,u))\leq C_{\eps}|\cjg u\,|\,f\cjd_2|\leq \frac{C_\eps}{c_{\Delta_\alpha}} \|f\|_2^2 \] 
if $c_{\Delta_\alpha}>0$, thus for $c_z>0$
\begin{equation}\label{boundH-zinverse}
\begin{split}
\|({\bf H}-z)^{-1}\|_{\mc{L}(e^{-\beta\rho}\mc{D}(\mc{Q}))}\leq & C_\eps^{1/2} c_z^{-1/2}\\
 \|({\bf H}-z)^{-1}\|_{\mc{L}(e^{-\beta\rho}L^2)}\leq & C_\eps c_z^{-1}.
 \end{split}
 \end{equation}
Let us consider a contour $\Gamma=\Gamma_0\cup \Gamma_+\cup \Gamma_-\subset \C$ with $a:=\tfrac{Q^2-\beta^2}{2}-\eps$, $\Gamma_{\pm}:=
a\pm iN+e^{\pm 3i\pi/8}\R^+\subset \C$ and $\Gamma_0=a+i[-N,N]$ for some $N>0$ large enough so that $c_z>0$ on $\Gamma$, and where $\Gamma$ oriented clockwise around $[a,+\infty)$. Using the holomorphic functional calculus, we have 
\[ e^{-t{\bf H}}=\frac{1}{2\pi i}\int_{\Gamma_+\cup \Gamma_-}e^{-tz}({\bf H}-z)^{-1}dz\]
and the integral converges both in $\mc{L}(e^{-\beta\rho}L^2)$ and $\mc{L}(e^{-\beta\rho}\mc{D}(\mc{Q}))$ using \eqref{boundH-zinverse}, with bound 
\[ \|e^{-t{\bf H}}\|_{e^{-\beta \rho}\mc{D}(\mc{Q})}\leq Ce^{-ta}\]
for some some $C$ depending only on $\eps>0$. Using duality, this gives \eqref{normetHweight2}.  
\end{proof}
In what follows, we will always write $\mathbf{R}(\alpha)$ for the resolvent, for both the operator acting on $L^2$ or acting on $e^{-\beta \rho}L^2$.

\begin{figure}
\centering
\begin{tikzpicture}
 \node[inner sep=10pt] (F1) at (0,0)
    {\includegraphics[width=.5\textwidth]{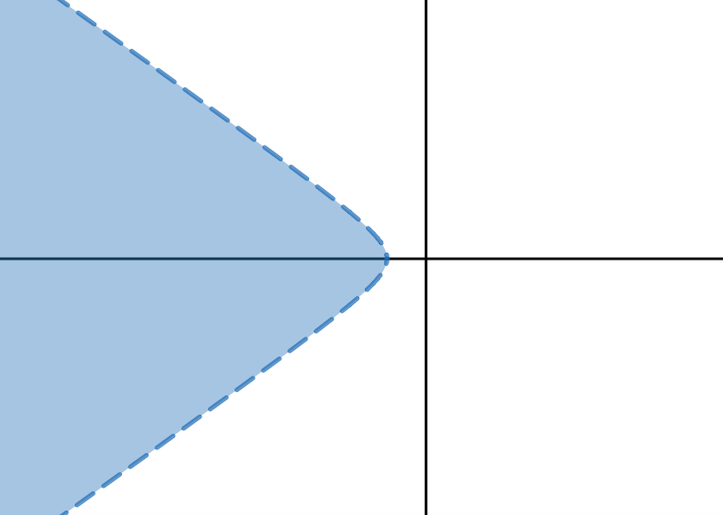}};
\node (F) at (1.7,2.3){ ${\rm Re}(\alpha)=Q$};
\node (F) at (-2.5,0.8){ ${\rm Re}((\alpha-Q)^2)>\beta^2$};
\end{tikzpicture}
  \caption{The blue region corresponds to the set of parameters $\alpha\in\C$ such that  ${\rm Re}((\alpha-Q)^2)>\beta^2$, i.e. region of validity of Lemma \ref{resolventweighted} (here $\beta=1$ on the plot).}
\end{figure}
\subsubsection{Poisson operator in the probabilistic region.}

For $\ell\in\N$,  we shall define the Poisson operator $\mc{P}_\ell(\alpha)$ in the resolvent set.
This operator is a way to construct the generalized eigenfunctions of ${\bf H}$: it takes an element $F\in E_\ell\subset L^2(\Omega_\T)$ and produces a function $u=\mc{P}_\ell(\alpha)F$ solving 
$({\bf H}-2\Delta_\alpha)u=0$ with a prescribed leading asymptotic in terms of $F$ as $c\to -\infty$. 

First, we explain in details our convention on square roots in $\C$ and 
since it will be important in the proof and to avoid confusions for the reader.
We will denote by $\sqrt{\cdot}$ the square root defined so that ${\rm Im}(\sqrt{z})>0$ if $z\in \C\setminus \R^+$, i.e. $\sqrt{re^{i\theta}}=\sqrt{r}e^{i\theta/2}$ for $\theta\in (0,2\pi)$ and $r>0$. In particular, one has $\sqrt{z^2}=z$ for ${\rm Im}(\sqrt{z})>0$ and this extends holomorphically to $z\in \C$. For $\la\in \R^+$, 
the function $\sqrt{z^2-\la}$ will be of special interest to us: it is well-defined and analytic
in $\{{\rm Im}(z)>0\}\cup (-\sqrt{\la},\sqrt{\la})$ 
(since $z^2\notin [\la,\infty)$ in this region) and it extends analytically in a
 neighborhood of the half lines $(-\infty,-\sqrt{\la})$ and $(\sqrt{\la},+\infty)$, for exemple in $z\in \sqrt{\la}+\R^+ e^{-i\theta}$  and $z\in -\sqrt{\la}+\R^+ e^{i(\pi+\theta)}$ for $\theta\in[0,\eps)$. With that convention, which will be used all along this Section on scattering for ${\bf H}$, we have $\sqrt{z^2-\la}>0$ if $z \in(\sqrt{\la},\infty)$ while $\sqrt{z^2-\la}<0$ if $z \in(-\infty,-\sqrt{\la})$. Later we will view these functions as holomorphic functions on a Riemann surface. 

We note the following elementary property, which will be useful in the following.

\begin{lemma}\label{racine}
For $z\in \C\setminus \R^+$, the following map is non-decreasing
\[ x \in\R^+ \mapsto {\rm Im}\sqrt{z-x}.\]
\end{lemma}
\begin{proof}
It suffices to differentiate in $x$.
 \end{proof}
 
Let $\chi\in C^\infty(\R)$ be equal to $1$ in $(-\infty,a-1)$ and equal to $0$ in $(a,+\infty)$ for some $a\in (0,1/2)$, then for $\alpha=Q+ip$ with ${\rm Im}(p)>0$ we choose  
\begin{equation}\label{betaell}
\beta_\ell>\max_{j=0,\dots,\ell}{\rm Im}\sqrt{p^2-2\la_j}-\gamma/2={\rm Im}\sqrt{p^2-2\la_\ell}-\gamma/2, \,\quad \textrm{ and }\beta_\ell\geq 0.
\end{equation}
Then for ${\rm Re}((\alpha-Q)^2)>\beta_\ell^2$ we define 
\begin{equation}\label{definPellproba}\begin{gathered}
\mc{P}_\ell(\alpha):\left\{ \begin{array}{lll}
E_\ell=\oplus_{j=0}^\ell \ker (\mathbf{P}-\la_j)& \to & e^{-(\beta_\ell+\gamma/2) \rho}\mc{D}(\mc{Q})\\
 F= \sum_{0\leq j\leq \ell} F_j &\mapsto & \chi F_-(\alpha) -\mathbf{R}(\alpha)(\mathbf{H}-2\Delta_{\alpha})(\chi F_-(\alpha)),
\end{array}\right.\\
\textrm{ with }F_-(\alpha):=\sum_{j=0}^{\ell}F_je^{ic\sqrt{p^2-2\la_j}}.
\end{gathered}\end{equation}
We will show in the following Lemma that this definition makes sense by using Lemma \ref{resolventweighted}. Before going to the proof of it, recall that ${\rm Im}\sqrt{p^2-2\la_j}>0$ for ${\rm Re}(\alpha)<Q$ by Lemma \ref{racine}, and note that the condition ${\rm Re}((\alpha-Q)^2)>\beta_\ell^2$ implies that 
${\rm Im} \sqrt{p^2-2\la_j}\geq {\rm Im}(p)> \beta_\ell$ for all $j=0,\dots,\ell$. We then emphasize that the main 
reason for $\mc{P}_\ell(\alpha)F$ to be defined and non-trivial is that $\chi F_-(\alpha)\in 
e^{-(\beta_\ell+\gamma/2) \rho}\mc{D}(\mc{Q})\setminus e^{-\beta_\ell \rho}\mc{D}(\mc{Q})$ and $(\mathbf{H}-2\Delta_{\alpha})(\chi F_-(\alpha))\in e^{-\beta_\ell \rho}\mc{D}'(\mc{Q})$ so that $\mathbf{R}(\alpha)(\mathbf{H}-2\Delta_{\alpha})(\chi F_-(\alpha))$ is well-defined but not equal to $\chi F_-(\alpha)$.
\begin{lemma}\label{firstPoisson}
For each $\ell\in\N$, let $\beta_\ell\geq 0$, then the operator $\mc{P}_\ell(\alpha)$ is well-defined,  bounded and holomorphic in the region 
\begin{equation}\label{regionLemma4.4}
\Big\{\alpha=Q+ip\in\C \, |\, {\rm Re}(\alpha-Q)<0,\,{\rm Re}((\alpha-Q)^2)>\beta_\ell^2\, ,\,\, \beta_\ell>{\rm Im}\sqrt{p^2-2\la_\ell}-\gamma/2\Big\}
\end{equation} 
and it satisfies in $e^{-(\beta_\ell+\gamma/2) \rho}\mc{D}'(\mc{Q})$
\begin{equation}\label{solutionHu=0} 
(\mathbf{H}-2\Delta_{\alpha})\mc{P}_\ell(\alpha)=0,
\end{equation}
and in the region $c\leq -1$, one has the asymptotic behaviour, with $F_j:=(\Pi_j-\Pi_{j-1})F$,
\begin{equation}\label{expansionofPell}
\mc{P}_\ell(\alpha)F=\sum_{j=0}^{\ell}F_je^{ic\sqrt{p^2-2\la_j}}+
F_+(\alpha), \quad F_+(\alpha)\in e^{-\beta_\ell \rho}\mc{D}(\mc{Q}).
\end{equation} 
\end{lemma}
\begin{proof}
First we observe that $(\mathbf{H}_0-2\Delta_{\alpha})F_-(\alpha)=0$  thus
\[(\mathbf{H}-2\Delta_{\alpha})\chi F_-(\alpha)=
-\tfrac{1}{2}\chi''(c) F_-(\alpha)+e^{\gamma c}V\chi F_-(\alpha)-\chi'(c)\pl_c F_-(\alpha).\]
We note that $\chi',\chi''$ have compact support in $\R$, and also for each $u\in \mc{D}(\mc{Q})$ 
\[ \begin{split}
|\cjg \chi e^{\gamma c+\beta_\ell \rho}V F_-(\alpha),u\cjd|\leq 
\int_{\R^-} e^{\gamma c+\beta_\ell \rho} |\mc{Q}_V(F_-(\alpha),u)| dc & \leq 
 \Big(\int_{\R^-} e^{\gamma c}\mc{Q}_V(u)dc\Big)^{\frac{1}{2}}\Big( \int_{\R^-}e^{\gamma c}\mc{Q}_V(e^{\beta _\ell \rho}F_-)dc\Big)^\frac{1}{2}\\
 \leq  \mc{Q}(u)^{1/2}\mc{Q}_{e^{\gamma c}V}(e^{\beta _\ell \rho}F_-)^{\frac{1}{2}} 
\end{split}\]
and $\mc{Q}_{e^{\gamma c}V}(e^{\beta _\ell \rho}F_-)<\infty$ by using 
\eqref{regionLemma4.4} and the fact that $\mc{Q}_V(\psi)<\infty$ for all $\psi\in E_\ell$.
We obtain that 
\[\chi F_-(\alpha)\in e^{-(\beta_\ell+\gamma/2)\rho}\mc{D}(\mc{Q}),\quad (\mathbf{H}-2\Delta_{\alpha})\chi F_-(\alpha)\in e^{-\beta_\ell \rho}\mc{D}'(\mc{Q}).\]
This shows, using Lemma \ref{resolventweighted}, that $\mathbf{R}(\alpha)(\mathbf{H}-2\Delta_{\alpha})(\chi F_-(\alpha))$ is well-defined as an element of $e^{-\beta_\ell \rho}\mc{D}(\mathcal{Q})$, with holomorphic dependence in $\alpha$, provided ${\rm Re}((\alpha-Q)^2)>\beta_\ell^2$. By construction, it clearly also solves \eqref{solutionHu=0} in $e^{-(\beta_\ell+\gamma/2) \rho}\mc{D}'(\mc{Q})$.
\end{proof}
Note that the error term $F_+(\alpha)$ in \eqref{expansionofPell} is smaller than  the bigger term in $F_-(\alpha)$ but is not necessarily neglectible with respect to all terms of $F_-(\alpha)$.

\begin{figure}
\centering
\begin{tikzpicture}
 \tikzstyle{PR}=[minimum width=2cm,text width=3cm,minimum height=0.8cm,rectangle,rounded corners=5pt,draw,fill=red!30,text=black,font=\bfseries,text centered,text badly centered]
  \tikzstyle{NS}=[minimum width=2cm,text width=3cm,minimum height=0.8cm,rectangle,rounded corners=5pt,draw,fill=blue!20,text=black,font=\bfseries,text centered,text badly centered]
 \tikzstyle{PRfleche}=[->,>= stealth,thick,red!60]
  \tikzstyle{NSfleche}=[->,>= stealth,thick,blue!60]
 \node[inner sep=0pt] (F1) at (0,0)
    {\includegraphics[width=.4\textwidth]{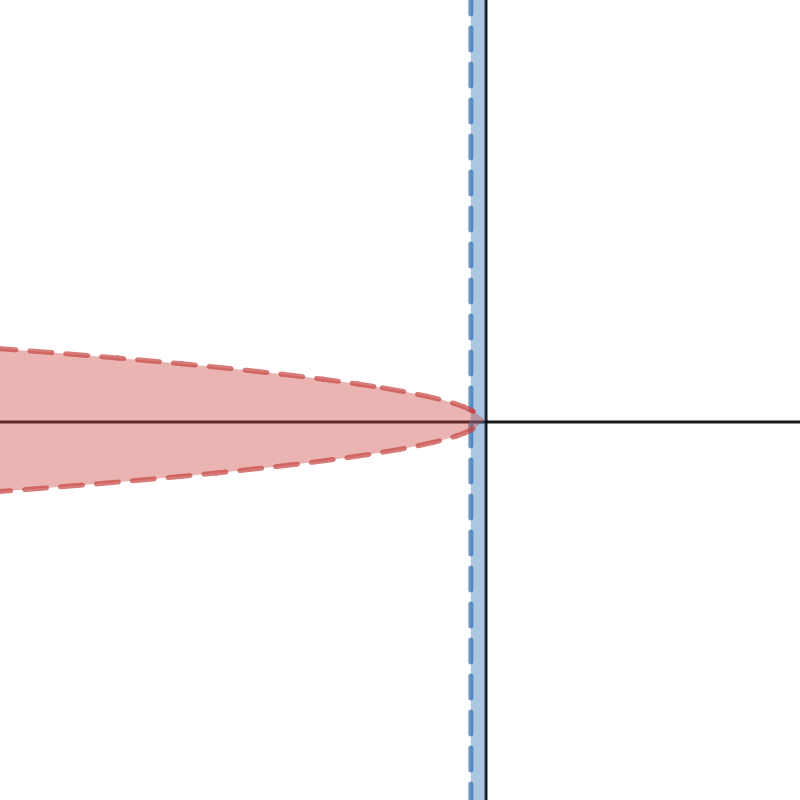}};
     \node[inner sep=0pt] (F2) at (8,0)
    {\includegraphics[width=.4\textwidth]{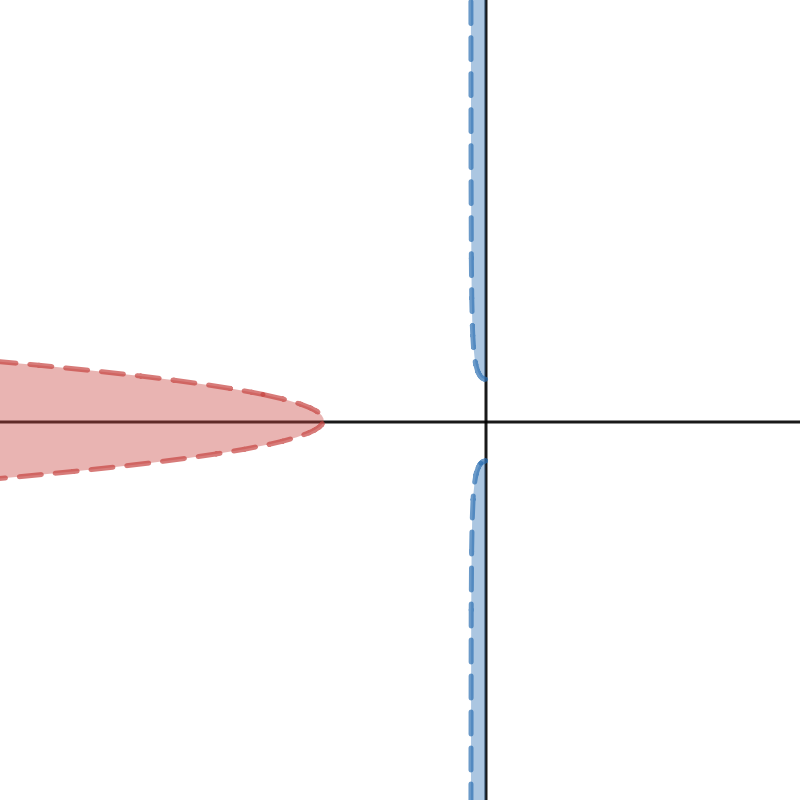}};
\node (F) at (1.7,2.3){ ${\rm Re}(\alpha)=Q$};
\node (G) at (9.7,2.3){ ${\rm Re}(\alpha)=Q$};
\node[PR] (P) at (3,-2) {Probabilistic regions \eqref{regionLemma4.4}};
\node[NS] (NSR) at (5,2) {Near spectrum regions \eqref{regionvalide}};
\node (P1) at (-2,-0.6) {};
\node (P2) at (6,-0.5) {};
\node (N1) at (0.7,3) {};
\node (N2) at (8.6,2) {};
\draw[PRfleche] (P) to [bend left=25] (P1);
\draw[PRfleche] (P)  to [bend right=25] (P2);
\draw[NSfleche] (NSR) to [bend right=25] (N1);
\draw[NSfleche] (NSR)  to [bend right=25] (N2);
\end{tikzpicture}
  \caption{Left picture: regions of definition of the Poisson operator $\mc{P}_0(\alpha)$ 
   (for the plot, we take $\gamma=1/2$ and $\beta_0$ is optimized to obtain the largest possible region). Right picture: regions of definition of the Poisson operator $\mc{P}_{\ell}(\alpha)$ with $\ell>0$ (for the plot we take $\la_\ell=4$,$\gamma=1/2$ and $\beta_\ell$ optimized). 
}
   \label{figure2}
\end{figure}
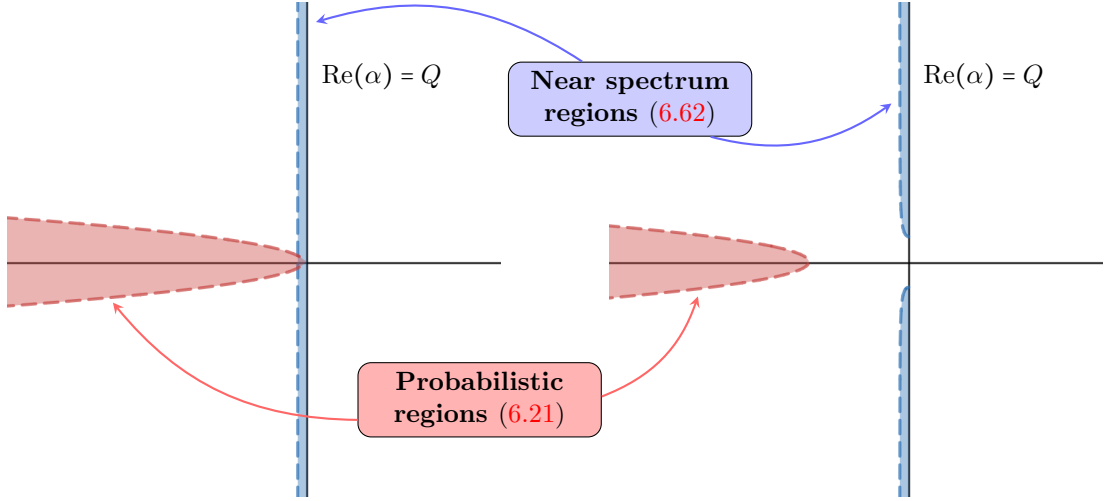


We also notice that where they are defined, we have for $j\geq 0, \ell\geq 0$
\begin{equation}\label{egalitePell}
\mc{P}_{\ell+j}(\alpha)|_{E_\ell}=\mc{P}_{\ell}(\alpha)|_{E_\ell}.
\end{equation}

In \eqref{definPellproba}, the definition of the operator $\mc{P}_\ell(\alpha)$ seemingly depends on the cutoff function $\chi$. In fact, we can show that this is not the case. We state a lemma below in this direction
\begin{lemma}\label{uniqPell}
For $\ell\in\N$, $\beta_\ell$ satisfying \eqref{betaell} and for ${\rm Re}((\alpha-Q)^2)>\beta_\ell^2$  the definition of the operator $\mc{P}_\ell(\alpha)_{|E_\ell} $ does not depend on $\chi$.
\end{lemma}
\begin{proof} Pick two functions $\chi,\hat{\chi}$ satisfying our assumptions and denote by $\mc{P}^\chi_\ell(\alpha),\mc{P}^{\hat{\chi}}_\ell(\alpha) $ the corresponding Poisson operators. Set $d(\chi):=\chi-\hat{\chi}$. Then observe that for $F\in E_\ell$ 
$$\mc{P}^\chi_\ell(\alpha)F-\mc{P}^{\hat{\chi}}_\ell(\alpha) F=d(\chi)F_-(\alpha)-\mathbf{R}(\alpha)(\mathbf{H}-2\Delta_{\alpha})( d(\chi)F_-(\alpha)).$$
Now we note that  $d(\chi)F_-(\alpha)\in \mc{D}(\mc{Q})$ since $d(\chi)(c)=0$ for $c\notin (-1,a)$ for some $a>0$ and $\mc{Q}_{V}$ is bounded on $E_\ell$. 
Then $\mathbf{R}(\alpha)(\mathbf{H}-2\Delta_{\alpha})( d(\chi)F_-(\alpha))=d(\chi)F_-(\alpha)$ since $(\mathbf{H}-2\Delta_{\alpha}):\mc{D}(\mc{Q})\to \mc{D}'(\mc{Q})$ is an isomorphism, hence $\mc{P}^\chi_\ell(\alpha)F-\mc{P}^{\hat{\chi}}_\ell(\alpha) F=0$.
\end{proof}

We can also rewrite the construction of the Poisson operator using the propagator. 
\begin{proposition}\label{Pellproba}
The following properties hold true:\\
1) Let $\ell\in\N$ and let $F\in L^2(\Omega_\T)\cap \ker(\mathbf{P}-\la_j)$ for $j\leq \ell$.  
Then we have the identity 
\[\mc{P}_\ell(\alpha)F= \lim_{t\to +\infty}e^{t 2\Delta_\alpha}e^{-t\mathbf{H}}\Big(e^{ic\sqrt{p^2-2\la_j}}\chi(c)F\Big)\]
in $e^{-(\beta+\gamma/2) \rho}\mc{D}'(\mc{Q})$ provided
${\rm Re}((\alpha-Q)^2)>\beta^2$ with $\beta>{\rm Im}(\sqrt{p^2-2\la_j})-\gamma/2$ and $\alpha=Q+ip$. Furthermore if ${\rm Im}(\sqrt{p^2-2\la_j})>\gamma$ then we can take $\chi=1$ in the above statement.\\
2) Let    $F\in L^2(\Omega_\T)\cap \ker(\mathbf{P})$. 
If $\alpha\in\R$ with $\alpha <Q$, then  $dc\otimes \P$ almost everywhere
\[\mc{P}_0(\alpha)F= \lim_{t\to +\infty}e^{t 2\Delta_\alpha}e^{-t\mathbf{H}}\Big(e^{(\alpha-Q)c} F\Big).\]

\end{proposition}

\begin{proof} We first prove 1). Recall that $\mathbf{H}=\mathbf{H}_0+e^{\gamma c}V$.
We have 
\[ (\mathbf{H}-2\Delta_\alpha)(\chi(c)e^{i\nu_j c}F)= \chi(c)e^{(i\nu_j +\gamma)c}V F  
-\tilde{\chi}(c)Fe^{i\nu_jc}\] 
where $\nu_j:=\sqrt{p^2-2\la_j}$ and $\tilde{\chi}(c):=\tfrac{1}{2}\chi''(c)+i\nu_j\chi'(c)\in C_c^\infty(\R)$, and this belongs to $e^{-\beta \rho}\mc{D}'(\mc{Q})$.
Using Lemma \ref{resolventweighted}, 
\[ \mc{P}_\ell(\alpha)F=\chi(c)e^{i\nu_j c}F-\mathbf{R}(\alpha)(e^{(i\nu_j+\gamma) c}V\chi F  -\tilde{\chi} e^{i\nu_jc}F)
\]
provided ${\rm Re}((\alpha-Q)^2)>\beta^2$ for any $\beta>{\rm Im}(\nu_j)-\gamma/2$. Noticing that the bound   \eqref{normetHweight} remains valid with $V=0$, we  can   make sense of 
$u(t):=e^{-t\mathbf{H}_0}(\chi(c)e^{i\nu_j c}F)$ as an element in $e^{-\beta_0 \rho}L^2(\R; E_j)$ for any $\beta_0>{\rm Im}(\nu_j)$. Then
\[ (\pl_t+2\Delta_\alpha)u(t)=e^{-t\mathbf{H}_0}(-\mathbf{H}_0+2\Delta_\alpha)(\chi(c)e^{i\nu_j c}F)=e^{-t\mathbf{H}_0}(e^{i\nu_j c}F \tilde{\chi}(c))
\]
thus we get by integrating in $t$
\begin{equation}\label{actionexp-tH_0} 
\begin{split}
e^{-t(\mathbf{H}_0-2\Delta_\alpha)}(\chi(c)e^{i\nu_j c}F)= & 
\chi(c)e^{i\nu_j c}F+\int_0^t e^{-s(\mathbf{H}_0-2\Delta_\alpha)}(e^{i\nu_jc}F\tilde{\chi}(c))ds\\
= & \chi(c)e^{i\nu_j c}F+(1- e^{-t(\mathbf{H}_0-2\Delta_\alpha)})\mathbf{R}_0(\alpha)(e^{i\nu_jc}F\tilde{\chi}(c))
\end{split}
\end{equation}
where $\mathbf{R}_0(\alpha):=(\mathbf{H}_0-2\Delta_\alpha)^{-1}$ is defined also on $e^{-\beta \rho}L^2$ by taking the proof of Lemma \ref{resolventweighted} in the case of the trivial potential $V=0$.  
We also note that $e^{-t\mathbf{H}_0}(e^{i\nu_jc}F\tilde{\chi}(c))$ and $e^{-t\mathbf{H}_0}\mathbf{R}_0(\alpha)(e^{i\nu_jc}F\tilde{\chi}(c))$ are in $L^2(\R; E_j)$ since $\mathbf{H}_0F=(Q^2/2+\la_j)F$  (and $\tilde{\chi}\in C_c^\infty(\R)$), i.e. all terms above are functions of $c$ with values in $E_j$.

We next claim that 
\[ e^{-t\mathbf{H}}(\chi(c)e^{i\nu_j c}F)=e^{-t\mathbf{H}_0}(\chi(c)e^{i\nu_j c}F)-\int_0^t 
e^{-(t-s)\mathbf{H}}e^{\gamma c}Ve^{-s\mathbf{H}_0}(\chi(c)e^{i\nu_j c}F)ds.\]
Indeed, first, all terms are well-defined: since $e^{-s\mathbf{H}_0}(\chi(c)e^{i\nu_j c}F)\in e^{-\beta_0 \rho} L^2(\R; E_j)\cap e^{-\beta_0 \rho} \mc{D}(\mc{Q})$ one has  $e^{\gamma c}Ve^{-s\mathbf{H}_0}(\chi(c)e^{i\nu_j c}F)\in e^{-\beta_0 \rho}\mc{D}'(\mc{Q})$, and we can then use \eqref{normetHweight2}. 
Then the identity above is obtained since both terms solve 
$(\pl_t+\mathbf{H})u(t)=0$ in $e^{-\beta_0 \rho}\mc{D}'(\mc{Q})$ with $u(0)=\chi(c)e^{i\nu_j c}F$ in $e^{-\beta_0 \rho}\mc{D}(\mc{Q})$.
By applying twice \eqref{actionexp-tH_0}, we thus obtain 
\[\begin{split}
e^{-t(\mathbf{H}-2\Delta_\alpha)}(\chi(c)e^{i\nu_j c}F)= & \chi(c)e^{i\nu_j c}F-e^{t2\Delta_\alpha}\int_0^t 
e^{-(t-s)\mathbf{H}}e^{\gamma c}Ve^{-s\mathbf{H}_0}(\chi(c)e^{i\nu_j c}F)ds\\
& + (1- e^{-t(\mathbf{H}_0-2\Delta_\alpha)})\mathbf{R}_0(\alpha)(e^{i\nu_jc}F\tilde{\chi}(c))\\
= & \chi(c)e^{i\nu_j c}F-\int_0^t 
e^{-(t-s)(\mathbf{H}-2\Delta_\alpha)}(e^{\gamma c}V\chi(c)e^{i\nu_j c}F)ds\\
 &+ (1- e^{-t(\mathbf{H}_0-2\Delta_\alpha)})\mathbf{R}_0(\alpha)(e^{i\nu_jc}F\tilde{\chi}(c)) \\ 
 & - \int_0^t 
e^{-(t-s)(\mathbf{H}-2\Delta_\alpha)}e^{\gamma c}V(1-e^{-s(\mathbf{H}_0-2\Delta_\alpha)})\mathbf{R}_0(\alpha)(e^{i\nu_jc}F\tilde{\chi}(c))ds.
\end{split}\]
Using \eqref{normetHweight} and \eqref{normetHweight2}  (applied with both $V\not=0$ and $V=0$)  and ${\rm Re}((\alpha-Q)^2)>\beta^2$, we have as bounded operators on respectively $e^{-\beta \rho}\mc{D}'(\mc{Q})$ and $e^{-\beta \rho}L^2$
\[ \begin{gathered}
\lim_{t\to +\infty}\int_0^t 
e^{-(t-s)(\mathbf{H}-2\Delta_\alpha)}ds=\lim_{t\to +\infty}(1-e^{-t(\mathbf{H}-2\Delta_\alpha)})\mathbf{R}(\alpha)=\mathbf{R}(\alpha)\\
\lim_{t\to +\infty} (1- e^{-t(\mathbf{H}_0-2\Delta_\alpha)})\mathbf{R}_0(\alpha)= \mathbf{R}_0(\alpha).
\end{gathered}\]
This gives in particular in $e^{-\beta \rho}\mc{D}'(\mc{Q})$
\[ \lim_{t\to +\infty}\int_0^t 
e^{-(t-s)(\mathbf{H}-2\Delta_\alpha)}(e^{\gamma c}V\chi(c)e^{i\nu_j c}F)ds=\mathbf{R}(\alpha)(e^{\gamma c}V\chi(c)e^{i\nu_j c}F).\]
Similarly one has in $e^{-\beta \rho}\mc{D}'(\mc{Q})$
\[ \lim_{t\to +\infty}\int_0^t 
e^{-(t-s)(\mathbf{H}-2\Delta_\alpha)}e^{\gamma c}V\mathbf{R}_0(\alpha)(e^{i\nu_jc}F\tilde{\chi}(c))ds=
\mathbf{R}(\alpha)e^{\gamma c}V\mathbf{R}_0(\alpha)(e^{i\nu_jc}F\tilde{\chi}(c)).\]
Finally we claim that 
\[\lim_{t\to +\infty}\int_0^t 
e^{-(t-s)(\mathbf{H}-2\Delta_\alpha)}e^{\gamma c}Ve^{-s(\mathbf{H}_0-2\Delta_\alpha)}\mathbf{R}_0(\alpha)(e^{i\nu_jc}F\tilde{\chi}(c))ds=0.\]
Indeed, we can apply Lebesgue theorem if one can show  
\[ \|e^{\gamma c}Ve^{-s(\mathbf{H}_0-2\Delta_\alpha)}\mathbf{R}_0(\alpha)(e^{i\nu_jc}F\tilde{\chi}(c))\|_{e^{\beta \rho}\mc{D}'(\mc{Q})}\leq e^{-\delta s}\]
for some $\delta>0$, since $\|e^{-(t-s)(\mathbf{H}-2\Delta_\alpha)}\|_{\mc{L}(e^{-\beta\rho}\mc{D}'(\mc{Q}))}\to 0$ by \eqref{normetHweight2}. But this estimate follows again from \eqref{normetHweight} with $V=0$ and the fact that $\mathbf{R}_0(\alpha)(e^{i\nu_jc}F\tilde{\chi}(c))\in e^{-\beta \rho}L^2(\R;E_j)$, which in turn  implies that 
\[e^{\gamma c}Ve^{-s(\mathbf{H}_0-2\Delta_\alpha)}\mathbf{R}_0(\alpha)(e^{i\nu_jc}F\tilde{\chi}(c))\in e^{\beta\rho}\mc{D}'(\mc{Q}).\] 
We have thus proved that 
\[\begin{split}
\lim_{t\to +\infty}e^{-t(\mathbf{H}-2\Delta_\alpha)}(\chi(c)e^{i\nu_j c}F)=  &\chi(c)e^{i\nu_j c}F-\mathbf{R}(\alpha)e^{\gamma c}V\big(e^{i\nu_jc}F\chi(c)+\mathbf{R}_0(\alpha)(e^{i\nu_jc}F\tilde{\chi}(c))\big)\\
& +\mathbf{R}_0(\alpha)(e^{i\nu_jc}F\tilde{\chi}(c)) .
\end{split}\]

We conclude by observing that  
\[\mathbf{R}(\alpha)(e^{\gamma c}V\chi e^{i\nu_jc}F  -\tilde{\chi} e^{i\nu_jc}F)=
-\mathbf{R}_0(\alpha)(e^{i\nu_jc}F\tilde{\chi})+\mathbf{R}(\alpha)e^{\gamma c}V\big(\chi e^{i\nu_jc}F+\mathbf{R}_0(\alpha)(e^{i\nu_jc}F\tilde{\chi}(c))\big),
\]
which can be established by applying $(\mathbf{H}-2\Delta_\alpha)$ to this equation and using the injectivity of 
$(\mathbf{H}-2\Delta_\alpha)$ on $e^{-\beta \rho} \mc{D}(\mc{Q})$ under our condition on $\alpha,\beta$. 

Notice that for ${\rm Im}(\sqrt{p^2-2\la_j})>\gamma$, we have $e^{(i\nu_j+\gamma)c}\mathbf{1}_{(0,+\infty)}(c) \in L^2(\R)$ so that the above argument applies with $\chi=1$. 

Now we prove 2). As $F\in   L^2(\Omega_\T)\cap \ker(\mathbf{P}-\la_0)$, we may assume  $F=1$ without loss of generality. With our assumptions, we can write $\alpha=Q+ip$ with  $p=i(Q-\alpha)$ and choose $(Q-\alpha)>\beta>(Q-\alpha)-\gamma$. By applying  1), we get that
\[\mc{P}_\ell(\alpha)1= \lim_{t\to +\infty}e^{t 2\Delta_\alpha}e^{-t\mathbf{H}}\Big(e^{(\alpha-Q)c}\chi(c) \Big)\]
in $e^{-\beta \rho}L^2$, hence $dc\otimes \P$ almost everywhere (up to extracting subsequence).  We have to show that we can replace $\chi$ by $1$. For this we will use the probabilistic representation \eqref{fkformula}: we have 
\begin{align*}
e^{t 2\Delta_\alpha}e^{-t\mathbf{H}}\Big(e^{(\alpha-Q)c}(1-\chi(c)) \Big)=&e^{-Qc} \E_\varphi\big[S_{e^{-t}} (e^{\alpha c}(1-\chi(c)))e^{-\mu M_\gamma ( \D_t^c)}\big]\\
=&e^{t 2\Delta_\alpha}e^{(\alpha-Q)c}\E_\varphi \big[e^{\alpha (X_t(0)-Qt)}(1-\chi(c+X_t(0)-Qt)))e^{-\mu M_\gamma ( \D_t^c)}\big].
\end{align*}
By Girsanov's transform this expression can be rewritten as
\begin{align*}
e^{t 2\Delta_\alpha}e^{-t\mathbf{H}}\Big(e^{(\alpha-Q)c}(1-\chi(c)) \Big)=&e^{(\alpha-Q)c}\E_\varphi \big[ \big(1-\chi(c+X_t(0)-(Q-\alpha)t)\big)\exp\big(-\mu \int_{ \D_t^c}|z|^{-\gamma\alpha} M_\gamma (dz)\big)\big].
\end{align*}
Recalling that $\chi=1$ on $(-\infty,a-1)$ and that $t\mapsto X_t(0)$ evolves as a Brownian motion independent of $\varphi$, then  estimating the exponential term by $1$ we obtain
\begin{align*}
\big|e^{t 2\Delta_\alpha}e^{-t\mathbf{H}}\Big(e^{(\alpha-Q)c}(1-\chi(c)) \Big)\big|\leq &e^{(\alpha-Q)c}\P\big( c+X_t(0)-(Q-\alpha)t\geq a-1\big).\end{align*}
The result easily follows from this estimate.
\end{proof}

\subsubsection{Meromorphic extension of the resolvent near the $L^2$ spectrum}
We denote by $Z$ the Riemann surface on which the functions $p\mapsto \sqrt{p^2-2\la_j}$ are single valued  for all $j$. This is a  ramified covering of $\C$ with ramification points $\{\pm\sqrt{2\la_j}\, |\, j\in\N\}$, and 
in which we  embed the region ${\rm Im}(p)>0$ that we call the \emph{physical sheet}. We will call $\pi: Z\to \C$ the projection of the covering. The construction of $Z$ can be done iteratively on $j$, as explained in Chapter 6.7 of Melrose's book \cite{Mel}. 
The map $p\mapsto \alpha:=Q+ip$ from ${\rm Im}(p)>0$ to ${\rm Re}(\alpha)<Q$ (now called the physical sheet for the variable $\alpha$) extends analytically as a map $Z\to \Sigma$ where $\Sigma$ is an isomorphic Riemann surface to $Z$ (it just amounts to a linear change of complex coordinates from $p$ to $\alpha$). We shall also denote by $\pi : \Sigma\to \C$ the projection.
Finally we choose a specific function $\chi$ of the form indicated  previously but we further impose that   $\chi\in C^\infty(\R)$ is equal to $1$ in $(-\infty,-1+\delta)$ and equal to $0$ in $(0,+\infty)$ (for some small $\delta>0$).

First, we recall the notation for the orthogonal projectors 
\[ \Pi_k= 1_{[0,\la_k]}(\mathbf{P}) : L^2(\Omega_\T,\P)\to L^2(\Omega_\T,\P)\] 
and we denote by $E_k$ their range (which are Hilbert spaces) and $E_k^\perp$ the range of $1-\Pi_k$. 
The goal of this section is to show the following:

\begin{proposition}\label{extensionresolvent}
Assume that $\gamma\in (0,\sqrt{2})$ and let $0\leq \beta<\gamma/2$. Then the following holds true:\\
1) The resolvent $\mathbf{R}(\alpha):=(\mathbf{H}-2\Delta_\alpha )^{-1}$ is bounded as a map 
$L^2(\R\times \Omega_\T)\to \mc{D}(\mathbf{H})$ for ${\rm Re}(\alpha)<Q$ and for $k\geq 0$ large enough, the operator
$(\mathbf{H}-2\Delta_{\pi(\alpha)})^{-1}$ admits a meromorphic continuation to the region 
\[\{\alpha=Q+ip \in \Sigma\,|\, |\pi(p)|^2\leq \la_k^{\frac{1}{4}}, \forall j\leq k, {\rm Im}\sqrt{p^2-2\la_j}>-\min(\beta,\gamma/2-\beta)
\}\] 
as a map 
\begin{align*}
\mathbf{R}(\alpha): e^{\beta \rho}L^2(\R\times \Omega_\T)\to e^{-\beta \rho}\mc{D}({\bf H}),
\end{align*}
and the residue at each pole is a finite rank operator. Moreover, for each $\psi\in C^\infty(\R)\cap L^\infty(\R)$ satisfying $\psi'\in L^\infty$ and ${\rm supp}(\psi)\subset (-\infty,A)$ for some $A\in \R$, 
\[\mathbf{R}(\alpha)(1-\Pi_k)\psi: \mc{D}'(\mc{Q})\to e^{-\beta \rho}\mc{D}(\mc{Q}).\] 
2) If $f\in e^{\beta \rho}\mc{D}'(\mc{Q})$ is supported in $c\in (-\infty,A)$ for some $A>0$ and is such that $\Pi_kf\in e^{\beta \rho}L^2(\R;E_k)$ 
and $(1-\Pi_k)f\in \mc{D}'(\mc{Q})$,  then for $\alpha$ as above and not a pole, one has in $c\leq 0$
\begin{equation}\label{asymptoticRf}
\mathbf{R}(\alpha)f=\sum_{j=0}^k a_j(\alpha,f)e^{-ic\sqrt{p^2-2\la_j}}+G(\alpha,f),
\end{equation}
with $a_j(\alpha,f)\in \ker (\mathbf{P}-\la_j)$, and $G(\alpha,f), \pl_cG(\alpha,f)\in e^{\beta\rho}L^2(\R^-; E_k)+L^2(\R^-; E_k^\perp)$, all depending meromorphically in $\alpha$ in the region they are defined.\\
3) There is no pole for $\mathbf{R}(\alpha)$ in $\{\alpha \in \Sigma\,|\, {\rm Re}(\alpha)\leq Q\}\setminus \cup_{j=0}^\infty \{Q\pm i\sqrt{2\la_j}\}$ and $Q\pm i\sqrt{2\la_j}$ can be at most a pole of order $1$.
\end{proposition}

To prove this Proposition, we will construct parametrices for the operator 
$\mathbf{H}-2\Delta_\alpha =\mathbf{H}-\tfrac{Q^2+p^2}{2}$ in several steps and will split the argument. More concretely, we will search for some bounded model operator $\til{\bf R}(\alpha):e^{\beta \rho}L^2(\R\times \Omega_\T)\to e^{-\beta \rho}\mc{D}(\mc{Q})$, holomorphic in $\alpha$ in the desired region of $\Sigma$, such that 
\[ (\mathbf{H}-2\Delta_\alpha)\til{\bf R}(\alpha)=1-{\bf K}(\alpha)\]
where ${\bf K}(\alpha) \in \mc{L}(e^{\beta \rho}L^2(\R\times \Omega_\T))$ is an analytic family of compact operators with $1-{\bf K}(\alpha_0)$ invertible at some $\alpha_0$ belonging to the physical sheet. Then the Fredholm analytic theorem will imply that $(1-{\bf K}(\alpha))^{-1}$ exists as a meromorphic family and 
${\bf R}(\alpha):=\til{\bf R}(\alpha)(1-{\bf K}(\alpha))^{-1}$ gives us the desired meromorphic extension of ${\bf R}(\alpha)$. Our strategy will be based on that method with slights modifications. The continuous spectrum of ${\bf H}$ near frequency $(Q^2+p^2)/2\in \R^+$ will come only from finitely many eigenmodes of $\mathbf{P}$, namely those $\la_j$ for which $2\la_j \leq p^2$. This suggests, in order to construct the approximation $\til{{\bf R}}(\alpha)$ to split 
the modes of ${\bf P}$ depending on the value of ${\rm Im}(\alpha-Q)$. The parametrix will be constructed in three steps as follows:
\begin{itemize}
\item First, we deal with the large eigenmodes for the operator 
$\mathbf{P}$ in the region $c\in (-\infty,0]$ of $L^2(\R \times \Omega_\T)$. We will show that this part does not contribute to the continuous spectrum at frequency $(Q^2+p^2)/2$, and we shall obtain a parametrix for that part by energy estimates. 
\item Then we consider the region $c\geq -1$ where we shall show that the model 
operator in that region (essentially ${\bf H}$ on $L^2([-1,\infty)\times\Omega_\T)$ with Dirichlet condition at $c=0$) has compact resolvent, providing a compact operator for the parametrix of that region.
\item Finally, we will deal with the $c\leq 0$ region corresponding to eigenmodes of $\mathbf{P}$ of order  $\mc{O}(|p|^2)$, where there is \emph{scattering} at $c=-\infty$ for frequency $(Q^2+p^2)/2$, producing continuous spectrum. The parametrix for this part is basically the exact inverse of ${\bf H_0}-2\Delta_\alpha$ restricted to finitely many modes of ${\bf P}$. 
\end{itemize}
 
For $s\geq 0$ and $I\subset \R$ an open interval, we will denote by $H^s(I; L^2(\Omega_\T))$ the Sobolev space of order $s$ in the $c$-variable
\[H^s(I; L^2(\Omega_\T)):=\{ u\in L^2(I\times \Omega_\T)\, |\, \forall j\leq \ell, \xi\mapsto \cjg \xi\cjd^{s}\mc{F}(u)(\xi)\in L^2(I\times \Omega_\T)\}\]
where $\mc{F}$ denotes Fourier transform in $c$, and similarly $H_0^s(I;L^2(\Omega_\T))$ will be the completion of $C_c^\infty(I;L^2(\Omega_\T))$
with respect to the norm $\|\cjg \xi\cjd^s \mc{F}(u)\|_{L^2(\R;L^2(\Omega_\T))}$.\\

\noindent\textbf{1) Large $\mathbf{P}$ eigenmodes in the region $c\leq 0$.}  
Let $\chi\in C^\infty(\mathbb{R},[0,1])$ which satisfies $\chi(c)=1$ for $c\leq -1+\delta$ for some $\delta\in (0,\tfrac{1}{2})$ and $\chi(c)=0$ in $[-1/2,\infty)$ and  $\tilde{\chi}\in C^\infty(\R,[0,1])$ such that $\tilde{\chi}=1$ on the support of 
$\chi$ and ${\rm supp}(\tilde{\chi})\subset \R^-$, and we now view these functions as multiplication operators by $\chi(c)$ on the spaces $e^{\beta \rho}L^2(\R^-\times\Omega_\T)$.
We will first show the following 
\begin{lemma}\label{LemmaRk}
1) There is a constant $C>0$ depending only on $|\tilde{\chi}'|_{\infty}, |\tilde{\chi}''|_{\infty}$ such that for each $k\in\N$, there is 
a bounded operator
\[\mathbf{R}_k^\perp(\alpha): L^2(\R;E_k^\perp)\to L^2(\R^-;E_k^\perp)\]
holomorphic in $\alpha=Q+ip\in\C$ in the region 
$\{{\rm Re}(\alpha)<Q\}\cup \{|\alpha-Q|^2\leq \la_k\}$,   
with 
\begin{align*}
& \tilde{\chi}\mathbf{R}_k^\perp(\alpha) \chi: L^2(\R\times\Omega_\T)
\to L^2(\R;E_k^\perp)\cap\mc{D}({\bf H}),
& \tilde{\chi}\mathbf{R}_k^\perp(\alpha) \chi: \mc{D}'(\mc{Q})\to \mc{D}(\mc{Q})\cap L^2(\R;E_k^\perp)
\end{align*} 
bounded, so that
\begin{equation}\label{firstparamlem}
(\mathbf{H}-\tfrac{Q^2+p^2}{2})\tilde{\chi}\mathbf{R}_k^\perp(\alpha)(1-\Pi_k)\chi=(1-\Pi_k)\chi+\mathbf{L}_{k}^\perp(\alpha)+\mathbf{K}_{k}^\perp(\alpha)\quad \text{ and }\quad \tilde{\chi}\mathbf{R}_k^\perp(\alpha)\Pi_k\chi=0
\end{equation}
with 
$\mathbf{L}_{k}^\perp(\alpha): \mc{D}'(\mc{Q})\to L^2(\R^-;E_k^\perp)$
and 
 $\mathbf{K}_{k}^\perp(\alpha):\mc{D}'(\mc{Q})\to e^{\beta \rho}L^2(\R;E_k)$ bounded and holomorphic in
$\alpha$ as above for each $0<\beta<\gamma/2$. 
In the region where
$|p^2|\leq \la_k$, 
one has the bound
\begin{equation}\label{boundonK1}
\|\mathbf{L}_{k}^\perp(\alpha)\|_{\mc{L}(L^2)}\leq C\la_k^{-1/2} 
\end{equation}
and $\mathbf{K}_{k}^\perp(\alpha)$ is compact as a map $L^2(\R\times\Omega_\T)\to 
e^{\beta \rho}L^2(\R;E_k)$.\\

2) Let $\beta\in\R$, then in the region ${\rm Re}((\alpha-Q)^2)>\beta^2-2\la_k+1$, the operator 
$\mathbf{R}_k^\perp(\alpha):e^{-\beta \rho} \mc{D}'(Q)\to e^{-\beta \rho}(L^2(\R^-;E_k^\perp)\cap \mc{D}(\mc{Q}))$ is a bounded holomorphic family, 
$\mathbf{K}_k^\perp(\alpha): e^{-\beta \rho}L^2(\R\times \Omega_\T)\to e^{-\beta \rho}L^2(\R^-;E_k)$ 
is a compact holomorphic family, $\mathbf{
L}_k^\perp(\alpha):e^{-\beta \rho}L^2(\R\times \Omega_\T)\to e^{-\beta \rho}L^2(\R^-;E_k^\perp)$ is bounded analytic with norm
\[\| \mathbf{L}_k^\perp(\alpha)\|_{\mc{L}(e^{-\beta \rho}L^2)}\leq \frac{C(1+|\beta|)}{\sqrt{{\rm Re}((\alpha-Q)^2)+2\la_k-\beta^2}}\]
for some $C>0$ depending only on $|\tilde{\chi}'|_\infty$ and $|\tilde{\chi}''|_\infty$.
\end{lemma}
\begin{proof}
Let us define $\mc{Q}_k^\perp$ the restriction of $\mc{Q}$ to the closed subspace 
\[ \mc{D}(\mc{Q}_k^\perp):=\mc{D}(\mc{Q})\cap L^2(\R^-;E_k^\perp)=\mc{D}(\mc{Q})\cap \ker \Pi_k\cap \ker r_{\R^+}\]
where $r_{\R^\pm}:L^2(\R\times \Omega_\T)\to L^2(\R^\pm\times \Omega_\T)$ is the restriction for $c\geq 0$.
Note that this is a closed form and is thus the quadratic form of a unique self-adjoint operator 
${\bf H}_k^\perp$, which maps $\mc{D}(\mc{Q}_k^\perp)\to \mc{D}'(\mc{Q}_k^\perp)$ and has a domain 
$\mc{D}(\mathbf{H}_k^\perp)\subset \mc{D}(\mc{Q}_k^\perp)$. 
For $u\in \mc{D}(\mc{Q}_k^\perp)$ we have $\|{\bf P}^{1/2}u\|_{2}^2\geq \la_k\|u\|^2_2$, thus  for each $\eps\geq 0$
\begin{equation}\label{boundsob}
\mc{Q}(u)\geq   \tfrac{1}{2}\|\pl_cu\|_{L^2}^2+( \tfrac{Q^2}{2}+(1-\eps)\la_k)\|u\|_{L^2}^2+\eps\|\mathbf{P}^{1/2}u\|^2_{L^2}+\mc{Q}_{e^{\gamma c}V}(u,u)
\end{equation}
and therefore the quadratic form $\mc{Q}_k^\perp(u)$ is bounded below 
by $\frac{\la_k}{2}\|u\|^2_{L^2}$ on $\mc{D}(\mc{Q}_k^\perp)$ (if $\epsilon$ is chosen small enough).
There is a natural injection $\iota: \mc{D}(\mc{Q}_k^\perp)\to \mc{D}(\mc{Q})$ which we view as an inclusion. 
Let $\chi,\til{\chi}\in C^\infty(\R)$ be equal to $1$ in $c<-1/2$ and supported in $(-\infty,0)$. Let $\til{\chi},\chi$ defined as $\chi$ but equal to $1$ on the support of $\chi$. We view $\chi$ as the operator of multiplication by $\chi(c)$ on $L^2(\R\times \Omega_\T)$, and similarly for $\til{\chi}$.
Let us check that $\iota\chi=\chi\iota: \mc{D}({\bf H}_k^\perp)\to \mc{D}({\bf H})$ and for $u\in \mc{D}({\bf H}_k^\perp)$ 
\begin{equation}\label{identite_entre_H} 
{\bf H}\iota(\chi u) = \iota ({\bf H}_k^\perp \chi u)+\chi e^{\gamma c}\Pi_k(V\iota(u))
\end{equation}   
(note that $\chi u\in \mc{D}({\bf H}_{k}^\perp)$ if $u\in \mc{D}({\bf H}_{k}^\perp)$ since ${\bf H}_{k}^\perp \chi u=\chi {\bf H}_{k}^\perp u-\frac{1}{2}\chi''u-\chi'\pl_cu$).
Here we use Lemma \ref{chiPikV}, which insures that $\chi e^{\gamma c}\Pi_k(V\iota(u))$ makes sense.
First, we have for $u\in \mc{D}(\mc{Q}_k^\perp)$ that $\iota(u)=(1-\Pi_k)\iota(u)$. 
For such $u$ and for $v\in \mc{D}(\mc{Q})$, one gets
\[ \begin{split}
\mc{Q}(\iota (\chi u),v)=& \tfrac{1}{2}\cjg \pl_c(1-\Pi_k)\chi u,\pl_c(1-\Pi_k)\til{\chi} v\cjd_{L^2(\R^-\times \Omega_\T)}+
\tfrac{Q^2}{2}\|(1-\Pi_k)\chi u,(1-\Pi_k)\til{\chi}v\cjd_{L^2(\R^-\times \Omega_\T)}+\\
& +\cjg {\bf P}^{1/2}(1-\Pi_k)\chi u,{\bf P}^{1/2}(1-\Pi_k)\til{\chi} v\cjd_{L^2(\R^-\times \Omega_\T)}+\int_{\R}\chi(c) e^{\gamma c}(\mc{Q}_V(u,(1-\Pi_k)v)+\mc{Q}_V(u,\Pi_kv))dc\\
=& \mc{Q}_{k}^\perp(\chi u,(1-\Pi_k)\til{\chi}v)+\int_{\R}e^{\gamma c}\mc{Q}_V(\chi u,\Pi_kv)\, dc\\
=& \cjg {\bf H}_k^\perp \chi u,(1-\Pi_k)\til{\chi}v\cjd_{2}+\int_{\R}e^{\gamma c}\mc{Q}_V(\chi u,\Pi_kv)\, dc=\cjg {\bf H}_k^\perp \chi u,v\cjd_{L^2(\R^-\times \Omega_\T)}+\int_{\R}e^{\gamma c}\mc{Q}_V(\chi u,\Pi_kv)\, dc
\end{split}\]
where in the last line, $\til{v}:=(1-\Pi_k)\til{\chi}v$ is viewed as an element in $\mc{D}(\mc{Q}_k^\perp)$. As in the proof of Lemma \ref{chiPikV}, there is a uniform constant $C_k$ depending only on $k,\gamma$ such that 
\[ \Big|\int_{\R^-}e^{\gamma c}\mc{Q}_V(u,\Pi_kv)\, dc\Big|\leq C_k|\mc{Q}_{e^{\gamma c}V}(i(u))|^{1/2} \|v\|_{2} \]
which implies that $\iota(\chi u)\in \mc{D}({\bf H})$ since $|\cjg {\bf H}_k^\perp \chi u,v\cjd_{L^2(\R^-\times \Omega_\T)}|\leq \|{\bf H}_k^\perp \chi u\|_2\|v\|_2$. Now we also notice that 
\[\int_{\R}e^{\gamma c}\mc{Q}_V(\chi u,\Pi_kv)\, dc= \cjg e^{\gamma c}V\chi \iota(u),\Pi_k \til{\chi}v\cjd_2\]
where the pairing makes sense since $e^{\gamma c}V\chi \iota(u)\in \mc{D}'(\mc{Q})$ and $\Pi_k\til{\chi}v\in \mc{D}(\mc{Q})$ by Lemma \ref{chiPikV}. Moreover, using $\til{\chi}\chi=\chi$ and $\Pi_k^*=\Pi_k$ this term is also equal to 
\[\cjg e^{\gamma c}V\chi \iota(u),\Pi_k \til{\chi}v\cjd_2=\cjg \chi e^{\gamma c}\Pi_k(V\iota(u)),v\cjd_2.\]
This shows the formula \eqref{identite_entre_H}.

The spectrum of ${\bf H}_k^\perp$ is contained in $[ \tfrac{Q^2}{2}+\la_k,\infty)$ due to \eqref{boundsob}. 
It will be said to have Dirichlet condition at $c=0$, 
by analogy with the Laplacian on finite dimensional manifolds. 
The resolvent $\mathbf{R}_k^\perp(\alpha):=(\mathbf{H}_k^\perp-\tfrac{Q^2+p^2}{2})^{-1}$ (with $\alpha=Q+ip$) is well defined as a bounded map
\begin{align*}
&\mathbf{R}_k^\perp(\alpha): L^2(\R^-;E_k^\perp)\to \mc{D}(\mathbf{H}_k^\perp)
&{\bf R}_k^\perp(\alpha): \mc{D}'(\mc{Q}_k^\perp)\to \mc{D}(\mc{Q}_k^\perp)
\end{align*}
if $p\in\C$ is such that $p^2\notin [2\la_k,\infty)$, with $L^2$
norm 
\begin{equation}\label{boundRkalpha} 
\|{\bf R}_k^\perp(\alpha)\|_{\mc{L}(L^2)}\leq 2/\la_k \textrm{ for }|p|^2\leq \la_k.
\end{equation}
It is also holomorphic in $\alpha$ in $\{{\rm Re}(\alpha)<Q\}\cup \{|\alpha-Q|^2<\la_k\}$.
We also define the dual map $\iota^T:\mc{D}'(\mc{Q})\to\mc{D}'(\mc{Q}_k^\perp)$ (which also maps $L^2(\R\times \Omega_\T)\to L^2(\R^-;E_k^\perp)$. We get 
\begin{align*}
& \iota \chi \mathbf{R}_k^\perp(\alpha)\iota^T(1-\Pi_k): L^2(\R^-\times\Omega_\T)\to \mc{D}(\mathbf{H})\cap L^2(\R;E_k^\perp)
& \iota\mathbf{R}_k^\perp(\alpha)\iota^T:\mc{D}'(\mc{Q})\to \mc{D}(\mc{Q})
\end{align*}
with the same properties. To avoid heavy notations, we shall now remove the $\iota,\iota^T$ maps in the notations so that we view $\chi \mathbf{R}_k^\perp(\alpha)(1-\Pi_k)$ as map $L^2(\R^-\times\Omega_\T)\to \mc{D}(\mathbf{H})\cap L^2(\R;E_k^\perp)$.

Using \eqref{boundsob} with $u={\bf R}_k^\perp(\alpha)f$ we also obtain that
\begin{equation} 
\label{boundplcRk}
\|\pl_c {\bf R}_k^\perp(\alpha)\|_{\mc{L}(L^2)}\leq \sqrt{\frac{4}{\la_k}}, \textrm{ for }|p|^2\leq \la_k/2.
\end{equation} 
Now we fix $\chi,\til{\chi}$ as defined before the Lemma.
Thus, using $\Pi_k{\bf H}_0={\bf H}_0\Pi_k$ and \eqref{identite_entre_H}, we get for $\alpha=Q+ip$ \[\begin{split}(\mathbf{H}-\tfrac{Q^2+p^2}{2})\tilde{\chi}\mathbf{R}_k^\perp(\alpha)(1-\Pi_k)\chi=& (1-\Pi_k)\chi-\tfrac{1}{2}[\pl_c^2,\tilde{\chi}]\mathbf{R}_k^\perp(\alpha)(1-\Pi_k)\chi+
e^{\gamma c}\tilde{\chi}\Pi_k V\mathbf{R}_k^\perp(\alpha)\chi (1-\Pi_k)\\
=:& (1-\Pi_k)\chi+\mathbf{L}_k^\perp(\alpha)+\mathbf{K}_k^\perp(\alpha).
\end{split}\]
Since $[\pl_c^2,\tilde{\chi}]$ is a first order operator with compact support in $c$ commuting with $\Pi_k$, we notice that $\mathbf{L}_k^\perp(\alpha):\mc{D}'(\mc{Q})\to L^2(\R^-;E_k^\perp)$ and
we can use \eqref{boundplcRk}, \eqref{boundRkalpha} to deduce that there is $C>0$ depending only on $|\tilde{\chi}'|_{\infty}, |\tilde{\chi}''|_{\infty}$ such that 
\[\|\mathbf{L}_k^\perp(\alpha)\|_{\mc{L}(L^2)}\leq C\la_k^{-1/2} \]
 as long as  $|p^2|\leq \la_k$.  Let us now deal with ${\bf K}_k^\perp(\alpha)$. 
First, notice that $\mathbf{K}_k^\perp(\alpha)$ maps $\mc{D}'(\mc{Q})$ to $e^{\gamma \rho/2}L^2(\R^-;E_k)$: indeed,
\[\mathbf{R}_k^\perp(\alpha)(1-\Pi_k)\chi: \mc{D}'(\mc{Q})\to \mc{D}(\mc{Q}), \quad  e^{\gamma c}\tilde{\chi}\Pi_k V: \mc{D}(\mc{Q})\to e^{\gamma \rho/2}L^2(\R^-;E_k)\]
using  Lemma \ref{chiPikV} with $\beta=\beta'=-\gamma/2$.
We would like to prove some regularization property in $c$ to deduce that $\mathbf{K}_k^\perp(\alpha)$ is compact on $L^2$ (or some weighted $L^2$ space). First, we have 
\[ \pl_c {\bf H}_k^\perp = {\bf H}_k^\perp\pl_c +\gamma e^{\gamma c}(1-\Pi_k)V(1-\Pi_k)\]
thus applying ${\bf R}_k^\perp(\alpha)$ on the left and right:
\[ {\bf R}_k^\perp(\alpha)\pl_c=\pl_c{\bf R}_k^\perp(\alpha)+\gamma {\bf R}_k^\perp(\alpha)e^{\gamma c}V{\bf R}_k^\perp(\alpha).\]
Here the last term really is ${\bf R}_k^\perp(\alpha)\iota^Te^{\gamma c}V\iota {\bf R}_k^\perp(\alpha)$, viewed as a bounded map $L^2(\R^-;E_k^\perp)\to \mc{D}(\mc{Q}_k^\perp)$ using $e^{\gamma c}V:\mc{D}(\mc{Q})\to \mc{D}'(\mc{Q})$.
Therefore (using $[\pl_c,V]=0$)
\[ \begin{split}
\pl_c{\bf K}_k^\perp(\alpha)=& \gamma {\bf K}_k^\perp(\alpha)+e^{\gamma c}\til{\chi}'\Pi_k V{\bf R}_k^\perp(\alpha)(1-\Pi_k)\chi\\
& +\til{\chi}e^{\gamma c}\Pi_k V{\bf R}_k^\perp(\alpha) \pl_c\chi-\gamma e^{\gamma c}\til{\chi}\Pi_k V{\bf R}_k^\perp(\alpha)e^{\gamma c}V{\bf R}_k^\perp(\alpha)\chi.
\end{split}\]
Just as above, the first two terms map $\mc{D}'(\mc{Q})$ to $e^{\gamma \rho/2}L^2(\R;E_k)$.
Next, $\pl_c \chi: L^2(\R\times \Omega_\T)\to \mc{D}'(\mc{Q})$, ${\bf R}_k^\perp(\alpha): \mc{D}'(\mc{Q})\to \mc{D}(\mc{Q})$ and 
$\til{\chi}e^{\gamma c}\Pi_k V: \mc{D}(\mc{Q}) \to L^2(\R\times \Omega_\T)$ are bounded (using again Lemma \ref{chiPikV}), so 
\[\tilde{\chi} e^{\gamma c}\Pi_k V{\bf R}_k^\perp(\alpha) \pl_c\chi: L^2(\R\times \Omega_\T)\to 
e^{\gamma \rho/2}L^2(\R; E_k)\]
is bounded. Finally, by Lemma \ref{chiPikV}
\[ \til{\chi}e^{\gamma c}\Pi_k V{\bf R}_k^\perp(\alpha): \mc{D}'(\mc{Q}_k^\perp)\to e^{\gamma \rho/2}L^2(\R;E_k), \quad 
e^{\gamma c}V{\bf R}_k^\perp(\alpha)\chi : L^2\to \mc{D}'(\mc{Q}_k^\perp)\]
is bounded so 
\[e^{\gamma c}\til{\chi}\Pi_k V{\bf R}_k^\perp(\alpha)e^{\gamma c}V{\bf R}_k^\perp(\alpha)\chi: L^2(\R\times \Omega_\T)\to e^{\gamma \rho/2}L^2(\R; E_k)\]
is also bounded and therefore $\pl_c{\bf K}_k^\perp(\alpha): L^2(\R\times \Omega_\T)\to e^{\gamma \rho/2}L^2(\R; E_k)$. This shows that 
\[{\bf K}_k^\perp(\alpha): L^2(\R\times \Omega_\T)\to e^{\gamma\rho/2}H_0^1(\R;E_k).\]

It is then easy to see this is a compact operator on $L^2$ as announced since $e^{\gamma\rho/2}H_0^1(\R;E_k)$ injects compactly into $e^{\beta \rho}L^2(\R\times \Omega_\T)$ for $\beta<\gamma/2$ by using that $E_k$ has finite dimension. We conclude that 
${\bf K}_k^\perp(\alpha)$ is compact as a map $L^2(\R\times \Omega_\T)\to e^{\beta \rho}L^2(\R\times \Omega_\T)$ if $\beta<\gamma/2$.
Moreover $\mathbf{K}_k^\perp(\alpha), \mathbf{L}_k^\perp(\alpha)$ are holomorphic 
in $\alpha\in \C$ in the region $\{{\rm Re}(\alpha)<Q\}\cup \{|\alpha-Q|^2<\la_k\}$ since ${\bf R}_k^\perp(\alpha)$ is. 
This concludes the proof of 1).\\

Let us next consider the region $\{{\rm Re}(\alpha)\leq Q\}$, and we proceed as in Lemma \ref{resolventweighted}. Let $\mathbf{H}_{k,\beta}^\perp:=e^{\beta \rho}\mathbf{H}_k^\perp e^{-\beta \rho}$
for $\beta\in \R$ which is also given by 
\[ \mathbf{H}_{k,\beta}^\perp= \mathbf{H}_k^\perp+(1-\Pi_k)\big(-\tfrac{\beta^2}{2}(\rho'(c))^2+\tfrac{\beta}{2} \rho''(c)+\beta \rho'(c)\pl_{c}\big)\] 
and the associated sesquilinear form on $\mc{D}(\mc{Q}_k^\perp)$
\[ \mc{Q}_{k,\beta}^\perp(u):=\mc{Q}_k^\perp(u)-\tfrac{\beta^2}{2}\|\rho'u\|^2_2+\tfrac{\beta}{2} \cjg \rho''u,u\cjd_2+\beta\cjg \pl_cu,\rho'u\cjd_2.\]
Note that on $\mc{D}(\mc{Q}_k^\perp)$, we have 
\[ {\rm Re}(\mc{Q}_{k,\beta}^\perp(u))\geq \mc{Q}_k^\perp(u)-\tfrac{\beta^2}{2}\|u\|_2^2\geq 
\tfrac{1}{2}\| \pl_c u\|^2_2+(\tfrac{Q^2-\beta^2}{2}+\la_k)\|u\|_2^2+\|e^{\frac{\gamma c}{2}}V^{\frac{1}{2}}u\|^2_2.\]
This implies by Lax-Milgram, just as in the proof of Lemma \ref{resolventweighted}, that if 
${\rm Re}((\alpha-Q)^2)>\beta^2-2\la_k$, then 
for each $f\in \mc{D}'(\mc{Q}^\perp_{k})$, there is a unique 
$u\in \mc{D}(\mc{Q}_k^\perp)$ such that \begin{align} 
&e^{\beta \rho}(\mathbf{H}_k^\perp-2\Delta_{\alpha}) e^{-\beta \rho}u=f, \textrm{ and if }f\in L^2\\
&\|u\|_{2}\leq \frac{2\|f\|_{2}}{{\rm Re}((\alpha-Q)^2)+2\la_k-\beta^2},\quad& \|\pl_cu\|_{2}\leq \frac{2\|f\|_{2}}{\sqrt{{\rm Re}((\alpha-Q)^2)+2\la_k-\beta^2}}.\label{estdc}
\end{align}
In particular, this shows that, for ${\rm Re}((\alpha-Q)^2)>\beta^2-2\la_k$, 
$\mathbf{R}_k^\perp(\alpha)$ extends as a map 
\[ \mathbf{R}_k^\perp(\alpha):  e^{-\beta \rho}\mc{D}'(\mc{Q}_{k}^\perp)\to e^{-\beta \rho}\mc{D}(\mc{Q}_k^\perp)
\]
with $\|\mathbf{R}_k^\perp(\alpha)\|_{\mc{L}(e^{\beta \rho}L^2)}\leq 2({\rm Re}((\alpha-Q)^2)+2\la_k-\beta^2)^{-1}$. If we further impose that
${\rm Re}((\alpha-Q)^2)>\beta^2-2\la_k+1$ then, since $e^{\beta \rho}[\pl_c,e^{-\beta \rho}]=-\beta \rho'$ and using \eqref{estdc},  
\[ \|\mathbf{L}_k^\perp(\alpha)\|_{\mc{L}(e^{-\beta \rho}L^2)}\leq \frac{2|\chi'|_\infty(1+|\beta|) +|\chi''|_\infty}{\sqrt{{\rm Re}((\alpha-Q)^2)+2\la_k-\beta^2}}.\]
Finally, the same argument as above for $\mathbf{K}_k^\perp(\alpha)$ shows that for ${\rm Re}((\alpha-Q)^2)+2\la_k-\beta^2>1$, the operator $\mathbf{K}_k^\perp(\alpha)$ is compact from $e^{-\beta \rho}L^2(\R\times\Omega_\T)$ to $e^{-\beta\rho}L^2(\R^-;E_k)$.
\end{proof}
\begin{remark} We notice that the operators ${\bf R}_k^\perp(\alpha)$, $\mathbf{K}_k^\perp(\alpha)$, ${\bf L}_k^\perp(\alpha)$ lift as holomorphic family of operators to the region 
$\{\alpha\in \Sigma\, |\, {\rm Re}(\pi(\alpha))<Q, |\pi(\alpha)-Q|^2<\lambda_k\}$ by simply composing with the projection $\pi: \Sigma\to \C$.\end{remark}
 
\noindent \textbf{2) The region $c\geq -1$.}
Next, consider the operator $\mathbf{H}-\tfrac{Q^2+p^2}{2}$ on $L^2([-1,\infty);L^2(\Omega_\T))$ with Dirichlet condition at $c=-1$ (i.e. the extension associated to the quadratic form on functions supported in $c\geq -1$), and $\hat{\chi}\in C^\infty(\R; [0,1])$  such that $(1-\hat{\chi})=1$ on ${\rm supp}(1-\chi)$ and 
$1-\hat{\chi}$ supported in $(-1,\infty)$ (otherwise stated, $\hat{\chi}=0$ on $(-1+\delta,+\infty)$ and $\hat{\chi}=1$ on $(-\infty,-1)$ ).

We will construct a quasi-compact approximate inverse to ${\bf H}$ in $[-1,\infty)$ by using energy estimates and the properties of $V$, in particular the fact  the region where $V>0$ is small are somehow small. We show the following: 
\begin{lemma}\label{regimec>-1}
There is a uniform constant $C>0$ and a bounded operator, independent of $\alpha$,
\[
\mathbf{R}_+: L^2(\R\times \Omega_\T)\to \mc{D}(\mc{{\bf H}}) ,   
\]
satisfying  
\begin{align*}
& (1-\hat{\chi}) \mathbf{R}_+(1-\chi): \mc{D}'(\mc{Q})\to \mc{D}(\mc{Q})
\end{align*} 
and for $\alpha=Q+ip\in\C$ and $k\geq 1$
\[(\mathbf{H}-\tfrac{Q^2+p^2}{2})(1-\hat{\chi})\mathbf{R}_+ (1-\chi)= (1-\chi)+\mathbf{K}_{+,k}(\alpha)+\mathbf{L}_{+,k}(\alpha)\]
where $\mathbf{K}_{+,k}(\alpha):L^2(\R\times \Omega_\T)\to L^2([-1,\infty)\times \Omega_\T)$ compact and holomorphic in $\alpha\in \C$, and the operator $\mathbf{L}_{+,k}(\alpha):L^2(\R\times \Omega_\T)\to L^2([-1,\infty)\times \Omega_\T)$ is bounded and holomorphic  in $\alpha\in \C$, such that 
\begin{equation}\label{boundbis}
 \|\mathbf{L}_{+,k}(\alpha)\|_{\mc{L}(L^2)}\leq   C (1+|p|^2\lambda_k^{-1/2}),\quad  \|\mathbf{L}_{+,k}(\alpha)^2\|_{\mc{L}(L^2)}\leq  C (1+|p|^2)\lambda_k^{-1/2}+C(1+|p|^4)\lambda_k^{-1}
 \end{equation}
for some uniform constant $C$ depending only on $\hat{\chi}$. Moreover 
$\mathbf{L}_{+,k}(\alpha)$ and $\mathbf{K}_{+,k}(\alpha)$ are bounded as maps $\mc{D}'(\mc{Q})\to L^2([-1,\infty)\times \Omega_\T)$. 
\end{lemma}
\begin{proof}
We consider ${\bf R}_+:=({\bf H}-\tfrac{Q^2}{2}+1)^{-1}: L^2\to \mc{D}({\bf H})$, which we also view as 
a map $\mc{D}'(\mc{Q})\to \mc{D}(\mc{Q})$. We have 
\begin{equation}\label{boundanddefR+}
\tfrac{1}{2}\|\pl_c\mathbf{R}_+f\|^2_{2}+\|\mathbf{R}_+f\|^2_{2}+\|\mathbf{P}^{1/2}\mathbf{R}_+f\|^2_{2}+\mc{Q}_{e^{\gamma c}V}(\mathbf{R}_+f)\leq \|f\|^2_{2},
\end{equation}
We write
\begin{equation}\label{errorK+}
\begin{split} 
(\mathbf{H}-\tfrac{Q^2+p^2}{2})(1-\hat{\chi})\mathbf{R}_+ (1-\chi)=& (1-\chi)+\tfrac{1}{2}[\pl_c^2,\hat{\chi}]\mathbf{R}_+(1-\chi)-(\tfrac{p^2}{2}-1)(1-\hat{\chi})\mathbf{R}_+ (1-\chi).\\
 =:& (1-\chi)+\mathbf{K}_+^{1}+\mathbf{K}_+^{2}(\alpha).
\end{split}
\end{equation}
Notice that $\mathbf{K}_+^{1},\mathbf{K}_+^2(\alpha)$ are bounded as maps $\mc{D}'(\mc{Q})\to L^2(\R\times \Omega_\T)$ by using that $[\pl_c^2,\hat{\chi}]:\mc{D}(\mc{Q})\to L^2(\R\times \Omega_\T)$ is bounded. 
By Lemma \ref{decroissance_c_grand}, there is $\beta>0$ such that have (here $c_+:=c\mathbf{1}_{c>0}$.)
\begin{equation}\label{weightedPik}
\Pi_k\mathbf{R}_+ : L^2(\R\times \Omega_\T)\to e^{-\beta c_+}L^2(\R; E_k)
\end{equation}
is bounded. Let $\psi\in C_c^\infty(\R)$ non-negative, equal to $1$ in $[-A,A]$ and $0$ for $|c|>2A$ and so that $\|\psi\|_\infty+\|\psi'\|_\infty+\|\psi''\|_\infty\leq 1$ (we shall take $A\to \infty$ later). By using Lemma \ref{chiPikV}, we have 
\[  \Pi_k \psi :\mc{D}(\mc{Q})\to \mc{D}(\mc{Q})\]
is bounded: indeed, $\pl_c(\psi u)=\psi' u+\psi\pl_cu\in L^2$ if $u\in \mc{D}(\mc{Q})$ and 
\[ \int_{\R} \psi^2(c)e^{\gamma c}\mc{Q}_V(u,u)dc \leq C_{\gamma,\psi}\mc{Q}(u,u)\]
for some $C_{\gamma,\psi}$ depending on $\gamma$ and $\psi$.

We are going to show that for all such $\psi$, 
\begin{equation}\label{rayas2005}
\|(1-\Pi_k)\psi {\bf R}_+\|_{L^2\to L^2}\leq 3/\sqrt{\la_k},
\end{equation}
which by taking $A\to \infty$ shows that $(1-\Pi_k){\bf R}_+$ has $\mc{L}(L^2)$ norm 
bounded by $3/\sqrt{\la_k}$.
Let $u=\mathbf{R}_+f$ with $f\in L^2$, then we claim that
\begin{equation}\label{bound-Pik}
\begin{gathered} 
\la_k\|(1-\Pi_k)\psi u\|_{2}^2 +\tfrac{1}{2}\|\pl_c(1-\Pi_k)\psi u\|_{2}^2 + \tfrac{1}{2}\mc{Q}_{e^{\gamma c}V}((1-\Pi_k)\psi u)    \\
\leq \|(1-\Pi_k)\psi f\|_{2}\|(1-\Pi_k)\psi u\|_{2}+\tfrac{1}{2}\mc{Q}_{e^{\gamma c}V}(\Pi_k\psi u)+2\|f\|^2_2
\end{gathered}
\end{equation} 
To prove this, we note that, in $\mc{D}'(\mc{Q})$ 
\[  (\mathbf{H}-\tfrac{Q^2}{2}+1)(1-\Pi_k)\psi u=
\psi(1-\Pi_k)f+\psi(\Pi_ke^{\gamma c}V-e^{\gamma c}V\Pi_k)u-\tfrac{1}{2}\psi''(1-\Pi_k)u-\psi'(1-\Pi_k)\pl_c \] 
thus pairing with $(1-\Pi_k)\psi u\in \mc{D}(\mc{Q})$, this gives 
\[\begin{split}
\mc{Q}((1-\Pi_k)\psi u, (1-\Pi_k)\psi u)=& \cjg (1-\Pi_k)\psi f,(1-\Pi_k)\psi u\cjd_2-
\mc{Q}_{e^{\gamma c}V}\Big(\psi \Pi_k u,\psi(1-\Pi_k)u\Big)\\
& -\tfrac{1}{2}\cjg \psi''(1-\Pi_k)u,(1-\Pi_k)\psi u\cjd_2-\cjg \psi'(1-\Pi_k)\pl_c u,(1-\Pi_k)\psi u\cjd_2
\end{split}\]
 This gives the bound \eqref{bound-Pik}, using Cauchy-Schwarz, $\|{\bf P}^{1/2}(1-\Pi_k)u\|^2_{L^2(\Omega_\T)}\geq \la_k\|u\|^2_{L^2(\Omega_\T)}$ and the fact that $\|u\|_{\mc{D}(\mc{Q})}\leq \|f\|_2$ (by \eqref{boundanddefR+}).
Now we do the same computation with $(1-\Pi_k)$ replaced by $\Pi_k$: 
\begin{equation}\label{bound1-Pik}
\begin{gathered}
\|\Pi_k \psi u\|_{2}^2 +\tfrac{1}{2}\|\pl_c\Pi_k\psi u\|_{2}^2 + \tfrac{1}{2}\mc{Q}_{e^{\gamma c}V}(\Pi_k\psi u)    \\
\leq \|\Pi_k\psi f\|_{2}\|\Pi_k\psi u\|_{2}+\tfrac{1}{2}\mc{Q}_{e^{\gamma c}V}((1-\Pi_k)\psi u)+2\|f\|^2_2.
\end{gathered}
\end{equation}
Combining \eqref{bound1-Pik} and \eqref{bound-Pik} and using $\|u\|_{2}\leq \|f\|_{2}$, we obtain
\begin{equation}\label{final1-Pik} 
\|(1-\Pi_k)\psi\mathbf{R}_{+} f\|_{2}=\|(1-\Pi_k)\psi u\|_{2}\leq \frac{3}{\sqrt{\la_k}}\|f\|_2
\end{equation}
Since $[\pl_c^2,\hat{\chi}]=\hat{\chi}''+2\hat{\chi}'\pl_c$ and $\hat{\chi}'=0$ on 
${\rm supp}(1-\chi)$, we have 
$(\mathbf{K}_+^1)^2=0$ and $\|\mathbf{K}_+^1\|_{\mc{L}(L^2)}\leq C$ (using \eqref{boundanddefR+}) for some uniform $C$ depending only on $\hat{\chi}$. 
By combining with \eqref{final1-Pik}, we deduce that 
\[ \begin{gathered}
\|(\mathbf{K}_+^1+ (1-\Pi_k)\mathbf{K}_+^2(\alpha))\|_{\mc{L}(L^2)}\leq C(1+|p|^2\lambda_k^{-1/2}), \\ 
\|(\mathbf{K}_+^1+ (1-\Pi_k)\mathbf{K}_+^2(\alpha))^2\|_{\mc{L}(L^2)}\leq C ((1+|p|^2)\lambda_k^{-1/2}+(1+|p|^2)^2\lambda_k^{-1})
\end{gathered}\] 
for some uniform $C$ depending only on $\hat{\chi}$.
Next we  consider the operator $\Pi_k\mathbf{K}_+^2(\alpha)$. 
Recall that, by \eqref{boundanddefR+}, 
\begin{equation}\label{weightedPik+der}
\pl_c\Pi_k\mathbf{R}_+=\Pi_k \pl_c{\bf R}_+: L^2(\R\times \Omega_\T)\to L^2(\R\times \Omega_\T)
\end{equation}
is bounded. Now we claim that the injection 
\begin{equation}\label{injectioncpt}
 F_k:=\{u\in e^{-\beta c}L^2([-1,\infty); E_k)\,|\, 
\pl_cu\in L^2([-1,\infty); E_k)\}\to e^{-\frac{\beta c}{2}}L^2([-1,\infty)\times \Omega_\T)
\end{equation}
is compact if we put the norm $\|u\|_{F_k}:=\|e^{\beta c}u\|_{2}+\|\pl_c u\|_{2}$
on $F_k$. Indeed, consider the operator $\eta_T{\rm Id}:F_k\to e^{-\frac{\beta c}{2}}L^2([-1,\infty); E_k)$ where $\eta_T(c)=\eta(c/T)$ if $\eta\in C_0^\infty((-2,2))$ is equal to $1$ on $(-1,1)$ and $0\leq\eta\leq 1$. Since $E_k$ has finite dimension, this is a compact operator by the compact embedding $H^1([-1,T);E_k)\to e^{-\frac{\beta c}{2}}L^2([-1,\infty);E_k)$,  and as $T\to \infty$ we have for $u\in F_k$
\[ \| e^{\frac{\beta}{2}c}(\eta_Tu-u)\|^2_{2}\leq e^{-\beta T}\int_{-1}^\infty (1-\eta_T)^2e^{2\beta c}\|u\|_{L^2(\Omega_\T)}^2 dc\leq e^{-\beta T}\|u\|^2_{F_k}
\]
thus the injection \eqref{injectioncpt} is a limit of compact operators for the operator norm topology, therefore is compact. 
By \eqref{weightedPik} and \eqref{weightedPik+der}, the operator $\Pi_k\mathbf{R}_+: 
L^2(\R\times \Omega_\T)\to e^{-\frac{\beta}{2}c_+}L^2([-1,\infty)\times \Omega_\T)$ is compact. 
We get that  
\[
\Pi_k \mathbf{K}_+^2(\alpha): L^2(\R \times \Omega_\T)\to L^2(\R \times E_k)
\]
is compact. This complete the proof by setting $\mathbf{K}_{+,k}(\alpha):=\Pi_k\mathbf{K}_+^2(\alpha)$ and 
$\mathbf{L}_{+,k}(\alpha):=\mathbf{K}_+^1+ (1-\Pi_k)\mathbf{K}_+^2(\alpha)$.
The holomorphy in $\alpha\in \C$ is clear since $\mathbf{K}_+^2(\alpha)$ is polynomial in $\alpha$.
\end{proof}

\begin{lemma}\label{decroissance_c_grand}
For $\alpha=Q+ip$ with $p\in i\R^*$, there is $\beta>0$ such that the resolvent $\Pi_k\mathbf{R}(\alpha): L^2(\R \times \Omega_\T)\to  e^{-\beta c_+}L^2(\R \times \Omega_\T)$ is bounded.
\end{lemma}

\begin{proof} We estimate the norm for $ e^{-\beta c_+}L^2(\R_+ \times \Omega_\T)$ since the estimate for   $ e^{-\beta c_+}L^2(\R^- \times \Omega_\T)$ results from Proposition \ref{l2alphamu}.  
Recall \eqref{resolvvspropagator} with the representation for the resolvent $
\mathbf{R}(\alpha)=\int_0^{\infty} e^{-t\mathbf{H}+t2 \Delta_{\alpha}}\dd t
$, valid for the range of $\alpha$ we consider. A key observation to establish our lemma is the following estimate on $  e^{-t\mathbf{H}}$, based on the  Feynman-Kac formula \eqref{FKgeneral}. Here we write for simplicity $V_t:=\int_{\D_t^c} |z|^{-\gamma Q}M_\gamma (\dd z)$. By using the    Markov property, we get that for $f\in L^2(\R \times \Omega_\T)$ and $a\in (0,1)$, 
\begin{align}
|e^{-t\mathbf{H}}f|=&| e^{-\frac{Q^2}{2}t}\E_{\varphi}\big[ f(c+B_t,\varphi_t)e^{-\mu e^{\gamma c}V_t}\big]| \nonumber\\
\leq & e^{-\frac{Q^2}{2}t} \E_{\varphi}\big[ |f(c+B_t,\varphi_t)|e^{-\mu e^{\gamma c}V_{(1-a)t}}\big]\nonumber \\
 =& e^{-\frac{Q^2}{2}(1-a)t} \E_{\varphi}\big[ e^{-at \mathbf{H}_0}|f|(c+B_{(1-a)t},\varphi_{(1-a)t}) e^{-\mu e^{\gamma c}V_{(1-a)t}}\big].  \label{markineg}
  \end{align}
 
 Now we want to exploit this relation in the integral representation of the resolvent.
Take $q\in[1,2)$. By using in turn the fact that $\Pi_k$ has finite dimensional  rank (in particular we have equivalence of norms on its rank) and the continuity of the map $\Pi_k: L^q(\Omega_\T)\to L^q(\Omega_\T)$ (see \cite[Th 5.14]{SJ}),  we obtain for some constant $C=C(t_0,\alpha,k,q,\mu,\gamma)$ which may change along the lines below
\begin{align*}
\int_0^\infty &e^{2\beta c}\E\Big[\Big(\Pi_k\int_0^{\infty} e^{-t\mathbf{H}+t2 \Delta_{\alpha}}f\dd t\Big)^2\Big]\,\dd c\\
\leq & C \int_0^\infty  e^{2\beta c}\E\Big[\big(\Pi_k\int_0^{\infty} e^{-t\mathbf{H}+t2 \Delta_{\alpha}}f\dd t\big)^q\Big] ^{2/q}\,\dd c\\
\leq & C \int_0^\infty  e^{2\beta c}\E\Big[\big| \int_0^{\infty} e^{-t\mathbf{H}+t2 \Delta_{\alpha}}f\dd t\big|^q\Big] ^{2/q}\,\dd c.
\end{align*}
 We are going to estimate this quantity by splitting the $\dd t$-integral above in two parts, $\int_0^{t_0}\dots$ and $\int_{t_0}^{\infty}\dots$ for some $t_0>0$, which we call respectively $A_1$ and $A_2$.
 
 Let us first focus on the first part.  Using  the relation \eqref{markineg} with $a=1/2$ produces the bound
\begin{align*}
A_1\leq C \int_0^\infty  e^{2 \beta  c}\E\Big[ \Big|\int_0^{t_0}   \E_{\varphi}\big[  e^{-\frac{t}{2}\mathbf{H}_0}|f|(c+B_{t/2},\varphi_{t/2})  e^{- \mu e^{\gamma c}V_{t/2}}  \big]  \,\dd t\Big|^q\Big]^{2/q} \dd c.
\end{align*}
 Then Jensen gives the estimate
 \begin{align*}
A_1\leq C \int_0^\infty  e^{2 \beta  c} \int_0^{t_0} \E\Big[   \big(  e^{-\frac{t}{2}\mathbf{H}_0}|f|(c+B_{t/2},\varphi_{t/2}) \big)^q e^{- q\mu e^{\gamma c}V_{t/2}}     \Big]^{2/q} \,\dd t\dd c.
\end{align*}
Now we use   H\"older's inequality  with conjugate exponents $2/q,  2/ (2-q)$ to estimate the above quantity by
 $$
A_1\leq C \int_0^\infty\int_0^{t_0}   e^{2\beta  c} \E\Big[  \big(e^{-\frac{t}{2}\mathbf{H}_0}|f|(c+B_{t/2},\varphi_{t/2})\big)^2\Big]    \Big(\E\big[e^{- \frac{2q }{2-q}\mu e^{\gamma c}V_{t/2}}     \big]\Big)^{\frac{2-q}{q}}\,\dd t\dd c.
$$
Using the elementary inequality $x^\theta e^{-x}\leq C$ for $\theta>0$, we deduce
 $$
A_1\leq C \int_0^\infty\int_0^{t_0}   e^{(2\beta-\theta \frac{2-q}{q}\gamma) c} \E\Big[  \big(e^{-\frac{t}{2}\mathbf{H}_0}|f|(c+B_{t/2},\varphi_{t/2})\big)^2\Big]    \Big(\E\big[  V_{t/2}^{-\theta}    \big]\Big)^{\frac{2-q}{q}}\,\dd t\dd c.
$$
 It remains to estimate the quantity $\E\big[  V_{t/2}^{-\theta}    \big]$. Notice that $V_{t/2}$ is a GMC over an annulus of radii $1$ and $e^{-t/2}$. Hence the quantity to be investigated is less than the same expression with $V_{t/2}$ replaced by a GMC over a ball of diameter $t/2$ contained in   this annulus. Standard  computations of GMC multifractal spectrum \cite[Th 2.14]{review} (also valid for negative moments) then yields that
\begin{equation}
\E\Big(  V_{t/2}^{-\theta}    \Big)\leq C t^{\psi(-\theta)}
\end{equation}
where $\psi(u):=(2+\tfrac{\gamma^2}{2})u-\tfrac{\gamma^2}{2}u^2 $ is the multifractal spectrum. Since $\psi$ is continuous and $\psi(0)=0$, we can find   $\theta>0$ such that $\psi(-\theta)>-1$. Then, for $2\beta=\frac{2-q}{q}\theta\gamma$, we deduce that
\begin{align*}
A_1
\leq & C  \int_0^\infty\int_0^{t_0}   \E\Big[ \big( e^{-\frac{t}{2}\mathbf{H}_0}|f|(c+B_{t/2},\varphi_{t/2})\big)^2\Big]  t^{\psi(-\theta)}\,\dd t\dd c\\
\leq & C  \int_0^{t_0}  \|e^{-\frac{t}{2}\mathbf{H}_0}|f|\|_2^2  t^{\psi(-\theta)}\,\dd t\\
\leq  & C\|f\|_2^2 \int_0^{t_0}   t^{\psi(-\theta)}\,\dd t,
\end{align*}
which proves our claim for $A_1$, namely the part corresponding to $\int_0^{t_0}$.

Concerning  $A_2$, we start as above
\begin{align*}
A_2:=&\int_0^\infty  e^{2\beta c}\E\Big[\Big(\Pi_k\int_{t_0}^{\infty} e^{-t\mathbf{H}+t2 \Delta_{\alpha}}f\dd t\Big)^2\Big]\,\dd c\\
\leq & C \int_0^\infty  e^{2\beta c}\Big(\E\Big[\big( \int_{t_0}^{\infty}  e^{-t\mathbf{H}+t2 \Delta_{\alpha}}f\dd t\big)^q\Big]\Big)^{2/q}\,\dd c
\end{align*}
for $q\in (1,2)$. For the values of $\alpha$ we consider, we have $2\Delta_\alpha\leq \frac{Q^2}{2}-\frac{|p|^2}{2}$ so that, using Jensen,  
\begin{align*}
A_2\leq C \int_0^\infty  e^{2\beta c}\Big(\E\big[ \int_{t_0}^{\infty}e^{-\frac{|p|^2}{2}t}  \big(\E_{\varphi}\big[ |f(c+B_t,\varphi_t)|e^{-\mu e^{\gamma c}V_{t}}\big] \big)^q   \dd t \big]\Big)^{2/q}\,\dd c
\end{align*}
Let us now fix $a\in (0,1)$ small  enough  such that $aQ^2<|p|^2/2$. Then, using  \eqref{markineg},  Jensen and the inequality $x^\theta e^{-x}\leq C$ for $\theta>0$, we deduce
\begin{align*}
A_2\leq   &  C \int_0^\infty  e^{2(\beta-\gamma\theta/q) c}\Big(\E\Big[ \int_{t_0}^{\infty}e^{-\frac{|p|^2}{4}t}   \E_{\varphi}\big[ \big(e^{-at\mathbf{H}_0}|f|(c+B_{(1-a)t},\varphi_{(1-a)t} )\big)^q V_{(1-a)t}^{- \theta}\big]    \dd t\Big] \Big)^{2/q}\,\dd c.
\end{align*}
Now we can proceed similarly  as above by using the fact that for $a<1$ and all $q>0$ $\sup_{t>t_0}\E[V_{(1-a)t}^{-\theta  }]<\infty$ (this quantity is decreasing in $t$ and GMC has negative moments) to get
\begin{align*}
A_2\leq &  C \int_0^\infty  e^{2(\beta-\gamma\theta/q) c}\Big(\int_{t_0}^{\infty}e^{-\frac{|p|^2}{4}t} \E\big[  \big(   e^{-at\mathbf{H}_0}|f|(c+B_{(1-a)t},\varphi_{(1-a)t} )\big)^2\big]^{q/2} \E\big[ V_{(1-a)t}^{-\frac{2\theta}{2-q}}\big] ^{\frac{2-q}{2}}   \dd t \Big)^{2/q}\,\dd c\\
\leq &  C \int_0^\infty  e^{2(\beta-\gamma\theta/q) c}\Big(\int_{t_0}^{\infty}e^{-\frac{|p|^2}{4}t} \E\Big[ \big(   e^{-at\mathbf{H}_0}|f|(c+B_{(1-a)t},\varphi_{(1-a)t} )\big)^2\Big]^{q/2}  \dd t \Big)^{2/q}\,\dd c.
\end{align*}
For $\beta=\theta\gamma/q$ and using Jensen, the above quantity is less than
$$ C  \int_{t_0}^{\infty}e^{-\frac{|p|^2}{4}t} \|e^{-at\mathbf{H}_0}|f|\|_2^2  \dd t  \leq C\|f\|_2^2\int_{t_0}^{\infty}e^{-\frac{|p|^2}{4}t}\,\dd t.$$
Hence the result.
 \end{proof}

\begin{remark} As above, the operators ${\bf R}_+, \mathbf{K}_{+,k}(\alpha)$ and $\mathbf{L}_{+,k}(\alpha)$, lift as holomorphic family of operators to $\Sigma$.\end{remark}

\noindent \textbf{3) Small $\mathbf{P}$ eigenmodes in the region $c\leq 0$, where there is scattering.} 
We will view the operator ${\bf H}$ as a perturbation of the free Hamiltonian 
$\mathbf{H}_{0}:= -\tfrac{1}{2}\pl_c^2 +\tfrac{Q^2}{2}+\mathbf{P}$ on $L^2(\R^-\times \Omega_\T)$ with Dirichlet condition at $c=0$. We show (recall that $\pi:\Sigma\to \C$ is the covering map)
\begin{lemma}\label{modelres}
1) Fix $k$ and $0<\beta<\gamma/2$. The operators  
\begin{equation}
\mathbf{R}_k(\alpha):= (\mathbf{H}_{0}-\tfrac{Q^2+p^2}{2})^{-1}\Pi_k: e^{\beta\rho}L^2(\R^{-}\times \Omega_\T)\to e^{-\beta \rho}L^2(\R^-;E_k)
\end{equation}
defined for ${\rm Im}(p)>0$ can be holomorphically continued to the region 
\begin{equation}\label{region_holo_R_k}\{\alpha=Q+ip \in \Sigma\,|\, \forall j\leq k, {\rm Im}\sqrt{p^2-2\la_j}>-\beta\}.
\end{equation} 
This continuation, still denoted  $\mathbf{R}_k(\alpha):e^{\beta \rho}L^2(\R\times\Omega_\T)\to e^{-\beta \rho}L^2(\R^-;E_k)$, satisfies 
\[\tilde{\chi}\mathbf{R}_k(\alpha)\chi:e^{\beta \rho}L^2(\R\times\Omega_\T)\to e^{-\beta \rho}(L^2(\R^-;E_k)\cap \mc{D}(\mc{Q})),\]
\[ (\mathbf{H}-\tfrac{Q^2+p^2}{2})\tilde{\chi}\mathbf{R}_k(\alpha)\chi =\Pi_k \chi+\mathbf{K}_{k,1}(\alpha)+{\bf K}_{k,2}(\alpha)\]
where $\mathbf{K}_{k,1}(\alpha)$, $\mathbf{K}_{k,2}(\alpha)$ are such that for ${\rm Im}\sqrt{p^2-2\la_j}>-\min(\beta,\gamma/2-\beta)$
\begin{align*}
& \mathbf{K}_{k,1}(\alpha):e^{\beta \rho}L^2(\R\times\Omega_\T)\to e^{\beta \rho}L^2(\R\times \Omega_\T)
\\
& {\bf K}_{k,2}(\alpha):e^{\beta \rho}L^2(\R\times \Omega_\T)\to e^{\beta \rho}\mc{D}'(\mc{Q})
\end{align*}
are holomorphic families of compact operators in \eqref{region_holo_R_k}, and we have  $\mathbf{K}_{k,i}(\alpha)(1-\Pi_k)=0$ for $i=1,2$, $(1-\Pi_k)\mathbf{K}_{k,1}(\alpha)=0$ and $\Pi_k\mathbf{K}_{k,2}(\alpha)=0$.\\
2) If $f\in e^{\beta \rho}L^2$, then there is $C_k>0$ depending on $k$, some $a_j(\alpha,f)$ and 
$G(\alpha,f)\in H^1(\R;E_k)$ depending linearly on $f$ and holomorphic in $\{\alpha=Q+ip \in \Sigma\,|\,  \forall j\leq k, {\rm Im}\sqrt{p^2-2\la_j}>-\min(\beta,\gamma/2-\beta)\}$ such that 
in the region $c\leq 0$
\begin{equation}\label{asympR_-}
 (\mathbf{R}_k(\alpha)f)=\sum_{\la_j\leq \la_k}a_j(\alpha,f)
e^{-ic\sqrt{p^2-2\la_j}}+G(\alpha,f),
\end{equation}
\[ \|G(\alpha,f)(c)\|_{L^2(\Omega_\T)}+\|\pl_cG(\alpha,f)(c)\|_{L^2(\Omega_\T)}\leq C_k e^{\beta \rho}
\|e^{-\beta \rho}\Pi_kf\|_{2}.\]
3) For each $\beta\in \R$, the operator $\mathbf{R}_k(\alpha)$ extends as a bounded analytic family 
\[ \mathbf{R}_k(\alpha): e^{-\beta \rho}L^2(\R^-\times \Omega_\T)\to e^{-\beta \rho}(L^2(\R^-;E_k)\cap \mc{D}(\mc{Q}))\]
in the region ${\rm Im}(p)>|\beta|$ and $\mathbf{K}_{k,1}(\alpha): e^{-\beta \rho}L^2(\R\times \Omega_\T)\to e^{-\beta \rho}L^2(\R\times\Omega_\T)$, $\mathbf{K}_{k,2}(\alpha): e^{-\beta \rho}L^2(\R\times \Omega_\T)\to e^{-\beta \rho}\mc{D}'(\mc{Q})$ are compact analytic families  in that same region.
\end{lemma}
\begin{proof}
We first consider $\mathbf{H}_{0}$ on $(-\infty,0]$ with Dirichlet condition at $c=0$.  Using the diagonalisation of $\mathbf{P}$ on $E_k$, we can compute the resolvent 
$(\mathbf{H}_{0}-\tfrac{Q^2+p^2}{2})^{-1}$ on $E_k$ by standard ODE methods (Sturm-Liouville theory): for ${\rm Im}(p)>0$, this is the diagonal operator given for $j\leq k$ and $f\in L^2(\R^-)$ and $\phi_j\in \ker (\mathbf{P}-\la_j)$
\[\begin{split} 
(\mathbf{H}_{0}-\tfrac{Q^2+p^2}{2})^{-1}f(c)\phi_j =\frac{2}{\sqrt{p^2-2\la_j}}& \phi_j\Big(\int_{-\infty}^c \sin(c\sqrt{p^2-2\la_j})e^{-ic'\sqrt{p^2-2\la_j}}f(c')dc'\\
 & +\int_{c}^0 e^{-ic\sqrt{p^2-2\la_j}}\sin(c'\sqrt{p^2-2\la_j})f(c')dc'\Big)
\end{split}\]
where our convention is that $\sqrt{z}$ is defined with the cut on $\R^+$, so that $\sqrt{p^2}=p$
if ${\rm Im}(p)>0$. For $j=0$, that is $\phi_0=1$, for each $\beta>0$ the resolvent restricted to $E_0$ admits an analytic continuation from ${\rm Im}(p)>0$ to ${\rm Im}(p)>-\beta$, as a map
\[ (\mathbf{H}_{0}-\tfrac{Q^2+p^2}{2})^{-1}\Pi_0: e^{-\beta |c|}L^2(\R^-\times \Omega_\T)\to e^{\beta |c|}L^2(\R^-;E_0).\]
This is easy to see by using Schur's lemma and the analyticity in $p$ for the Schwartz kernel 
\[ \kappa_0(c,c'):= \indic_{\{c\geq c'\}}e^{-\beta (|c|+|c'|)} \sin(cp)e^{-ic'p} +\indic_{\{c'\geq c\}}e^{-\beta (|c'|+|c|)} \sin(c'p)e^{-icp}\]
of the operator $\Pi_0 e^{-\beta |c|}(\mathbf{H}_{0}-\tfrac{Q^2+p^2}{2})^{-1}e^{-\beta |c|}\Pi_0$ that we view as on operator on $L^2(\R^-)$. Moreover, one directly also 
obtains that it maps $e^{-\beta |c|}L^2(\R^-\times \Omega_\T)\to e^{-\beta |c|}H^2(\R^-;E_0)\cap H_0^1(\R^-;E_0)$.
 Similarly, the operators 
\[ (\mathbf{H}_{0}-\tfrac{Q^2+p^2}{2})^{-1} \phi_j\cjg \phi_j,\cdot\cjd:e^{-\beta |c|}L^2(\R^-\times \Omega_\T)\to e^{\beta |c|}L^2(\R^-;\C \phi_j)\]
are analytic in $p$, which implies that 
\[\mathbf{R}_k(\alpha):= (\mathbf{H}_{0}-\tfrac{Q^2+p^2}{2})^{-1}\Pi_k: e^{-\beta |c|}L^2(\R^-\times \Omega_\T)\to e^{\beta |c|}H^2(\R^-;E_k)\cap  H_0^1(\R^-;E_k)\]
admits an analytic extension in $p$ to the region $\{p\in Z\, |\, \forall j\geq 0, {\rm Im}(\sqrt{p^2-2\la_j})>-\beta\}$ of the ramified Riemann surface $Z$. The proof of Lemma \ref{chiPikV} also shows  that $\mc{Q}_{e^{\gamma c}V}(\til{\chi}\Pi_ku)<\infty$ if $u\in L^2(\R\times \Omega_\T)$ and clearly also ${\bf P}^{\frac{1}{2}}\Pi_k\in \mc{L}(L^2(\Omega_\T))$, thus we deduce that 
\[\tilde{\chi} \mathbf{R}_k(\alpha)\chi:e^{-\beta |c|}L^2(\R^-\times \Omega_\T)\to e^{\beta |c|} (L^2(\R^-;E_k)\cap \mc{D}(\mc{Q}))\]
is bounded.
We have (using $\Pi_k\mathbf{R}_k(\alpha)=\mathbf{R}_k(\alpha)$ and Lemma \ref{chiPikV})
\[\begin{split} 
(\mathbf{H}-\tfrac{Q^2+p^2}{2})\tilde{\chi} \mathbf{R}_k(\alpha)\chi= &\, \Pi_k\chi-\tfrac{1}{2} [\pl_c^2,\tilde{\chi}]\Pi_k \mathbf{R}_k(\alpha)\chi+
e^{\gamma c}V\Pi_k\tilde{\chi}\mathbf{R}_k(\alpha)\chi.\\
=& \, \Pi_k\chi+ \tfrac{1}{2} [\pl_c^2,\tilde{\chi}]\Pi_k \mathbf{R}_k(\alpha)\chi+
\Pi_ke^{\gamma c} V\Pi_k\tilde{\chi}\mathbf{R}_k(\alpha)\chi+\tilde{\chi}(1-\Pi_k)e^{\gamma c}V\Pi_k\tilde{\chi}\mathbf{R}_k(\alpha)\chi\\
= & \, \Pi_k \chi+\mathbf{K}_{k,1}(\alpha)+\mathbf{K}_{k,2}(\alpha)
\end{split}\]
where $\mathbf{K}_{k,2}(\alpha):=\tilde{\chi}(1-\Pi_k)e^{\gamma c}V\Pi_k\tilde{\chi}\mathbf{R}_k(\alpha)\chi$
satisfies $\Pi_k {\bf K}_{k,2}(\alpha)=0$. 
The operator $[\pl_c^2,\tilde{\chi}]\Pi_k \mathbf{R}_k(\alpha)\chi$ is compact on $e^{-\beta|c|}L^2(\R;L^2(\Omega_\T))$ since $e^{\beta |c|}[\pl_c^2,\tilde{\chi}]$ is a compactly supported first order operator in $c$, $E_k={\rm Im}(\Pi_k)$ is finite dimensional in $L^2(\Omega_\T)$
and $\mathbf{R}_k(\alpha): e^{-\beta|c|}L^2(\R^-\times \Omega_\T)\to e^{\beta |c|}H^2(\R^-;E_k)$ (this amounts to the compact injection $H^2([-1,0];E_k)\to H^1(\R^-;E_k)$). 

The operator $e^{(\beta-\gamma/2)|c|}\Pi_k\tilde{\chi}\mathbf{R}_k(\alpha)\chi e^{-\beta|c|}$ is also compact as maps 
$L^2(\R \times \Omega_\T)\to L^2(\R; E_k)$ and $L^2\to \mc{D}(\mc{Q})$ by using the same type of argument as for proving 
the compact injection \eqref{injectioncpt}: indeed, one has the pointwise bound on its Schwartz kernel restricted to $\ker (\mathbf{P}-\la_j)$
\[ \begin{split}
|\kappa_j(c,c')|\leq & Ce^{(\beta-\gamma/2)|c|-\beta|c'|}(e^{{\rm Im}(\sqrt{p^2-2\la_j})(|c|-|c'|)}+e^{-{\rm Im}(\sqrt{p^2-2\la_j})(|c|+|c'|)})\indic_{|c'|\geq |c|}\\
& + Ce^{(\beta-\gamma/2)|c'|-\beta|c|}(e^{{\rm Im}(\sqrt{p^2-2\la_j})(|c'|-|c|)}+e^{-{\rm Im}(\sqrt{p^2-2\la_j})(|c|+|c'|)})\indic_{|c|\geq |c'|}.
\end{split}\]
We see that for ${\rm Im}(\sqrt{p^2-2\la_j})\geq 0$, if $0<\beta<\gamma/2$, this is bounded by 
$C\max(e^{(\beta-\gamma/2)|c|-\beta|c'|},e^{(\beta-\gamma/2)|c'|-\beta|c|})$, and is thus the integral kernel of a compact operator on $L^2(\R^-)$ since it is Hilbert-Schmidt (the kernel being in $L^2(\R\times\R)$). If now ${\rm Im}(\sqrt{p^2-2\la_j})<0$, the same argument shows that a sufficient condition to be compact is that 
\[ {\rm Im}\sqrt{p^2-2\la_j}>-\beta \textrm{ and } {\rm Im}\sqrt{p^2-2\la_j}>\beta-\gamma/2.\]  
Combining with Lemma \ref{chiPikV}, we deduce that if $0<\beta<\gamma/2$, then 
\[e^{\gamma c}V\Pi_k \tilde{\chi}{\bf R}_k(\alpha)\chi : e^{\beta \rho}L^2(\R\times\Omega_\T)\to e^{\beta \rho}\mc{D}'(\mc{Q}) \] 
\[\textrm{ and } \Pi_k e^{\gamma c}V\Pi_k \tilde{\chi}{\bf R}_k(\alpha)\chi : e^{\beta\rho }L^2(\R\times\Omega_\T)\to e^{\beta\rho}L^2(\R^-; E_k)\]
are compact operators. Thus ${\bf K}_{k,1}(\alpha):e^{\beta\rho }L^2(\R\times\Omega_\T)\to e^{\beta\rho}L^2(\R^-; E_k)$ is also compact. This proves 1).\\
  
Now, if $f\in e^{\beta \rho}L^2(\R;E_k)$ we have for $c\leq 0$ and writing $f=\sum_{j\leq k}f_j$
with $f_j\in \ker (\mathbf{P}-\la_j)$
\begin{equation}\label{asymptRk}
\begin{split}
(\mathbf{R}_k(\alpha)\chi f)(c)=& 2\sum_{j\leq k}
\frac{e^{-ic\sqrt{p^2-2\la_j}}}{\sqrt{p^2-2\la_j}}\int_{-\infty}^0 \sin{(c'\sqrt{p^2-2\la_j})}\chi(c')f_j(c')dc'\\
& + 2\sum_{j\leq k}
\frac{e^{-ic\sqrt{p^2-2\la_j}}}{\sqrt{p^2-2\la_j}}\int_{-\infty}^c\sin\Big((c-c')\sqrt{p^2-2\la_j}\Big)f_j(c')dc'
\end{split}\end{equation}
and the term in the second line, denoted $G(c)$, satisfies for $c<-1$
\[\begin{split}
 \|G(c)\|_{L^2(\Omega_\T)}\leq & 2\sum_{j\leq k} \int_{-\infty}^c 
e^{|{\rm Im}\sqrt{p^2-2\la_j}|(c-c')+\beta c'} \|e^{-\beta \rho(c')}f_j(c')\|_{L^2(\Omega_\T)}dc'\\
\leq & C_k\|e^{-\beta \rho}f\|_{2} e^{-\beta|c|}.
\end{split}
\]
and the same bounds hold for $\|\pl_cG(c)\|_{L^2(\Omega_\T)}$; this completes the proof of 2).

We remark that if ${\rm Im}(p)>0$, then the operator 
\[ (\mathbf{H}_{0}-\tfrac{Q^2+p^2}{2})^{-1}\Pi_k: e^{\beta |c|}L^2(\R^-\times \Omega_\T)\to e^{\beta |c|}L^2(\R^-;E_k)\]
is bounded and analytic in $p$, provided $0<\beta<\min_{j\leq k}{\rm Im}(\sqrt{p^2-2\la_j})={\rm Im}(p)$. Indeed, 
this follows again from Schur's lemma applied to the Schwartz kernel 
\[ \indic_{\{c\geq c'\}}e^{-\beta (|c|-|c'|)} \sin(c\sqrt{p^2-2\la_j})e^{-ic'\sqrt{p^2-2\la_j}} +\indic_{\{c'\geq c\}}e^{-\beta (|c'|-|c|)} \sin(c'\sqrt{p^2-2\la_j})e^{-ic\sqrt{p^2-2\la_j}}.\]
Using again that $\Pi_k: E_k\to L^2(\Omega_\T)$ is bounded, this implies that 
\[\mathbf{R}_k(\alpha): e^{\beta |c|}L^2(\R^-\times \Omega_\T)\to e^{\beta |c|}
(L^2(\R^-;E_k)\cap \mc{D}(\mc{Q}))\]
is an analytic bounded family in $p$ in the region $0<\beta<{\rm Im}(p)$. The same argument works with $0<-\beta<{\rm Im}(p)$ in case $\beta<0$.
The operator $\mathbf{K}_{k,1}(\alpha):e^{-\beta \rho}L^2(\R\times \Omega_\T)\to e^{-\beta \rho}L^2(\R \times \Omega_\T)$ is compact by the same argument as above since it is Hilbert-Schmidt and 
$\mathbf{K}_{k,2}(\alpha):e^{-\beta \rho}L^2(\R\times \Omega_\T)\to e^{-\beta \rho}\mc{D}'(\mc{Q})$ is compact.
\end{proof}

\noindent \textbf{4) Proof of Proposition \ref{extensionresolvent}}. 
For $\beta\in \R$, define the Hilbert space for fixed $k$
\[\mc{H}_{k,\beta}:=e^{\beta \rho}L^2(\R;E_k)\oplus L^2(\R;E_k^\perp)\]
with scalar product
\[ \cjg f,f'\cjd_{\mc{H}_{k,\beta}}:= \int_{\R}e^{-2\beta \rho(c)}\cjg \Pi_k f,\Pi_kf'\cjd_{L^2(\Omega_\T)}dc +\cjg(1-\Pi_k)f,(1-\Pi_k)f' \cjd_{L^2(\R\times\Omega_\T)}.\] 
We now fix $\beta$ and $\beta'$ as in the statement of Proposition \ref{extensionresolvent}. 
We now use the operators of Lemma \ref{LemmaRk}, Lemma \ref{regimec>-1} and Lemma \ref{modelres}: let $\chi;\tilde{\chi},\hat{\chi}$ be the cutoff functions of these Lemmas and let $\check{\chi}\in C^\infty(\R)$ equal to $1$ on ${\rm supp}(\tilde{\chi})$ and supported in $\R^-$. 
We define 
\[ \til{\mathbf{R}}(\alpha):=\tilde{\chi}\mathbf{R}_k^\perp(\pi(\alpha))\chi+(1-\hat{\chi})\mathbf{R}_+(1-\chi)+\tilde{\chi}\mathbf{R}_k(\alpha)\chi-\check{\chi}{\bf R}_k^\perp(\alpha){\bf K}_{k,2}(\alpha)\]
which in $\{\alpha=Q+ip \in \Sigma \,|\, \forall j\leq k, {\rm Im}\sqrt{p^2-2\la_j}>-\beta\}$
is bounded and holomorphic (in $\alpha$) as a map $\til{{\bf R}}(\alpha): \mc{H}_{k,\beta}\to \mc{H}_{k,-\beta}\cap e^{-\beta \rho}\mc{D}(\mc{Q})$. It moreover satisfies the identity
\[ \begin{split}
(\mathbf{H}-2\Delta_{\pi(\alpha)})\til{\mathbf{R}}(\alpha)=& 1+\mathbf{\mathbf{L}}_k^\perp(\pi(\alpha))+\mathbf{K}_k^\perp(\pi(\alpha))+\mathbf{K}_{+,k}(\pi(\alpha))+\mathbf{L}_{+,k}(\pi(\alpha))+\mathbf{K}_{k,1}(\alpha)\\
& -(\check{\mathbf{L}}_k^\perp(\pi(\alpha))+\check{\mathbf{K}}_k^\perp(\pi(\alpha))){\bf K}_{k,2}(\alpha)
\end{split}\]
where $\check{\mathbf{L}}_k^\perp(\pi(\alpha))$ and $\check{\mathbf{K}}_k^\perp(\pi(\alpha))$ are the operators of Lemma  \ref{LemmaRk} with $\tilde{\chi}$ (resp. $\chi$) is replaced by $\check{\chi}$ (resp. $\tilde{\chi}$). 
Let us define 
\begin{equation}\label{def:tildeKk}
\tilde{\mathbf{K}}_k(\alpha):=\mathbf{K}_{k,1}(\alpha)+\mathbf{K}_{+,k}(\pi(\alpha))+\mathbf{K}_k^\perp(\pi(\alpha))-(\check{\mathbf{L}}_k^\perp(\pi(\alpha))+\check{\mathbf{K}}_k^\perp(\pi(\alpha))){\bf K}_{k,2}(\alpha).
\end{equation} 
By Lemma \ref{LemmaRk}, Lemma \ref{regimec>-1} and Lemma \ref{modelres}, $\tilde{\mathbf{K}}_k(\alpha):\mc{H}_{k,\beta}\to \mc{H}_{k,\beta}$ is compact  and holomorphic in $\alpha$ (recall that $\mathbf{K}_{k,j}(\alpha)(1-\Pi_k)=0$). We claim that for each $\psi\in C^\infty(\R)\cap L^\infty(\R)$ satisfying $\psi'\in L^\infty$ and ${\rm supp}(\psi)\subset (-\infty,A)$ for some $A\in \R$, 
\begin{equation}\label{boundonD'(Q)prop3.10}
\tilde{\mathbf{K}}_k(\alpha)(1-\Pi_k):\mc{D}'(\mc{Q})\to \mc{H}_{k,\beta} \quad \textrm{ and }\til{{\bf R}}(\alpha)(1-\Pi_k)\psi:\mc{D}'(\mc{Q})\to e^{-\beta \rho}\mc{D}(\mc{Q})
\end{equation} 
are bounded. 
Indeed, Lemma \ref{LemmaRk} and Lemma \ref{regimec>-1} show that $\mathbf{K}_{+,k}(\pi(\alpha))+\mathbf{K}_k^\perp(\pi(\alpha))$
is bounded as a map $\mc{D}'(\mc{Q})\to \mc{H}_{k,\beta}$ and that $\tilde{\chi}\mathbf{R}_k^\perp(\pi(\alpha))\chi+(1-\hat{\chi})\mathbf{R}_+(1-\chi)$ is bounded as a map $\mc{D}'(\mc{Q})\to \mc{D}(\mc{Q})$, while Lemma \ref{chiPikV} shows that $(1-\Pi_k)\psi:\mc{D}'(\mc{Q})\to \mc{D}'(\mc{Q})$.

Now  if $|\pi(\alpha)-Q|^2\leq \la_k^{1/4}$ and  if $k$ is  large enough, the operator $\til{\mathbf{L}}_k(\alpha):= \mathbf{L}_k^\perp(\pi(\alpha))+\mathbf{L}_{+,k}(\alpha)$ is bounded as map 
\begin{align}\label{def:tildeLk}
&\til{\mathbf{L}}_k(\alpha):\mc{H}_{k,\beta}\to L^2(\R;E_k^\perp) 
& \til{\mathbf{L}}_k(\alpha):  \mc{D}'(\mc{Q})\to \mc{H}_{k,\beta}
\end{align} 
with holomorphic dependance in $\alpha$, and with bound  (recall \eqref{boundonK1} and \eqref{boundbis})
\[ \|\tilde{\mathbf{L}}_k(\alpha)^2\|_{\mc{H}_{k,\beta}\to L^2(\R;E_k^\perp)}<1/2.\]

In particular, $(1+\tilde{\mathbf{L}}_k(\alpha))(1-\tilde{\mathbf{L}}_k(\alpha))=1-\tilde{\mathbf{L}}_k(\alpha)^2$ is invertible on $\mc{H}_{k,\beta}$ with holomorphic inverse given by the Neumann series $\sum_{j=0}^\infty \tilde{\mathbf{L}}_k(\alpha)^{2j}$; 
we write $(1+\mathbf{T}_k(\alpha)):=(1-\tilde{\mathbf{L}}_k(\alpha))(1-\tilde{\mathbf{L}}_k(\alpha)^2)^{-1}$, with $\mathbf{T}_k(\alpha)$ mapping boundedly  $\mc{H}_{k,\beta}\to \mc{H}_{k,\beta}$. Moreover  we have 
\[ (\mathbf{H}-2\Delta_{\pi(\alpha)})\til{\mathbf{R}}(\alpha)(1+\mathbf{T}_k(\alpha))=1+\tilde{\mathbf{K}}_k(\alpha)(1+\mathbf{T}_k(\alpha))\]
 and the remainder $\hat{\mathbf{K}}(\alpha):=\tilde{\mathbf{K}}_k(\alpha)(1+\mathbf{T}_k(\alpha))$ is now compact on $\mc{H}_{k,\beta}$, 
 and $1+\hat{\mathbf{K}}(\alpha)$ is thus Fredholm of index $0$. 
 
 Let $p_0=iq$ for some $q\gg \beta$, the operator $\mathbf{H}-\tfrac{Q^2}{2}$ being 
 self-adjoint on its domain $\mc{D}(\mathbf{H})$ and non-negative, $\mathbf{H}-\tfrac{Q^2+p_0^2}{2}$ is invertible with inverse denoted $\mathbf{R}(\alpha_0)$ if $\alpha_0=Q+ip_0=Q-q$. 
 Now, let $(\psi_j)_{j\leq J}\subset \mc{H}_{k,\beta}$  be an orthonormal basis of 
$\ker (1+\hat{\mathbf{K}}(\alpha_0)^*)$, and 
$(\varphi_j)_{j\leq J}\subset \mc{H}_{k,\beta}$ an orthonormal basis  of $\ker (1+\hat{\mathbf{K}}(\alpha_0))$. 
For each $j$, there is $w_j\in \mc{D}(\mathbf{H})$ such that $(\mathbf{H}-\tfrac{Q^2+p_0^2}{2})w_j=\psi_j$. 
If $\theta\in C^\infty(\R)$ equal $1$ in $c\in (-\infty,-1)$ and is supported in $c\in \R^-$, we have in $\mc{D}'(\mc{Q})$
\begin{equation}\label{paris} 
(\mathbf{H}_0-\tfrac{Q^2+p_0^2}{2})\theta w_j=\theta \psi_j-\theta e^{\gamma c}Vw_j-\tfrac{1}{2}[\pl_c^2,\theta]w_j
\end{equation}
and this implies by projecting this relation on $E_k$ with $\Pi_k$  that, setting $\psi_{j,k}=\Pi_k\psi_j$ and $\psi_{j,k}^\perp=(1-\Pi_k)\psi_j$, similarly $w_{j,k}=\Pi_kw_j$ and $w_{j,k}^\perp=(1-\Pi_k)w_j$
\begin{equation}\label{projPik}
(\mathbf{H}_0-\tfrac{Q^2+p_0^2}{2})\theta w_{j,k}=\theta \psi_{j,k}-\theta \Pi_k(e^{\gamma c}Vw_j)-\tfrac{1}{2}[\pl_c^2,\theta]w_{j,k}.
\end{equation}
By Lemma \ref{chiPikV} with $\beta'=\beta=\gamma/2$, we see that $\Pi_k(e^{\gamma c} Vw_j)\in e^{\gamma \rho/2}L^2$. Since $[\pl_c^2,\theta]$ is a first order differential operator with compact support, we get $[\pl_c^2,\theta]w_{j,k}\in L^2$ as $\til{\theta} w_{j,k}\in \mc{D}(\mc{Q})$ if $\til{\theta}\in C_c^\infty(\R)$ by Lemma \ref{chiPikV}. The right hand side of \eqref{projPik} is then in $e^{\beta\rho}L^2$. 
This shows in particular that 
$\theta w_{j,k}\in H^2(\R^-;E_k)\cap H_0^1(\R^-;E_k)$ and since $E_k$ is finite dimensional, it is
direct to check that 
\[\mathbf{R}_{k}(Q+ip_0)(\mathbf{H}_0-\tfrac{Q^2+p_0^2}{2})\theta w_{j,k}=\theta w_{j,k}\] 
with $\mathbf{R}_{k}(Q+ip_0)$ the operator of Lemma \ref{modelres}.
We obtain in the region $c\in \R^-$
\[\theta w_{j,k}=\mathbf{R}_k(Q+ip_0)\big(\theta \psi_{j,k}-\theta e^{\gamma c}\Pi_k(Vw_j)-\tfrac{1}{2}[\pl_c^2,\theta]w_{j,k}\big).\]
Using the properties of $\mathbf{R}_k(\alpha)$ in Part 3) of Lemma \ref{modelres} (applied with $-\beta$ instead of $\beta$), 
we see that for each $0<\beta<\gamma/2$  
\[ \theta w_{j,k}\in e^{\beta \rho}L^2(\R^-;E_k)\] 
and thus we deduce that
\[  w_j\in e^{\beta \rho}L^2(\R;E_k)\oplus L^2(\R;E_k^\perp)=\mc{H}_{k,\beta}.\]
If we consider the finite rank operator $\mathbf{W}$ defined by $\mathbf{W}f:=\sum_{j=1}^Jw_j \cjg f,\varphi_j\cjd_{  \mc{H}_{k,\beta}}$ 
for $f\in \mc{H}_{k,\beta}$, 
we have 
\[ (\mathbf{H}-\tfrac{Q^2+p_0^2}{2})\mathbf{W}=\mathbf{Y} \quad \textrm{ with } \mathbf{Y}f:=\sum_{j=1}^J\psi_j\cjg f,\varphi_j\cjd_{ \mc{H}_{k,\beta}}.\]
But now it is direct to check that $1+\hat{\mathbf{K}}(\alpha_0)+\mathbf{Y}$ is invertible on 
$\mc{H}_{k,\beta}$ and we obtain that  
\[(\mathbf{H}-2\Delta_{\pi(\alpha)})\big(\til{\mathbf{R}}(\alpha)(1+\mathbf{T}_k(\alpha))+\mathbf{W}\big)=1+\hat{\mathbf{K}}(\alpha)+\mathbf{Y}-2(\Delta_{\pi(\alpha)}-\Delta_{\alpha_0}){\bf W}.\]
The remainder ${\bf K}(\alpha):=\hat{\mathbf{K}}(\alpha)+\mathbf{Y}-2(\Delta_{\pi(\alpha)}-\Delta_{\alpha_0}){\bf W}$ is compact on $\mc{H}_{k,\beta}$, analytic in $\alpha$ in the desired region and 
$1+\mathbf{K}(\alpha)$ is invertible for 
$\alpha=\alpha_0$, therefore we can apply the Fredholm analytic theorem to conclude that 
the family of operator $(1+\mathbf{K}(\alpha))^{-1}$ exists as a  meromorphic family of bounded operators on $\mc{H}_{k,\beta}$ for $\alpha$ in 
$\{\alpha=Q+ip \in \Sigma\,|\, |\pi(\alpha)-Q|^2\leq \la_k^{1/4}, \forall j\leq k, {\rm Im}\sqrt{p^2-\la_j}>-\beta\}$  except on a discrete set of poles with finite rank polar part. We can thus set 
\begin{equation}\label{formulefinaleRalpha} 
\mathbf{R}(\alpha):=\big(\til{\mathbf{R}}(\alpha)(1+\mathbf{T}_k(\alpha))+ \mathbf{W}\big)(1+\mathbf{K}(\alpha))^{-1}
\end{equation}
which satisfies the desired properties. To prove the boundedness of ${\bf R}(\alpha)(1-\Pi_k)\psi:\mc{D}'(\mc{Q}) \to e^{-\beta\rho}\mc{D}(\mc{Q})$ for $\psi\in L^\infty(\R)\cap C^\infty(\R)$ with $\psi'\in L^\infty$ and ${\rm supp}(\psi)\subset (-\infty,A)$ for $A\in \R$, we write 
\[ {\bf R}(\alpha)=\til{{\bf R}}(\alpha)-{\bf R}(\alpha)(\tilde{{\bf K}}_k(\alpha)+\tilde{{\bf L}}_k(\alpha))\]
and we have seen in \eqref{boundonD'(Q)prop3.10} and \eqref{def:tildeLk} that $(\tilde{{\bf K}}_k(\alpha)+\tilde{{\bf L}}_k(\alpha)): \mc{D}'(\mc{Q}) \to \mc{H}_{k,\beta}$ is bounded and $\til{{\bf R}}(\alpha)(1-\Pi_k)\psi:\mc{D}'(\mc{Q})\to e^{-\beta\rho}\mc{D}(\mc{Q})$ is bounded, while $(1-\Pi_k)\psi:\mc{D}'(\mc{Q})\to \mc{D}'(\mc{Q})$ by Lemma \ref{chiPikV}. Finally ${\bf R}(\alpha):\mc{H}_{k,\beta}\to e^{\beta \rho}\mc{D}(\mc{Q})$ and the same holds for ${\bf W}$. This shows the announced property of ${\bf R}(\alpha)(1-\Pi_k)\psi$.
We can use the mapping properties of $\til{\mathbf{R}}(\alpha)$ and $\mathbf{W}$, together with \eqref{asympR_-} and \eqref{formulefinaleRalpha} to deduce \eqref{asymptoticRf}.

We finally need to prove that there is no pole in the half plane ${\rm Re}(\alpha)\leq Q$ except possibly at the points $\alpha=Q\pm i\sqrt{2\lambda_j}$.
First, by the spectral theorem, one has for each $f\in e^{\beta \rho}L^2\subset L^2$ with 
$\beta>0$ and each $\alpha$ satisfying ${\rm Re}(\alpha)<Q$
\[  \|{\bf R}(\alpha)f\|_{e^{\-\beta \rho}L^2}\leq C\|{\bf R}(\alpha)f\|_{2}\leq \frac{C\|f\|_{2}}{|{\rm Im}(\alpha)|.|{\rm Re}(\alpha)-Q|}\]
which  implies that a pole $\alpha_0=Q+ip_0$ with $p_0\notin \{\pm\sqrt{2\lambda_j}\,| \,j\geq 0\}$ on ${\rm Re}(\alpha)=Q$ must be at most of order $1$, while at $p_0=\pm \sqrt{2\lambda_j}$ it can be at most of order $2$ on $\Sigma$. Since $\til{\mathbf{R}}(\alpha)(1+\mathbf{T}_k(\alpha))+ \mathbf{W}$ is analytic, a pole of $\til{\mathbf{R}}(\alpha)$ can only come from a pole of $(1+\mathbf{K}(\alpha))^{-1}$, with polar part being a finite rank operator. We now assume that $p_0\notin \{\pm\sqrt{2\lambda_j}\,| \,j\geq 0\}$.
Let us denote by ${\bf Z}$ the finite rank residue ${\bf Z}={\rm Res}_{\alpha_0}{\bf R}(\alpha)$. Then $(\mathbf{H}-\tfrac{Q^2+p_0^2}{2}){\bf Z}=0$, which means that each element in ${\rm Ran}(\bf Z)$ is a $w\in e^{-\beta\rho}\mc{D}({\bf H})$ such that $(\mathbf{H}-\tfrac{Q^2+p_0^2}{2})w=0$. There are finite rank operators ${\bf Z}_0,\dots,
{\bf Z}_{N}$ on $\mc{H}_{k,\beta}$ so that for $\psi\in C^\infty((-\infty,-2)_c;[0,1])$
\[\psi{\bf Z}=\sum_{n=0}^N \psi\pl_{\alpha}^{n}\tilde{\mathbf{R}}(\alpha_0){\bf Z}_n=\sum_{n=1}^N \psi \pl_{\alpha}^{n}\mathbf{R}_k(\alpha_0)\chi{\bf Z}_n+\psi\tilde{\mathbf{R}}(\alpha_0){\bf Z}_0+\psi{\bf Z}_{L^2}\]
where ${\bf Z}_{L^2}$ is a finite rank operator mapping to $\mc{H}_{k,\beta}\subset L^2$.
For $f\in \mc{H}_{k,\beta}$, the expression of $\pl_{\alpha}^{j}\mathbf{R}_k(\alpha_0)f$ is explicit 
from \eqref{asymptRk}, and one directly checks by differentiating \eqref{asymptRk} in $\alpha$ that it is of the form (for $c<-2$)
\[ (\pl_{\alpha}^{n}\tilde{\mathbf{R}}(\alpha_0)f)(c)=\sum_{\la_j\leq \la_k}\sum_{m\leq n}\til{a}_{j,m}(\alpha,f)c^{m}
e^{-ic\sqrt{p_0^2-2\la_j}}+\til{G}(\alpha_0,f)\] 
for some $\til{a}_{j,m}(\alpha_0,f)\in \ker ({\bf P}-\la_j)$ and $\tilde{\chi}\til{G}(\alpha_0,f)\in \mc{H}_{k,\beta}$ satisfying ${\bf H}\tilde{\chi}\til{G}(\alpha_0,f)\in \mc{H}_{k,\beta}$. This implies that, in $c<-2$,  $w\in {\rm Ran}({\bf Z})$ is necessarily of the form 
\[w=\sum_{j\leq k}\sum_{m\leq N}b_{j,m}c^{m}
e^{-ic\sqrt{p_0^2-2\la_j}}+\hat{G}\]
for some $b_{j,m}\in\ker ({\bf P}-\la_j)$ and $\hat{G}\in \mc{H}_{k,\beta}$ with ${\bf H}\hat{G}\in \mc{H}_{k,\beta}$. Using that $\tilde{\chi}(c)e^{\gamma c}\Pi_k(Vw)\in e^{\beta c}L^2$, we see from the equation $(\mathbf{H}_0-\tfrac{Q^2+p_0^2}{2})\Pi_k(w)=-e^{\gamma c}\Pi_k(Vw)$ that 
\[ (\mathbf{H}_0-\tfrac{Q^2+p_0^2}{2})\Big(\sum_{\la_j\leq \la_k}\sum_{m\leq N}b_{j,m}c^{m}
e^{-ic\sqrt{p_0^2-2\la_j}}\Big)\Big|_{c\leq 0}\in e^{\beta c}L^2\]
and by using the explicit expression of ${\bf H}_0$, it is clear that necessarily 
$b_{j,m}=0$ for all $m\not=0$. Then we may apply Lemma \ref{boundarypairing} with $u_1=u_2=w$ to deduce that $b_{j,0}=0$, and therefore $w\in \mc{D}(\bf{H})$, which implies $w=0$ by Lemma \ref{embedded}.

It remains to show that $\alpha=Q\pm i\sqrt{2\lambda_j}$ is a pole of order  at most $1$. To simplify, we write the argument for $\alpha=Q$, the proof is the same for all $j$. The method is basically the same as in the proof of \cite[Proposition 6.28]{Mel}: the resolvent has Laurent expension ${\bf R}(\alpha)=(\alpha-Q)^{-2}{\bf Q}+(\alpha-Q)^{-1}{\bf R}'(\alpha)$ for some holomorphic operator ${\bf R}'(\alpha)$ near $\alpha=Q$ and ${\bf Q}$ has finite rank, then we also have $\|{\bf R}(\alpha)\phi\|_{L^2}\leq |Q-\alpha|^{-2}$ for $\alpha<Q$ and all $\phi\in \mc{C}$, thus we can deduce that 
\[ {\bf Q}\phi=\lim_{\alpha\to Q^-}(\alpha-Q)^2{\bf R}(\alpha)\phi.\]
The limit holds in $e^{-\delta\rho}L^2$ for all $\delta>0$ small, but the right hand side has actually a bounded $L^2$-norm, so ${\bf Q}\phi\in L^2$ and thus ${\rm Ran}({\bf Q})\subset L^2$. Since we also have ${\bf H}{\bf Q}=0$ from Laurent expanding $({\bf H}-2\Delta_\alpha){\bf R}(\alpha)={\rm Id}$ at $\alpha=Q$, we conclude that ${\bf Q}=0$ by using Lemma \ref{embedded}.
\qed

\subsubsection{The resolvent in the physical sheet on weighted spaces} 

We shall conclude this section on the resolvent of ${\bf H}$ by analyzing its boundedness on weighted spaces $e^{-\beta \rho}L^2$  in the half-plane $\{{\rm Re}(\alpha)<Q\}$. We recall that 
Lemma \ref{resolventweighted} was precisely proving such boundedness but the region of validity in $\alpha$  of this Lemma was not covering the whole physical-sheet, and in particular not the region close to the line ${\rm Re}(\alpha)=Q$. Just as in Lemma \ref{firstPoisson}, 
the main application of such boundedness on weighted spaces is to define the Poisson operator 
$\mc{P}_\ell(\alpha)$, and we aim to define it in a large connected region of 
$\{{\rm Re}(\alpha)\leq Q\}$ relating the probabilistic region and the line $\alpha \in Q+i\R$ corresponding to the $L^2$-spectrum of ${\bf H}$.

\begin{proposition}\label{physicalsheet} 
Let $\beta\in \R$ and ${\rm Re}(\alpha)<Q$, then the resolvent $\mathbf{R}(\alpha)$ of $\mathbf{H}$ extends as an analytic family of bounded operators 
\begin{align*}
\mathbf{R}(\alpha): e^{-\beta \rho}L^2(\R\times \Omega_\T)\to e^{-\beta \rho}\mc{D}(\mc{Q})  
\end{align*}
in the region ${\rm Re}(\alpha)<Q-|\beta|$, and it satisfies for each $\psi\in C^\infty(\R)\Cap L^\infty(\R)$ such that $\psi'\in L^\infty$ and ${\rm supp}(\psi)\subset (-\infty,A)$ for some $A\in \R$,
\[\mathbf{R}(\alpha)(1-\Pi_k)\psi: e^{-\beta \rho}\mc{D}'(\mc{Q})\to e^{-\beta \rho}\mc{D}(\mc{Q}).\]
\end{proposition} 
\begin{proof} We proceed as in the proof of Proposition \ref{extensionresolvent}: 
we let for ${\rm Re}(\alpha)<Q$
\[ \tilde{{\bf R}}(\alpha):=\tilde{\chi}\mathbf{R}_k^\perp(\alpha)\chi+(1-\hat{\chi})\mathbf{R}_+(1-\chi)+\tilde{\chi}\mathbf{R}_k(\alpha)\chi-\check{\chi}{\bf R}_k^\perp(\alpha){\bf K}_{k,2}(\alpha)\]
and we get 
\[ (\mathbf{H}-\tfrac{Q^2+p^2}{2})\tilde{{\bf R}}(\alpha)={\rm Id}+\tilde{{\bf K}}_k(\la)+\tilde{{\bf L}}_k(\alpha)\]
where we used the operators of the proof of Proposition \ref{extensionresolvent} (see \eqref{def:tildeKk} and \eqref{def:tildeLk}). Now we take ${\rm Re}(\alpha)<Q$ and $|Q-\alpha|<A_0$ for some fixed constant $A_0>0$ that can be chosen arbitrarily large, and we let $k>0$ large enough so that 
\begin{equation}\label{dingue}
A_0^2 +1<\min (\frac{\la_k^{1/2}}{16(1+C_2)},\la_k^{1/4}), \quad \la_k> (16^2C_{1}^2(1+C_{2}^2)+1)(|\beta|+1)^2,
\end{equation}
where the constant $C_{1}$, $C_{2}$ above are the constants respectively given in Lemma \ref{LemmaRk} and Lemma \ref{regimec>-1}. 
The conditions in \eqref{dingue} ensures both the condition ${\rm Re}((\alpha-Q)^2)>\beta^2-2\lambda_k+1$ of Lemma  \ref{LemmaRk} is satisfied and  the operator 
$\tilde{\chi}\mathbf{R}_k^\perp(\alpha)\chi:e^{-\beta \rho} \mc{D}'(\mc{Q})\to e^{-\beta \rho}(L^2(\R^-;E_k^\perp)\cap \mc{D}(\mc{Q}))$ is a bounded holomorphic family,  and the norm estimate appearing in 2) of Lemma  \ref{LemmaRk} gives
 \begin{equation} \label{borne1}
 \| \mathbf{L}_k^\perp(\alpha)\|_{\mc{L}(e^{-\beta \rho}L^2)}\leq \frac{C_{1}(1+|\beta|)}{\sqrt{{\rm Re}((\alpha-Q)^2)+2\la_k-\beta^2}}\leq \frac{1}{16(1+C_{2})}.
  \end{equation}
The condition  $|\beta|<Q-{\rm Re}(\alpha)$ (equivalent to ${\rm Im}(p)>|\beta|$) makes sure that we can apply 3)  of Lemma \ref{modelres}: in particular  the operator $\mathbf{R}_k(\alpha): e^{-\beta \rho}L^2(\R^-\times \Omega_\T)\to e^{-\beta \rho}(L^2(\R^-;E_k)\cap \mc{D}(\mc{Q}))$ is a bounded holomorphic family.
Also, Lemma \ref{regimec>-1} ensures that $(1-\hat{\chi})\mathbf{R}_+(1-\chi): e^{-\beta \rho}\mc{D}'(\mc{Q})\to e^{-\beta \rho}\mc{D}(\mc{Q})$ is a bounded holomorphic family (note that both cutoff functions $(1-\hat{\chi})$ and $(1-\chi)$ kill the $c\to -\infty$ behaviour and this is why Lemma \ref{regimec>-1} extends to weighted spaces $e^{-\beta \rho}L^2(\R\times \Omega_\T)$). Also the first condition in \eqref{dingue} ensures the norm estimate (as given by \eqref{boundbis})
 \begin{equation} \label{borne2}
 \|\mathbf{L}_{+,k}(\alpha)\|_{\mc{L}(L^2)}\leq  2 C_2,\quad  \|\mathbf{L}_{+,k}(\alpha)^2\|_{\mc{L}(L^2)}\leq  \frac{1}{8}.
 \end{equation}
As a consequence
\[ \tilde{\mathbf{R}}(\alpha): e^{-\beta \rho}L^2\to e^{-\beta \rho}\mc{D}(\mc{Q})\]
is bounded and holomorphic in $U:=\{\alpha\in \C\,|\,  |Q-\alpha|<A_0, {\rm Re}(\alpha)<Q-|\beta|\}$.
Furthermore \eqref{borne1} and \eqref{borne2} provide the estimate
\[ \|(\mathbf{L}_{+,k}(\alpha)+\mathbf{L}_k^\perp(\alpha))^2\|_{\mc{L}(e^{-\beta \rho}L^2)}<1/2.\]
Moreover, 2) of Lemma \ref{LemmaRk}, Lemma \ref{regimec>-1} and 3) of Lemma \ref{modelres} also give that 
$\tilde{\mathbf{K}}_k(\alpha)$ is compact on the Hilbert space 
$e^{-\beta \rho}L^2(\R\times \Omega_\T)$. Exactly the  same  argument as in the proof of Proposition \ref{extensionresolvent} gives that 
\[ (\mathbf{H}-\frac{Q^2+p^2}{2})\tilde{{\bf R}}(\alpha)(1+\mathbf{T}_k(\alpha))=1+ \tilde{{\bf K}}_k(\alpha)(1+\mathbf{T}_k(\alpha))\]
for some $\mathbf{T}_k$ bounded holomorphic on $e^{-\beta \rho}L^2$ in $U$. Since  by Lemma \ref{resolventweighted} we know that $(\mathbf{H}-2\Delta_\alpha)$ is invertible on $e^{-\beta \rho}L^2$
for some $\alpha_0\in U$, one can always add a finite rank operator $\mathbf{W}: e^{-\beta \rho}L^2\to e^{-\beta \rho}\mc{D}(\mathbf{H})$, so that 
\[(\mathbf{H}-2\Delta_\alpha)(\mathbf{\til{R}}(\alpha)(1+\mathbf{T}_k(\alpha))+\mathbf{W})=1+\mathbf{K}(\alpha)\]
for some compact remainder $\mathbf{K}(\alpha)$ on $e^{-\beta \rho}L^2$, analytic in $\alpha\in U$ in the desired region and $1+\mathbf{K}(\alpha)$ being invertible for $\alpha=\alpha_0\in U$. This implies by analytic Fredholm theorem that 
\[ \mathbf{R}(\alpha)=(\tilde{\mathbf{R}}(\alpha)(1+\mathbf{T}_k(\alpha))+\mathbf{W})(1+\mathbf{K}(\alpha))^{-1}: e^{-\beta \rho}L^2(\R\times \Omega_\T)\to e^{-\beta \rho}\mc{D}(\mc{Q})\]
is meromorphic for $\alpha\in U$. Now, using the density of the embeddings 
$e^{|\beta| \rho}L^2\subset L^2\subset e^{-|\beta| \rho}L^2$ and using that $\mathbf{R}(\alpha)$ is holomorphic in $U$ as a bounded operator on $L^2$, it is direct to check that 
$\mathbf{R}(\alpha):e^{-\beta \rho}L^2\to e^{-\beta \rho}\mc{D}(\mc{Q})$ is analytic in $U$. Since $A_0$ (and thus $U$) can be chosen arbitrarily large as long as the constraint ${\rm Re}(\alpha)<Q-|\beta|$ is satisfied, we obtain our desired result. To prove that ${\bf R}(\alpha)(1-\Pi_k)\psi$ maps 
$e^{-\beta \rho}\mc{D}'(\mc{Q})$ to $e^{-\beta \rho}\mc{D}(\mc{Q})$, we proceed as in the proof of Proposition \ref{extensionresolvent}  and write
\[ {\bf R}(\alpha)=\til{{\bf R}}(\alpha)-{\bf R}(\alpha)(\tilde{{\bf K}}_k(\alpha)+\tilde{{\bf L}}_k(\alpha)).\]
We have seen that $\til{{\bf R}}(\alpha)(1-\Pi_k)\psi:e^{-\beta \rho}\mc{D}'(\mc{Q})\to e^{-\beta \rho}\mc{D}(\mc{Q})$. The same arguments (just as in the proof of Proposition \ref{extensionresolvent}) also prove that that operators 
$\tilde{{\bf K}}_k(\alpha),\tilde{{\bf L}}_k(\alpha)$ are bounded as operators 
$e^{-\beta \rho}\mc{D}'(\mc{Q})\to e^{-\beta \rho}L^2$, thus we obtain that ${\bf R}(\alpha)(1-\Pi_k)\psi:e^{-\beta \rho}\mc{D}'(\mc{Q})\to e^{-\beta \rho}\mc{D}(\mc{Q})$ is bounded.
\end{proof}

\subsection{The Poisson operator} 

We have seen in Lemma \ref{firstPoisson} that it is possible to construct a family of Poisson operators $\mc{P}_\ell(\alpha)$ in what we called the \emph{probabilistic region}, which contains a half line $(-\infty,Q-c_\ell)$ for some $c_\ell\geq 0$ depending on $\ell$. The construction was using the resolvent acting on weighted $L^2$-spaces.  In this section, we will use Proposition \ref{extensionresolvent} and Proposition \ref{physicalsheet} to prove that the Poisson operators 
extend holomorphically in $\alpha$ in a connected region of ${\rm Re}(\alpha)\leq Q$ containing the probabilistic region and the line $Q+i\R$.

\subsubsection{Regime close to the continuous spectrum of $H$}

We first start with a technical lemma that allows to define the Poisson operator on the continuous spectrum $Q+i\R$:
\begin{lemma}\label{boundarypairing}
Let $p\in \R$ and for $m=1,2$,   let $u_m\in e^{-\delta \rho}L^2(\R\times\Omega_\T)$ with $\delta>0$ such that:\\
1) for each $\theta\in C^\infty(\R;[0,1])$ supported in $(a,+\infty)$ for some $a\in\R$ then $\theta u_m\in \mc{D}(\mc{Q})$\\
2)  $u_m$ satisfies 
\[ (\mathbf{H}-\tfrac{Q^2+p^2}{2})u_m=r_m \in e^{\delta \rho}L^2(\R\times \Omega_\T).\]
Set $k=\max\{j\geq0\, |\, 2\la_j\leq p^2\}$. Then $u_m$ has asymptotic behaviour
\begin{equation}\label{decompui}
u_m=\sum_{j,2 \la_j\leq p^2}\Big(a_m^j 
e^{-ic\sqrt{p^2-2\la_j}}+b_m^j 
e^{ic\sqrt{p^2-2\la_j}}\Big)+G_m
\end{equation}
with $a_m^j,b_m^j\in \ker (\mathbf{P}-\la_j)$, and both $G_m,\pl_cG_m\in e^{\delta \rho}L^2(\R\times \Omega_\T)+L^2(\Omega_\T;E_k^\perp)$.
Then we have 
\[\cjg u_1\,|\,r_2\cjd-\cjg r_1\,|\,u_2\cjd=i \sum_{j, 2\la_j\leq p^2}\sqrt{p^2-2\la_j}\Big(\cjg a_1^j\,|\,a_2^j\cjd_{L^2(\Omega_\T)}-
\cjg b_1^j\,|\,b_2^j\cjd_{L^2(\Omega_\T)}\Big).\]
\end{lemma}
\begin{proof} Let $\theta\in C^\infty(\R)$ be non-negative 
satisfying $\theta_{T}=1$ on $[-T,\infty)$ and ${\rm supp}(\theta_{T})\subset [-T-\eps,\infty)$
where $T>0$ is a large parameter and $\eps>0$ small, and let $\til{\theta}_{T}(\cdot)=\theta_T(\cdot\,+1)$. In particular we have $\til{\theta}_{T}\theta_{T}=\theta_{T}$. First, $\til{\theta}_{T} u_m\in H^1(\R; L^2(\Omega_\T))$  satisfies 
\[(\mathbf{H}- \tfrac{Q^2+p^2}{2})(\til{\theta}_{T} u_m)=\til{\theta}_{T}r_m-\tfrac{1}{2}[\pl_c^2,\til{\theta}_{T}]u_m \in L^2(\R\times \Omega_\T)\] 
thus $\til{\theta}_Tu_m\in \mc{D}(\mathbf{H})$ (we used that $[\pl_c^2,\til{\theta}_{T}]$ is a first order differential operator with compactly supported coefficients). This implies, using $[\mathbf{H},\til{\theta}_T]\theta_T=0=\theta_T[\mathbf{H},\til{\theta}_T]$, that 
\begin{equation}\label{Greensformula}
\begin{split}
\cjg u_1,r_2\cjd-\cjg r_1,u_2\cjd = & \lim_{T\to  \infty}
\cjg \theta_{T} u_1, \mathbf{H}(\til{\theta}_{T}u_2)\cjd-\cjg \mathbf{H}(\til{\theta}_{T}u_1),\theta_{T} u_2\cjd \\ =
 & -\lim_{T\to  \infty}\tfrac{1}{2}
\cjg [\pl_c^2,\theta_{T}] u_1, u_2\cjd. 
\end{split}\end{equation}
We write $u_m=u_m^0+G_m$ by using \eqref{decompui}. Then we claim that, as $T\to \infty$,
\[ |\cjg [\pl_c^2,\theta_{T}] u_1^0, G_2\cjd|+|\cjg [\pl_c^2,\theta_{T}] G_1, u_2\cjd|\to 0.\]
Indeed, we have $(|u_1^0|+|\pl_cu_1^0|)G_2\in L^1(\R\times \Omega_\T)$ and $(|G_1|+|\pl_cG_1|)u_2\in L^1(\R\times \Omega_\T)$, and the support of $[\pl_c^2,\theta_{T}]$ is contained in $[-T-1,-T]$. We are left in \eqref{Greensformula} to study the limit of $\cjg [\pl_c^2,\theta_{T}] u_1^0, u_2^0\cjd$.
But now we have $[\pl_c^2,\theta_{T}] u^0_1=\theta''_{T}u^0_1+2\theta'_{T}\pl_cu^0_1$ and for fixed $T>0$ 
it is direct to check, using integration by parts and the fact that $(\mathbf{H}_0-\tfrac{Q^2+p^2}{2})u_m^0=0$ that
\[\begin{split} 
\cjg [\pl_c^2,\theta_{T}] u_1^0, u_2^0\cjd=& \int_{-T-1}^{-T}\pl_c\big(\theta_T\cjg \pl_cu^0_1,u^0_2\cjd_{L^2(\Omega_\T)}-\theta_T\cjg u^0_1,\pl_cu^0_2\cjd_{L^2(\Omega_\T)}\big)dc\\
=& \cjg \pl_cu_1^0(-T),u_2^0(-T)\cjd_{L^2(\Omega_\T)}-\cjg u_1^0(-T),\pl_cu_2^0(-T)\cjd_{L^2(\Omega_\T)}.
\end{split}\]
A direct computation gives that this is equal to 
\[ 2i \sum_{j, \la_j\leq p^2}\sqrt{p^2-2\la_j}\Big(
\cjg b_1^j,b_2^j\cjd_{L^2(\Omega_\T)}-\cjg a_1^j,a_2^j\cjd_{L^2(\Omega_\T)}\Big).\]  
This completes the proof. 
\end{proof}

Now we extend the construction of the Poisson operator \eqref{definPellproba} in a neighborhood of the line spectrum $\alpha\in Q+i\R$.
\begin{proposition}\label{poissonprop}
Let $0<\beta<\gamma/2$ and $\ell\in\N$.
Then there is an analytic family of operators $\mc{P}_\ell(\alpha)$  
\[\mc{P}_\ell(\alpha): E_\ell\to e^{-\beta\rho}\mc{D}(\mc{Q})\]
in the region 
\begin{equation}\label{regionvalide}
\Big\{\alpha\in \C\,\Big|\, {\rm Re}(\alpha)< Q, {\rm Im}\sqrt{p^2-2\la_\ell}<\beta\Big\}\cup \Big\{Q+ip\in Q+i\R\, \Big| \, |p|\in \bigcup_{j\geq \ell} (\sqrt{2\la_j},\sqrt{2\la_{j+1}})\Big\},
\end{equation}
continous at each $Q\pm i\sqrt{2\la_j}$ for $j\geq \ell$, satisfying $(\mathbf{H}-\tfrac{Q^2+p^2}{2})\mc{P}_\ell(\alpha)F=0$ and 
\begin{equation}\label{expansionP} 
\mc{P}_\ell(\alpha)F=\sum_{j\leq \ell}\Big(F^-_je^{ic\sqrt{p^2-2\la_j}}+F_j^+(\alpha)
e^{-ic\sqrt{p^2-2\la_j}}\Big)+G_\ell(\alpha,F)
\end{equation}
with $F^-_j=\Pi_{\ker(\mathbf{P}-\la_j)}F$, $F_j^+(\alpha)\in \ker (\mathbf{P}-\la_j)$, and 
$G_\ell(\alpha,F),\partial_c G_\ell(\alpha,F)\in e^{\frac{\beta}{2}\rho(c)}L^2(\R\times \Omega_\T)+L^2(\R;  E_\ell^{\perp})$. In particular, $\mc{P}_\ell(\alpha)\in e^{-({\rm Im}(\sqrt{p^2-\la_\ell})+\eps)\rho}\mc{D}(\mc{Q})$ for all $\eps>0$.
Moreover,  
for each $\theta\in C_c^\infty(\R)$, one has $\theta \mc{P}_\ell(\alpha)F\in \mc{D}(\mathbf{H})$. Such a solution $u\in e^{-\beta\rho}L^2(\R\times\Omega_\T)$ to the equation $(\mathbf{H}-\tfrac{Q^2+p^2}{2})u=0$ with the asymptotic expansion \eqref{expansionP} is unique.
The operator $\mc{P}_\ell(\alpha)$ admits a meromorphic extension to the region  
\begin{equation}\label{regiondextension}
\Big\{\alpha=Q+ip \in \Sigma\, |\, \forall j \in\N \cup\{0\}, {\rm Im}\sqrt{p^2-2\la_j}\in (\beta/2-\gamma,\beta/2)\Big\}\end{equation}
and $\mc{P}_\ell(\alpha)F$ satisfies \eqref{expansionP} in that region. Finally, $F_j^+(\alpha)$ depends meromorphically on $\alpha$ in the region above.  
\end{proposition}
\begin{proof}
We start by setting $u_-(\alpha):=\sum_{j=0}^\ell F^-_je^{ic\sqrt{p^2-2\la_j}}$, and let $\chi\in C^\infty(\R)$ equal to $1$ in $(-\infty,-1)$ and with ${\rm supp}(\chi)\subset \R^-$.
We get 
\[(\mathbf{H}-\tfrac{Q^2+p^2}{2})(\chi u_-(\alpha))= -\frac{1}{2}\chi''(c)u_-(\alpha)- \chi'(c)\pl_cu_-(\alpha)+
e^{\gamma c}V\chi u_-(\alpha).\]
The first two terms are in $e^{N\rho}L^2(\R\times \Omega_\T)$ for all $N$, the term $e^{\gamma c}V\chi u_-(\alpha)$ can be decomposed (for each $k\geq \ell$) as 
\[ \Pi_k(e^{\gamma c}V\chi u_-(\alpha))+(1-\Pi_k)(e^{\gamma c}V\chi u_-(\alpha)).\]
Using that $u_-(\alpha)=\Pi_k u_-(\alpha)$ if $k\geq \ell$ together with Lemma \ref{chiPikV}, we see that the first term satisfies $\Pi_k(e^{\gamma c}V\chi u_-(\alpha))\in e^{\beta \rho}L^2(\R^-;E_k)$
 and the second term $(1-\Pi_k)(e^{\gamma c}V\chi u_-(\alpha))\in e^{\beta \rho}\mc{D}'(\mc{Q})$, provided that $\beta<\gamma/2$ and that ${\rm Im}\sqrt{p^2-\la_j})\leq \gamma/2$ for $j\leq \ell$.
We can thus define, using Proposition \ref{extensionresolvent} (with $k\gg \ell$ large enough), 
if ${\rm Im}(\sqrt{p^2-2\la_j})\in (-\min(\beta,\frac{\gamma}{2}-\beta),\gamma/2)$  for all $j\leq k$ and $|\pi(p)|^2\leq \la_k^{1/4}$
\[u_+(\alpha):=\mathbf{R}(\alpha)(\mathbf{H}-\tfrac{Q^2+p^2}{2})(\chi u_-(\alpha))\in e^{-\beta\rho}\mc{D}(\mc{Q})\]
so that $u(\alpha):=\chi u_-(\alpha)-u_+(\alpha)$ solves $(\mathbf{H}-\tfrac{Q^2+p^2}{2})u(\alpha)=0$ in $e^{-\frac{\gamma}{2}\rho}\mc{D}'(\mc{Q})$.
We use Proposition \ref{extensionresolvent} with $k\gg \ell$ large enough, and we see 
that $u_+(\alpha)$ is of the form, in $c\leq 0$,
\[
u_+(\alpha)= \sum_{j\leq k}\til{a}_j(\alpha,F)e^{-ic\sqrt{p^2-2\la_j}}+G(\alpha,F)= \sum_{j\leq \ell}\til{a}_j(\alpha,F)e^{-ic\sqrt{p^2-2\la_j}}+G_\ell(\alpha,F)
\]
with $\til{a}_j(\alpha,F)\in \ker (\mathbf{P}-\la_j)$ and $G(\alpha,F), \pl_cG(\alpha,F)\in e^{\beta\rho}L^2(\R^-\times \Omega_\T)+ L^2(\R^-;E_k^\perp)$ and 
$G_\ell (\alpha,F)\in e^{\beta\rho}L^2(\R^-\times \Omega_\T)+L^2(\R^-;E_\ell^\perp)$ if ${\rm Re}(\alpha)<Q$.  Here we have used the fact that ${\rm Im}\sqrt{p^2-2\la_j}>0$ 
(since either ${\rm Re}(\alpha)<Q$ or $p^2<2\la_j$ for all $j>\ell$ if $\alpha=Q+ip$ with $p\in\R\setminus [-\sqrt{2\lambda_\ell},\sqrt{2\lambda_\ell}]$) to place all terms corresponding to all $\ell<j\leq k$, which belong to $L^2(\R;E_\ell^\perp)$, in the remainder term $G_\ell(\alpha,F)$.
This shows that 
\begin{equation}\label{defPalpha}
\mc{P}_\ell(\alpha)F:=u(\alpha)=\chi(c)\sum_{j\leq \ell}F^-_je^{ic\sqrt{p^2-2\la_j}}-
\mathbf{R}(\alpha)(\mathbf{H}-2\Delta_\alpha )\Big(\chi(c)\sum_{j\leq \ell}F^-_je^{ic\sqrt{p^2-2\la_j}}\Big)
\end{equation}
satisfies all the required properties. The analyticity in $\alpha$ except possibly at the points  $Q\pm i\sqrt{2\la_j}$ for $j\in\N_0$ follows from Proposition \ref{extensionresolvent}, in particular 3) of that Proposition. At the points $Q\pm i\sqrt{2\la_j}$, the analyticity on the surface $\Sigma$ is a consequence of the Lemma \ref{Poissonaj}, in particular \eqref{expression P vs resolvante} and the fact that $Q\pm i\sqrt{2\la_j}$ is at most a pole of ordre $1$ of ${\bf R}(\alpha)$ and  $a_j(\alpha,\varphi)$. We notice that the expression of $\mc{P}_\ell(\alpha)$ is the same as in \eqref{definPellproba}, thus when the regions of $\alpha$ considered in Lemma 
\ref{firstPoisson} and here have an intersection, then this corresponds to the same operator, by analytic continuation.
 
The uniqueness of the solution with such an asymptotic is direct if ${\rm Re}(\alpha)<Q$: the difference of two such solutions  would be in $\mc{D}(\mc{Q})$ and the operator $\mathbf{H}$ has no 
$L^2$ eigenvalues (Lemma \ref{embedded}), hence the difference is identically $0$. For the case $\alpha=Q+ip$ with $p\in\R$, denote by $\hat u(\alpha)$ the difference of two such solutions. Then $\hat u(\alpha) $ can be written under the form
\[\hat u(\alpha)=\sum_{j\leq \ell} \hat F_j^+(\alpha)
e^{-ic\sqrt{p^2-2\la_j}} +\hat G_\ell(\alpha,F)\]
where  $\hat F_j^+(\alpha)\in \ker (\mathbf{P}-\la_j)$ and 
$\hat G_\ell(\alpha,F)\in e^{\beta\rho(c)}L^2(\R\times \Omega_\T)+L^2(\R;  E_\ell^{\perp})$.  We can split the sum above as $\sum_{j, 2\lambda_j\leq p^2}\cdots+\sum_{j, p^2< j \leq 2\lambda_\ell}\cdots $.  The sum $\sum_{j, p^2< j \leq 2\lambda_\ell}\cdots $ belongs to some $e^{\delta \rho}L^2$ as well as its $\partial_c$ derivative. We can use 
Lemma \ref{boundarypairing} to see that $\sum_{j, 2\lambda_j\leq p^2}\|\hat F_j^+(\alpha)\|_{L^2(\Omega_\T)}^2=0$, hence again $\hat u(\alpha)\in L^2$ and we can conclude as previously.

The meromorphic extension of $\mc{P}_\ell(\alpha)$ is a direct consequence of the meromorphic extension of ${\bf R}(\alpha)$ in Proposition \ref{extensionresolvent}.
\end{proof}

We notice that for $\alpha=Q+ip$ with $p\in\R$, the function $\bbar{\mc{P}(\bbar{\alpha})F}$ is another solution of $(\mathbf{H}-\tfrac{Q^2+p^2}{2})u=0$ satisfying 
\[\bbar{\mc{P}_\ell(\bbar{\alpha})F}=\sum_{j\leq \ell}\Big(\bbar{F^-_j}e^{-ic\sqrt{p^2-\la_j}}+\bbar{F_j^+(\bbar{\alpha})}
e^{ic\sqrt{p^2-\la_j}}\Big)+\bbar{G_\ell(\bbar{\alpha},F)}.\]
This implies that for each $F=\sum_{j\leq \ell}F_j^-\in E_\ell$, there is a unique solution $u=\widehat{\mc{P}}_\ell(\alpha)F$ to $(\mathbf{H}-\tfrac{Q^2+p^2}{2})u=0$ of the form 
\begin{equation}\label{regimebar}
\widehat{\mc{P}}_\ell(\alpha)F=\sum_{j\leq \ell}\Big(F^-_je^{-ic\sqrt{p^2-\la_j}}+\widehat{F}_j^+(\alpha)
e^{ic\sqrt{p^2-\la_j}}\Big)+\widehat{G}_\ell(\alpha,F)
\end{equation}
with $\widehat{G}_\ell(\alpha,F)\in e^{\frac{\beta}{2}\rho(c)}L^2(\R\times \Omega_\T)+L^2(\R;  E_\ell^{\perp})$ and $\widehat{F}_j^+\in \C\rho_j$, and $\hat{\mc{P}}_\ell(\alpha)$ extends meromorphically on an open set of $\Sigma$ just like $\mc{P}_\ell(\alpha)$. 

\begin{lemma}\label{Poissonaj}
Let $\ell\in \N$, $0<\beta<\gamma$ and $\alpha$ in \eqref{regionvalide}, 
the Poisson operator $\mc{P}_{\ell}(\alpha)$ can be obtained from the resolvent as follows: 
for $F=\sum_{j\leq \ell}F_j^- \in E_\ell$ and $\varphi\in e^{\beta\rho}L^2$ 
\begin{equation}\label{expression P vs resolvante} 
\cjg \mc{P}_\ell(\alpha)F,\varphi\cjd_2=  i\sum_{j\leq \ell} \sqrt{p^2-2\la_j}\Big\cjg F_j^-,a_j(\bbar{\alpha},\varphi)\Big\cjd_{L^2(\Omega_\T)}
\end{equation}
where $a_j(\alpha,\varphi)$ are the functionals obtained from \eqref{asymptoticRf}, holomorphic in $\alpha$ and linear in $\varphi$.
\end{lemma}
\begin{proof} 
Let $\alpha=Q+ip$ with $p\in \R\setminus ([-\sqrt{2\la_\ell},\sqrt{2\la_\ell}]\cup_{j\geq \ell}\{\pm\sqrt{2\la_j}\})$ and  
let us take $F=\sum_{j\leq \ell} F_j^-$ 
with $ F_j^-\in {\rm Ker}(\mathbf{P}-\la_j)$ for $j\leq \ell$. Then from the construction of 
$\mc{P}_\ell(\alpha) F$ (with a function $\chi=\chi(c)$) in the proof of Proposition \ref{poissonprop}   
\[ \begin{split}
\cjg \mc{P}_\ell(\alpha)F,\varphi\cjd & = \Big\cjg \sum_{j\leq \ell}F^-_je^{ic\sqrt{p^2-2\la_j}}\chi, \varphi\Big\cjd- \Big\cjg \mathbf{R}(\alpha)(\mathbf{H}-2\Delta_\alpha)\sum_{j\leq \ell}F^-_je^{ic\sqrt{p^2-2\la_j}}\chi,\varphi\Big\cjd\\
& =\Big\cjg \sum_{j\leq \ell}F^-_je^{ic\sqrt{p^2-2\la_j}}\chi, \varphi\Big\cjd- \Big\cjg ( \mathbf{H}-2\Delta_\alpha)\sum_{j\leq \ell}F^-_je^{ic\sqrt{p^2-2\la_j}}\chi, \mathbf{R}(\bbar{\alpha})\varphi\Big\cjd.
\end{split}
\]
Here we used $ \mathbf{R}(\alpha)^*= \mathbf{R}(\bbar{\alpha})={\bf R}(2Q-\alpha)$. Let $\theta_T$ be as in the proof of Lemma \ref{boundarypairing}. 
We have  
  \begin{align*}
\lim_{T\to \infty}&\Big\cjg \theta_T(c)(\mathbf{H}-2\Delta_\alpha)\sum_{j\leq \ell}
F^-_je^{ic\sqrt{p^2-2\la_j}}\chi, \mathbf{R}(\bbar{\alpha})\varphi\Big\cjd\\
=& \Big\cjg \sum_{j\leq \ell}F^-_je^{ic\sqrt{p^2-2\la_j}}\chi,\varphi\Big\cjd - \frac{1}{2}\lim_{T\to \infty}\Big\cjg \sum_{j\leq \ell}F^-_je^{ic\sqrt{p^2-2\la_j}}\chi,
\theta''_T \mathbf{R}(\bbar{\alpha})\varphi\Big\cjd\\
&- \lim_{T\to \infty}\Big\cjg \sum_{j\leq \ell}F^-_je^{ic\sqrt{p^2-2\la_j}}\chi,
\theta'_T\pl_c \mathbf{R}(\bbar{\alpha})\varphi\Big\cjd.
\end{align*} 

Using now the asymptotic form \eqref{asymptoticRf},  the last two limits above
can be rewritten as 
\[
\lim_{T\to \infty}\Big\cjg \sum_{j\leq \ell}F^-_je^{ic\sqrt{p^2-2\la_j}}\chi,
\theta''_T\mathbf{R}(\bbar{\alpha})\varphi\Big\cjd=\lim_{T\to \infty} \sum_{j\leq \ell}\Big\cjg 
F^-_je^{ic\sqrt{p^2-2\la_j}},
\theta''_Ta_j(\bbar{\alpha},\varphi)e^{ic\sqrt{p^2-2\la_j}}\Big\cjd\]
\[
\lim_{T\to \infty}\Big\cjg \sum_{j\leq \ell}F^-_je^{ic\sqrt{p^2-2\la_j}}\chi,
\theta'_T\pl_c\mathbf{R}(\bbar{\alpha})\varphi\Big\cjd=\lim_{T\to \infty}\sum_{j\leq \ell}\Big\cjg F^-_je^{ic\sqrt{p^2-2\la_j}},
\theta'_Ta_j(\bbar{\alpha},\varphi)\pl_c(e^{ic\sqrt{p^2-2\la_j}})\Big\cjd
\]
and this easily yields 
\[ \cjg \mc{P}_\ell(\alpha)F,\varphi\cjd =i\sum_{j\leq \ell} \sqrt{p^2-\la_j}\Big\cjg F_j^-,a_j(\bbar{\alpha},\varphi)\Big\cjd_{L^2(\Omega_\T)}.
\]
where $a_j$ are the functional obtained from \eqref{asymptoticRf}. The result then extends holomorphically to the region \eqref{regionvalide} and meromorphically to \eqref{regiondextension}
\end{proof}

\subsubsection{The Poisson operator far from ${\rm Re}(\alpha)=Q$.}
We have seen in Lemma \ref{firstPoisson} that the Poisson operator can be defined far from the spectrum. The problem is that the region of analyticity of $\mc{P}_\ell(\alpha)$ in Lemma \ref{firstPoisson} does not intersect (for $\ell$ large at least) the region of analyticity 
of $\mc{P}_\ell(\alpha)$ from Proposition \ref{poissonprop}. The proposition below extends the construction of the Poisson operator to a region overlapping both regions in  Lemma \ref{firstPoisson}  and Proposition \ref{poissonprop} (see figure \ref{figure3}). 

\begin{proposition}\label{Poissonintermediaire}
For $\ell$ fixed, the Poisson operator $\mc{P}_\ell(\alpha)$ of Lemma \ref{firstPoisson} extends analytically to the region 
\begin{equation}\label{regionetendue} 
\Big\{\alpha=Q+ip\in \C\,\, |\, \,{\rm Re}(\alpha)<Q\, ,\, {\rm Im}(p)>{\rm Im}(\sqrt{p^2-2\la_\ell})-\gamma/2\Big\}
\end{equation}
as an function in $e^{-({\rm Im}(\sqrt{p^2-2\la_\ell})+\eps)\rho}\mc{D}(\mc{Q})$ for all $\eps>0$.
\end{proposition}
\begin{proof}
As before, for $F=\sum_{j=0}^\ell F_j^-\in E_\ell$ with $F_j^-\in {\rm Ker}(\mathbf{P}-\lambda_j)$, we set $u_-(\alpha):=\sum_{j\leq \ell}F^-_je^{ic\sqrt{p^2-2\la_j}}$, and let $\chi\in C^\infty(\R)$ equal to $1$ in $(-\infty,-1)$ and with ${\rm supp}(\chi)\subset \R^-$.
We get 
\[(\mathbf{H}-\tfrac{Q^2+p^2}{2})(\chi u_-(\alpha))= -\tfrac{1}{2}\chi''(c)u_-(\alpha)- \chi'(c)\pl_cu_-(\alpha)+e^{\gamma c}V\chi u_-(\alpha).\]
The first two terms in the right hand side are in $e^{N\rho}L^2(\R\times \Omega_\T)$ for all $N>0$ (indeed, compactly 
supported in $c$), while the last term can be decomposed as 
\[  \Pi_ke^{\gamma c}V\chi u_-(\alpha)+(1-\Pi_k)\chi e^{\gamma c}V u_-(\alpha)\in e^{(-{\rm Im}\sqrt{p^2-2\la_\ell}+\gamma/2)\rho}L^2(\R;E_k)+e^{(-{\rm Im}\sqrt{p^2-2\la_\ell}+\gamma/2)\rho}\mc{D}'(\mc{Q})
\]
by using Lemma \ref{chiPikV}.
Using Proposition \ref{physicalsheet}, we can thus define, with the same formula as in Lemma \ref{definPellproba} and Proposition \ref{poissonprop}, the Poisson operator 
\[u_+(\alpha):=\mathbf{R}(\alpha)(\mathbf{H}-\tfrac{Q^2+p^2}{2})(\chi u_-(\alpha))\in e^{(-{\rm Im}\sqrt{p^2-2\la_\ell}+\gamma/2)\rho}\mc{D}(\mc{Q}),\]
\[ \mc{P}_\ell(\alpha)F:=u_-(\alpha)\chi - u_+(\alpha)\]
in the region 
\[   {\rm Im}(p)={\rm Re}(Q-\alpha)> \max_{j\leq \ell}{\rm Im}\sqrt{p^2-2\la_j}-\gamma/2={\rm Im}\sqrt{p^2-2\la_\ell}-\gamma/2.\]
\end{proof}
\begin{remark}\label{nonempty}
Notice that this region of holomorphy is non-empty and connected, as for $\ell$ and $|p|=R\gg \la_\ell$ 
\[ {\rm Im}(p)-{\rm Im}\sqrt{p^2-2\la_\ell}+\gamma/2=\gamma/2+\mc{O}(\frac{\la_\ell}{R^2})>0.\]\end{remark}

\begin{figure}
\centering
\begin{tikzpicture}
 \tikzstyle{PR}=[minimum width=2cm,text width=3cm,minimum height=0.8cm,rectangle,rounded corners=5pt,draw,fill=red!30,text=black,font=\bfseries,text centered,text badly centered]
  \tikzstyle{NS}=[minimum width=2cm,text width=3cm,minimum height=0.8cm,rectangle,rounded corners=5pt,draw,fill=blue!20,text=black,font=\bfseries,text centered,text badly centered]
    \tikzstyle{AR}=[minimum width=2cm,text width=3cm,minimum height=0.8cm,rectangle,rounded corners=5pt,draw,fill=SeaGreen!50,text=black,font=\bfseries,text centered,text badly centered]
 \tikzstyle{PRfleche}=[->,>= stealth,thick,red!60]
  \tikzstyle{NSfleche}=[->,>= stealth,thick,blue!60]
 \tikzstyle{Afleche}=[->,>= stealth,thick,SeaGreen]
 \node[inner sep=0pt] (F1) at (0,0)
    {\includegraphics[width=.6\textwidth]{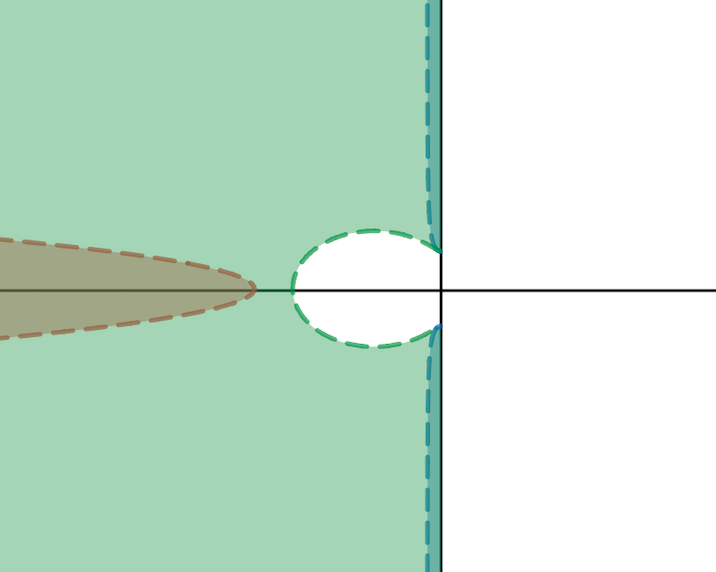}};
 \node (F) at (2,2.3){ ${\rm Re}(\alpha)=Q$};
\node[PR] (P) at (3,-2) {Probabilistic region  \eqref{regionLemma4.4}};
\node[NS] (NSR) at (5,3) {Near spectrum region \eqref{regionvalide}};
\node[AR] (AR) at (4,1) {Analyticity region  \eqref{regionetendue} for $\mc{P}_\ell(\alpha)$};
\node (P1) at (-3,-0.5) {};
\node (N1) at (1,3) {};
\node (A1) at (-2,2) {};
\draw[PRfleche] (P) to [bend left=25] (P1);
\draw[NSfleche] (NSR) to [bend right=25] (N1);
\draw[Afleche] (AR) to [bend left=25] (A1);
\end{tikzpicture}
  \caption{The green colored region correspond to the  region \eqref{regionetendue} of validity of Lemma \ref{Poissonintermediaire} for the analyticity of the Poisson operator $\mc{P}_\ell(\alpha)$ with $\ell>0$ (for the plot, we take $\la_\ell=4$, $\gamma=1/2$). It overlaps the probabilistic \eqref{regionLemma4.4} and near spectrum \eqref{regionvalide} regions.}
   \label{figure3}
\end{figure}
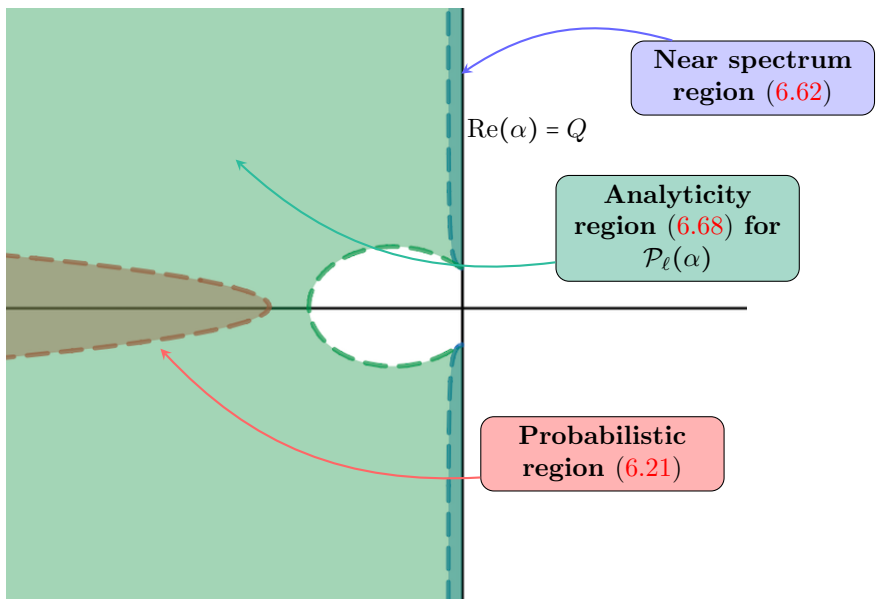


\subsection{The Scattering operator}\label{section_scattering}

\begin{definition}
Let $\ell\in \N$ and $\alpha=Q+ip$ with $p\in\R\setminus (-\sqrt{2\la_\ell},\sqrt{2\la_\ell})$. 
The scattering operator $\mathbf{S}_\ell(\alpha): E_\ell \to E_\ell$ for the $\ell$-th layer (also called $\ell$-scattering operator) is the operator defined as follows:
let $F=\sum_{j\leq \ell}F_j\in E_\ell$ (with $F_j\in {\rm Ker}(\mathbf{P}-\la_j)$) and let $F_j^-:=(p^2-2\la_j)^{-1/4}F_j$, then we set  
\[ \mathbf{S}_\ell(\alpha)F:=\left\{ \begin{array}{ll}
\sum_{j \leq \ell}F_j^+(\alpha)(p^2-2\la_j)^{1/4}, & \textrm{ if }p>\sqrt{2\la_\ell}, \\
\sum_{j\leq \ell}\widehat{F}_j^+(\alpha)(p^2-2\la_j)^{1/4}, & \textrm{ if }p<-\sqrt{2\la_\ell}.
\end{array}\right.\]
where $F_j^+(\alpha),\hat{F}_j^+(\alpha)$ are the functions in \eqref{expansionP} and \eqref{regimebar}. We will call more generally 
$$\mathbf{S}(\alpha):=\left\{\begin{array}{lll}
\bigcup_{\ell, 2\la_\ell<p^2}\ker ({\bf P}-\la_j) &\to & \bigcup_{\ell, 2\la_\ell<p^2}
\ker({\bf P}-\la_j)  \\
F\in E_\ell & \mapsto & \mathbf{S}_\ell(\alpha)F\end{array}\right.
$$
 the scattering operator, where we use  $\mathbf{S}_\ell(\alpha)_{|E_{\ell'}}= \mathbf{S}_{\ell'}(\alpha) $ for $\ell'<\ell$.
\end{definition} 
Let us define the map $\omega_\ell: \Sigma\to \Sigma$ by the following property: if 
$r_j(\alpha)=\sqrt{p^2-2\la_j}$ are the analytic functions on $\Sigma$ used to define 
this ramified covering (with $\alpha=Q+ip$ in the physical sheet),  
then $\omega_\ell(\alpha)$ is the point in $\Sigma$ so that 
\[ r_j(\omega_\ell(\alpha))=\left\{ \begin{array}{ll}
 -r_j(\alpha) & \textrm{ if }j \leq \ell, \\
 r_j(\alpha) & \textrm{ if }j>\ell
\end{array}\right..\]
As a consequence of Proposition \ref{poissonprop} and Lemma \ref{boundarypairing}, we obtain the 
\begin{corollary}\label{mainthscat}
For each $\ell\in\N$, the $\ell$-th scattering operator $\mathbf{S}_\ell(\alpha)$ is unitary on $E_\ell$ if $\alpha=Q+ip$ ($p\in\R$) is such that $\la_\ell<p^2<\min\{\la_j\, |\, \la_\ell<\la_j\}$.  
It also satisfies the functional equation 
\begin{equation}\label{fcteqS} 
\mathbf{S}_\ell(\alpha)\mathbf{S}_\ell(\omega_\ell(\alpha))={\rm Id}.
\end{equation}
Moreover it extends meromorphically in \eqref{regiondextension} if $\beta<\gamma$. It also satisfies the following functional equation for each $F=\sum_{j=0}^\ell F_j\in E_\ell$
\begin{equation}\label{fcteqP} 
\mc{P}_\ell(\alpha)\sum_{j=0}^\ell (p^2-2\la_j)^{-1/4}F_j=\mc{P}_\ell(\omega_\ell(\alpha))\mathbf{S}_\ell(\alpha)\sum_{j=0}^\ell F_j.
\end{equation} 
\end{corollary}
\begin{proof}
The unitarity of $\mathbf{S}_\ell(\alpha)$ on the line ${\rm Re}(\alpha)=Q$ follows directly from Lemma \ref{boundarypairing} applied with $u_\ell=\mc{P}_\ell(\alpha)F$. The functional equation \eqref{fcteqS} reads $\mathbf{S}_\ell(\alpha)\mathbf{S}_\ell(2Q-\alpha)={\rm Id}$ on the line ${\rm Re}(\alpha)=Q$
and that comes directly from the uniqueness statement in Proposition \ref{poissonprop} on the line. The extension  of ${\bf S}(\alpha)$ with respect to $\alpha$ comes directly from the meromorphy 
of the $a_j(\alpha,F)$ in Proposition \ref{extensionresolvent}. The functional identity extends meromorphically under the formula \eqref{fcteqS}. 
The functional equation \eqref{fcteqS} also comes from uniqueness of the  Poisson operator. 
\end{proof}

\begin{theorem}\label{spectralmeasure}

For each $j\in \N$, let $(h_{jk})_{k=1,\dots,k(j)}$ be an orthonormal basis of 
$\ker_{L^2(\Omega_\T)}(\mathbf{P}-\la_j)$. The spectral resolution holds   for all $ \varphi,\varphi'\in  e^{\beta\rho}L^2(\Omega_\T\times\R)$ with $\beta>0$ 
\begin{equation}\label{spectraldecomposH}
 \cjg \varphi\,|\,\varphi'\cjd_2 = \frac{1}{2\pi}\sum_{j=0}^{\infty} \sum_{k=1}^{k(j)}\int_{0}^\infty
\Big\cjg \varphi \,|\,\mc{P}_j\Big(Q+i\sqrt{p^2+2\la_j}\Big)h_{jk}\Big\cjd 
\Big\cjg \mc{P}_j\Big(Q+i\sqrt{p^2+2\la_j}\Big)h_{jk}\,|\,\varphi'\Big\cjd {\rm d}p.
\end{equation}
As a consequence the spectrum of ${\bf H}$ is absolutely continuous.
\end{theorem}
\begin{proof}
We recall the Stone formula: for $\varphi,\varphi' \in e^{\beta \rho}L^2$ for $\beta>0$
\[ \begin{split}
\cjg \varphi\,|\,\varphi'\cjd_2=& \frac{1}{2\pi i}\lim_{\eps\to 0^+}\int_{0}^{\infty}\cjg [(\mathbf{H}-\tfrac{Q^2}{2}-t-i\eps)^{-1}-(\mathbf{H}-\tfrac{Q^2}{2}-t+i\eps)^{-1}]\varphi\,|\,\varphi'\cjd {\rm d}t\\
=& \frac{1}{2\pi i}\lim_{\eps\to 0^+}\int_{0}^{\infty}\cjg [(\mathbf{H}-\tfrac{Q^2+p^2}{2}-i\eps)^{-1}-(\mathbf{H}-\tfrac{Q^2+p^2}{2}+i\eps)^{-1}]\varphi\,|\,\varphi'\cjd p{\rm d}p\\
&= \frac{1}{2\pi i}\int_0^\infty \cjg [(\mathbf{R}(Q+ip)-\mathbf{R}(Q-ip))]\varphi\,|\,\varphi'\cjd p{\rm d}p.
\end{split}\]
Here $\alpha=Q+ip$ (with $p>0$) has to be viewed as an element in $\Sigma$ obtained by limit $Q+ip-\eps$ as $\eps\to 0^+$ and, if $\ell$ is the largest integer such that $2\la_\ell\leq p^2$, we write $\bbar{\alpha}$ for the point $\omega_\ell(\alpha)$ on $\Sigma$. 
For $\alpha=Q+ip$ with $p\in \R^+$, we have for $\varphi \in e^{\beta \rho}L^2$ with $\beta>0$
\[ \begin{split}
\frac{1}{2i}\cjg (\mathbf{R}(\alpha)-\mathbf{R}(\bbar{\alpha}))\varphi\,|\,\varphi\cjd&= {\rm Im}\cjg \mathbf{R}(\alpha)\varphi\,|\,\varphi\cjd\\
&= {\rm Im}\cjg \mathbf{R}(\alpha)(\mathbf{H}-2\Delta_\alpha)\mathbf{R}(\bbar{\alpha})\varphi\,|\,\varphi\cjd \\
&= {\rm Im}\cjg (\mathbf{H}-2\Delta_\alpha)\mathbf{R}(\bbar{\alpha})\varphi\,|\,\mathbf{R}(\bbar{\alpha})\varphi\cjd \end{split}.\]
Here we have used that $(\mathbf{H}-2\Delta_\alpha)\mathbf{R}(\bbar{\alpha})={\rm Id}$ on $e^{\beta\rho} L^2$ provided $p\in \R$ and that 
$\cjg \mathbf{R}(\bbar{\alpha})\varphi\,|\,\varphi'\cjd=\cjg \varphi,{\bf R}(\alpha)\varphi'\cjd$ 
for $\varphi\,|\,\varphi'\in e^{\beta \rho}L^2$, this last fact coming from the identity 
$\mathbf{R}(\bbar{\alpha})=\mathbf{R}(\alpha)^*$ for ${\rm Re}(\alpha)<Q$ and passing to the limit ${\rm Re}(\alpha)\to Q$.  
Let $\theta_T(c)$ as in the proof of Lemma \ref{boundarypairing}. We have 
\[\begin{split} 
\cjg (\mathbf{H}-2\Delta_\alpha)\mathbf{R}(\bbar{\alpha})\varphi,\mathbf{R}(\bbar{\alpha})\varphi\cjd& =\lim_{T\to +\infty}
\cjg \theta_T(\mathbf{H}-2\Delta_\alpha)\mathbf{R}(\bbar{\alpha})\varphi,\mathbf{R}(\bbar{\alpha})\varphi\cjd\\
&= \lim_{T\to +\infty}
\cjg \theta_T\mathbf{R}(\bbar{\alpha})\varphi\,|\,\varphi\cjd+\lim_{T\to \infty}\tfrac{1}{2}\cjg [\pl^2_c,\theta_T]\mathbf{R}(\bbar{\alpha})\varphi\,|\,\mathbf{R}(\bbar{\alpha})\varphi\cjd .
\end{split}\]
 Using \eqref{asymptoticRf} and arguing as in the proof of Lemma \ref{boundarypairing}, as $T\to \infty$ we get
\[\begin{split}
\cjg (\mathbf{H}-2\Delta_\alpha)\mathbf{R}(\bbar{\alpha})\varphi,\mathbf{R}(\bbar{\alpha})\varphi\cjd
=& \cjg \mathbf{R}(\bbar{\alpha})\varphi,\varphi\cjd
+\tfrac{1}{2}\lim_{T\to \infty}\sum_{j\leq \ell}\|a_j(\bbar{\alpha},\varphi)\|^2_{L^2(\Omega_\T)}
  \pl_c(e^{ic\sqrt{p^2-2\la_j}})_{|c=-T}e^{iT\sqrt{p^2-2\la_j}}\\
& -\tfrac{1}{2}\lim_{T\to \infty}\sum_{j\leq \ell}\|a_j(\bbar{\alpha},\varphi)\|_{L^2(\Omega_\T)}^2 e^{-iT\sqrt{p^2-2\la_j}}\pl_c(e^{-ic\sqrt{p^2-2\la_j}})_{|c=-T}\\
 =& \cjg \mathbf{R}(\bbar{\alpha})\varphi,\varphi\cjd-i\sum_{j\leq \ell}{\sqrt{p^2-2\la_j}}\|a_j(\bbar{\alpha},\varphi)\|_{L^2(\Omega_\T)}^2.
\end{split}\]
We conclude that  
\begin{equation}\label{ImRalpha} 
-{\rm Im}\cjg \mathbf{R}(\bbar{\alpha})\varphi,\varphi\cjd= \tfrac{1}{2}\sum_{j\leq \ell}{\sqrt{p^2-2\la_j}}\|a_j(\bbar{\alpha},\varphi)\|^2_{L^2(\Omega_\T)}.
\end{equation}
By Lemma \ref{Poissonaj}, we have 
\begin{equation}\label{Pell*}
\mc{P}_\ell(\alpha)^*\varphi=-i\sum_{j\leq \ell}\,  a_j(\bbar{\alpha},\varphi)\sqrt{p^2-2\la_j}.
\end{equation}
By polarisation and by denoting $\Pi_j$ the orthogonal projectors on ${\rm Ker}(\mathbf{P}-\lambda_j)$, we deduce from \eqref{ImRalpha} and \eqref{Pell*} that 
$$-{\rm Im}\cjg \mathbf{R}(\bbar{\alpha})\varphi,\varphi\cjd= \tfrac{1}{2}\sum_{j\leq \ell}\frac{1}{\sqrt{p^2-2\la_j}}\cjg \Pi_j \mc{P}_\ell(\alpha)^*\varphi, \Pi_j \mc{P}_\ell(\alpha)^*\varphi' \cjd_{L^2(\Omega_\T)}.$$
Rewriting $\mc{P}_\ell(\alpha)^*\varphi=\sum_{j=0}^\ell\sum_{k=1}^{k(j)}\cjg\varphi,\mc{P}_\ell(\alpha)h_{jk}\cjd_2 h_{jk}$, we obtain
\[ \cjg \varphi,\varphi'\cjd_2 = \frac{1}{2\pi}\int_{0}^\infty \sum_{\ell=0}^{\infty}\indic_{[\sqrt{2\la_{\ell}},\sqrt{2\la_{\ell+1}})}(p)\sum_{j=0}^\ell\sum_{k=1}^{k(j)}
\cjg \varphi,\mc{P}_\ell(Q+ip)h_{jk}\cjd_2 \bbar{\cjg \varphi',\mc{P}_\ell(Q+ip)h_{jk}\cjd_2} \frac{p}{\sqrt{p^2-2\la_j}} dp.\]
This finally can be rewritten, using \eqref{egalitePell}, as 
\[\begin{split}
\cjg \varphi,\varphi'\cjd_2  = &\frac{1}{2\pi}\sum_{j=0}^{\infty}\sum_{\ell\geq j}\int_{0}^\infty \indic_{[\sqrt{2\la_{\ell}},\sqrt{2\la_{\ell+1}})}(p)\sum_{k=1}^{k(j)}
\cjg \varphi,\mc{P}_\ell(Q+ip)h_{jk}\cjd \bbar{\cjg \varphi',\mc{P}_\ell(Q+ip)h_{jk}\cjd} \frac{p}{\sqrt{p^2-2\la_j}} dp\\
= &\frac{1}{2\pi}\sum_{j=0}^{\infty}\int_{\sqrt{2\la_j}}^\infty \sum_{k=1}^{k(j)}
\cjg \varphi,\mc{P}_j(Q+ip)h_{jk}\cjd \bbar{\cjg \varphi',\mc{P}_j(Q+ip)h_{jk}\cjd} \frac{p}{\sqrt{p^2-2\la_j}}dp\\
=& \frac{1}{2\pi}\sum_{j=0}^{\infty} \sum_{k=1}^{k(j)}\int_{0}^\infty
\Big\cjg \varphi ,\mc{P}_j\Big(Q+i\sqrt{r^2+2\la_j}\Big)h_{jk}\Big\cjd 
\bbar{\Big\cjg \varphi',\mc{P}_j\Big(Q+i\sqrt{r^2+2\la_j}\Big)h_{jk}\Big\cjd} dr
\end{split}\]
where we performed the change of variables in the last line $r=\sqrt{p^2-2\la_j}$.
\end{proof}

\subsection{Holomorphic parametrization of the eigenstates}\label{sub:holomorphic}
We conclude this section by defining the generalized eigenfunctions $\Psi_{\alpha,{\bf k},{\bf l}}$ by using the standard orthonormal basis $(\psi_{{\bf kl}})_{{\bf k,l}\in \mc{N}}$  of $L^2(\Omega_\T) $ made of Hermite polynomials and introduced below \eqref{firstlength}. We set 
\[ \Psi_{\alpha,{\bf k},{\bf l}}:=\mc{P}_j(Q-\sqrt{(Q-\alpha)^2-2\la_{{\bf kl}}})\psi_{{\bf kl}}\] 
where $\ell$ is defined by $\la_\ell=\la_{{\bf kl}}$ and $\sqrt{z}$ is used here with the convention that the cut is on $\R^-$ (ie. $\sqrt{Re^{i\theta}}=\sqrt{R}e^{i\theta/2}$ for $\theta\in (-\pi,\pi)$). These are eigenfunctions of ${\bf H}$ with eigenvalues $2\Delta_\alpha+\la_{{\bf kl}}$. Using \eqref{regionvalide} and \eqref{regionetendue} for the holomorphy of $\mc{P}_\ell(\cdot)$, we obtain the 
\begin{proposition}\label{holomorphiceig}
Let $\ell\geq 0$ such that $\la_{{\bf kl}}=|{\bf k}|+|{\bf l}|=\la_\ell$. For each $\eps>0$, the function $\Psi_{\alpha,{\bf k},{\bf l}}\in e^{- \beta \rho}\mc{D}(\mc{Q})$ is an eigenfunction of ${\bf H}$ with eigenvalue $2\Delta_\alpha+\la_\ell$ for all $\beta>Q-{\rm Re}(\alpha)$, it is holomorphic on the set
\begin{equation}\label{defWell}
W_\ell:= \Big\{ \alpha \in \C \setminus \mc{D}_\ell\, |\,  
{\rm Re}(\alpha)\leq Q, {\rm Re}\sqrt{(Q-\alpha)^2-2\la_\ell}>{\rm Re}(Q-\alpha)-\gamma/2\Big\}
\end{equation}
where $\mc{D}_\ell:=\bigcup_{j\geq \ell}\{Q\pm i\sqrt{2(\la_j-\la_\ell)}\}$ is a discrete set where  
$\Psi_{\alpha,{\bf k},{\bf l}}$ is continuous in $\alpha$ with square root singularities.
The set $W_\ell$ is a connected subset of the half-plane ${\rm Re}(\alpha)\leq Q$, containing $(Q+i\R)\setminus \mc{D}_\ell$ and the real half-line $(-\infty,Q-\frac{2\la_\ell}{\gamma}-\frac{\gamma}{4})$. 
\end{proposition}
Finally, one has for $P\in \R$, $\Psi_{Q+iP,{\bf k},{\bf l}}=\mc{P}_j(Q+i\sqrt{P^2+2\la_{{\bf kl}}})\psi_{{\bf kl}}$ and one can rewrite \eqref{spectraldecomposH} as 
\begin{equation}\label{spectraldecomposH2} 
\cjg \varphi\,|\,\varphi'\cjd_2 = \frac{1}{2\pi}\sum_{{\bf k,l}\in \mc{N}} \int_{0}^\infty
\cjg \varphi \,|\,\Psi_{Q+iP,{\bf k},{\bf l}}\cjd 
\cjg \Psi_{Q+iP,{\bf k},{\bf l}}\,|\,\varphi'\cjd {\rm d}P.
\end{equation}

\section{Probabilistic representation of the Poisson operator}\label{sec:proba}
In Section \ref{sec:scattering}, we constructed the generalized eigenstates $\Psi_{\alpha,{\bf k},{\bf l}}$ by means of the Poisson operator (see Prop \ref{holomorphiceig}) of the Liouville Hamiltonian $\mathbf{H}$ on the spectrum line $\alpha\in Q+i\R$ and showed that these generalized eigenstates can be analytically continued in the parameter $\alpha $ over the region $W_\ell$ defined by \eqref{defWell} for $|{\bf k}|+|{\bf l}|=\la_\ell$. As in the case of the $\mathbf{H}^0$-eigenstates, we will need to perform a change of basis. So, similarly to Prop \ref{prop:mainvir0} item 4, we set for $\nu,\tilde{\nu}\in\mc{T}$ Young diagrams,  with $N:=|\nu|+|\tilde\nu|$,
\begin{equation}\label{defdescendents}
\begin{gathered}
\Psi_{\alpha,\nu,\tilde{\nu}}: =  \sum_{\k,\l, |{\bf k}|+|{\bf l}|=N}M^{N}_{\alpha,\k\l,\nu\tilde\nu}\Psi_{\alpha,\k,\l},
\end{gathered}
\end{equation}  
with the convention that  $\Psi_{\alpha,\emptyset, \emptyset}=\Psi_{\alpha,0,0}$, which will be denoted by $\Psi_{\alpha}$. The $(\Psi_{\alpha,\nu,\tilde{\nu}})_{\nu,\tilde\nu}$ (with $\nu\not=\emptyset$ or $\tilde\nu\not=\emptyset$) will be called descendant states. Since the coefficients $M^{N}_{\alpha,\k\l,\nu\tilde\nu}$ are analytic  in $\alpha\in\C$ (see Prop \ref{prop:mainvir0}), $\Psi_{\alpha,\nu,\tilde{\nu}}$ satisfy the same holomorphy properties (namely Prop \ref{holomorphiceig}) as $\Psi_{\alpha,{\bf k},{\bf l}}$ provided that $|{\bf k}|+|{\bf l}|=|\nu|+|\tilde\nu|$. The main goal of this section is to compute the correlation functions of these descendant states and our strategy can be summarized as follows.

 It turns out that for $\alpha$ real in the physical region, which we will call probabilistic region, we will be able to give a probabilistic representation for the descendant states thanks to  the {\it intertwining} construction based on the descendant states $\Psi^0_{\alpha,\nu,{\tilde{\nu}}}$ of the GFF theory $\mu=0$:
\begin{equation} \label{secward:intert}
 \lim_{t\to +\infty}e^{t (2\Delta_\alpha+|\nu|+|{\tilde{\nu}}|)}e^{-t\mathbf{H}}\Psi^0_{\alpha,\nu,{\tilde{\nu}}}=\Psi_{\alpha,\nu,{\tilde{\nu}}}
\end{equation}
for $\alpha$ real and negative enough (see Proposition \ref{descconvergence}). Indeed, in Subsection \ref{SET},  we express  the free descendant states in terms of contour integrals of the Stress Energy Tensor (SET) of the GFF theory, which we plug then in \eqref{secward:intert}. Applying    the Feynman-Kac formula for the propagator $e^{-t\mathbf{H}}$ in the expression \eqref{secward:intert}, we obtain a probabilistic expression for the descendant states $\Psi_{\alpha,\nu,\tilde{\nu}}$   in terms of contour integrals of SET in LCFT\footnote{Recast in the language of CFTs, this is somewhat equivalent to stating that descendant states can be obtained via the action of Virasoro generators on primary states.}   (see Corollary \ref{propcontour}).   The Ward identities (Proposition \ref{proofward}) then allows us to translate these contour integrals
in terms of differential operators: more precisely, correlation functions of descendant states $\Psi_{\alpha,\nu,\tilde{\nu}}$ can be obtained in terms of differential operators acting on the correlation functions of primary states $\Psi_{\alpha}$. Finally  we will analytically continue these relations from the region $\{\alpha\in\R;\alpha<Q\}$ back to the spectrum  line and obtain the consequences in Subsection \ref{sub:desc3point} about the structure of 3 point correlation functions involving descendant states.
 
\subsection{Highest weight states}\label{sub:HW}
Recall the definition \eqref{psialphadef}  of the highest weight state $\Psi^0_\alpha(c,\varphi)=e^{(\alpha-Q)c}$   for $\alpha\in\C$ of the $\mu=0$ theory. From Proposition \ref{Pellproba} item 2 (applied with $F=1$), we know that for $\alpha\in\R$ with $\alpha<Q$, the state 
\begin{align}\label{Psialpha}
\Psi_\alpha=\mathcal{P}_0(\alpha)1
\end{align}
is given by the large time limit 
\begin{equation}\label{intertprimary}
\Psi_\alpha=\lim_{t\to\infty}e^{2t\Delta_\alpha}e^{-t\mathbf{H}}\Psi^0_\alpha,\quad dc\otimes \P_\T \text{ a.e.}
\end{equation}
where $\Delta_\alpha$ denotes the conformal weight \eqref{deltaalphadef}. In physics (or representation theory)  terminology  the state
$\Psi_\alpha$ is the highest weight state corresponding to the  {\it primary field} $V_\alpha$. Combining \eqref{intertprimary} with the Feynman-Kac formula \eqref{fkformula} leads to the probabilistic representation  for $\alpha<Q$
\begin{equation}\label{defvalpha}
\Psi_\alpha(c,\varphi)  :=e^{(\alpha-Q)c} \E_\varphi\Big[ \exp\Big(-\mu e^{\gamma c}\int_{\D} |x|^{-\gamma\alpha }M_\gamma(\dd x)\Big)\Big].
\end{equation} 
We recall here that the integrability of $ |x|^{-\gamma\alpha }$ with respect to $M_\gamma(\dd x)$ is detailed in  \cite{DKRV}.
\begin{remark}
In forthcoming work, we will show that, for $\alpha\in (\tfrac{2}{\gamma},Q)$, we have as $c\to-\infty$
$$ \Psi_\alpha (c,\varphi) = e^{(\alpha-Q)c}+e^{(Q-\alpha)c} {R}(\alpha) +e^{(Q-\alpha)c}o(1 )$$
with $o(1 )\to 0$   in $L^2(\Omega_\T)$ as $c\to-\infty$ with $R$  the reflection coefficient defined in \cite{dozz}, which thus appears as the scattering coefficient of constant functions: for $\ell=0$ we have $\mathbf{S}_0(\alpha)=R(\alpha) \mathbf{Id}$. More generally, we will show that the scattering matrix is diagonal.
\end{remark}

\subsection{Descendant states}\label{sub:desc}
Recall (Subsection \ref{repth}) that for $\mu=0$ we have descendant states given by
\begin{align*}
\Psi^0_{\alpha,\nu,\tilde\nu}(c,\varphi)=\mathcal{Q}_{\alpha,\nu,\tilde{\nu}}(\varphi)e^{(\alpha-Q)c}
\end{align*}
where $\mathcal{Q}_{\alpha,\nu,\tilde{\nu}}$ are eigenstates of the operator $\bf P$:
\begin{align*}
{\bf P}\mathcal{Q}_{\alpha,\nu,\tilde{\nu}}=(|\nu|+|\tilde{\nu}|)\mathcal{Q}_{\alpha,\nu,\tilde{\nu}}
\end{align*}
so that 
\begin{align*}
{\bf H}^0\Psi^0_{\alpha,\nu,\tilde{\nu}}=(2\Delta_\alpha+|\nu|+|\tilde{\nu}|)\Psi^0_{\alpha,\nu,\tilde{\nu}}.
\end{align*}
From  Proposition \ref{Pellproba} we infer the following   
\begin{proposition}\label{descconvergence}
Let $\alpha<\big(Q-\frac{2(|\nu|+|\tilde{\nu}|)}{\gamma}-\frac{\gamma}{4}\big)\wedge (Q-\gamma)$. Then the limit
\begin{equation}\label{descremi}
 \lim_{t\to +\infty}e^{t (2\Delta_\alpha+|\nu|+|{\tilde{\nu}}|)}e^{-t\mathbf{H}}\Psi^0_{\alpha,\nu,{\tilde{\nu}}}=\Psi_{\alpha,\nu,{\tilde{\nu}}}
\end{equation}
holds in $e^{-(\beta+\gamma/2)\rho}L^2(\R\times\Omega_\T)$ for $\beta>Q-\alpha-\gamma/2$.
\end{proposition}
\begin{proof}
Write  $|\nu|+|{\tilde{\nu}}|=\lambda_j$ for some $j$. We apply  Proposition \ref{Pellproba} item 1 but we have to make a small notational warning: indeed recall that in Section  \ref{sec:scattering}, eigenvalues are parametrized by $2\Delta_\alpha$ whereas here eigenvalues correspond to $ 2\Delta_\alpha+|\nu|+|{\tilde{\nu}}| = 2\Delta_\alpha+\lambda_j $. So let us call $\alpha'$ the $\alpha$ in the statement of Proposition  \ref{Pellproba}, and write it as $\alpha'=Q+ip$ with $p^2=2\la_j-(Q-\alpha)^2$ ($p\in i\R$) in such a way that 
$$2\Delta_{\alpha'}=2\Delta_{\alpha}+\lambda_j,\quad\text{otherwise stated }\sqrt{p^2-2\lambda_j}=i(Q-\alpha).$$
In particular ${\rm Im}\sqrt{p^2-2\lambda_j}=Q-\alpha>\gamma$ so that we can choose   $\chi=1$ in Proposition \ref{Pellproba} item 1.
Let  $\ell\geq 1$, $j\leq \ell$ and $F=\mathcal{Q}_{\alpha,\nu,{\tilde{\nu}}}$. Then the limit \eqref{descremi} exists in $e^{-(\beta+\gamma/2) \rho}L^2$  if
$\beta>{\rm Im}\sqrt{p^2-2\la_j}-\gamma/2=Q-\alpha-\gamma/2>0$ and $-p^2>\beta^2$.  In conclusion we get $\alpha<Q-\gamma/2=2/\gamma$   and  $-p^2>(Q-\alpha-\gamma/2)^2$. Substituting $ p^2=2\la_j-(Q-\alpha)^2$ in the latter, we arrive at the relation $(Q-\alpha)^2-2\lambda_j>(Q-\alpha-\gamma/2)^2$, which can be solved to find our  condition.
 \end{proof}

  \subsection{Stress Energy Field}\label{SET}

In this section we construct a probabilistic representation for the Virasoro descendants \eqref{psibasis}. This can be done in terms of a local field,  the {\it stress-energy tensor}, formally given for $z\in\D$ by
\begin{equation}\label{stress}
T(z):=Q\partial^2_{z}X(z)-(\partial_z X(z))^2+\E[(\partial_zX(z))^2].
\end{equation}
The stress tensor does not make sense as a random field but can be given sense at the level of correlation functions as the limit of   $T_\epsilon(z)$ defined by \eqref{stress} with $X$ replaced by  a regularized field $X_\epsilon$ which we take in the form
\begin{align}\label{XepsDdef}
X_{\epsilon}(z):=  \langle X, f_{\epsilon,z}\rangle_{\D}
\end{align}
for suitable test function $f_{\epsilon,z}$.  
We will use two regularizations in what follows. For the first one we take $f_{\epsilon,z}(u)=\frac{1}{\epsilon^2} \varrho(\frac{z-u}{\epsilon} )$ with non-negative $\varrho\in C_c^\infty(\C)$ with $\varrho(0)=1$ and  $\int_{\C} \varrho(x) \dd x=1$.  For this regularization we have for $t\geq 0$ the following scaling relation
\begin{align}\label{scalingrelation}
S_{-t}T_\epsilon(z)=e^{-2t}T_{e^{-t}\epsilon}(e^{-t}z).
\end{align}
For the second regularization we write $z=e^{-t+i\theta}$ and take  $f_{\epsilon,z}(e^{-s+i\theta'})=\frac{1}{2\pi\epsilon}\rho(\frac{t-s}{\epsilon})\sum_{|n|<\epsilon^{-1}}e^{in(\theta-\theta')}$ with $\rho\in C_c^\infty(\R)$, $\rho(0)=1$ and  $\int_{\R} \rho(x) \dd x=1$. This regularization has the property that $T_\epsilon(z)$ depends only on the Fourier components $\varphi_n$ with $|n|<1/\epsilon$ when we decompose $X=P\varphi+X_\D$. We denote also by $\bar T(z)$  the complex conjugate of $T(z)$.

Our goal is  to express the action of the Virasoro generators \eqref{virassoro} and \eqref{virassorotilde} on the state $U_0F$  in terms of the states
\begin{align}\label{TbarT}
U_0\big(\prod_{i=1}^kT(u_i)\prod_{j=1}^l \bar T(v_j) F\big) 
\end{align}
which will be defined as limits of regularized expressions.
We start by specifiying a suitable class of $F$ for which \eqref{TbarT} makes sense. 
 Let $\delta<1$ .We introduce the set
$\caE_\delta$ defined by  
\begin{equation}\label{defEcall} 
\caE_\delta:=\big\{f\in C_0^\infty(\delta \D) \mid f(e^{-t+i\theta})=\sum_{n\in\Z}f_n(t)e^{in\theta}, \text{ with }f_n\in C_0^\infty((-\ln\delta,\infty)) \text{ and }f_n=0\text{ for } |n| \text{ large enough}\big\}.
\end{equation}
Define $\caF_\delta\subset  \caF_{  \D}$ by  
\begin{equation}\label{Fform}
 \caF_\delta= {\rm span}\ \{ \prod_{i=1}^l \langle g_i,  c+X\rangle_\D e^{\langle f,c+X \rangle _\D}; \; l \geq 0, \:  f,g_i \in \caE_\delta \big \} .
\end{equation}
with $l \geq 0$, $f,g_i \in \caE_\delta$.
We note that $U_0F\in e^{\beta|c|}L^2(\R\times\Omega_\T)$ for $\beta>|f_0-Q|$ where   $f_0=\langle f,1 \rangle_\D$ and is in the domain of the Virasoro operators  $\mathbf{L}_{-\nu}^0\tilde{\mathbf{L}}_{-\tilde\nu}^0$ defined in Subsection \ref{repth} since it depends on a finite number of $\varphi_n$. Let  
 \begin{align}\label{caOt}
 \caO_\delta=\{({\bf u},{\bf v})\in\C^{m+n}\,|\, \delta<|u_j|,|v_j|<1,\; \forall j\neq j', \:  \  |u_j|\neq |u_{j'}|, |v_j|\neq |v_{j'}|, |u_j|\neq |v_{j'}| \}.
\end{align}
We have the simple:

\begin{proposition}\label{basicstresss}
Let $F\in\caF_{\delta}$. 
Then the functions
\begin{align*}
({\bf u},{\bf v})\to 
U_0\big(\prod_{i=1}^kT_{\epsilon_i}(u_i)\prod_{j=1}^l \bar T_{\epsilon'_j}(v_j) F\big)
\end{align*}
are continuous on  $\caO_\delta$ with values  in $e^{\beta|c|}L^2(\R\times\Omega_\T)$ for $\beta>| f_0-Q|$ and converge uniformly on compact subsets of $\caO_\delta$  as $(\epsilon_i)_i$ and $(\epsilon'_j)_j$ tend successively to $0$ in whatever order. The limit is independent of the regularization or the order along which limits have been taken and denoted by \eqref{TbarT} and it   defines a  function holomorphic in $\bf u$ and anti-holomorphic in $\bf v$ in the region $\caO_\delta$ taking values in  $e^{\beta|c|}L^2(\R\times\Omega_\T)$.
\end{proposition}

\begin{proof}
To keep the notation simple we consider only the case $j=0$ in \eqref{TbarT} and the case where $F\in \caF_{\delta  }$ is of the form $F=e^{\langle f,c+X \rangle _\D}$ with $f\in  \caE_\delta$. By replacing $f$ by $f+\sum_i\lambda_i g_i$ with $g_i \in \caE_\delta$ and differentiating at $\lambda_i=0$ we can deduce the result for general $F$.  For such $F$ we have
\begin{align}\label{ozerof}
(U_0F)(c,\varphi)=e^{(f_0-Q)c}e^{\langle P\varphi,f \rangle_\D}e^{\hf \langle f, G_\D f \rangle_\D}.
\end{align}
By Gaussian integration by parts (see \eqref{basicipp}) the right-hand side of \eqref{TbarT} is a sum of terms 
\begin{align}\label{messy}
{\rm const}\times\prod_i(\partial^{a_i}P\varphi_{\epsilon_i}(u_i))^{b_i}\prod_{j<k}(\partial^{b_{jk}}_{u_j}\partial^{c_{jk}}_{u_k}(f_{\epsilon_j,u_j},G_\D f_{\epsilon_k,u_k})_\D
^{d_{jk}})\prod_l\partial^{d_l}_{u_l}(f_{\epsilon_l,u_l},G_\D f)_\D
^{e_l}U_0F
\end{align}
where $a_i,b_{jk}, c_{jk}\in \{1,2\}$, $b_i,d_{jk},e_l\in \{0,1,2\}$, $d_l\in \{1,2\}$ and $G_\D$ is the Dirichlet Green function \eqref{dirgreen}. 
The functions $(f_{\epsilon_j,u_j},G_\D f_{\epsilon_k,u_k})_\D
$ and $(f_{\epsilon_l,u_l},G_\D f)_\D$  converge uniformly on compacts of $\caO_\delta$ to smooth functions.  From \eqref{dirgreen}  
 we get
\begin{align}\label{delG}
\partial_zG_\D (z,u)=-\hf(\frac{1}{z-u}-\frac{1}{z-\frac{1}{\bar u}}).
\end{align}
Hence the limit of the second product in \eqref{messy} is holomorphic since $b_{jk}, c_{jk}>0$. For the third product, we get convergence to terms of the form 
$$\int_\D \partial_u G_\D(u,z)f(z) \dd z=-\tfrac{1}{2}\sum_{n=0}^\infty (f_n u^{-n-1}+f_{-n-1}u^n)
$$
or its $\partial_u$ derivative where $f_n=\langle f,u^n \rangle_\D$ and $f_{-n}=\langle f,\bar u^n \rangle_\D$ and the sum is finite and holomorphic.  
Finally,  recalling \eqref{harmonic} it is easy to check that the first product converges in $L^2(\Omega_\T)$ uniformly on compacts of $\caO_\delta$ to a holomorphic function $g({\bf u})$   and $g({\bf u})U_0F\in  e^{\beta|c|}L^2(\R\times\Omega_\T)$.
\end{proof}


Now we will consider contour integrals of observables of the type \eqref{TbarT}, for which we use the following notation: for $f:\D\to\C$ and  $\delta>0$ 
\begin{align*}
\oint_{|z|=\delta }f(z) \dd z: =i \delta \int_0^{2\pi}f( \delta e^{i\theta}) e^{i \theta} \dd \theta,\ \  \oint_{|z|={\delta}}f(z)\dd \bar z:=i {\delta} \int_0^{2\pi}f({\delta} e^{i\theta})e^{-i \theta} \dd \theta.
\end{align*}
Then we have

\begin{lemma}\label{keepcooldontStress}
 Let $F\in\caF_{\delta}$   and $\delta'>\delta$.  
Then for  all $n>0$
\begin{align}\label{basicT}
\tfrac{1}{2\pi i}\oint_{|z|=\delta'} z^{1-n}U_0(T(z)F) \dd z=\mathbf{L}_{-n}^0U_0F,\ \ \tfrac{1}{2\pi i}\oint_{|z|=\delta'} \bar z^{1-n}U_0(\bar{T}(z)F)d\bar z=\tilde{\mathbf{L}}_{-n}^0U_0F.
\end{align}
\end{lemma}
\begin{proof}
Here again (for simplicity) we consider the case where $F\in \caF_{\delta \D}$ is of the form $F=e^{\langle f,c+X \rangle _\D}$ with $f\in  \caE_\delta$. 
Let us write the  integration by parts terms explicitly. First
\begin{align*}
U_0(\partial_zX(z)F) 
=&\partial_zP\varphi(z)U_0 F+\int_\D  
\partial_zG_\D (z,u)
f(u) \dd u\,U_0F. 
\end{align*}
From \eqref{delG}  
 we get
\begin{align}\label{delG2}
\partial_zG_\D (z,u)
=-\hf\sum_{n=0}^\infty(u^nz^{-n-1}+\bar u^{n+1}z^n)
\end{align}
which converges since $|u|<\delta<\delta'=|z|<1$. Therefore
$$\int_\D \partial_z G_\D(z,u)f(u) \dd u=-\tfrac{1}{2}\sum_{n=0}^\infty (f_n z^{-n-1}+f_{-n-1}z^n)$$
where $f_n=\langle f,u^n \rangle_\D$ and $f_{-n}=\langle f,\bar u^n \rangle_\D$ for $n\geq 0$.
Recalling \eqref{harmonic}, we have obtained
\begin{align}\label{lll}
U_0(\partial_zX(z)F)&=\sum_{n\geq 0}(n\varphi_nz^{n-1}-\hf (f_n z^{-n-1}+f_{-n-1}z^n))U_0F\\&=\sum_{n\in\Z}z^{n-1}(n\varphi_n1_{n>0}-\hf f_{-n} )U_0F.\nonumber
\end{align}
By \eqref{ozerof}
\begin{align}\label{ozerof1}
(U_0F)(c,\varphi)=e^{(f_0-Q)c}e^{\sum_{n>0}(\varphi_nf_{n}+\varphi_{-n}f_{-n})}e^{\hf \langle f, G_\D f \rangle_\D}
\end{align}
so that
\begin{align}\nonumber
U_0(\partial_zX(z)F)&=(-\hf f_0z^{-1}+\sum_{n\neq 0}z^{n-1}(n\varphi_n1_{n>0}-\hf \partial_{-n}))U_0F\\&=i\sum_{n\in\Z}z^{n-1}{\bf A}_{-n}U_0F \label{llll}
\end{align}
where we recall that ${\bf A}_0=\frac{i}{2}(\partial_c+Q)$. Hence
\begin{align}\label{QtermT}
U_0(Q\partial^2_zX(z)F)=&-iQ \sum_{n\in \Z}(n+1)z^{-n-2} \mathbf{A}_{n}U_0 F.
\end{align}

Next consider the quadratic terms in $T$ and use Gaussian integrating by parts twice to get  
\begin{align*}
U_0(((\partial_z X(z))^2-\E[(\partial_zX(z))^2])F)
&= \Big(\partial_z P\varphi(z)+\int_\D \partial_z G_\D(z,u)f(u) \dd u
\Big)^2U_0F\\&=(\sum_{n\in\Z}z^{n-1}(n\varphi_n1_{n>0}-\hf f_{-n} ))^2U_0F\\&=\sum_{n,m}z^{n+m-2} (-\mathbf{A}_{-n} \mathbf{A}_{-m}+\tfrac{m}{2}\delta_{m,-n}1_{m>0})U_0F\\&=-\sum_{n,m}z^{n+m-2} :\mathbf{A}_{-n} \mathbf{A}_{-m}:U_0F
\end{align*}
where we used \eqref{lll} in the second step,  \eqref{llll} in the third step and $\mathbf{A}_{m} \mathbf{A}_{-m}=\mathbf{A}_{-m} \mathbf{A}_{m}+\frac{m}{2}$ for $m>0$ in the last step. The sum converges in $e^{\beta|c|}L^2(\R\times\Omega_\T)$. Combining this with \eqref{QtermT} the claim \eqref{basicT} follows upon doing the contour integral. The claim for $\bar T$ is proved in the same way.
\end{proof}

%

Let us now introduce some notation for the general correlations \eqref{TbarT}. Denote ${\bf u}=(u_1,\dots u_k)\in\D^k$. 
We define  nested contour integrals for $f:\D^k\times\D^j\to\C$  by 
\begin{align}\label{def_contour}
\oint_{|\mathbf{u}|=\boldsymbol{\delta} }\oint_{|\mathbf{v}|=\tilde{\boldsymbol{\delta}}}f ({\bf u},{\bf v})\dd {\bf \bar v}\dd {\bf u}:=\oint_{|u_k|=\delta_k}\dots \oint_{|u_1|=\delta_1} \oint_{|v_j|=\tilde{\delta}_{j}}\dots  \oint_{|v_1|=\tilde{\delta}_1}f({\bf u},{\bf v}) \dd \bar v_1\dots \dd \bar v_{j}  \dd u_1\dots \dd u_k.
\end{align}
where $\boldsymbol{\delta} :=(\delta_1,\dots, \delta_k)$ with $0<\delta_1<\dots<\delta_k<1$ and similarly for $\tilde{\boldsymbol{\delta}}$. Furthermore we always suppose $\delta_i\neq \tilde\delta_j$ for all $i,j$.
Next, for $\boldsymbol{\epsilon}\in (\R^+)^k$ and $\boldsymbol{\epsilon}'\in (\R^+)^{j}$ we set 
$T_{\boldsymbol{\epsilon}}( \mathbf{u}):=\prod_{i=1}^k T_{\epsilon_i}(u_i)$ (and similarly for the anti-holomorphic part) 
and given two Young diagrams $\nu=(\nu_i)_{1\leq i\leq k},\tilde\nu=(\tilde \nu_i)_{1\leq i\leq j},$ we denote $\mathbf{u}^{1-\nu}=\prod u_i^{1-\nu_i}$ and  $\bar{\mathbf{v}}^{1-\tilde{\nu}}=\prod \bar {v}_i^{1-\tilde{\nu}_i}$. 
 With these notations we have:

\begin{proposition}\label{FinalTbarT}
 Let $F\in\caF_{\delta}$ and 
 $\delta <\delta_1 \wedge \tilde{\delta}_1$. Then 
 \begin{align}\label{tbart11111}
(2\pi i)^{-k-j}\oint_{|\mathbf{u}|=\boldsymbol{\delta}}\oint_{|\mathbf{v}|=\tilde{\boldsymbol{\delta}}} \mathbf{u}^{1-\nu}\bar{\mathbf{v}}^{1-\tilde{\nu}} U_0\big(T({\bf u})\bar T({\bf v}) F\big)\dd {\bf \bar v}\dd {\bf u}={\bf L}_{-\nu }^0 {\bf L}_{-\tilde\nu}^0U_0 F.
\end{align}
\end{proposition}

\begin{proof}
For simplicity consider again the case with only $T$ insertions. We proceed by induction in $k$. By Lemma \ref{keepcooldontStress} the claim holds for $k=1$. Suppose it holds for $k-1$.  We use the second regularization introduced above. This entails that $T_\epsilon(u_l)\in \caF_{\delta_{l+1}} $ for $\epsilon$ small enough.  By Proposition \ref{basicstresss}  we have
\begin{align*}
\oint_{|\mathbf{u}|=\boldsymbol{\delta}} \mathbf{u}^{1-\nu} U_0\big(T({\bf u}) F\big)\dd  {\bf u}&=\lim_{\boldsymbol{\epsilon}^{(k)}\to 0} \lim_{\epsilon_k\to 0}
\oint_{|\mathbf{u}|=\boldsymbol{\delta}} \mathbf{u}^{1-\nu} U_0\big(T_{\epsilon_k}(u_k)T_{\boldsymbol{\epsilon}^{(k)}}({\bf u}^{(k)}) F\big)\dd  {\bf u}\\
&=\lim_{\boldsymbol{\epsilon}^{(k)}\to 0}
\oint_{|\mathbf{u}|=\boldsymbol{\delta}} \mathbf{u}^{1-\nu} U_0\big(T(u_k)T_{\boldsymbol{\epsilon}^{(k)}}({\bf u}^{(k)}) F\big)\dd  {\bf u}
\end{align*}
where we introduced the notation 
$\boldsymbol{\epsilon}^{(k)}=(\epsilon_1,\dots, \epsilon_{k-1})$, ${\bf u}^{(k)}=(u_1,\dots, u_{k-1})$. We have, for $\boldsymbol{\epsilon}^{(k)}$ small enough that $T_{\boldsymbol{\epsilon}^{(k)}}({\bf u}^{(k)})F\in \caF_{\delta_k}$.
Hence by 
Lemma \ref{keepcooldontStress}
\begin{align*}
\frac{1}{2\pi i}\oint_{|u_k|=\delta_k}u_k^{1-\nu_k}U_0\big(T(u_k)T_{\boldsymbol{\epsilon}^{(k)}}({\bf u}^{(k)}) F\big)du_k={\bf L}_{-\nu_k}U_0\big(T_{\boldsymbol{\epsilon}^{(k)}}({\bf u}^{(k)}) F\big) .
\end{align*}
 From \eqref{messy} we infer that $U_0\big(T_{\boldsymbol{\epsilon}^{(k)}}({\bf u}^{(k)}) F)=P_k(\boldsymbol{\epsilon}^{(k)},\varphi)U_0F$ where $P_k(\boldsymbol{\epsilon}^{(k)},\varphi)$ is a polynomial in finitely many variables $\varphi_n$ (depending on $k$)  with coefficients continuous in ${\bf u}^{(k)}$ (and depending on $\boldsymbol{\epsilon}^{(k)}$).
Therefore we may commute ${\bf L}_{-\nu_k} $ and the integration to get
\begin{equation*}
(2\pi i)^{-k} \oint_{|\mathbf{u}|=\boldsymbol{\delta}} \mathbf{u}^{1-\nu} U_0\big(T(u_k)T_{\boldsymbol{\epsilon}^{(k)}}({\bf u}^{(k)}) F\big)\dd  {\bf u}= {\bf L}_{-\nu_k} \left ( (2\pi i)^{-k+1} \oint_{|\mathbf{u}^{(k)}|=\boldsymbol{\delta}^{(k)}} (\mathbf{u}^{(k)})^{1-\nu^{(k)}} U_0\big(T_{\boldsymbol{\epsilon}^{(k)}}({\bf u}^{(k)}) F\big)\dd  {\bf u}^{(k)}   \right ).
\end{equation*}
By the induction  hypothesis the term 
\begin{equation}\label{lala}
I_{\boldsymbol{\epsilon}^{(k)}}:=(2\pi i)^{-k+1} \oint_{|\mathbf{u}^{(k)}|=\boldsymbol{\delta}^{(k)}} (\mathbf{u}^{(k)})^{1-\boldsymbol{\nu}^{(k)}} U_0\big(T_{\boldsymbol{\epsilon}^{(k)}}({\bf u}^{(k)} )F\big)\dd  {\bf u}^{(k)}  
\end{equation}
converges to $\mathbf{L}_{-\boldsymbol{\nu}^{(k)}} U_0F$ as $\boldsymbol{\epsilon}^{(k)}$ goes to $0$. 
Moreover we claim that, for all fixed $k$, there exists $M,N<\infty$ s.t. $I_{\epsilon^{(k)}}=Q_k(\boldsymbol{\epsilon}^{(k)},\varphi)U_0F$ where $Q_k(\boldsymbol{\epsilon}^{(k)},\varphi)$ is a polynomial in  $\varphi_n$ with $|n|<N$ and of degree less than $M$. Furthermore the coefficients of   $Q_k(\boldsymbol{\epsilon}^{(k)},\varphi)$ converge as $\boldsymbol{\epsilon}^{(k)}\to 0$.  Therefore we can commute  ${\bf L}_{-\nu_k}$ and the limit to get 
\begin{align*}
\lim_{\boldsymbol{\epsilon}^{(k)}\to 0} {\bf L}_{-\nu_k}I_{\epsilon^{(k)}}={\bf L}_{-\nu_k}\lim_{\boldsymbol{\epsilon}^{(k)}\to 0}I_{\epsilon^{(k)}}={\bf L}_{-\nu_k}\mathbf{L}_{-\nu^{(k)}} U_0(F)=\mathbf{L}_{-\boldsymbol{\nu}} U_0(F)
\end{align*}
which completes the induction step.  

To prove the above claim we use \eqref{messy}. Let $\caG$ be the graph with vertex set $\{1,\dots,k\}$ and edges $\{i,j\}$ with $d_{ij}>0$. Let $\Gamma$ be a connected component of $\caG$. Then the number of vertices $i$ in $\Gamma$ s.t. $b_i>0$ is no more than two. Consider a connected component for which the number is two. This corresponds to a subproduct in \eqref{messy} of the form
\begin{align}\label{moremess}
f_{\boldsymbol{\epsilon}}(u_{i_1},\dots,u_{i_l})=\partial P\varphi_{\epsilon_{i_1}}(u_{i_1})\partial P\varphi_{\epsilon_{i_l}}(u_{i_l})\prod_{a=1}^{l-1}\partial_{u_{i_a}}\partial_{u_{i_{a+1}}}(f_{\epsilon_{i_a},u_{i_a}},G_\D f_{\epsilon_{i_{a+1}},u_{i_{a+1}}})_\D.
\end{align}
We have 
\begin{align*}
\partial P\varphi_{\epsilon}(u)=\sum_{n=1}^{M_\varepsilon}a_n(\epsilon)\varphi_nu^{n-1}
\end{align*}
where $M_\epsilon<\infty$ if $\epsilon>0$ and $a_n(\epsilon)$ converges as $\epsilon\to 0$.
Similarly, for $|u|<|u'|$,
\begin{align*}
\partial_{u}\partial_{u'}(f_{\epsilon,u},G_\D f_{\epsilon',u'})_\D=\sum_{m=0}^{N_{\epsilon,\epsilon'}}b_m(\epsilon,\epsilon')u^m{u'}^{-m-2}
\end{align*}
where $N_{\epsilon,\epsilon'}<\infty$ if $\epsilon,\epsilon'>0$ and $b_n(\epsilon,\epsilon')$ converge as $\epsilon,\epsilon'\to 0$. Insert these to \eqref{moremess}. 
To simplify notation denote  $u_{i_a}=v_a$, $\epsilon_{i_a}=\varepsilon_a$, $\delta_{i_a}=d_a $ and $\nu_{i_a}=\eta_a$.
The contour integral of $f_{\boldsymbol{\epsilon}}$ becomes
\begin{align*}
\oint_{|\mathbf{v}|=\bf d} \mathbf{v}^{1-\boldsymbol{\eta}} f_{\boldsymbol{\varepsilon}}({\bf v})d{\bf v}=\sum_{n_1,n_l} \varphi_{n_1}\varphi_{n_l}a_{n_1,n_l}(\boldsymbol{\varepsilon})
\end{align*}
where 
\begin{align*}
a_{n_1,n_l}(\boldsymbol{\varepsilon})=\sum_{m_1,\dots, m_{l-1}>0} a_{n_1,n_l}({\bf m},\boldsymbol{\varepsilon})
\int_{[0,2\pi]^l} e^{i(n_1\theta_1+n_l\theta_l+\sum_{i=1}^{l-1}\alpha_i(m_i\theta_i-(m_{i}+2)\theta_{i+1})+\sum_{i=1}^l(1-\eta_i)\theta_i}\dd\boldsymbol{\theta}
\end{align*}
and $\alpha_i=\pm 1$ depending on whether $|v_a|<|v_{a+1}|$ or the opposite. The $\boldsymbol{\theta}$ integral gives the constraints $n_1+\alpha_1m_1+1-\eta_1=0$,  $n_l-\alpha_{l-1}(m_{l-1}+2)+1-\eta_l=0$ and 
$\alpha_im_i-\alpha_{i-1}(m_{i-1}+2)+(1-\eta_i)=0$ for $i=2,\dots,l-1$. The last two imply
 $n_l-\alpha_1m_1+C(\boldsymbol{\alpha},\boldsymbol{\eta})=0$
  and combining this with the first constraint we get $a_{n_1,n_l}=0$ if $|n_1|, |n_l|>C(\boldsymbol{\alpha},\boldsymbol{\eta})$, where $C(\boldsymbol{\alpha},\boldsymbol{\eta})$ denotes a generic constant  depending only on $\boldsymbol{\alpha},\boldsymbol{\eta}$. Convergence of the $a_{n_1,n_l}(\boldsymbol{\varepsilon})$ as $\boldsymbol{\varepsilon}\to 0$ then follows from the convergence of the $(a_n)_n$'s and $(b_m)$'s and the fact that $a_{n_1,n_l}(\boldsymbol{\varepsilon})$ is a polynomial of these coefficients. This finishes the proof for the case $\Gamma$ has two vertices with $b_i>0$. The two other cases are similar.
\end{proof}


In the sequel we will apply Proposition \ref{FinalTbarT} to the function 
$$F=S_{e^{-t}}U_0^{-1}\Psi^0_\alpha=e^{\alpha \int_0^{2\pi} (c+X(e^{-t+i\theta}))\frac{d\theta}{2\pi}-\alpha Q t}
$$ where $\Psi^0_\alpha(c,\varphi)=e^{(\alpha-Q)c}$ (in this case integration against a function $f \in  \caE_\delta$ is replaced by an average on a circle but the previous considerations apply also). Then $U_0F=e^{-t{\bf H}^0}\Psi^0_\alpha=e^{-2t\Delta_\alpha}\Psi^0_\alpha$. Thus we arrive to the representation for the Virasoro descendants 
\begin{align}\label{finalcontour}
\Psi^0_{\alpha,\nu,\tilde\nu}=e^{2t\Delta_\alpha}
(2\pi i)^{-k-j}\oint_{|\boldsymbol{u}|=\boldsymbol{\delta}}\oint_{|\boldsymbol{v}|=\tilde{\boldsymbol{\delta}}} \mathbf{u}^{1-\nu}\bar{\mathbf{v}}^{1-\tilde\nu} U_0\big(T({\bf u})\bar T({\bf v}) S_{e^{-t}}U_0^{-1} \Psi^0_\alpha\big)\dd{\bf \bar v} \dd{\bf u}
\end{align}
where now $e^{-t}<\delta_1\wedge \tilde{\delta}_1$.

\subsection{Conformal Ward Identities}\label{sub:ward}
  
In this section we state the main identity relating a LCFT correlation function with a $V_\alpha$ insertion to a scalar product with the descendant states $\Psi_{\alpha,\nu,\tilde\nu}$ given by \eqref{defdescendents}.
 We have 
\begin{lemma} \label{T-lemma} We have 
 \begin{align*}
e^{-t{\bf H}}
U_0\big(T({\bf u})\bar T({\bf v}) S_{e^{-s}}U_0^{-1}\Psi^0_\alpha\big)=
\lim_{{\boldsymbol{\epsilon}\to 0}}\lim_{{\boldsymbol{\epsilon}'\to 0}}e^{-t{\bf H}}
U_0\big(T_{\boldsymbol{\epsilon}}({\bf u})\bar T_{\boldsymbol{\epsilon'}}({\bf v}) S_{e^{-s}}U_0^{-1}\Psi^0_\alpha\big)
\end{align*}
 where the limit is in $e^{-\beta \rho }L^2(\R\times \Omega_\T)$ for all $ \beta >Q-\alpha$, uniformly in $({\bf u},{\bf v})\in\caO_{e^{-s}}$  (recall \eqref{caOt}) and the LHS is  analytic in ${\bf u}$  and anti-analytic in ${\bf v}$ on $\caO_{e^{-s}}$.
 \end{lemma}
 \begin{proof}
This  follows from Proposition \ref{basicstresss} and \eqref{normetHweight} applied with $ \beta >Q-\alpha$.
\end{proof}
The following lemma gives a probabilistic expression for $e^{-t{\bf H}}\Psi^0_{\alpha,\nu,\tilde\nu} $ (recall our convention for contour integrals in \eqref{def_contour}):
\begin{lemma}\label{TTLemma} Let $\delta_k\wedge \tilde\delta_{\tilde k}<e^{-t}$. Then 
\begin{align*}
 e^{-t\mathbf{H}}\Psi^0_{\alpha,\nu,\tilde\nu}
 =& \frac{e^{-(2\Delta_{\alpha}+|\nu|+|\tilde\nu|)t}}{(2\pi i)^{k+j}}
 \oint_{|\mathbf{u}|=\boldsymbol{\delta}}   \oint_{|\mathbf{v}|=\boldsymbol{\tilde\delta}}  
\mathbf{u}^{1-\nu}\bar{\mathbf{v}}^{1-\tilde\nu}  e^{-Q c} \E_\varphi\Big( T(\mathbf{u})\bar T(\mathbf{v}) V_\alpha(0) e^{-\mu e^{\gamma c}M_\gamma(\D\setminus\D_t)}\Big) \dd   \bar{\mathbf{v}}\dd   \mathbf{u}
 \end{align*}
where
\begin{align*}
 \E_\varphi\Big( T(\mathbf{u})\bar T(\mathbf{v}) V_\alpha(0) e^{-\mu e^{\gamma c}M_\gamma(\D\setminus\D_t)}\Big) :=
\lim_{{\boldsymbol{\epsilon}\to 0}}\lim_{{\boldsymbol{\epsilon}'\to 0}} \E_\varphi\Big( T_{\boldsymbol{\epsilon}}(\mathbf{u})\bar T_{\boldsymbol{\epsilon'}}(\mathbf{v}) V_\alpha(0) e^{-\mu e^{\gamma c}M_\gamma(\D\setminus\D_t)}\Big) 
 \end{align*}
and the limit exists in $e^{-\beta \rho }L^2(\R\times \Omega_\T)$ for all $\beta >Q-\alpha$ and is analytic in $\bf u$ and  anti-analytic in $\bf v$ in the region $\caO_{e^{-t}}$.
  \end{lemma}

 \begin{proof} For the sake of readability we write the proof in the case when $\tilde\nu=0$.  
Thus, consider
\begin{equation}
\Psi^0_{\alpha,\nu,0}=\frac{e^{2\Delta_{\alpha} s}}{(2\pi i)^{k}}
 \oint_{|\mathbf{u}|=\boldsymbol{\delta}}  
\mathbf{u}^{1-\nu} 
U_0\Big(  T(\mathbf{u})   S_{e^{-s}}U_0^{-1}\Psi^0_\alpha\Big) \dd   \mathbf{u} .
\end{equation}
We use the regularisation $T_\epsilon$ where \eqref{scalingrelation} holds. By Lemma \ref{T-lemma}
 \begin{align*}
 e^{-t\mathbf{H}}U_0\Big(  T(\mathbf{u})   S_{e^{-s}}U_0^{-1}\Psi^0_\alpha\Big)=&\lim_{{\boldsymbol{\epsilon}\to 0}} e^{-t\mathbf{H}}U_0\Big(  T_{\boldsymbol{\epsilon}}(\mathbf{u})   S_{e^{-s}}U_0^{-1}\Psi^0_\alpha\Big)
 \\ 
 = &\lim_{{\boldsymbol{\epsilon}\to 0}}e^{-t\mathbf{H}}U\Big(T_{\boldsymbol{\epsilon}}(\mathbf{u})(S_{e^{-s}}U_0^{-1}\Psi^0_\alpha)e^{\mu e^{\gamma c}M_\gamma(\D)}\Big)
 \\
  =& \lim_{{\boldsymbol{\epsilon}\to 0}} U\Big( S_{e^{-t}}(T_{\boldsymbol{\epsilon}}(\mathbf{u})(S_{e^{-s}}U_0^{-1}\Psi^0_\alpha)e^{\mu e^{\gamma c}M_\gamma(\D)})\Big) 
  \\
 =&\lim_{{\boldsymbol{\epsilon}\to 0}} e^{-2t} U\Big(  T_{e^{-t}\boldsymbol{\epsilon}}(e^{-t}\mathbf{u}) (S_{e^{-s-t}}U_0^{-1}\Psi^0_\alpha)S_{e^{-t}}(e^{\mu e^{\gamma c} M_\gamma(\D)})\Big) \\
 = &e^{-2t}  e^{-Q c}\lim_{{\boldsymbol{\epsilon}\to 0}} \E_\varphi\Big(  T_{e^{-t}\boldsymbol{\epsilon}}(e^{-t}\mathbf{u}) (S_{e^{-s-t}}U_0^{-1}\Psi^0_\alpha)e^{-\mu e^{\gamma c} M_\gamma(\D\setminus \D_t)})\Big)
 \\
 := &e^{-2t}  e^{-Q c} \E_\varphi\Big(  T(e^{-t}\mathbf{u}) (S_{e^{-s-t}}U_0^{-1}\Psi^0_\alpha)e^{-\mu e^{\gamma c} M_\gamma(\D\setminus \D_t)})\Big)
  \end{align*}
 where we used \eqref{scalingrelation} in the fourth identity. 
    By Lemma \ref{T-lemma} the last expression is analytic in $ \mathbf{u}$   and since $e^{2\Delta_{\alpha} s} e^{-t\mathbf{H}}U_0\Big(  T(\mathbf{u})   S_{e^{-s}}U_0^{-1}e_\alpha\Big)$ is independent on $s$
we can take the limit $s\to\infty$.  For this we note that
$$
S_{e^{-s-t}}U_0^{-1}\Psi^0_\alpha=e^{\alpha (c-(t+s)Q)}e^{\alpha (1,X( e^{-t-s}\,\cdot))_\T}=e^{-2\Delta_{\alpha}(s+t)}e^{\alpha c}e^{\alpha (1,X( e^{-t-s}\,\cdot))_\T-\hf\alpha^2\E (1,X( e^{-t-s}\,\cdot))_\T^2}
$$
so that
\begin{align*}
e^{2\Delta_{\alpha} s} \E_\varphi\Big(  T(e^{-t}\mathbf{u}) (S_{e^{-s-t}}U_0^{-1}\Psi^0_\alpha)e^{-\mu e^{\gamma c} M_\gamma(\D\setminus \D_t)})\Big)=e^{-2\Delta_{\alpha} t}\E_\varphi\Big(  T(e^{-t}\mathbf{u}) V_\alpha(0)e^{-\mu e^{\gamma c} M_\gamma(\D\setminus \D_t)})\Big)
\end{align*} 
 and the last expression is analytic in $ \mathbf{u}$.
 Hence by a change of variables in the $\bf u$-integral
 \begin{align*}
 e^{-t\mathbf{H}}\Psi^0_{\alpha,\nu,0}
 =&
  \frac{e^{2\Delta_{\alpha} s}}{(2\pi i)^{k}}
 \oint_{|\mathbf{u}|=\boldsymbol{\delta}}  
\mathbf{u}^{1-\nu}  e^{-t\mathbf{H}}
U_0\Big(  T(\mathbf{u})   S_{e^{-s}}U_0^{-1}\Psi^0_\alpha\Big) \dd   \mathbf{u} \\= & 
  \frac{e^{-(2\Delta_{\alpha}+|\nu|)t}}{(2\pi i)^{k}}
 \oint_{|\mathbf{u}|=\boldsymbol{e^{-t}\delta}}  
\mathbf{u}^{1-\nu}  e^{-Q c} \E_\varphi\Big( T(\mathbf{u}) V_\alpha(0) e^{-\mu e^{\gamma c}M_\gamma(\D\setminus\D_t)}\Big) \dd   \mathbf{u}\\
 =& 
  \frac{e^{-(2\Delta_{\alpha}+|\nu|)t}}{(2\pi i)^{k}}
 \oint_{|\mathbf{u}|=\boldsymbol{\delta}}  
\mathbf{u}^{1-\nu}  e^{-Q c} \E_\varphi\Big( T(\mathbf{u}) V_\alpha(0) e^{-\mu e^{\gamma c}M_\gamma(\D\setminus\D_t)}\Big) \dd   \mathbf{u}
 \end{align*}
where in the last step we used analyticity to move the contours to ${|\mathbf{u}|=\boldsymbol{\delta}}$. 
\end{proof}

In what follows, for fixed $n\geq 1$, we will denote
\begin{equation}\label{defZ}
 \mathcal{Z} :=\{\mathbf{z}=(z_1,\dots,z_n)\,|\,\forall i\not = j,\,\, z_i\not =z_j\text{ and }\forall i,\,\, |z_i|<1\}  .
 \end{equation}
Denoting $\theta (\mathbf{z})=(\theta(z_1),\dots,\theta(z_n))\in  \C^n$ we have $\theta\mathcal{Z}=\{(z_1,\dots,z_n)\,|\,\forall i\not = j,\,\, z_i\not =z_j\text{ and }\forall i,\,\, |z_i|>1
 \} . $

 For    $\boldsymbol{\alpha}=(\alpha_1,\dots,\alpha_n)\in\R^n$ such that $\alpha_i<Q$ for all $i$ we define the function  $U_{\boldsymbol{\alpha}}(\mathbf{z}): \R\times \Omega_\T\to \R$ by
 \begin{align}\label{defUward}
U_{\boldsymbol{\alpha}}(\mathbf{z},c,\varphi):=&\lim_{\epsilon\to 0}e^{-Qc}\E_\varphi \Big[\Big(\prod_{i=1}^nV_{\alpha_i,\epsilon}(z_i)\Big)e^{-\mu e^{\gamma c}M_\gamma(\D )} \Big], \quad \text{ for }  \mathbf{z} \in  \mathcal{Z}  
\end{align}
where $V_{\alpha_i,\epsilon}$ stands for the regularized vertex operator \eqref{Vregul}. 
Let us set 
\begin{equation}\label{defs}
s:=\sum_{i=1}^n\alpha_i.
\end{equation}

 \begin{remark}\label{correlalpha}
 It follows directly from the construction of correlation functions that for   $\mathbf{z} \in  \theta\mathcal{Z}$
 $$\langle V_{\alpha}(0)\prod_{i=1}^nV_{\alpha_i }(z_i)\rangle_{\gamma,\nu}=\Big(\prod_{i=1}^n|z_i|^{-4\Delta_{\alpha_i}}\Big)\langle\Psi_{\alpha} | U_{\boldsymbol{\alpha}}(\theta(\mathbf{z}))\rangle_2
  $$
 and these expressions are finite if $\alpha+s>2Q$ and $\alpha, \alpha_i<Q$.
 \end{remark}

 \begin{lemma}\label{integrcf}
 Let $  \mathbf{z} \in  \theta\mathcal{Z}$. Then almost everywhere in $c,\varphi$ and for all  $R>0$ 
 $$U_{\boldsymbol{\alpha}}(\theta(\mathbf{z}))(c,\varphi)\leq e^{(s-Q)(c\wedge 0)-R(c\vee 0) } A(\varphi)
 $$ where $A\in L^2 (\Omega_\T)$.
  \end{lemma}
\begin{proof}
Let $r=\max_i|\theta(z_i)|$ and $\iota(\varphi)=\inf_{x\in\D_r}P\varphi(x)$ and $\sigma(\varphi)=\sup_{x\in\D_r}P\varphi(x)$ with $\D_r$ the disk centered at $0$ with radius $r$.
Then
\begin{align*}
U_{\boldsymbol{\alpha}}(\theta(\mathbf{z}))(c,\varphi)
\leq Ce^{-Qc}e^{(c+\sigma(\varphi))s}\E e^{-\mu e^{\gamma (c+\iota(\varphi))}Z}
\end{align*}
where the expectation is over the Dirichlet GFF $X_\D$ and
\begin{align*}
Z=\int_{\D_r} (1-|z|^2)^{\frac{\gamma^2}{2}}e^{\sum_i\gamma\alpha_i G_\D(z,\theta(z_i))}M_{\gamma,\D}(\dd z )
\end{align*}
where $M_{\gamma,\D}$ is the GMC of $X_\D$. For $c<0$ we use the trivial bound  $$U_{\boldsymbol{\alpha}}(\theta(\mathbf{z}))(c,\varphi)\leq Ce^{(s-Q)c}e^{\sigma(\varphi)s}$$ and for $c>0$ we note that 
$Z$ has all negative moments so that for $a>0$
\begin{align*}
\E e^{-aZ}=\E (aZ)^{-n}(aZ)^{n}e^{-aZ}\leq n!\E (aZ)^{-n}\leq C_na^{-n}
\end{align*}
implying
\begin{align*}
U_{\boldsymbol{\alpha}}(\theta(\mathbf{z}))(c,\varphi)\leq C_ne^{c(s-Q-\gamma n)}e^{s\sigma(\varphi)-n\gamma \iota(\varphi)}
\end{align*}
Since $e^{s \sigma(\varphi)-n\iota(\varphi)}$ is in $L^2(\P)$ for all $s,n$ the claim follows.
\end{proof} 
Define now the modified Liouville expectation (with now $\D_t$ the disk centered at $0$ with radius $e^{-t}$)
\begin{align}\label{modifiedlcft}
\langle F\rangle_t=\int_\R e^{-2Qc}\E\Big[ F(c,X)e^{-\mu e^{\gamma c}M_{\gamma}(\C\setminus\D_t)}\Big]\dd c.
\end{align}
Also, in the contour integrals below,  for vectors $\boldsymbol{\delta}$, $\boldsymbol{\widetilde{ \delta}}$ defining the radii of the respective contours, we will put a subscript $t$ when these variables are multiplied by $e^{-t}$, namely $\boldsymbol{\delta}_t:=e^{-t}\boldsymbol{\delta}$ and similarly for  $\boldsymbol{\widetilde{ \delta}}_t$.
Then we get
\begin{corollary}\label{propcontour}
Let $  \mathbf{z} \in  \theta\mathcal{Z}$. For $\alpha\in\R$ such that $\alpha<\big(Q-\frac{2(|\nu|+|\tilde{\nu}|)}{\gamma}-\frac{\gamma}{4}\big)\wedge (Q-\gamma)$  and  $\alpha+\sum_i\alpha_i>2Q$, we have
\begin{align*} 
\langle&\Psi_{\alpha,\nu,{\nu'}} |  U_{\boldsymbol{\alpha}}(\theta(\mathbf{z}))\rangle_2\\
=&
\Big(\prod_{i=1}^n|z_i|^{4\Delta_{\alpha_i}}\Big)\times\frac{1}{(2\pi i)^{k+j}}\lim_{t\to \infty} \oint_{|\mathbf{u}|=\boldsymbol{\delta}_t}  
\oint_{| { \mathbf{v}}|=\boldsymbol{{\delta}'}_t}  \mathbf{u}^{1-\nu}\bar{\mathbf{v}}^{1-{\nu'}} \langle  T(\mathbf{u})\bar T({ \mathbf{v}}) V_{\alpha}(0)\prod_{i=1}^nV_{\alpha_i }(z_i) \rangle_t  \, \dd \bar{ \mathbf{v}}\dd    \mathbf{u}. 
\end{align*}
\end{corollary}
\begin{proof}
Combining Proposition \ref{descconvergence} and Lemma \ref{integrcf} the existence of the limit 
\begin{align*}
\langle\Psi_{\alpha,\nu,{\nu'}} | U_{\boldsymbol{\alpha}}(\theta(\mathbf{z}))\rangle_2 =&\lim_{t\to\infty}e^{(2\Delta_{\alpha}+|\nu|+|\nu'|)t}\langle e^{-t\mathbf{H}}\Psi^0_{\alpha,\nu,\nu'} |  U_{\boldsymbol{\alpha}}(\theta(\mathbf{z}))\rangle_2
 \end{align*}
 follows. By Lemma \ref{TTLemma} the RHS is given by the RHS of \eqref{propcontour}.
\end{proof}

Here is the main result of this section: 
 
\begin{proposition}\label{proofward}
Let $  \mathbf{z} \in  \theta\mathcal{Z}$. For $\alpha\in\R$ such that $\alpha<\big(Q-\frac{2(|\nu|+|\tilde{\nu}|)}{\gamma}-\frac{\gamma}{4}\big)\wedge (Q-\gamma)$  and  $\alpha+\sum_i\alpha_i>2Q$ and for all $i$ $\alpha_i<Q$, we have  in the distributional sense 
$$\langle\Psi_{\alpha,\nu,\tilde{\nu}} | U_{\boldsymbol{\alpha}}(\theta(\mathbf{z}))\rangle_2=\Big(\prod_{i=1}^n|z_i|^{4\Delta_{\alpha_i}}\Big)\times
\mathbf{D}_{\nu}\tilde{\mathbf{D}}_{{\tilde{\nu}}}\langle V_{\alpha}(0)\prod_{i=1}^nV_{\alpha_i }(z_i)\rangle_{\gamma,\mu}   $$
where the differential operators $\mathbf{D}_{\nu}$, $\tilde{\mathbf{D}}_{\tilde{\nu}}$ are defined by 
\begin{equation}
\mathbf{D}_{\nu}= \mathbf{D}_{\nu_k}\dots \mathbf{D}_{\nu_1}\quad \text{ and } 
\quad \tilde{\mathbf{D}}_{\tilde{\nu}}= \tilde{\mathbf{D}}_{\tilde{\nu}_j}\dots \tilde{\mathbf{D}}_{\tilde{\nu}_1}
\end{equation}
where for $n\in\N$
\begin{align}
\mathbf{D}_{n}=&\sum_{i=1}^n\Big(-\frac{1}{z_i^{n-1}}\partial_{z_i}+\frac{(n-1) }{z_i^n}\Delta_{\alpha_i}\Big)\\
\tilde{\mathbf{D}}_{n}=&\sum_{i=1}^n\Big(-\frac{1}{\bar z_i^{n-1}}\partial_{\bar z_i}+\frac{(n-1) }{\bar z_i^n}\Delta_{\alpha_i}\Big)
\end{align}
 \end{proposition}
 
\begin{proof}    Section \ref{subproofward} will be devoted to the proof of this proposition.
\end{proof}
  
   \subsection{Computing the 3-point correlation functions of descendant states}\label{sub:desc3point}
Now we exploit Proposition \ref{proofward} to give \emph{exact} analytic expressions for  the 3-point correlation functions of descendant states. For this, we first need to introduce some notation: for $\nu= (\nu_i)_{i \in \llbracket 1,k\rrbracket}$ a Young diagram and some real $\Delta,\Delta', \Delta''$ we set
\begin{equation}\label{poidsv}
v(\Delta,\Delta', \Delta'', \nu):=  \prod_{j=1}^k  (\nu_j \Delta'-\Delta+\Delta''+ \sum_{u<j} \nu_u )  .
\end{equation}
With this notation, we can state the following key result: 

\begin{proposition}[Conformal Ward identities]\label{Ward}
Assume $\alpha_1,\alpha_2<Q$ with $\alpha_1+\alpha_2>Q$ and $|z|<1$.  For all $P>0$, 
 the scalar product $\langle \Psi_{Q+iP, \nu,\tilde\nu}\, |\, U_{\alpha_1,\alpha_2}(0,z) \rangle_2$ is explicitly given by the following expression
 \begin{align}
\langle \Psi_{Q+iP, \nu,\tilde\nu}\, |\, U_{\alpha_1,\alpha_2}(0,z) \rangle_2 &=v(\Delta_{\alpha_1}, \Delta_{\alpha_2}, \Delta_{Q+iP},\tilde{\nu}) v(\Delta_{\alpha_1}, \Delta_{\alpha_2}, \Delta_{Q+iP},\nu)\\
& \times \frac{1}{2} C^{ \mathrm{DOZZ}}_{\gamma,\mu}( \alpha_1,\alpha_2, Q+iP  ) \bar{z}^{|\nu|} z^{|\tilde{\nu}|}  |z|^{2 (\Delta_{Q+iP}-\Delta_{\alpha_1}-\Delta_{\alpha_2})}   \nonumber
 \end{align}
where $\Delta_{\alpha}$ are conformal weights \eqref{deltaalphadef} and $ C^{ \mathrm{DOZZ}}_{\gamma,\mu}( \alpha_1,\alpha_2, Q+iP  )$ are the DOZZ structure constants (see appendix \ref{dozz} for the definition).
 \end{proposition}

The remaining part of this subsection is devoted to the proof of this proposition. 
We first need the following lemma concerning analycity of $U_{\boldsymbol{\alpha}}(\mathbf{z})$ in the parameter $\boldsymbol{\alpha}\in\C^n$, proved in Appendix \ref{app:analytic}.

\begin{lemma}\label{lem:analytic}
For fixed $\mathbf{z}\in \mathcal{Z}$ the mapping $\boldsymbol{\alpha}\mapsto U_{\boldsymbol{\alpha}}(\mathbf{z})\in e^{ \beta'\rho}L^2(\R\times\Omega_\T)$ extends analytically to a complex neighborhood $\mathcal{A}^n_U$ in $\C^n$ of the set $\{\boldsymbol{\alpha}\in\R^n\mid \forall i, \alpha_i<Q\}$, for arbitrary $\beta'<{\rm Re}(s)-Q$ (recall \eqref{defs}). This analytic extension is continuous in $(\boldsymbol{\alpha},\boldsymbol{z})\in \mathcal{A}^n_U\times\mathcal{Z}$.
\end{lemma}

The first conclusion we want to draw is the fact that the pairing of $\Psi_\alpha$ with  $U_{\boldsymbol{\alpha}}(\theta(\mathbf{z})) $ (in the case $n=2$) is related to the DOZZ formula when $\alpha$ is on the spectrum line $Q+i\R$.

\begin{lemma}
Here we fix $n=2$ and we consider $\mathbf{z}\in \theta\mathcal{Z}$. The mapping $(\alpha,\boldsymbol{\alpha})\mapsto \langle \Psi_\alpha | U_{\boldsymbol{\alpha}}(\theta(\mathbf{z}))\rangle_2$ is 
continuous in the set
$$\Xi:= \{(\alpha,\boldsymbol{\alpha})\in \C \times \mathcal{A}^2_U\mid   {\rm Re}(\alpha)\leq Q,{\rm Re}(\alpha+\alpha_1+\alpha_2)>2Q\}$$
 and analytic in the set $\Xi\cap \{\alpha\in \C\setminus\mathcal{D}_0\}$. Moreover, in  the set $\Xi$, we have the relation
$$\langle \Psi_\alpha | U_{\boldsymbol{\alpha}}(\theta(\mathbf{z}))\rangle_2= |z_1|^{2\Delta_{\alpha_2}-2\Delta_\alpha+2\Delta_{\alpha_1}}|z_1-z_2|^{2\Delta_{\alpha}-2\Delta_{\alpha_1}-2\Delta_{\alpha_2}}|z_2|^{2\Delta_{\alpha_1}-2\Delta_\alpha+2\Delta_{\alpha_2}}\tfrac{1}{2}C_{\gamma,\mu}^{{\rm DOZZ}} (\alpha,\alpha_1,\alpha_2 ) .$$
In particular this relation holds for $\alpha=Q+iP$ with $P\in (0,+\infty)$ and $\alpha_1,\alpha_2\in (-\infty,Q)$ with $\alpha_1+\alpha_2>Q$. 
\end{lemma}
   
\begin{proof}
By Proposition \ref{holomorphiceig}  
 the mapping $\alpha\mapsto \Psi_\alpha\in e^{-\frac{\beta}{2}\rho}L^2(\R\times\Omega_\T)$ is analytic on $W_0$, i.e.   in the region $\{\alpha\in\C\setminus \mathcal{D}_0 \mid Q-{\rm Re}(\alpha)<\beta \}$ and continuous on $\mathcal{D}_0$. Combining with Lemma \ref{lem:analytic} (with $n=2$) produces directly the region of analycity/continuity we claim.
 Furthermore when all the parameters $\alpha,\alpha_1,\alpha_2$ are real, by Remark \ref{correlalpha} we  have  
  $$\langle V_{\alpha}(0) V_{\alpha_1 }(z_1)V_{\alpha_2 }(z_2)\rangle_{\gamma,\nu}=\Big(\prod_{i=1}^2|z_i|^{-4\Delta_{\alpha_i}}\Big)\langle\Psi_{\alpha} | U_{\boldsymbol{\alpha}}(\theta(\mathbf{z}))\rangle_2.
$$  
Also, for real parameters, the LHS coincide with the DOZZ formula \cite{dozz}, namely
 $$\langle V_{\alpha}(0) V_{\alpha_1 }(z_1)V_{\alpha_2 }(z_2)\rangle_{\gamma,\nu}=|z_1|^{2\Delta_{\alpha_2}-2\Delta_\alpha-2\Delta_{\alpha_1}}|z_1-z_2|^{2\Delta_{\alpha}-2\Delta_{\alpha_1}-2\Delta_{\alpha_2}}|z_2|^{2\Delta_{\alpha_1}-2\Delta_\alpha-2\Delta_{\alpha_2}}\tfrac{1}{2}C_{\gamma,\mu}^{{\rm DOZZ}} (\alpha,\alpha_1,\alpha_2 ) .$$
This proves the claim.
\end{proof}

Now we would like to use the Ward identities, i.e. Proposition \ref{proofward}, to express the correlations of descendant fields with two insertions $\langle \Psi_{\alpha,\nu,\tilde{\nu}} | U_{\boldsymbol{\alpha}}(\theta(\mathbf{z}))\rangle_2$ (here with $n=2$) in terms of differential operators applied to to correlation of primaries $\langle \Psi_\alpha | U_{\boldsymbol{\alpha}}(\theta(\mathbf{z}))\rangle_2$ when the parameter $\alpha$ is close to the spectrum line $\alpha\in Q+i\R$. This is not straightforward because Proposition \ref{proofward} is not only restricted to real values of the parameter but also because the constraint on $\alpha$, which forces it to be negatively large, implies to have $n$ large  in order for the global Seiberg bound $\alpha+\sum_i\alpha_i>2Q$ to be satisfied. Transferring Ward's relations close to the   spectrum line is thus our next task.

For this, fix a pair of Young diagrams $\nu,\tilde\nu$ and $\ell $ such that $\lambda_\ell=|\nu|+|\tilde\nu|$. By Proposition   \ref{holomorphiceig} there exists a  connected open set $\mathcal{A}_{\nu,\tilde{\nu}}:=W_0\cap W_\ell\subset\C$ such that the mappings $\alpha\mapsto \Psi_{\alpha,\nu,\tilde{\nu}}\in e^{-\beta\rho}L^2(\R\times\Omega_\T)$ and $\alpha\mapsto \Psi_{\alpha}\in e^{-\beta\rho}L^2(\R\times\Omega_\T)$  for $\beta>Q- \mathrm{Re}(\alpha)$ are analytic on $\mathcal{A}_{\nu,\tilde{\nu}}$ and furthermore

\begin{itemize}
\item $\mathcal{A}_{\nu,\tilde{\nu}}$ contains a complex neighborhood of the spectrum line $ \{\alpha=Q+iP\mid P\in(0,+\infty)\}$, with the discrete set $\mathcal{D}_0\cup\mathcal{D}_\ell$ removed and the mappings extend continuously to 
 $\mathcal{D}_0\cup\mathcal{D}_\ell$.
\item  $\mathcal{A}_{\nu,\tilde{\nu}}$ contains a complex neighborhood of   the real half-line $(-\infty,Q-\frac{2\la_\ell}{\gamma}-\frac{\gamma}{4})$  
\end{itemize}
Therefore, for arbitrary fixed $n$ and $\mathbf{z}\in\theta\mathcal{Z}$,  the pairings $(\alpha,\boldsymbol{\alpha})\mapsto \langle\Psi_{\alpha,\nu,\tilde{\nu}} |U_{\boldsymbol{\alpha}}(\theta(\mathbf{z}))\rangle_2$ and $(\alpha,\boldsymbol{\alpha})\mapsto \langle\Psi_{\alpha} |U_{\boldsymbol{\alpha}}(\theta(\mathbf{z}))\rangle_2$ are holomorphic in the region $\mathcal{A}_{\nu,\tilde{\nu}}\odot\mathcal{A}^n_U:=\{(\alpha,\boldsymbol{\alpha})\in \mathcal{A}_{\nu,\tilde{\nu}}\times\mathcal{A}^n_U\mid 2Q<{\rm Re}(s+\alpha)\}$. 
Let us consider the subsets  
 $$\mathcal{S}:=\{(\alpha,\boldsymbol{\alpha}) \mid \forall i \,\,\alpha_i\in\R\text{ and } \alpha_i<Q,s-Q>0,\alpha\in Q+i (0,\infty)  \},\quad \mathcal{S}_\star=\mathcal{S}\cap\{(\alpha,\boldsymbol{\alpha}) \mid\alpha\not \in \mathcal{D}_0\cup\mathcal{D}_\ell\}$$ 
 and 
 $$\mathcal{R}_{\nu,\tilde{\nu}}:=\{(\alpha,\boldsymbol{\alpha}) \mid \forall i \,\,\alpha_i\in\R\text{ and } \alpha_i<Q,\alpha+s-2Q>0,\alpha\in \R,\alpha< (Q-\frac{2\la_\ell}{\gamma}-\frac{\gamma}{4})\wedge (Q-\gamma) \}.$$ 
They are both in the same connected component of $\mathcal{A}_{\nu,\tilde{\nu}}\odot\mathcal{A}^n_U$. Furthermore, $ \mathcal{S}_\star$ is obviously non-empty whereas the condition $(n-2)Q>-(Q-\frac{2\la_\ell}{\gamma}-\frac{\gamma}{4})\wedge (Q-\gamma)$ ensures that $\mathcal{R}_{\nu,\tilde{\nu}}$ is non-empty, which we assume from now on.

   Now we  exploit the Ward identities, valid on $\mathcal{R}_{\nu,\tilde{\nu}}$.  Let us consider a smooth compactly supported function $\varphi$ on $\theta\mathcal{Z}$. The mapping 
 $$(\alpha,\boldsymbol{\alpha})\in\mathcal{A}_{\nu,\tilde{\nu}}\odot\mathcal{A}_U^n\mapsto \int \langle\Psi_{\alpha,\nu,\tilde{\nu}} |U_{\boldsymbol{\alpha}}(\theta(\mathbf{z})) \rangle_2\bar{\varphi}(\mathbf{z})\,\dd \mathbf{z}$$
 is thus analytic. Furthermore on $\mathcal{R}_{\nu,\tilde{\nu}}$ and by Proposition \ref{proofward}, it coincides with the mapping  
 \begin{align*}
 (\alpha,\boldsymbol{\alpha})\in\mathcal{A}_{\nu,\tilde{\nu}}\odot\mathcal{A}_U^n
 \mapsto & 
 \int  \langle V_{\alpha}(0)\prod_{i=1}^nV_{\alpha_i }(z_i)\rangle_{\gamma,\mu}  \overline{\tilde{\mathbf{D}}_{{\tilde{\nu}}}^\ast \mathbf{D}_{\nu}^\ast  \Big( \varphi(\mathbf{z})  \prod_{i=1}^n|z_i|^{4\Delta_{\alpha_i}}\Big)}\,\dd \mathbf{z}\\
 =& \int \Big(\prod_{i=1}^n|z_i|^{-4\Delta_{\alpha_i}}\Big)\langle\Psi_{\alpha} | U_{\boldsymbol{\alpha}}(\theta(\mathbf{z}))\rangle_2  \overline{\tilde{\mathbf{D}}_{{\tilde{\nu}}}^\ast \mathbf{D}_{\nu}^\ast  \Big( \varphi(\mathbf{z})  \prod_{i=1}^n|z_i|^{4\Delta_{\alpha_i}}\Big)}\,\dd \mathbf{z}
 \end{align*}
  where  we have  introduced the (adjoint) operator $\mathbf{D}_\nu^\ast$ by
 $$\int_{\C^n} \mathbf{D}_\nu f(\mathbf{z})\bar{\varphi}(z)\,\dd \mathbf{z}=\int_{\C^n}  f(\mathbf{z})\overline{\mathbf{D}_\nu^\ast \varphi}(z)\,\dd \mathbf{z}$$
 for all functions $f$ in the domain of $\mathbf{D}_\nu$ and all smooth compactly supported functions $ \varphi$ in $\C^n$ (and similarly for $\tilde{\mathbf{D}}_{\tilde{\nu}}$). Therefore both mappings are analytic and coincide on $\mathcal{R}_{\nu,\tilde{\nu}}$, thus on the connected component of $\mathcal{A}_{\nu,\tilde{\nu}}\odot\mathcal{A}_U$ containing $\mathcal{R}_{\nu,\tilde{\nu}}$, therefore  on $\mathcal{S}_\star$ and finally on $\mathcal{S}$ by continuity. Notice that, on $\mathcal{S}$, we can take all the $\alpha_i$'s equal to $0$ but the first two of them provided they satisfy $\alpha_1+\alpha_2>Q$. This fact being valid for all test function $\varphi$, we deduce that the relation 
\begin{align*}
\langle \Psi_{Q+iP,\nu,\tilde{\nu}} |U(V_{\alpha_1}&(\theta(z_1))V_{\alpha_2}(\theta(z_2)))\rangle_2\\
=&\tfrac{1}{2}C_{\gamma,\mu}^{{\rm DOZZ}} (Q+iP,\alpha_1,\alpha_2 )|z_1|^{4\Delta_{\alpha_1}}|z_2|^{4\Delta_{\alpha_2}}\\
&\times \mathbf{D}_{\nu}  \tilde{\mathbf{D}}_{{\tilde{\nu}}}  \Big( |z_1|^{2\Delta_{\alpha_2}-2\Delta_{Q+iP}-2\Delta_{\alpha_1}}|z_1-z_2|^{2\Delta_{Q+iP}-2\Delta_{\alpha_1}-2\Delta_{\alpha_2}}|z_2|^{2\Delta_{\alpha_1}-2\Delta_{Q+iP}-2\Delta_{\alpha_2}}\Big)
\end{align*}
holds for almost every $z_1,z_2$ and $\alpha_1,\alpha_2<Q$ such that $\alpha_1+\alpha_2>Q$, and thus for every $z_1,z_2\in\theta\mathcal{Z}$ as both sides are continuous in these variables. From this relation and after some elementary algebra to compute the last term, sending $z_2\to\infty$, we end up with the claimed relation.\qed


\section{Proof of Theorem \ref{bootstraptheoremintro}}\label{sectionproofbootstrap}

\subsection{Definition of Conformal Blocks}\label{sub:defblock}
Before proving Theorem \ref{bootstraptheoremintro}, we give the definition of the conformal blocks \eqref{blocksintro} which is based on material introduced in subsections \ref{sectionDiagvirasoro}  and 
\ref{sub:desc3point}. 
The conformal blocks are defined as the formal power series 
\begin{equation}\label{defblocks}
 \mathcal{F}_{Q+iP}(\Delta_{\alpha_1}, \Delta_{\alpha_2},\Delta_{\alpha_3},\Delta_{\alpha_4}z)= \sum_{n=0}^\infty  \beta_n(\Delta_{Q+iP},\Delta_{\alpha_1}, \Delta_{\alpha_2},\Delta_{\alpha_3},\Delta_{\alpha_4})z^n
 \end{equation}
with
\begin{equation}\label{expressionbeta}
\beta_n(\Delta_{Q+iP},\Delta_{\alpha_1}, \Delta_{\alpha_2},\Delta_{\alpha_3},\Delta_{\alpha_4}):= \sum_{|\nu|, |\nu'|=n}  v(\Delta_{\alpha_1}, \Delta_{\alpha_2}, \Delta_{Q+iP},\nu) F_{Q+iP}^{-1} (\nu,\nu') v(\Delta_{\alpha_4},\Delta_{\alpha_3}, \Delta_{Q+iP}, \nu').
\end{equation}
where the matrix $(F_{Q+iP}^{-1} (\nu,\nu'))_{|\nu|,|\nu'|=n}$  is the inverse of the scalar product matrix \eqref{scapo} and 
the function $v$ is explicitely defined by expression \eqref{poidsv}.
The definition of $\beta_n$ is not explicit because there is no known explicit formula for the inverse matrix  $F_{Q+iP}^{-1}$ and the convergence of  \eqref{defblocks} is an open problem. Below we'll prove the series converges  in the unit disc a.e. in $P$.

\subsection{Proof of Theorem \ref{bootstraptheoremintro}}
We start with a Lemma on the spectral decompostion:
\begin{lemma}\label{spclemma}
Let $u_1,u_2\in e^{\delta c_-}L^2(\R\times \Omega_\T)$ with $\delta>0$. 
Then
\begin{align*}
\langle u_1 \,|\, u_2\rangle_2=\frac{1}{2\pi}
 \lim_{N,L\to \infty}
 \sum_{|\nu'|=|\nu|\leq N} \sum_{|\tilde\nu'|=|\tilde\nu|\leq N}
\int_0^L \langle   u_1\,|\, \Psi_{Q+iP,\nu ,\tilde{\nu}}\rangle  _{2}\langle \Psi_{Q+iP,\nu' , \tilde{\nu}'}\,|\, u_2\rangle_{2}F^{-1}_{Q+iP}(\nu,\nu')F^{-1}_{Q+iP}(\tilde\nu,\tilde\nu')\, \dd P.
 \end{align*} 
\end{lemma}

\proof 
We write the spectral representation  \eqref{spectraldecomposH2} as
\begin{equation}\label{keyidentity}
\begin{split}
\cjg u_1\,|\,u_2\cjd_2 
=&\lim_{N,L\to\infty} \frac{1}{2\pi}\sum_{ |\k|+|\l|\leq N}  \int_{0}^L     
\cjg u_1\,|\,\Psi_{Q+iP,\k,\l}\cjd_2   \cjg \Psi_{Q+iP,\k,\l}\,|\,u_2\cjd_{2} {\rm d}P.
\end{split}
\end{equation}
Let 
  $F_{Q+iP}^{-1/2}(\nu,\nu')$ 
be the square root of the positive definite matrix
$(F_{Q+iP}^{-1}(\nu,\nu'))_{|\nu|=|\nu'|=N}$ and set
 for $|\nu|=N$, $|\tilde\nu|=N'$ 
\begin{equation*}
H_{Q+iP,\nu,\tilde{\nu}}: =   \sum_{|\nu_1|=N, |\nu_2|= N'}  F_{Q+iP}^{-1/2}(\nu, \nu_1)F_{Q+iP}^{-1/2}(\tilde{\nu}, \nu_2) \Psi_{Q+iP,\nu_1,\nu_2} .
\end{equation*}
where $ \Psi_{Q+iP,\nu_1,\nu_2} $ are defined in
 \eqref{defdescendents}.  We get
 the identity
\begin{align*}
\sum_{ |\k|+|\l |=N }      
\cjg u_1\,|\,\Psi_{Q+iP,\k,\l}\cjd_2   \cjg \Psi_{Q+iP,\k,\l}\,|\,u_2\cjd_{2}= \sum_{|\nu| + |\tilde \nu| = N}  \cjg u_1 \,| \, H_{Q+iP,\nu,\tilde{\nu}} \cjd_2 \cjg  \, H_{Q+iP,\nu,\tilde{\nu}}  \,| u_2\cjd_2 .
\end{align*}
Hence we have 
\begin{align}\label{equalityyy}  \sum_{|\nu| + |\tilde \nu| \leq N} \int_0^L   \cjg u_1 \,| \, H_{Q+iP,\nu,\tilde{\nu}} \cjd_2 \cjg  \, H_{Q+iP,\nu,\tilde{\nu}}  \,| u_2\cjd_2   {\rm d}P= \sum_{|\k|+ |\l| \leq N}  \int_0^L  \cjg u_1 \,| \, \Psi_{Q+iP,\k,\l} \cjd_2  \cjg  \Psi_{Q+iP,\k,\l}   \,|\, u_2\cjd_2    {\rm d}P .
\end{align}
On the other hand 
\begin{align}
 &  \sum_{|\nu|,|\tilde \nu| \leq N}\int_0^L  \cjg u_1 \,| \, H_{Q+iP,\nu,\tilde{\nu}} \cjd_2 \cjg  \, H_{Q+iP,\nu,\tilde{\nu}}  \,|\, u_2\cjd_2   {\rm d}P \label{casej,j'leqN}  \\
& =  \sum_{|\nu_1|=|\nu_3|\leq N}\sum_{  |\nu_2|=|\nu_4|\leq N} \int_0^L  \cjg u_1 \,| \, \Psi_{Q+iP,\nu_1,\nu_2} \cjd_2 \cjg  \, \Psi_{Q+iP,\nu_3, \nu_4}  \,|\, u_2\cjd_2 F_{Q+iP}^{-1}(\nu_1, \nu_3) F_{Q+iP}^{-1}(\nu_2, \nu_4)   {\rm d}P.\nonumber
\end{align}
Let $u_1=u_2=u$. Then
\begin{equation}\label{threeterms}
  \int_0^L \sum_{|\nu|+ |\tilde \nu| \leq N} | \cjg u \,| \, H_{Q+iP,\nu,\tilde{\nu}} \cjd_2 |^2   {\rm d}P \leq   \int_0^L \sum_{\substack{|\nu|\leq N, \\ |\tilde \nu| \leq N}} | \cjg u \,| \, H_{Q+iP,\nu,\tilde{\nu}} \cjd_2 |^2   {\rm d}P \leq   \int_0^L \sum_{|\nu|+ |\tilde \nu| \leq 2N} | \cjg u \,| \, H_{Q+iP,\nu,\tilde{\nu}} \cjd_2 |^2   {\rm d}P.
\end{equation}
By \eqref{equalityyy} and \eqref{keyidentity} the limits  as $L,N\to \infty$ of all the three sums in \eqref{threeterms} are $2\pi \cjg u\,|\,u\cjd_2$. Hence by polarisation the  limits  as $L,N\to \infty$ of the left hand sides of \eqref{casej,j'leqN}  and \eqref{equalityyy} exist and equal $2\pi \cjg u_1\,|\,u_2\cjd_2$. This proves the claim. \qed

Let  $|z|<1$, $t\in (0,1)$ and $ |u|>1$ with $u\neq t^{-1}$ .The $4$-point correlation function is given as a scalar product
\begin{equation}\label{scalarproduct}
\langle  V_{\alpha_1}(0)  V_{\alpha_2}(z) V_{\alpha_3}(t^{-1}) V_{\alpha_4}(u)  \rangle_{\gamma,\mu}  =   t^{4 \Delta_{\alpha_3}} |u|^{-4 \Delta_{\alpha_4}}  \Big\langle  U_{\alpha_1, \alpha_2}(0,z) \,|\, U_{\alpha_4, \alpha_3}\big({\bar{u}}^{-1},t\big)   \Big\rangle_2
\end{equation}
where  $U_{\alpha_1, \alpha_2}$ is defined in  \eqref{defUward}.
As in \eqref{4pointinfty} we have
\begin{align*}
 \langle  V_{\alpha_1}(0)  V_{\alpha_2}(z) V_{\alpha_3}(t^{-1}) V_{\alpha_4}(\infty)  \rangle_{\gamma,\mu} 
 = t^{4 \Delta_{\alpha_3}} \Big\langle  U_{\alpha_1, \alpha_2}(0,z\big) \,|\, U_{\alpha_4, \alpha_3}\big(0,t\big)   \Big\rangle_2.
\end{align*}
 Recall from Lemma \ref{integrcf} that for $|z|<1$ the vector $U_{\alpha_1, \alpha_2}(0,z) \in e^{\delta c_-}L^2(\R\times \Omega_\T)$ for some $\delta>0$  if $\alpha_1+\alpha_2>Q$, and the same holds for $U_{\alpha_4, \alpha_3}(0,t)$ if $\alpha_3+\alpha_4>Q$.  
Using Lemma \ref{spclemma} we get
\begin{align}\label{doublelimite}
 \Big\langle  U_{\alpha_1, \alpha_2}(0,z\big) \,|\, U_{\alpha_4, \alpha_3}\big(0,t\big)   \Big\rangle_2=\frac{1}{2\pi} \lim_{N,L\to \infty} I_{N,L}(\boldsymbol{\alpha},z,t)
 \end{align} 
 where
 \begin{align}\label{doublelimite1}
I_{N,L}(\boldsymbol{\alpha},z,t)
= \sum_{\substack{
 |\nu'|=|\nu|\leq N,\\
 |\tilde{\nu}'|=|\tilde{\nu}|\leq N}}\int_0^L \langle    U_{\alpha_1,\alpha_2}(0,z)\,|\, \Psi_{Q+iP,\nu ,\tilde{\nu}}\rangle  _{2}\langle \Psi_{Q+iP,\nu' , \tilde{\nu}'}\,|\, U_{\alpha_4,\alpha_3}(0,t)\rangle_{2}F^{-1}_{Q+iP}(\nu,\nu')F^{-1}_{Q+iP}(\tilde\nu,\tilde\nu')\, \dd P.
 \end{align} 
Using Proposition \ref{Ward} for the  scalar products we can write this as
\begin{align*}
I_{N,L}&(\boldsymbol{\alpha},z,t)
=  \tfrac{1}{4} t^{4 \Delta_{\alpha_3}} t^{-2 \Delta_{\alpha_3}-2 \Delta_{\alpha_4}   }|z|^{-2\Delta_{\alpha_1}-2\Delta_{\alpha_2}}
\int_{0}^L  \bbar{ C_{\gamma,\mu}^{ \mathrm{DOZZ}}( \alpha_1,\alpha_2, Q+iP)}C_{\gamma,\mu}^{ \mathrm{DOZZ}}( \alpha_4,\alpha_3, Q+iP)
 \\&
  |tz|^{2\Delta_{Q+iP} }\sum_{ |\nu'|=|\nu|\leq N}
 z^{|\nu|}
t^{ |\nu'|}v(\Delta_{\alpha_1}, \Delta_{\alpha_2}, \Delta_{Q+iP},\nu) F_{Q+iP}^{-1}(\nu,\nu')v(\Delta_{\alpha_4}, \Delta_{\alpha_3}, \Delta_{Q+iP},\nu')\\&
 \sum_{
 |\tilde{\nu}'|=|\tilde{\nu}|\leq N}
\bar{z}^{|\tilde \nu|}t^{|\tilde{\nu'}|}v(\Delta_{\alpha_1}, \Delta_{\alpha_2}, \Delta_{Q+iP},\tilde{\nu}) 
F_{Q+iP}^{-1}(\tilde{\nu}, \tilde{\nu}')  v(\Delta_{\alpha_4}, \Delta_{\alpha_3}, \Delta_{Q+iP},\tilde{\nu}') 
dP\\&=\tfrac{1}{4} t^{2 \Delta_{\alpha_3}-2 \Delta_{\alpha_4}   }|z|^{-2\Delta_{\alpha_1}-2\Delta_{\alpha_2}}
\int_{0}^L  \bbar{ C_{\gamma,\mu}^{ \mathrm{DOZZ}}( \alpha_1,\alpha_2, Q+iP)}C_{\gamma,\mu}^{ \mathrm{DOZZ}}( \alpha_4,\alpha_3, Q+iP) |tz|^{2\Delta_{Q+iP} }\\& \left |  \sum_{n=1}^N \beta_n(\Delta_{Q+iP},\Delta_{\alpha_1}, \Delta_{\alpha_2},\Delta_{\alpha_3},\Delta_{\alpha_4}) (zt)^{n}    \right |^2 dP
\end{align*}
where we noted that the $v$ factors are real. The coefficients $\beta_n$ are given by \eqref{defblocks} so
 formally  the bootstrap formula follows now by taking $N,L\to\infty$. Yet, rigorously, there is still a gap to bridge since we do not know the convergence of the series \eqref{defblocks}. We adapt here a Cauchy-Schwarz type argument \footnote{Communicated to us by S. Rychkov.} to establish the convergence  a.s. in $P$.
For this, we take first $\alpha_3=\alpha_2$ and $\alpha_1=\alpha_4$, with $\alpha_1+\alpha_2>Q$. 
Then 
\begin{align*}
I_{N,L}&(\alpha_1,\alpha_2,,\alpha_2,\alpha_1,z,t)
=  \frac{1}{4} t^{4 \Delta_{\alpha_2}}|tz|^{-2\Delta_{\alpha_1}-2\Delta_{\alpha_2}}\\&
\int_{0}^L   |C_{\gamma,\mu}^{ \mathrm{DOZZ}}( \alpha_1,\alpha_2, Q-iP  )|^2  |tz|^{2 \Delta_{Q+iP}} \left |  \sum_{n=0}^N \beta_n(\Delta_{Q+iP},\Delta_{\alpha_1}, \Delta_{\alpha_2},\Delta_{\alpha_2},\Delta_{\alpha_1}  ) (zt)^{n}    \right |^2 dP
\end{align*}
If $z\in(0,1)$ thi expression   is increasing in the variables $N,L$ 
and $zt$ since $\beta_n\geq 0$; moreover, it converges as $N,L$ goes to infinity. This implies that the series 
\[ \sum_{n=0}^\infty z^n \beta_n(\Delta_{Q+iP},\Delta_{\alpha_1}, \Delta_{\alpha_2},\Delta_{\alpha_2},\Delta_{\alpha_1})\] 
is absolutely converging for $|z|<1$ for almost all $P\geq 0$ and 
\begin{align*}
& \langle  V_{\alpha_1}(0)  V_{\alpha_2}(z) V_{\alpha_2}(t^{-1}) V_{\alpha_1}(\infty)  \rangle_{\gamma,\mu}= \\
&\frac{1}{8\pi} t^{4 \Delta_{\alpha_2}}|tz|^{-2\Delta_{\alpha_1}-2\Delta_{\alpha_2}} \int_{0}^\infty   |C_{\gamma,\mu}^{ \mathrm{DOZZ}}( \alpha_1,\alpha_2, Q-iP  )|^2 |tz|^{2 \Delta_{Q+iP}} \left |  \sum_{n=0}^\infty \beta_n(\Delta_{Q+iP},\Delta_{\alpha_1}, \Delta_{\alpha_2},\Delta_{\alpha_2},\Delta_{\alpha_1}  ) (zt)^{n}    \right |^2 dP.
\end{align*}
This leads to the formula \eqref{bootstrapidentityintro} for the case $\alpha_3=\alpha_2$ and $\alpha_1=\alpha_4$  by using the identity 
\begin{equation*}
\langle  V_{\alpha_1}(0)  V_{\alpha_2}(tz) V_{\alpha_2}(1) V_{\alpha_1}(\infty)  \rangle_{\gamma,\mu}= t^{-4 \Delta_{\alpha_2}} \langle  V_{\alpha_1}(0)  V_{\alpha_2}(z) V_{\alpha_2}(t^{-1}) V_{\alpha_1}(\infty)  \rangle_{\gamma,\mu}.  
\end{equation*}

For the general case  we use
Cauchy-Schwartz
\[
\begin{split}
|\beta_n(\Delta_{Q+iP}P,\Delta_{\alpha_1}, \Delta_{\alpha_2},\Delta_{\alpha_3},\Delta_{\alpha_4}  ) |
& \leq {\beta_n(\Delta_{Q+iP},\Delta_{\alpha_1}, \Delta_{\alpha_2},\Delta_{\alpha_2},\Delta_{\alpha_1}  )}^\hf {\beta_n(\Delta_{Q+iP},\Delta_{\alpha_4}, \Delta_{\alpha_3},\Delta_{\alpha_3},\Delta_{\alpha_4}  )}^\hf    \\
& \leq \frac{1}{2}  (\beta_n(\Delta_{Q+iP},\Delta_{\alpha_1}, \Delta_{\alpha_2},\Delta_{\alpha_2},\Delta_{\alpha_1}  )+ \beta_n(\Delta_{Q+iP},\Delta_{\alpha_4}, \Delta_{\alpha_3},\Delta_{\alpha_3},\Delta_{\alpha_4}  )).   
\end{split}\]
so that the general case follows from the  case $\alpha_1=\alpha_4$ and $\alpha_2=\alpha_3$. \qed

\section{Proof of Proposition \ref{proofward}}\label{subproofward}

\subsection{ Preliminary remarks}
 
   Before proceeding to  computations, we stress that the reader should keep in mind that the SET field $T(u)$  is not a proper random field. In particular the expectation in \eqref{propcontour} is a notation for the object constructed in the limit as $\boldsymbol{\epsilon}\to0$ and $t\to\infty$. In LCFT the construction of the correlation functions of the  SET is subtle. This was done in \cite{KRV} only for one or two SET insertions. However the situation is here  much simpler and we will not have to rely on \cite{KRV}. The reason is that  we need to deal with correlation functions with a regularized LCFT expectation $\langle \cdot\rangle_t$ where we have replaced $M_\gamma(\C)$ by $M_\gamma(\C\setminus \D_t)$ in \eqref{modifiedlcft}, and all the  SET insertions that we will consider are located in $\D_t$ which as we will see makes them much more regular than in the full LCFT. 
   
The regularized SET  field $T_\epsilon(u)$ is a proper random field and  its correlation functions in the presence of the vertex operators  are defined as limits of the corresponding ones with regularized vertex operators
\begin{align}\label{regular}
\langle  T_{\boldsymbol{\epsilon}}(\mathbf{u})\bar T_{\boldsymbol{\epsilon'}}({ \mathbf{v}}) V_{\alpha}(0)\prod_{i=1}^nV_{\alpha_i }(z_i) \rangle_t =\lim_{\epsilon''\to 0}\langle  T_{\boldsymbol{\epsilon}}(\mathbf{u})\bar T_{\boldsymbol{\epsilon'}}({ \mathbf{v}}) V_{\alpha,\epsilon''}(0)\prod_{i=1}^nV_{\alpha_i ,\epsilon''}(z_i) \rangle_t .
\end{align}
The existence of this limit follows from the representation of the expectation on the RHS as a  GFF expectation of an explicit function of a a GMC integral \cite{DKRV}.  And in particular the limit is independent of the regularization procedure used for the vertex operators.  For simplicity we will use in this section  the following:
\begin{align}\label{Vregulward}
V_{\alpha,\epsilon}(x)=\epsilon^{\frac{\alpha^2}{2}}e^{\alpha \phi_\epsilon(x)}=|x|_+^{-4\Delta_\alpha}e^{\alpha c}e^{\alpha X_\epsilon(x)-\frac{\alpha^2}{2}\E X_\epsilon(x)^2}(1+\caO(\epsilon))
\end{align}
where $X_\epsilon$ is the same regularization as in $T_\epsilon$.  The $\caO(\epsilon)$ will drop out from all terms in the $\epsilon\to 0$ limit and will not be displayed below.

 The proof of Proposition \ref{proofward} consists of using Gaussian integration by parts to the the $ T_{\boldsymbol{\epsilon}}(\mathbf{u})$ and $\bar T_{\boldsymbol{\epsilon'}}({ \mathbf{v}}) $ factors in \eqref{regular} to which we now turn.


\subsection{ Gaussian integration by parts}
 
\eqref{regular} is analysed using Gaussian integration by parts.  For a centered Gaussian vector $(X,Y_1,\dots, Y_N)$ and a smooth function $f$ on $\R^N$, the Gaussian integration by parts formula is $$\E[X \,f(Y_1,\dots,Y_N)]=\sum_{k=1}^N\E[XY_k]\E[\partial_{k}f(Y_1,\dots,Y_N)]. $$
Applied to the LCFT this leads to the following formula. Let $\phi=c+X-2Q\log|z|_+$ be the Liouville field  and $F$ a smooth function on $\R^N$.  Define for $u,v\in\C$
    \begin{align}\label{greenderi}
    C(u,v)=-\frac{1}{2}\frac{1}{u-v},\ \  C_{\epsilon,\epsilon'}(u,v)=\int \rho_{\epsilon}(u-u') \rho_{\epsilon'}(v-v')C(u',v')\dd u\dd v
 \end{align}
 with $( \rho_{\epsilon})_\epsilon$ a mollifying family of the type $\rho_\epsilon(\cdot)=\epsilon^{-2}\rho(|\cdot|/\epsilon)$.
Then   for $z,x_1,\dots,x_N\in\C$
\begin{align} \label{basicipp}
\langle \partial_z\phi_\epsilon(z) F(\phi_{\epsilon'}(x_1),\dots,\phi_{\epsilon'}(x_N))\rangle_t
&=
\sum_{k=1}^NC_{\epsilon,\epsilon'}(z,x_k)\langle \partial_{k}F(\phi_{\epsilon'}(x_1),\dots,\phi_{\epsilon'}(x_N))\rangle_t
\\
&-\mu\gamma\int_{\C_t} C_{\epsilon,0}(z,x)\langle V_\gamma(x) F(\phi_{\epsilon'}(x_1),\dots,\phi_{\epsilon'}(x_N))\rangle_t \dd x  \nonumber
\end{align}
where $F$ in the applications below is such that all the terms here are well defined. Note that $ \partial_uG(u,v)=C(u,v)+\hf u^{-1}1_{|u|>1}$. The virtue of the Liouville field is that the annoying metric dependent terms $u^{-1}1_{|u|>1}$ drop out from the formulae. This fact is nontrivial and it was proven in \cite[Subsection 3.2]{KRV},  for the case $t=\infty$ with $F$ corresponding to product of vertex operators. The proof goes the same way to produce \eqref{basicipp} with a finite $t$.

The first application of this formula is a direct proof of the existence of the $\boldsymbol{\epsilon},\boldsymbol{\epsilon}'\to 0$ limit of \eqref{regular} which will be useful also later in the proof.  We have
\begin{proposition}
The functions \eqref{regular} converge uniformly on compact subsets of $({\bf u},{\bf v},{\bf z})\in \D_t\times  \D_t \times \theta\caZ$
\begin{align}\label{regularlimit}
\lim_{\boldsymbol{\epsilon},\boldsymbol{\epsilon}'\to 0}\langle  T_{\boldsymbol{\epsilon}}(\mathbf{u})\bar T_{\boldsymbol{\epsilon'}}({ \mathbf{v}}) V_{\alpha}(0)\prod_{i=1}^nV_{\alpha_i }(z_i) \rangle_t :=\langle  T(\mathbf{u})\bar T({ \mathbf{v}}) V_{\alpha}(0)\prod_{i=1}^nV_{\alpha_i }(z_i) \rangle_t 
\end{align}
where the limit is analytic in ${\bf u}\in \caO_{e^{-t}}$ and anti-analytic in ${\bf v}\in \caO_{e^{-t}}$.
\end{proposition}
\begin{proof}
We consider for simplicity  only the case $\tilde{\nu}=0$. The LHS is defined as the limit $\epsilon''\to 0$ in \eqref{regular} but we will  for clarity work directly with $\epsilon''=0$ as it will be clear below that this limit trivially exists. Indeed, the functions $C_{\epsilon,\epsilon'}(u,v)$
are smooth  for $\epsilon, \epsilon'>0$ and they  converge  together with their derivatives uniformly on  $|u|<e^{-s}$, $|v|\geq e^{-t}$ for all $s>t$  to the derivatives of $C(u,v)$. 

We will now apply \eqref{basicipp} to all the $\phi_\epsilon$ factors in the SET tensors in\eqref{regularlimit} one after the other. To make this systematic let us introduce the notation
 $$(\caO_1,\caO_2,\caO_3)  :=( \partial_z\phi_\epsilon, \partial_z^2\phi_\epsilon, (\partial_z\phi_\epsilon)^2-\E(\partial_z\phi_\epsilon)^2).$$
 Applying the integration by parts formula to $\partial_z\phi_\epsilon(u_k)$ (or to $\partial_z^2\phi_\epsilon(u_k)$ if $i_k=2$ below) we obtain
 \begin{align*}
 \langle \prod_{j=1}^k \caO_{i_j}(u_j)V_\alpha(0)\prod_{i=1}^nV_{\alpha_i }(z_i) \rangle_t=&\sum_{l=1}^{k-1}  \langle \{\caO_{i_k}(u_k)\caO_{i_l}(u_l)\} \prod_{j\neq k,l} \caO_{i_j}(u_j)V_\alpha(0)\prod_{i=1}^nV_{\alpha_i }(z_i) \rangle_t \\+&\sum_{l=0}^n \alpha_l \langle \{\caO_{i_k}(u_k)\phi(z_l)\}\prod_{j\neq k} \caO_j(u_j)V_\alpha(0)\prod_{i\neq l}V_{\alpha_i }(z_i) \rangle_t\\-&\mu\gamma \int_{\C_t} \langle \{\caO_{i_k}(u_k)\phi(x)\} \prod_{j\neq k} \caO_j(u_j)V_\alpha(0)V_\gamma(x)\prod_{i\neq l}V_{\alpha_i }(z_i) \rangle_t\dd x
 \end{align*}
where $\alpha_0=\alpha$ and $z_0=0$ and we introduced the following notation
\begin{align*}
 \{\caO_1(u)\caO_1(v)\} &:=\partial_vC_{\epsilon,\epsilon}(u,v) & & \{\caO_2(u)\caO_2(v)\} :=\partial_u\partial^2_{v}C_{\epsilon,\epsilon}(u,v)\\
 \{\caO_1(u)\caO_3(v)\} &:=2\caO_1(v)\partial_vC_{\epsilon,\epsilon}(u,v) & &  \{\caO_1(u)\caO_2(v)\} :=\partial^2_{v}C_{\epsilon,\epsilon}(u,v)\\
 \{\caO_3(u)\caO_1(v)\} &:=\caO_1(u)\partial_vC_{\epsilon,\epsilon}(u,v) & & \{\caO_2(u)\caO_1(v)\} :=\partial_{u}\partial_{v}C_{\epsilon,\epsilon}(u,v) \\
\{\caO_3(u)\caO_3(v)\} &:=2\caO_1(u)\caO_1(v)\partial_vC_{\epsilon,\epsilon}(u,v)& &  \{\caO_2(u)\caO_3(v)\} :=2\caO_1(v)\partial_{u}\partial_{v}C_{\epsilon,\epsilon}(u,v)\\
\{\caO_3(u)\caO_2(v)\} &:=\caO_1(u)\partial^2_{v}C_{\epsilon,\epsilon}(u,v)  & &  
\end{align*}
etc. Similarly 
$$\{\caO_1(u)\phi(x)\} := C_{\epsilon,0}(u,x), \quad\{\caO_2(u)\phi(x)\} := \partial_u C_{\epsilon,0}(u,x), \quad\{\caO_3(u)\phi(x)\} :=\caO_1(u)C_{\epsilon,0}(u,x). $$

Since $T_\epsilon=Q\caO_2-\caO_3$ we can iterate the integration by parts formula  to obtain an expression for  $\langle  T_{\boldsymbol{\epsilon}}(\mathbf{u})
 V_{\alpha}(0)\prod_{i=1}^nV_{\alpha_i }(z_i) \rangle_t $ as a sum of terms 
 of the form
\begin{align}\label{basicterm}
C\prod_{\alpha, \beta} \partial^{i_{\alpha\beta}}C_{\epsilon,\epsilon}(u_\alpha,u_\beta)\prod_{\alpha,l}  \partial^{j_{\alpha,l }}C_{\epsilon,0}(u_\alpha,z_l)\int_{\C_t^m}\prod_{\alpha ,k} \partial^{l_{\alpha,k }}C_{\epsilon,0}(u_{\alpha},x_k) \langle  V_\alpha(0)\prod_{j=1}^mV_{\gamma }(x_j)\prod_{i=1}^nV_{\alpha_i }(z_i) \rangle_t\dd{\bf x}
\end{align}
where the products run through some subsets of the index values. We have for all compact $K\subset \D_t$ 
$$ \sup_{u\in K, x\in\C_t}\sup_\epsilon|\partial_u^{l}\partial_{\bar u}^{l'}C_{\epsilon,0}(u,x)| <\infty.$$
Hence the  expression \eqref{basicterm} is bounded together with all its derivatives in $\bf u$ uniformly in $\epsilon$ by
\begin{align}\label{aprior}
C\int_{\C_t^m}\langle  V_\alpha(0)\prod_{j=1}^mV_{\gamma }(x_j)\prod_{i=1}^nV_{\alpha_i }(z_i) \rangle_t \dd{\bf x}.
\end{align}
which by the lemma below is finite. The limit of \eqref{basicterm} and all its derivatives exist by dominated convergence. Clearly the $\partial_{\bar u}$ derivatives vanish so the limit   is analytic in $\bf u$.
\end{proof}

We need the following KPZ identity for the a priori bound \eqref{aprior} (recall \eqref{defs})
\begin{lemma}\label{kpzlemma}
The functions ${\bf x}\in \C_t^n\to \langle  V_\alpha(0)\prod_{j=1}^mV_{\gamma }(x_i)\prod_{i=1}^nV_{\alpha_i }(z_i) \rangle_t$ are integrable and
\begin{align}\label{KPZ}
\int_{\C_t^m} \langle  V_\alpha(0)\prod_{j=1}^mV_{\gamma }(x_i)\prod_{i=1}^nV_{\alpha_i }(z_i) \rangle_t\dd{\bf x}=C\int_{\C_t^m} \langle  V_\alpha(0)\prod_{i=1}^nV_{\alpha_i }(z_i) \rangle_t\dd{\bf x}
\end{align}
where $C=(\mu\gamma)^{-m}\prod_{l=0}^{m-1}(\alpha+s+ \gamma l-2Q)$.
\end{lemma}
 \begin{proof}
 See \cite[Lemma 3.3]{KRV}. Briefly, 
 $$\langle  V_\alpha(0)\prod_{i=1}^nV_{\alpha_i }(z_i) \rangle_t=\mu^{\frac{2Q-s-\alpha}{\gamma}}\langle  V_\alpha(0)\prod_{i=1}^nV_{\alpha_i }(z_i) \rangle_t\big|_{\mu=1}$$
 and the LHS of \eqref{KPZ} equals $(-\partial_\mu)^m\langle  V_\alpha(0)\prod_{i=1}^nV_{\alpha_i }(z_i) \rangle_t$.
 \end{proof}

 The representation \eqref{basicterm} will not be useful for a direct proof of the Ward identity due to the $V_\gamma$ insertions. We will rather use the integration by parts inductively, first to $T_\epsilon(u_k)$  which corresponds to the largest contour 
 in the contour integrals in expression \eqref{propcontour}, and by showing then that at each step the   $V_\gamma$ insertions give rise to the derivatives in the Proposition \ref{proofward}. We give now the inductive step to prove this claim, stated for simplicity for the case $\tilde{\nu}=0$. For this, we introduce the (adjoint) operator $\mathbf{D}_n^\ast$
 defined (by duality) by
 $$\int_{\C^n} \mathbf{D}_n f(\mathbf{z})\bar{\varphi}(z)\,\dd \mathbf{z}=\int_{\C^n}  f(\mathbf{z})\overline{\mathbf{D}_n^\ast \varphi}(z)\,\dd \mathbf{z}$$
 for all functions $f$ in the domain of $\mathbf{D}_n$ and all smooth compactly supported functions $ \varphi$ in $\C^n$.
  Then 
 
 \begin{proposition}\label{wardinduction}
  Let $\varphi\in C_0^\infty(\theta\caZ)$ be a smooth compactly supported function in $\theta\caZ\subset\C^n$ and define
  \begin{align*}
\caT_t(\nu,\varphi):=\int \Big(\frac{1}{(2\pi i)^{k}}\oint_{|\mathbf{u}|=\boldsymbol{\delta}_t}   \mathbf{u}^{1-\nu}  \langle   T(\mathbf{u}) V_\alpha(0)\prod_{i=1}^nV_{\alpha_i }(z_i) \rangle_t\dd{\bf u}\Big)\,\bar{\varphi}({\bf z})\dd{\bf z}.
\end{align*}
Then for $\nu^{(k)}=(\nu_1,\dots,\nu_{k-1})$
\begin{align*}
\caT_t(\nu,\varphi)=\caT_t(\nu^{(k)},{\bf D}^\ast_{\nu_k}\varphi)+\caB_t(\nu,\varphi)
\end{align*}
where
\begin{align}\label{btbound}
|\caB_t(\nu,\varphi)|\leq Ce^{(\alpha+|\nu|-2)t}\int \langle   V_\alpha(0)\prod_{i=1}^nV_{\alpha_i }(z_i) \rangle_t \,|\varphi({\bf z})|\dd{\bf z}.
\end{align}
 \end{proposition}
 \vskip 2mm
\noindent{\it Proof of Proposition \ref{proofward}}. Iterating Proposition \ref{wardinduction} we get 
 \begin{align*}
 \caT_t(\nu,\varphi)=\int \langle   V_\alpha(0)\prod_{i=1}^nV_{\alpha_i }(z_i) \rangle_t \,\overline{{\bf D}^\ast_\nu\varphi}({\bf z})\dd{\bf z}+\caB_t(\varphi)
\end{align*}
 where $\caB_t(\varphi)$ satisfies  
 \begin{equation*}
 \caB_t(\varphi) \leq Ce^{(\alpha+|\nu|-2)t} \sum_{\ell=1}^k\int \langle   V_\alpha(0)\prod_{i=1}^nV_{\alpha_i }(z_i) \rangle_t \,  | {\bf D}^\ast_{\nu_{\ell+1}} \cdots {\bf D}^\ast_{\nu_k} \varphi({\bf z})|\dd{\bf z}
 \end{equation*}
where  by convention $ {\bf D}^\ast_{\nu_{\ell+1}} \cdots {\bf D}^\ast_{\nu_k} \varphi=\varphi$ if $\ell=k$. The functions ${\bf z}\to \langle   V_\alpha(0)\prod_{i=1}^nV_{\alpha_i }(z_i) \rangle_t$ are continuous on   $\theta\caZ$ and converge uniformly as $t\to\infty$ on compact subsets of $\theta\caZ$ to  the function $ \langle   V_\alpha(0)\prod_{i=1}^nV_{\alpha_i }(z_i) \rangle$. Hence $ \caT_t(\nu,\cdot)$ converges in the Frechet topology of $\caD'(\theta\caZ)$ to the required limit since $ \caB_t(\varphi)$ goes to $0$ as $t$ goes to infinity (recall that $\alpha+|\nu|-2<0$). \qed
 
\subsection{ Proof of Proposition \ref{wardinduction}}
 
 We start the proof of Proposition \ref{wardinduction} by
 applying Gaussian integration by parts formula twice to the $T(u_k)$. 
This produces plenty of terms which we group in four contributions:
  \begin{align}\label{IPPstress}
   \langle   T(\mathbf{u})& V_\alpha(0)\prod_{i=1}^nV_{\alpha_i }(z_i) \rangle_t   =R(\mathbf{u},\mathbf{z})  + M(\mathbf{u},\mathbf{z})  + N(\mathbf{u},\mathbf{z})+D(\mathbf{u},\mathbf{z})  .
  \end{align}
  In $R(\mathbf{u})$ we group all the contractions with $C(u_k,u_l)$ between $u_k$ and $u_l$, $l<k$ and  $u_k$ and $0$. These terms do not contribute to the contour integral of the $u_k$ variable since they give rise to integrals of the form
 \begin{align*}
 \oint_{|u_k|=e^{-t}\delta_k}u_k^{1-\nu_k}(u_k-v)^{-a}(u_k-w)^{-b}\dd u_k
\end{align*} 
  where $v,w\in\{0,u_1,\dots,u_{k-1}\}$ and $a+b\geq 2$. Since $|v|,|w|<e^{-t}\delta_k$ and $\nu_k>0$, this integral vanishes. We conclude
  \begin{align}\label{Rcontour}
 \oint_{|u_k|=e^{-t}\delta_k}u_k^{1-\nu_k}R({\bf u},\mathbf{z})\dd u_k=0.
\end{align} 
For the benefit of the reader we display all the terms in $R({\bf u},\mathbf{z})$ in the Appendix \ref{ippjunk}.

 Let us now introduce the notations ${\bf u}^{(k)}:=(u_1,\dots, u_{k-1})$ and  ${\bf u}^{(k,\ell)}:=(u_1,\dots,u_{\ell-1},u_{\ell+1},\dots, u_{k-1})$. The second contribution in \eqref{IPPstress} collects the contractions hitting only one $V_{\alpha_p}$:
\begin{align}\label{IPPstressM}
 M(\mathbf{u},\mathbf{z})=  \sum_{p=1}^{n} (\frac{Q\alpha_p}{2}- \frac{\alpha_p^2}{4}) \frac{1}{(u_k-z_p)^2}  \langle    T(\mathbf{u}^{(k)})  V_\alpha(0)\prod_{i=1}^nV_{\alpha_i }(z_i) \rangle_t . 
\end{align}
We can then do the $u_k$ integral explicitly to obtain
 \begin{align}\label{confweightward}
& \int\Big(\frac{1}{(2\pi i)^{k}}\oint_{|\mathbf{u}|=\boldsymbol{\delta}_t}  
   \mathbf{u}^{1-\nu}\M(\mathbf{u},\mathbf{z}) \,  d   \mathbf{u}\Big)\bar{\varphi}({\bf z})d{\bf z}  \\
   & = \int \left ( \sum_{p=1}^n\frac{\nu_k-1}{z_p^{\nu_k}}\Delta_{\alpha_p} \right ) 
 \langle    T(\mathbf{u}^{(k)})  V_\alpha(0)\prod_{i=1}^nV_{\alpha_i }(z_i) \rangle_t \bar{\varphi}({\bf z})\dd{\bf z}  
\end{align}
Note that this term contains the constant part of the differential operator $\mathbf{D}_{\nu_k}$. 

The third contribution is given by terms where all contractions hit $V_\gamma$:
\begin{align}\nonumber
N(\mathbf{u},\mathbf{z})=&(\frac{\mu\gamma^2}{4}-\frac{\mu\gamma Q}{2}) \int_{\C_t} \frac{1}{(u_k-x)^2}\langle   T(\mathbf{u}^{(k)})  V_\alpha(0)V_\gamma(x)\prod_{i=1}^nV_{\alpha_i }(z_i) \rangle_t \,\dd x
\\=
  & - \mu   \int_{\C_t} \frac{1}{(u_k-x)^2}\langle   T(\mathbf{u}^{(k)})  V_\alpha(0)V_\gamma(x)\prod_{i=1}^nV_{\alpha_i }(z_i) \rangle_t \,\dd x. \label{Nterm} 
\end{align}

Finally $D$ gathers all the other terms
\begin{align}\label{IPPstressrest}
 D(\mathbf{u},\mathbf{z}) = \sum_{i=1}^9T_i(\mathbf{u},\mathbf{z})
 \end{align}
 with
\begin{align} 
 T_1(\mathbf{u},\mathbf{z})&=
 -\sum_{\ell=1}^{k-1}\sum_{p=1}^n      \frac{Q\alpha_p}{(u_k-u_\ell)^3(u_k-z_p)} \langle   T(\mathbf{u}^{(\ell,k)})  V_\alpha(0)\prod_{i=1}^nV_{\alpha_i }(z_i) \rangle_t 
\nonumber \\
 T_2(\mathbf{u},\mathbf{z})  &= \sum_{\ell=1}^{k-1} \sum_{p=1}^{n} \frac{\alpha_p}{(u_k-u_\ell)^2(u_k-z_p)} \langle   \partial_{z}X(u_{\ell})  T(\mathbf{u}^{(\ell,k)})  V_\alpha(0)\prod_{i=1}^nV_{\alpha_i }(z_i) \rangle_t  
\nonumber\\
T_3(\mathbf{u},\mathbf{z}) &=-  \sum_{p=1}^{n}  \frac{\alpha\alpha_p}{2} \frac{1}{(u_k-z_p)u_k}  \langle    T(\mathbf{u}^{(k)})  V_\alpha(0)\prod_{i=1}^nV_{\alpha_i }(z_i) \rangle_t  
\nonumber\\
T_4(\mathbf{u},\mathbf{z})& =\frac{\mu\gamma}{2} \sum_{p=1}^{n} \alpha_p\int_{\C_t} \frac{1}{(u_k-z_p)(u_k-x)} \langle    T(\mathbf{u}^{(k)})  V_\alpha(0)V_\gamma(x)\prod_{i=1}^nV_{\alpha_i }(z_i) \rangle_t  \,\dd x  
\nonumber\\
T_5(\mathbf{u},\mathbf{z}) &=-  \sum_{p\not=p'=1}^{n}  \frac{\alpha_p\alpha_{p'}}{4} \frac{1}{(u_k-z_p)(u_k-z_{p'})}  \langle    T(\mathbf{u}^{(k)})  V_\alpha(0)\prod_{i=1}^nV_{\alpha_i }(z_i) \rangle_t 
\nonumber\\ 
T_6(\mathbf{u},\mathbf{z}) &   =\mu Q\gamma \sum_{\ell=1}^{k-1}\int_{\C_t}    \frac{1}{(u_k-u_\ell)^3(u_k-x)} \langle   T(\mathbf{u}^{(\ell,k)})  V_\alpha(0)V_\gamma(x)\prod_{i=1}^nV_{\alpha_i }(z_i) \rangle_t \,\dd x
\nonumber\\
T_7(\mathbf{u},\mathbf{z}) &=-\mu\gamma \sum_{\ell=1}^{k-1} \int_{\C_t} \frac{1}{(u_k-u_\ell)^2(u_k-x)} \langle   \partial_{z}X(u_{\ell})  T(\mathbf{u}^{(\ell,k)})  V_\alpha(0)V_\gamma(x)\prod_{i=1}^nV_{\alpha_i }(z_i) \rangle_t  \,\dd x
\nonumber\\
T_8(\mathbf{u},\mathbf{z}) &=\frac{\mu\gamma \alpha}{2}   \int_{\C_t} \frac{1}{u_k(u_k-x)} \langle    T(\mathbf{u}^{(k)})  V_\alpha(0)V_\gamma(x)\prod_{i=1}^nV_{\alpha_i }(z_i) \rangle_t  \,\dd x
\nonumber\\
T_9(\mathbf{u},\mathbf{z}) &= -\frac{\mu^2\gamma^2}{4}  \int_{\C_t}\int_{\C_t} \frac{1}{(u_k-x)(u_k-x')} \langle    T(\mathbf{u}^{(k)})  V_\alpha(0)V_\gamma(x)V_\gamma(x')\prod_{i=1}^nV_{\alpha_i }(z_i) \rangle_t  \,\dd x .
\label{Dterms}
  \end{align}
  
  We need to show that  $N$ and $D$   will give rise (after contour integration) to the  $\partial_{z_i}$-derivatives in the expression $\mathbf{D}_{\nu_k}\langle   T(\mathbf{u}^k)  V_\alpha(0)\prod_{i=1}^nV_{\alpha_i }(z_i) \rangle_t $. 
To show this we need to analyse  $N$ further. 

Regularizing the vertex insertions  (beside the $V_\gamma$ insertion, we also regularize the $V_{\alpha_i}$'s for later need) 
in $N(\mathbf{u},\mathbf{z})$ given by \eqref{Nterm}, and performing an integration by parts (Green formula) in the $x$ integral  we get  
\begin{align*} 
 N(\mathbf{u},\mathbf{z}) =& - \mu  \lim_{\epsilon\to 0} \int_{\C_t} \partial_x\frac{1}{u_k-x }\langle   T(\mathbf{u}^{(k)})  V_\alpha(0)V_{\gamma,\epsilon}(x)\prod_{i=1}^nV_{\alpha_i,\epsilon }(z_i) \rangle_t \,\dd x\\
=& B_t(\mathbf{u},\mathbf{z}) + \mu \gamma  \lim_{\epsilon\to 0} \int_{\C_t}  \frac{1}{u_k-x } \langle   T(\mathbf{u}^{(k)})  V_\alpha(0) \partial_x\phi_\epsilon(x)V_{\gamma,\epsilon}(x)\prod_{i=1}^nV_{\alpha_i ,\epsilon}(z_i) \rangle_t \,\dd x\\=: &B_t(\mathbf{u},\mathbf{z}) + \tilde N(\mathbf{u},\mathbf{z})
  \end{align*} 
 where the boundary term appearing in Green formula has   $\epsilon\to 0$ limit given by
 \begin{align}\label{bryterm}
 B_t(\mathbf{u},\mathbf{z}) :=
 i\mu\oint_{|x|=e^{-t}}
\frac{1}{u_k-x} \langle   T(\mathbf{u}^{(k)})  V_\alpha(0)V_\gamma(x)\prod_{i=1}^nV_{\alpha_i }(z_i) \rangle_t \,\dd\bar x
\end{align}
and we used 
  \eqref{Vregul} to write
\begin{align*}
\partial_xV_{\gamma,\epsilon}(x)=\alpha\partial_x\phi_\epsilon(x)V_{\gamma,\epsilon}(x).
\end{align*}

In $\tilde N(\mathbf{u},\mathbf{z})$ we integrate by parts the $ \partial_x\phi_\epsilon(x)$ and end up with
  \begin{align}\nonumber
\tilde N(\mathbf{u},\mathbf{z})
=&  - \mu Q \gamma\sum_{\ell=1}^{k-1}  \int_{\C_t}  \frac{1}{(u_k-x)(x-u_\ell)^3 } \langle   T(\mathbf{u}^{(\ell,k)})  V_\alpha(0)V_\gamma(x)\prod_{i=1}^nV_{\alpha_i }(z_i) \rangle_t \,\dd x
\\
& +\mu  \gamma\sum_{\ell=1}^{k-1}  \int_{\C_t}  \frac{1}{(u_k-x)(x-u_\ell)^2 } \langle \partial_zX(u_\ell)  T(\mathbf{u}^{(\ell,k)})  V_\alpha(0)V_\gamma(x)\prod_{i=1}^nV_{\alpha_i }(z_i) \rangle_t \,\dd x \nonumber
\\
& -\frac{\mu  \gamma \alpha}{2}  \int_{\C_t}  \frac{1}{(u_k-x) x } \langle  T(\mathbf{u}^{(k)})  V_\alpha(0)V_\gamma(x)\prod_{i=1}^nV_{\alpha_i }(z_i) \rangle_t \,\dd x \nonumber
\\
&  +\frac{\mu^2  \gamma^2  }{2}  \int_{\C_t} \int_{\C_t}  \frac{1}{(u_k-x) (x-x') } \langle  T(\mathbf{u}^{(k)})  V_\alpha(0)V_\gamma(x)V_\gamma(x')\prod_{i=1}^nV_{\alpha_i }(z_i) \rangle_t \,\dd x \nonumber
\\
&+ {\mu  \gamma  }\lim_{\epsilon\to 0}  \sum_{p=1}^n \int_{\C_t}  \frac{\alpha_p}{u_k-x }C_{\epsilon,0}(x,z_p)  \langle  T(\mathbf{u}^{(k)})  V_\alpha(0)V_\gamma(x)\prod_{i=1}^nV_{\alpha_i,\epsilon }(z_i) \rangle_t \,\dd x\nonumber
\\
&=: 
T'_6(\mathbf{u},\mathbf{z})+T'_7(\mathbf{u},\mathbf{z})+T'_{8}(\mathbf{u},\mathbf{z})+T'_{9}(\mathbf{u})+T'_{4}(\mathbf{u},\mathbf{z})\label{T4leq}
\end{align}
where again we took the  $\epsilon\to 0$ limit in the terms where it was obvious. In particular this identity proves that the limit on the RHS, denoted by $T'_{4}(\mathbf{u},\mathbf{z})$, exists. The numbering of these terms and the ones below will be used when comparing with \eqref{Dterms}.

\subsubsection{Derivatives of correlation functions}

We want to compare the expression \eqref{IPPstress} to  derivatives of the function $\langle   T(\mathbf{u}^{(k)})  V_\alpha(0)\prod_{i=1}^nV_{\alpha_i }(z_i) \rangle_t $. 
We need to treat separately the cases $\nu_k\geq 2$ and $\nu_k=1$.

\subsubsection{Case $\nu_k\geq 2$ }

 We have
\begin{lemma}\label{deriv2}
Let
   \begin{align*}
I_\epsilon(\mathbf{u},{\bf z}):=
\sum_{p=1}^n\frac{1}{u_k-z_p}&  \partial_{z_p}\langle   T(\mathbf{u}^k) V_\alpha(0)\prod_{i=1}^nV_{\alpha_i,\epsilon }(z_i) \rangle_t .
  \end{align*}  
  Then $\lim_{\epsilon\to 0}I_\epsilon(\mathbf{u},{\bf z}):= I(\mathbf{u},{\bf z})$
exists and defines a continuous function in ${\bf z}\in \theta\caZ$ satisfying
  \begin{align}\label{contourI}
  \int\Big(\frac{1}{(2\pi i)^{k}}\oint_{|\mathbf{u}|=\boldsymbol{\delta}_t}{\bf u}^{1-\nu} I(\mathbf{u},{\bf z})d{\bf u} \Big)\bar{\varphi}({\bf z})\dd {\bf z}=\caT_t(\nu^{(k)},\hat{\bf D}_{\nu_k}\varphi)
\end{align}
for all $\varphi\in C_0^\infty(\theta\caZ)$ with $\hat{\bf D}_{n}={\bf D}^\ast_{n}-(n-1)\sum_i\Delta_{\alpha_i}z_i^{-n}$.
\end{lemma}
\begin{proof}
We have
   \begin{align*}
I_\epsilon(\mathbf{u},{\bf z})=\sum_{p=1}^n\frac{\alpha_p}{u_k-z_p}&  \langle   T(\mathbf{u}^{(k)}) V_\alpha(0)\partial_{z_p}\phi_\epsilon(z_p)\prod_{i=1}^nV_{\alpha_i,\epsilon }(z_i) \rangle_t   = K_\epsilon(\mathbf{u},{\bf z})+L_\epsilon(\mathbf{u},{\bf z})
  \end{align*}  
  where we integrate by parts the $\partial_{z_p}\phi_\epsilon(z_p)$ and $K_\epsilon({\bf u},{\bf z})$ collects the terms with an obvious $\epsilon\to 0$ limit $K({\bf u},{\bf z})$:
    \begin{align*}
 K(\mathbf{u},{\bf z})  =&- \sum_{p=1}^n\sum_{\ell=1}^{k-1}\frac{Q\alpha_p}{(u_k-z_p)(z_p-u_\ell)^3}  \langle   T(\mathbf{u}^{(\ell,k)})  V_\alpha(0)\prod_{i=1}^nV_{\alpha_i }(z_i) \rangle_t   
   \\
&+ \sum_{p=1}^n\sum_{\ell=1}^{k-1}\frac{\alpha_p}{(u_k-z_p)(z_p-u_\ell)^2}  \langle  \partial_{z}X(u_\ell) T(\mathbf{u}^{(\ell,k)})  V_\alpha(0)\prod_{i=1}^nV_{\alpha_i }(z_i) \rangle_t     
 \\  
&- \sum_{p=1}^n \frac{\alpha_p\alpha}{2}\frac{1}{(u_k-z_p)z_p}  \langle   T(\mathbf{u}^{(k)})  V_\alpha(0)\prod_{i=1}^nV_{\alpha_i }(z_i) \rangle_t     
\\
&- \sum_{p\not= p'=1}^n \frac{\alpha_p\alpha_{p'}}{2}\frac{1}{(u_k-z_p)(z_p-z_{p'})}  \langle  T(\mathbf{u}^{(k)})  V_\alpha(0)\prod_{i=1}^nV_{\alpha_i }(z_i) \rangle_t
   \\
  =:&D_1(\mathbf{u},{\bf z})+D_2(\mathbf{u},{\bf z})+D_3(\mathbf{u},{\bf z})
  +D_5(\mathbf{u},{\bf z}) ,
  \end{align*}  
whereas 
\begin{align*}
 L_\epsilon(\mathbf{u},{\bf z}) 
=-\mu\gamma\sum_{p=1}^n  
\alpha_p\int_{\C_t}
\frac{1}{u_k-z_p}C_{\epsilon,0}(z_p,x)  \langle  T(\mathbf{u}^{(k)})  V_\alpha(0)V_\gamma(x)\prod_{i=1}^nV_{\alpha_i ,\epsilon}(z_i) \rangle_t     \,\dd x   .
  \end{align*}  
Since $C_{0,0}(z_p,x) =-\frac{1}{2}\frac{1}{z_p-x}$  and since it is not clear that $\frac{1}{z_p-x} \langle  T(\mathbf{u}^{k})  V_\alpha(0)V_\gamma(x)\prod_{i=1}^nV_{\alpha_i }(z_i) \rangle_t    $ is integrable the $\epsilon\to 0$ limit of $L_\epsilon$ is problematic\footnote{Actually, this fact was shown in \cite{dozz} without the SET insertions and could be proven here as well but we will follow another route because the recursion to prove this extension is painful.}. However, we can compare it with the term ${T_4'}$ in \eqref{T4leq}.  Writing 
$$\frac{1}{u_k-z_p}=\frac{1}{u_k-x}+\frac{z_p-x}{(u_k-z_p)(u_k-x)},
$$
we conclude that $ L_\epsilon$ converges:
\begin{align*}
\lim_{\epsilon\to 0} L_\epsilon(\mathbf{u},{\bf z})
 =&-\mu \gamma\lim_{\epsilon\to 0}\sum_{p=1}^n  
\alpha_p  \Big(\int_{\C_t}
\frac{1}{u_k-x}C_{\epsilon,0}(z_p,x)  \langle  T(\mathbf{u}^{(k)})  V_\alpha(0)V_\gamma(x)\prod_{i=1}^nV_{\alpha_i,\epsilon }(z_i) \rangle_t     \,\dd x
\\
& +  \int_{\C_t}
\frac{z_p-x}{(u_k-z_p)(u_k-x)}C_{\epsilon,0}(z_p,x)  \langle  T(\mathbf{u}^{(k)})  V_\alpha(0)V_\gamma(x)\prod_{i=1}^nV_{\alpha_i }(z_i) \rangle_t     \,\dd x\Big) 
\\
 =&T_4'(\mathbf{u},{\bf z})+T_4(\mathbf{u},{\bf z}).
  \end{align*}  
Indeed, setting $z=z_p-x$  the function 
\begin{align*}
(z_p-x)C_{\epsilon,0}(z_p,x)= -\hf\int \rho_\epsilon(y)\frac{z}{z+y}dy=\frac{i}{2\epsilon^2}\int_{\R^+}\rho(\tfrac{r}{\epsilon})\oint_{|u|=1} \frac{z}{z+ru}\frac{du}{u}rdr=-\pi \int_{\R^+}\rho(r)1_{r<|z|/\epsilon}rdr
\end{align*}
is uniformly bounded and converges almost everywhere to $-\hf$.

The same argument can be repeated to the smeared functions to show that (because convergence is uniform over compact subsets of $\theta\mathcal{Z}$)
\begin{align*}
\lim_{\epsilon\to 0}
\sum_{p=1}^n\int\Big(\frac{\alpha_p}{u_k-z_p}  \partial_{z_p}\langle   T(\mathbf{u}^{(k)}) V_\alpha(0)\prod_{i=1}^nV_{\alpha_i,\epsilon }(z_i) \rangle_t  \Big)\bar{\varphi}({\bf z})d{\bf z}:=\ell(\nu,\varphi)
\end{align*}
exists. Then integrating $\partial_{z_p}$ by parts and using that $\lim_{\epsilon\to 0}
\langle   T(\mathbf{u}^k) V_\alpha(0)\prod_{i=1}^nV_{\alpha_i,\epsilon }(z_i) \rangle_t  $ exists we conclude
\begin{align*}
\ell(\nu,\varphi)=
-\int\langle   T(\mathbf{u}^k) V_\alpha(0)\prod_{i=1}^nV_{\alpha_i,\epsilon }(z_i) \rangle_t  \sum_{p=1}^n  \partial_{z_p}(\frac{\alpha_p}{u_k-z_p} 
\varphi({\bf z}))\dd {\bf z}
\end{align*}
which proves \eqref{contourI}.
\end{proof}

We have obtained the relation
\begin{align}
N(\mathbf{u},\mathbf{z})-I(\mathbf{u},{\bf z})&=B_t(\mathbf{u},\mathbf{z})+\sum_{i=6}^9T_i(\mathbf{u},\mathbf{z})-\sum_{i=1}^3D_i(\mathbf{u},\mathbf{z})-D_5(\mathbf{u},\mathbf{z})
 -T_4(\mathbf{u},{\bf z}).\label{N-I}
\end{align}
Let us consider the expression
\begin{align*}
K(\mathbf{u},{\bf z}):=&\langle   T(\mathbf{u})V_\alpha(0)\prod_{i=1}^nV_{\alpha_i }(z_i) \rangle_t  -I(\mathbf{u},{\bf z})- M(\mathbf{u},\mathbf{z}) 
\\=&R(\mathbf{u},\mathbf{z})   + N(\mathbf{u},\mathbf{z})+D(\mathbf{u},\mathbf{z}) -I(\mathbf{u},{\bf z}).
\end{align*}
By \eqref{confweightward}  and \eqref{contourI} we have
 \begin{align}\label{contourI1}
  \int\Big(\oint_{|\mathbf{u}|=\boldsymbol{\delta}_t}{\bf u}^{1-\nu} K(\mathbf{u},{\bf z})d{\bf u} \Big)\varphi({\bf z})\dd {\bf z}&=\caT(\nu,\varphi)-\caT(\nu^{(k)},\mathbf{D}^\ast_{\nu_k}\varphi).
   \end{align}
On the other hand
combining \eqref{IPPstress}, \eqref{IPPstressrest} and \eqref{N-I} 
we obtain
\begin{align}
K&=R+\sum_{i=1}^3(T_i-D_i)+ T_5-D_5+\sum_{i=6}^9(T_i+T'_i)
 +B_t\label{finalsum}.
\end{align}
Now some simple algebra, see Appendix \ref{ippjunk} gives:
\begin{align}\label{cancel1}T'_{9}=-T_{9}, \ \ 
T_5&=D_5\\ \label{cancel3}
\oint_{|u_k|=e^{-t}\delta_k}u_k^{1-\nu_k}({T}_i(\mathbf{u},\mathbf{z})+{T'}_i(\mathbf{u},\mathbf{z}))\,\dd u_k&=0\ \ i=6,7,8,
\\ \label{cancel4}
\oint_{|u_k|=e^{-t}\epsilon_k}u_k^{1-\nu_k}\big(T_i(\mathbf{u},\mathbf{z}) -D_i(\mathbf{u},\mathbf{z})\big) \,\dd u_k&=0\ \ i=1,2,3.
\end{align}
Hence using  these relations and \eqref{Rcontour}  we conclude
 \begin{align*}
  \int\Big(\oint_{|\mathbf{u}|=\boldsymbol{\delta}_t}{\bf u}^{1-\nu} K(\mathbf{u},{\bf z})\dd{\bf u} \Big)\bar{\varphi}({\bf z})\dd {\bf z}&= \int\Big(\oint_{|\mathbf{u}|=\boldsymbol{\delta}_t}{\bf u}^{1-\nu} B_t(\mathbf{u},{\bf z})\dd{\bf u} \Big)\bar{\varphi}({\bf z})\dd{\bf z} =:B_t(\nu,\varphi).
 \end{align*}
Thus to prove Proposition \ref{proofward} for $\nu_k\geq 2$ we need to prove the bound    \eqref{btbound} for $B_t(\nu,\varphi)$.
 Recalling \eqref{bryterm} we get by residue theorem
\begin{align*}
\oint_{|\mathbf{u}|=\boldsymbol{\delta}_t}  
   \mathbf{u}^{1-\nu} B_t(\mathbf{u},{\bf z})  \,  \dd   \mathbf{u}
   =-2\pi\oint_{|\mathbf{u}^{(k)}|=\boldsymbol{\delta}^{(k)}_t}  
   (\mathbf{u}^{(k)})^{1-\nu^{(k)}} \oint_{|x|=e^{-t}} x^{1-\nu_k}\langle   T(\mathbf{u}^{(k)})  V_\alpha(0)V_\gamma(x)\prod_{i=1}^nV_{\alpha_i }(z_i) \rangle_t \,\dd\bar x  \,  \dd   \mathbf{u}.
\end{align*}
By \eqref{basicterm} (at $\epsilon=0$ and an extra $V_{\alpha_{n+1}}(z_{n+1})=V_\gamma(x)$) the expectation on the RHS is a sum of terms of the form
\begin{align*}
\int_{\C_t^m}I(\mathbf{u}^{(k)},{\bf x},{\bf z},x)\langle  V_\alpha(0)V_{\gamma }(x)\prod_{\ell=1}^mV_{\gamma }(x_\ell)\prod_{i=1}^{n}V_{\alpha_i }(z_i) \rangle_t\dd{\bf x}
\end{align*}
where
\begin{align*}
I(\mathbf{u}^{(k)},{\bf x},{\bf z},x)=
C\prod_{\alpha, \beta} \frac{1}{(u_\alpha-u_\beta)^{k_{\alpha\beta}}}\prod_{\alpha,i}\frac{1}{(u_\alpha-z_i)^{l_{\alpha i}}}\prod_{\alpha,\ell}\frac{1}{(u_\alpha-x_\ell)^{m_{\alpha \ell}} }\prod_{\alpha}\frac{1}{(u_\alpha-x)^{n_{\alpha }} }
\end{align*}
where $\sum k_{\alpha\beta}+\sum l_{\alpha k}+\sum m_{\alpha l}+\sum n_{\alpha }=2(k-1)$. Performing the $u$-integrals in the order $u_{k-1},u_{k-2},\dots$ by the residue theorem we get
\begin{align*}
\oint_{|\mathbf{u}^{(k)}|=\boldsymbol{\delta}^{(k)}_t}  
   (\mathbf{u}^{(k)})^{1-\nu^{(k)}} I(\mathbf{u}^{(k)},{\bf x},{\bf z},x) \dd   \mathbf{u}=\sum C(a,{\bf b},{\bf c})x^{-a}\prod_\ell x_\ell^{-b_\ell}\prod_iz_i^{-c_i}
\end{align*}
with $a+\sum b_\ell+\sum c_i=|\nu^{(k-1)}|$. Since $|x_\ell|\geq e^{-t}$ and $|x|= e^{-t}$ we conclude
\begin{align*}
|\oint_{|\mathbf{u}|=\boldsymbol{\delta}_t}  
   \mathbf{u}^{1-\nu} B_t(\mathbf{u},\mathbf{z})  \,  \dd   \mathbf{u}|\leq Ce^{t(|\nu|-2)}
 \max_{m\leq 2(k-1)}\sup_{|x|=e^{-t}}\int_{\C_t^m}  \langle  V_\alpha(0)V_\gamma(x)\prod_{\ell=1}^mV_{\gamma }(x_\ell)\prod_{i=1}^nV_{\alpha_i }(z_i) \rangle_t \dd{\bf x}.
   \end{align*}
By Lemma \ref{kpzlemma} 
\begin{align*}
\int_{\C_t^m}  \langle  V_\alpha(0)V_\gamma(x)\prod_{\ell=1}^mV_{\gamma }(x_\ell)\prod_{i=1}^nV_{\alpha_i }(z_i) \rangle_t d{\bf x}&=C \langle  V_\alpha(0)V_\gamma(x)\prod_{i=1}^nV_{\alpha_i }(z_i) \rangle_t\leq C |x|^{-\gamma\alpha}=Ce^{t\gamma\alpha}.
   \end{align*}
where we used the  formula \eqref{probarepresentation} and this estimate is uniform over the compact subsets of $\mathbf{z}\in\theta\mathcal{Z}$. Hence
\begin{align*}
|\oint_{|\mathbf{u}|=\boldsymbol{\delta}_t}  
   \mathbf{u}^{1-\nu} B_t(\mathbf{u},\mathbf{z})  \,  d   \mathbf{u}|\leq Ce^{t(|\nu|+\alpha-2)}
   \end{align*}
  as claimed.

\subsubsection{Case $\nu_k=1$}

Here we need to regularize also the Liouville expectation: let $\langle -\rangle_{t,\epsilon}$ be as in \eqref{modifiedlcft} except we replace $e^{\gamma c}M_\gamma(\C_t)$ by the regularized version $\int_{\C_t}V_{\gamma,\epsilon}(x)\dd x$. We use following variant of Lemma \ref{deriv2}. 
 \begin{lemma}\label{deriv2'}
Let $
I'_\epsilon(\mathbf{u},{\bf z}):=\frac{1}{u_k}
\sum_{p=1}^n  \partial_{z_p}\langle   T(\mathbf{u}^k) V_\alpha(0)\prod_{i=1}^nV_{\alpha_i,\epsilon }(z_i) \rangle_{t,\epsilon} 
  $. 
  Then $\lim_{\epsilon\to 0}I'_\epsilon(\mathbf{u},{\bf z}):= I'(\mathbf{u},{\bf z})$
exists and defines a continuous function in ${\bf z}\in \theta\caZ$ satisfying
  \begin{align}\label{contourI'}
  \int\Big(\frac{1}{(2\pi i)^{|\nu|}}\oint_{|\mathbf{u}|=\boldsymbol{\delta}_t}{\bf u}^{1-\nu} I'(\mathbf{u},{\bf z})\dd {\bf u} \Big)\bar{\varphi}({\bf z})\dd {\bf z}=\caT_t(\nu^{(k)},{\bf D}_{\nu_k}^\ast\varphi)
\end{align}
for all $\varphi\in C_0^\infty(\theta\caZ)$.
\end{lemma}
 \begin{proof} The proof is similar to Lemma \ref{deriv2} but cancellations occur for other reasons and we explain how. First, the integration by parts gives
 \begin{align*}
 I'_\epsilon(\mathbf{u},{\bf z})=K'_\epsilon(\mathbf{u},{\bf z})+L'_\epsilon(\mathbf{u},{\bf z})
\end{align*}
 where $\lim_{\epsilon\to 0}K'_\epsilon=K'$ exists and is given by
 \begin{align*}
 K'(\mathbf{u},{\bf z})   =&- \sum_{p=1}^n\sum_{\ell=1}^{k-1}\frac{Q\alpha_p}{u_k(z_p-u_\ell)^3}  \langle   T(\mathbf{u}^{(\ell,k)})  V_\alpha(0)\prod_{i=1}^nV_{\alpha_i }(z_i) \rangle_t   
   \\
&+ \sum_{p=1}^n\sum_{\ell=1}^{k-1}\frac{\alpha_p}{u_k(z_p-u_\ell)^2}  \langle  \partial_{z}X(u_\ell) T(\mathbf{u}^{(\ell,k)})  V_\alpha(0)\prod_{i=1}^nV_{\alpha_i }(z_i) \rangle_t      
 \\  
&- \sum_{p=1}^n \frac{\alpha_p\alpha}{2}\frac{1}{u_kz_p}  \langle   T(\mathbf{u}^{(k)})  V_\alpha(0)\prod_{i=1}^nV_{\alpha_i }(z_i) \rangle_t       
\\
&- \sum_{p\not= p'=1}^n \frac{\alpha_p\alpha_{p'}}{2}\frac{1}{u_k(z_p-z_{p'})}  \langle  T(\mathbf{u}^{(k)})  V_\alpha(0)\prod_{i=1}^nV_{\alpha_i }(z_i) \rangle_t    
\\
 & =:C_1(\mathbf{u},{\bf z})+C_2(\mathbf{u},{\bf z})+C_3(\mathbf{u},{\bf z})+
  C_5(\mathbf{u},{\bf z})
  \end{align*}
whereas
\begin{align*}
L'_\epsilon(\mathbf{u},{\bf z})= -
\frac{\mu\gamma }{u_k} \sum_{p=1}^n  {\alpha_p }\int_{\C_t}C_{\epsilon,\epsilon}( z_p,x)\langle  T(\mathbf{u}^{(k)})  V_\alpha(0)V_{\gamma,\epsilon}(x)\prod_{i=1}^nV_{\alpha_i,\epsilon }(z_i) \rangle_{t,\epsilon}     \,\dd x 
\end{align*} 
 is the term that needs analysis.  Let us define
\begin{align*}
B'_t(\mathbf{u},{\bf z}):=-\frac{i\mu }{u_k}\oint_{|x|=e^{-t}} \langle  T(\mathbf{u}^{(k)})  V_\alpha(0)V_\gamma(x)\prod_{i=1}^nV_{\alpha_i }(z_i) \rangle_t     \,\dd\bar x.
\end{align*} 
 Then
 \begin{align*}
B'_t(\mathbf{u},{\bf z})=&-\frac{i\mu }{u_k}\lim_{\epsilon\to 0} \oint_{|x|=e^{-t}} \langle  T(\mathbf{u}^{(k)})  V_{\alpha}(0)V_{\gamma,,\epsilon}(x)\prod_{i=1}^nV_{\alpha_i,\epsilon }(z_i) \rangle_{t,\epsilon}     \,\dd\bar x\\=&\frac{\mu}{u_k}\lim_{\epsilon\to 0} \int_{\C_t} \partial_x\langle  T(\mathbf{u}^{(k)})  V_{\alpha}(0)V_{\gamma,,\epsilon}(x)\prod_{i=1}^nV_{\alpha_i,\epsilon }(z_i) \rangle_{t,\epsilon}     \,\dd x\\
 =& -\frac{\mu Q\gamma }{u_k}\sum_{\ell=1}^{k-1} \int_{\C_t} \frac{1}{(x-u_\ell)^3}  \langle  T(\mathbf{u}^{(\ell,k)})  V_\alpha(0)V_\gamma(x)\prod_{i=1}^nV_{\alpha_i }(z_i) \rangle_t     \,\dd x
\\
&+\frac{\mu\gamma}{u_k} \sum_{\ell=1}^{k-1} \int_{\C_t} \frac{1}{(x-u_\ell)^2}  \langle \partial_zX(u_\ell) T(\mathbf{u}^{(\ell,k)})  V_\alpha(0)V_\gamma(x)\prod_{i=1}^nV_{\alpha_i }(z_i) \rangle_t     \,\dd x
\\
&-\mu\frac{ \gamma \alpha}{2u_k}   \int_{\C_t} \frac{1}{x} \langle    T(\mathbf{u}^{(k)})  V_\alpha(0)V_\gamma(x)\prod_{i=1}^nV_{\alpha_i }(z_i) \rangle_t  \,\dd x
\\ 
&-\frac{\mu\gamma}{u_k} \lim_{\epsilon\to 0}\sum_{p=1}^n  {\alpha_p }\int_{\C_t}C_{\epsilon,\epsilon}( z_p,x) \langle  T(\mathbf{u}^{(k)})  V_\alpha(0)V_{\gamma,\epsilon}(x)\prod_{i=1}^nV_{\alpha_i ,\epsilon}(z_i) \rangle_{t,\epsilon}     \,\dd x   \\
&=:P_6(\mathbf{u},{\bf z})+P_7(\mathbf{u},{\bf z})+P_{8}(\mathbf{u},{\bf z})+P_4(\mathbf{u},{\bf z}) 
\end{align*}
This proves the existence of $ \lim_{\epsilon\to 0}L'_\epsilon=L'=P_4$ and furthermore
\begin{align}\label{l4p4}
I'=\lim_{\epsilon\to 0}I'_\epsilon=C_1+C_2+C_3+C_5+B'_t-P_6-P_7-P_{8}.
\end{align}
 The claim \eqref{contourI'} follows as in Lemma \ref{deriv2}.
 \end{proof}
 Let us consider the expression
\begin{align*}
K'(\mathbf{u},{\bf z}):=&\langle   T(\mathbf{u})V_\alpha(0)\prod_{i=1}^nV_{\alpha_i }(z_i) \rangle_t  +I'(\mathbf{u},{\bf z})- M(\mathbf{u},\mathbf{z}) 
\\=&R(\mathbf{u},\mathbf{z})   + N(\mathbf{u},\mathbf{z})+D(\mathbf{u},\mathbf{z}) +I'(\mathbf{u},{\bf z}).
\end{align*}
By \eqref{confweightward}  and \eqref{contourI'} we have
 \begin{align}\label{contourI12}
  \int\Big(\oint_{|\mathbf{u}|=\boldsymbol{\delta}_t}{\bf u}^{1-\nu} K(\mathbf{u},{\bf z})\dd{\bf u} \Big)\bar{\varphi}({\bf z})\dd{\bf z}&=\caT(\nu,\varphi)-\caT(\nu^{(k)},\mathbf{D}^\ast_{\nu_k}\varphi).
   \end{align}
On the other hand
combining \eqref{IPPstress}, \eqref{IPPstressrest} and \eqref{N-I} 
we obtain
\begin{align}
K&=R+N +\sum_{i=1}^9 T_i +\sum_{i=1}^3 C_i+C_5  -\sum_{i=6}^8 P_i
 +B'_t\label{finalsum2}
\end{align}
As before it is easy to  check that the following relations hold (see Appendix \ref{ippjunk})
 \begin{align}\label{tinuk1}
\rcircleleftint_{|u_k|=e^{-t}\epsilon_k} T_i(\mathbf{u},{\bf z})\,\dd u_k=&0\quad \text{ for }i=4,5,9 
\\
\rcircleleftint_{|u_k|=e^{-t}\epsilon_k}\big(T_i(\mathbf{u},{\bf z})+C_i(\mathbf{u},{\bf z})\big)\,\dd u_k=&0\quad \text{ for }i=1,2,3\label{tinuk2}
\\
\rcircleleftint_{|u_k|=e^{-t}\epsilon_k}\big(T_i(\mathbf{u},{\bf z})-P_i(\mathbf{u},{\bf z})\big)\,\dd u_k=&0\quad \text{ for }i=6,7,8\label{tinuk2'}
\\
\rcircleleftint_{|u_k|=e^{-t}\epsilon_k} C_5(\mathbf{u},{\bf z}) \,\dd u_k=&0  \label{tinuk3}\\
\rcircleleftint_{|u_k|=e^{-t}\epsilon_k} N(\mathbf{u},{\bf z}) \,\dd u_k=&0 . \label{tinuk1'}
\end{align}
We can now conclude as in the case $\nu_k>1$. \qed



\appendix
\section{Elementary lemmas on the Virasoro algebra}
\begin{lemma}\label{LemAppendixVir}
Let $t_1, \cdots, t_k \in \Z$ be such that $t_1+\cdots+t_k>0$. Then
\begin{equation}\label{sumpositiveappendix}
\mathbf{L}^{0,\alpha}_{t_1} \cdots \mathbf{L}_{t_k}^{0,\alpha}1=0, \quad \tilde{\mathbf{L}}^{0,\alpha}_{t_1} \cdots \tilde{\mathbf{L}}_{t_k}^{0,\alpha}1=0.
\end{equation}
\end{lemma}

\proof
We will prove this relation by recursion on $k$ and work with $\mathbf{L}^{0,\alpha}_{n}$ (the case of the $\tilde{\mathbf{L}}^{0,\alpha}_{n} $ is identical). For $k=1$, the relation comes from the fact that $\mathbf{L}^{0,\alpha}_{t}1=0$ for $t>0$. For $k \geq 2$, let $j$ be the biggest index such that $t_j>0$. If $j=k$ we are done. Otherwise, we have using \eqref{virasoro} on $ \mathbf{L}^{0,\alpha}_{t_j},  \mathbf{L}^{0,\alpha}_{t_{j+1}}$
\begin{align*}
 \mathbf{L}^{0,\alpha}_{t_1} \cdots \mathbf{L}_{t_k}^{0,\alpha}1   =& \mathbf{L}^{0,\alpha}_{t_1} \cdots  \mathbf{L}^{0,\alpha}_{t_{j-1}}  \mathbf{L}^{0,\alpha}_{t_{j+1}}  \mathbf{L}^{0,\alpha}_{t_{j}} \mathbf{L}^{0,\alpha}_{t_{j+2}}    \cdots\mathbf{L}_{t_k}^{0,\alpha}1     \\
  = & (t_j-t_{j+1}) \mathbf{L}^{0,\alpha}_{t_1} \cdots  \mathbf{L}^{0,\alpha}_{t_{j-1}}  \mathbf{L}^{0,\alpha}_{t_j+t_{j+1}} \mathbf{L}^{0,\alpha}_{t_{j+2}}    \cdots\mathbf{L}_{t_k}^{0,\alpha}1    \\
& + \frac{c_L}{12} \delta_{t_j=-t_{j+1}} (t_j^3-t_j) \mathbf{L}^{0,\alpha}_{t_1} \cdots \mathbf{L}^{0,\alpha}_{t_{j-1}} \mathbf{L}^{0,\alpha}_{t_{j+2}}  \cdots \mathbf{L}_{t_k}^{0,\alpha}1.   
\end{align*}
The second term in the above equality is equal to $0$ by the recursion hypothesis. The third term in the above equality is equal to $0$ since if $t_j=-t_{j+1}$ then we can apply the recursion hypothesis on $(t_1, \cdots, t_{j-1}, t_{j+2}, \cdots, t_k)$ since in this case $t_1+ \cdots+t_{j-1}+ t_{j+2}+\cdots+ t_k= t_1+\cdots+t_k$. Therefore by iterating the above procedure with the couple $(t_j,t_{j+2})$, etc... we end up with
\[
 \mathbf{L}^{0,\alpha}_{t_1} \cdots \mathbf{L}_{t_k}^{0,\alpha}1=  \mathbf{L}^{0,\alpha}_{t_1} \cdots \mathbf{L}^{0,\alpha}_{t_{j-1}} \mathbf{L}^{0,\alpha}_{t_{j+1}}    \cdots\mathbf{L}_{t_k}^{0,\alpha} \mathbf{L}_{t_j}^{0,\alpha}1  =0.  \qedhere
\]

\begin{lemma}\label{LemAppendixVir1}
We suppose that $t_1, \cdots, t_k \in \Z$ are such that $t_1+\cdots t_k=0$.
\begin{equation*}
\mathbf{L}_{t_1}^{0,\alpha} \cdots \mathbf{L}_{t_k}^{0,\alpha}1 = \sum_{k \geq 0} a_k ( \mathbf{L}_{0}^{0,\alpha})^k 1
\end{equation*}
where the $a_k$ only depend on the commutation relations and not on $\alpha$.
\end{lemma}

\proof
We suppose that for all $t_1, \cdots, t_r$ such that $|t_i| \not = 0$ for all $i$, $t_1+ \cdots +t_r=0$ and $|t_1|+ \cdots + |t_r| \leq N$ the lemma holds. We consider $t_1, \cdots, t_k$ such that $t_1+ \cdots+ t_k=0$, $|t_i| \not = 0$ for all $i$ and $|t_1|+ \cdots + |t_k| \leq  N+1$. We consider $j$ the largest integer such that $t_j>0$. Hence for all $l >j$ we have $t_l<0$.  
We have
\begin{align*}
 \mathbf{L}_{t_1}^{0,\alpha}  \cdots \mathbf{ L}_{t_j}^{0,\alpha} \mathbf{L}_{t_{j+1}}^{0,\alpha}  \cdots \mathbf{L}_{t_k}^{0,\alpha}  & =    \mathbf{L}_{t_1}^{0,\alpha}  \cdots  \mathbf{L}_{t_{j+1}}^{0,\alpha}\mathbf{L}_{t_{j}}^{0,\alpha}  \cdots \mathbf{L}_{t_k}^{0,\alpha} \\
 & + (t_j-t_{j+1})  \mathbf{L}_{t_1}^{0,\alpha}  \cdots  \mathbf{L}_{t_j+t_{j+1}}^{0,\alpha}   \cdots \mathbf{L}_{t_k}^{0,\alpha}  \\
&+ \frac{c_L}{12} (t_j^3-t_j) \delta_{t_j=-t_{j+1}} \mathbf{ L}_{t_1}^{0,\alpha}  \cdots  \mathbf{L}_{t_{j-1}}^{0,\alpha }\mathbf{L}_{t_{j+2}}^{0,\alpha}  \cdots \mathbf{L}_{t_k}^{0,\alpha} .
\end{align*}
If $t_j=-t_{j+1}$ then the third term is of the desired form by the recursive assumption; if  $t_j \not =-t_{j+1}$ then the third term is $0$. If $t_j+t_{j+1} \not = 0$ then the second term is of the desired form by the recursive assumption. If not then one can use $\mathbf{L}_n^{0,\alpha}L_0^{0,\alpha}=\mathbf{L}_0^{0,\alpha}\mathbf{L}_n^{0,\alpha} +n \mathbf{L}_{n}^{0,\alpha}$ to get
\begin{equation}\label{L0}
 \mathbf{L}_{t_1}^{0,\alpha}  \cdots  \mathbf{L}_{t_{j-1}}^{0,\alpha}  \mathbf{L}_{0}^{0,\alpha}   \mathbf{L}_{t_{j+2}}^{0,\alpha} \cdots \mathbf{L}_{t_k}^{0,\alpha} =  (t_1+ \cdots +t_{j-1})  \mathbf{L}_{t_1}^{0,\alpha}  \cdots  \mathbf{L}_{t_{j-1}}^{0,\alpha}  \mathbf{ L}_{t_{j+2}}^{0,\alpha} \cdots \mathbf{L}_{t_k}^{0,\alpha} +\mathbf{L}_0^{0,\alpha}  \mathbf{L}_{t_1}^{0,\alpha}  \cdots  \mathbf{L}_{t_{j-1}}^{0,\alpha}   \mathbf{L}_{t_{j+2}}^{0,\alpha} \cdots \mathbf{L}_{t_k}^{0,\alpha}  
\end{equation}
and one gets the result by the recursive assumption. To conclude the recursion step, one just needs to iterate the above procedure of shifting the $\mathbf{L}^{0,\alpha}_{t_j}$ to the right on the first term $ \mathbf{L}_{t_1}^{0,\alpha}  \cdots  \mathbf{L}_{t_{j+1}}^{0,\alpha}\mathbf{L}_{t_{j}}^{0,\alpha}  \cdots \mathbf{L}_{t_k}^{0,\alpha}$ and using $\mathbf{L}^{0,\alpha}_{t_j}1=0$. 

In the general case (i.e. when some $t_i$ can be equal to $0$), one can shift the $L_{0}^{0,\alpha}$ terms to the left using \eqref{L0}.   \qed

\section{The DOZZ formula}\label{dozz}
We set $\ell(z)=\frac{\Gamma (z)}{\Gamma (1-z)}$ where $\Gamma$ denotes the standard Gamma function. We introduce Zamolodchikov's special holomorphic function $\Upsilon_{\frac{\gamma}{2}}(z)$ by the following expression for $0<{\rm Re} (z)< Q$
\begin{equation}\label{def:upsilon}
\ln \Upsilon_{\frac{\gamma}{2}} (z)  = \int_{0}^\infty  \left ( \Big (\frac{Q}{2}-z \Big )^2  e^{-t}-  \frac{( \sinh( (\frac{Q}{2}-z )\frac{t}{2}  )   )^2}{\sinh (\frac{t \gamma}{4}) \sinh( \frac{t}{\gamma} )}    \right ) \frac{dt}{t}.
\end{equation}
The function $\Upsilon_{\frac{\gamma}{2}}$  is then defined on all $\C$ by analytic continuation of the expression \eqref{def:upsilon} as expression \eqref{def:upsilon}  satisfies the following remarkable functional relations: 
\begin{equation}\label{shiftUpsilon}
\Upsilon_{\frac{\gamma}{2}} (z+\frac{\gamma}{2}) = \ell\big( \frac{\gamma}{2}z\big) (\frac{\gamma}{2})^{1-\gamma z}\Upsilon_{\frac{\gamma}{2}} (z), \quad
\Upsilon_{\frac{\gamma}{2}} (z+\frac{2}{\gamma}) = \ell\big(\frac{2}{\gamma}z\big) (\frac{\gamma}{2})^{\frac{4}{\gamma} z-1} \Upsilon_{\frac{\gamma}{2}} (z).
\end{equation}
The function $\Upsilon_{\frac{\gamma}{2}}$ has no poles in $\C$ and the zeros of $\Upsilon_{\frac{\gamma}{2}}$ are simple (if $\gamma^2 \not \in \Q$) and given by the discrete set $(-\frac{\gamma}{2} \N-\frac{2}{\gamma} \N) \cup (Q+\frac{\gamma}{2} \N+\frac{2}{\gamma} \N )$. 
With these notations, the DOZZ formula is defined for $\alpha_1,\alpha_2,\alpha_3 \in \C$ by the following formula where we set $\bar{\alpha}=\alpha_1+\alpha_2+\alpha_3$ 
\begin{equation}\label{theDOZZformula}
C_{\gamma,\mu}^{{\rm DOZZ}} (\alpha_1,\alpha_2,\alpha_3 )  = \Big(\pi \: \mu \:  \ell\big(\frac{\gamma^2}{4}\big)  \: \big(\frac{\gamma}{2}\big)^{2 -\frac{\gamma^2}{2}} \Big)^{\frac{2 Q -\bar{\alpha}}{\gamma}}   \frac{\Upsilon_{\frac{\gamma}{2}}'(0)\Upsilon_{\frac{\gamma}{2}}(\alpha_1) \Upsilon_{\frac{\gamma}{2}}(\alpha_2) \Upsilon_{\frac{\gamma}{2}}(\alpha_3)}{\Upsilon_{\frac{\gamma}{2}}(\frac{\bar{\alpha}}{2}-Q) 
\Upsilon_{\frac{\gamma}{2}}(\frac{\bar{\alpha}}{2}-\alpha_1) \Upsilon_{\frac{\gamma}{2}}(\frac{\bar{\alpha}}{2}-\alpha_2) \Upsilon_{\frac{\gamma}{2}}(\frac{\bar{\alpha}}{2} -\alpha_3)   }
\end{equation}      
 The DOZZ formula is meromorphic with poles corresponding to the zeroes of the denominator of expression \eqref{theDOZZformula}. Note that it is symmetric in $\alpha_1,\alpha_2,\alpha_3$ and real valued when $\alpha_j$ are real.

\section{Integration by parts calculations
}\label{ippjunk}
{\bf R-terms.} Here we list explicitly the terms in the integration by parts formula in the proof of Proposition \ref{wardinduction} giving zero contribution to the contour integral:  
  \begin{align*}
 R(\mathbf{u},\mathbf{z}) :=&  3Q^2\sum_{\ell=1}^{k-1}  \frac{1}{(u_k-u_\ell)^4} \langle   T(\mathbf{u}^{(\ell,k)})  V_\alpha(0)\prod_{i=1}^nV_{\alpha_i }(z_i) \rangle_t
  \\
  &   -2Q\sum_{\ell=1}^{k-1}  \frac{1}{(u_k-u_\ell)^3} \langle   T(\mathbf{u}^{(\ell,k)}) \partial_{z}X(u_\ell) V_\alpha(0)\prod_{i=1}^nV_{\alpha_i }(z_i) \rangle_t  \\
   &+\frac{Q\alpha}{2}\frac{1}{u_k^2} \langle   T(\mathbf{u}^{(k)})  V_\alpha(0)\prod_{i=1}^nV_{\alpha_i }(z_i) \rangle_t
    \\
     &  - Q^2\sum_{\ell,\ell'\not=1}^{k-1}  \frac{1}{(u_k-u_\ell)^3(u_k-u_{\ell'})^3} \langle  T(\mathbf{u}^{\ell,\ell',k})  V_\alpha(0)\prod_{i=1}^nV_{\alpha_i }(z_i) \rangle_t 
     \\
     & +2Q\sum_{\ell,\ell'=1}^{k-1}  \frac{1}{(u_k-u_\ell)^3(u_k-u_{\ell'})^2} \langle   \partial_{z}X(u_{\ell'}) T(\mathbf{u}^{\ell,\ell',k})  V_\alpha(0)\prod_{i=1}^nV_{\alpha_i }(z_i) \rangle_t  
\\
& - Q\alpha \sum_{\ell=1}^{k-1}  \frac{1}{(u_k-u_\ell)^3u_k} \langle   T(\mathbf{u}^{(\ell,k)})  V_\alpha(0)\prod_{i=1}^nV_{\alpha_i }(z_i) \rangle_t 
\\
& - \sum_{\ell,\ell'=1}^{k-1}  \frac{1}{(u_k-u_\ell)^2(u_k-u_{\ell'})^2} \langle   \partial_{z}X(u_{\ell}) \partial_{z}X(u_{\ell'}) T(\mathbf{u}^{(\ell,\ell',k)})  V_\alpha(0)\prod_{i=1}^nV_{\alpha_i }(z_i) \rangle_t   
\\
& + \alpha\sum_{\ell=1}^{k-1} \frac{1}{(u_k-u_\ell)^2 u_k } \langle   \partial_{z}X(u_{\ell})  T(\mathbf{u}^{(\ell,k)})  V_\alpha(0)\prod_{i=1}^nV_{\alpha_i }(z_i) \rangle_t   
\\
&-   \frac{\alpha^2}{4}   \frac{1}{u_k^2} \langle    T(\mathbf{u}^{(k)})  V_\alpha(0)\prod_{i=1}^nV_{\alpha_i }(z_i) \rangle_t .
 \end{align*}
 
\vskip 2mm
\noindent {\bf Proof of \eqref{cancel1}-\eqref{cancel4}}. The claim 
$T'_{9}=-T_{9}$  results from the relation 
$$\frac{1}{(u_k-x)(u_k-x')}=\frac{1}{x-x'}\Big(\frac{1}{u_k-x}-\frac{1}{u_k-x'}\Big)$$
and the fact that the mapping $(x,x')\mapsto \langle  T(\mathbf{u}^{(k)})  V_\alpha(0)V_\gamma(x)V_\gamma(x')\prod_{i=1}^nV_{\alpha_i }(z_i) \rangle_t$ is symmetric.
 
  Also, $T_5=D_5$ comes from the relation 
$$\frac{1}{(u_k-z_p)(u_k-z_{p'})}=\frac{1}{z_p-z_{p'}}\Big(\frac{1}{u_k-z_p}-\frac{1}{u_k-z_{p'}}\Big)$$ and a re-indexation of the double sum.

Using 
$$ \frac{1}{(u_k-u_\ell)^2(u_k-x)}= \frac{1}{(u_k-u_\ell)^2(u_\ell-x)} - \frac{1}{(u_k-u_\ell)(x-u_\ell)^2}+\frac{1}{(x-u_\ell)^2(u_k-x)} $$
we find that
\begin{align*}
T_7+T'_7 
=-\mu  \gamma\sum_{\ell=1}^{k-1}  \int_{\C_t} \Big( \frac{1}{(u_k-u_\ell)^2(u_\ell-x)} - \frac{1}{(u_k-u_\ell)(x-u_\ell)^2}\Big)\langle \partial_zX(u_\ell)  T(\mathbf{u}^{(\ell,k)})  V_\alpha(0)V_\gamma(x)\prod_{i=1}^nV_{\alpha_i }(z_i) \rangle_t \,\dd x
\end{align*}
satisfies
$$
\oint_{|u_k|=e^{-t}\delta_k}u_k^{1-\nu_k}({T}_7(\mathbf{u},\mathbf{z})+{T}'_{7}(\mathbf{u},\mathbf{z}))\,\dd u_k=0$$
by moving the contour to $\infty$ ($\nu_k>1$ is used here).

Using that
$$\frac{1}{u_k(u_k-x)}-\frac{1}{x(u_k-x)}=-\frac{1}{u_kx} $$
and $\oint_{|u_k|=e^{-t}\delta_k}u_k^{1-\nu_k}\frac{1}{u_k}\dd u_k=0$ for $\nu_k>1$
we deduce that 
$$
\oint_{|u_k|=e^{-t}\delta_k}u_k^{1-\nu_k}({T}_{8}(\mathbf{u},\mathbf{z})+{T}'_{8}(\mathbf{u},\mathbf{z}))\,\dd u_k=0.$$

The relation 
$$ \frac{1}{(u_k-u_\ell)^3(u_k-x)}= \frac{1}{(u_k-u_\ell)^3(u_\ell-x)} - \frac{1}{(u_k-u_\ell)^2(x-u_\ell)^2}+\frac{1}{(u_k-u_\ell)(u_\ell-x)^3}-\frac{1}{(u_k-x)(u_\ell-x)^3} $$
entails in the same way \eqref{cancel3} for $i=6$.

Finally  the relations  \eqref{cancel4} follow by computing residues at the pole $u_k=z_p$.

\vskip 2mm
\noindent {\bf Proof of \eqref{tinuk1}-\eqref{tinuk1'}}. 

The relations \eqref{tinuk1} and \eqref{tinuk1'} holds because all the corresponding  $T_i(\mathbf{u})$ are holomorphic in $u_k\in \D_t$.  For \eqref{tinuk3} we observe that
\begin{align*}
\frac{1}{2\pi i}\rcircleleftint_{|u_k|=e^{-t}\epsilon_k} C_5(\mathbf{u}) \,\dd u_k= - \sum_{p\not= p'=1}^n \frac{\alpha_p\alpha_{p'}}{2}\frac{1}{ (z_p-z_{p'})}  \langle  T(\mathbf{u}^{k})  V_\alpha(0)\prod_{i=1}^nV_{\alpha_i }(z_i) \rangle_t   
\end{align*}
and that this expression is null for antisymmetry reasons. \eqref{tinuk2} and  \eqref{tinuk2'} follow from the residue at $z_p$ and $x$ respectively
 \begin{align*}
\rcircleleftint_{|u_k|=e^{-t}\epsilon_k} T_1(\mathbf{u},\mathbf{z})\,\dd u_k=&  2\pi i\sum_{\ell=1}^{k-1}\sum_{p=1}^n      \frac{Q\alpha_p}{(z_p-u_\ell)^3} \langle   T(\mathbf{u}^{(\ell,k)})  V_\alpha(0)\prod_{i=1}^nV_{\alpha_i }(z_i) \rangle_t  =-\rcircleleftint_{|u_k|=e^{-t}\epsilon_k} C_1(\mathbf{u},\mathbf{z})\,\dd u_k
\\
\rcircleleftint_{|u_k|=e^{-t}\epsilon_k} T_2(\mathbf{u},\mathbf{z})\,\dd u_k=& 
-2\pi i\sum_{\ell=1}^{k-1} \sum_{p=1}^{n} \frac{\alpha_p}{(z_p-u_\ell)^2} \langle   \partial_{z}X(u_{\ell})  T(\mathbf{u}^{(\ell,k)})  V_\alpha(0)\prod_{i=1}^nV_{\alpha_i }(z_i) \rangle_t  =-\rcircleleftint_{|u_k|=e^{-t}\epsilon_k} C_2(\mathbf{u},\mathbf{z})\,\dd u_k
\\
\rcircleleftint_{|u_k|=e^{-t}\epsilon_k} T_3(\mathbf{u},\mathbf{z})\,\dd u_k=&    2\pi i  \sum_{p=1}^{n}  \frac{\alpha\alpha_p}{2} \frac{1}{z_p}  \langle    T(\mathbf{u}^{(k)})  V_\alpha(0)\prod_{i=1}^nV_{\alpha_i }(z_i) \rangle_t    =-\rcircleleftint_{|u_k|=e^{-t}\epsilon_k} C_3(\mathbf{u},\mathbf{z})\,\dd u_k,
\end{align*}
thus proving \eqref{tinuk2}. 
Finally, we compute
\begin{align*}
\frac{1}{2\pi i}\rcircleleftint_{|u_k|=e^{-t}\epsilon_k} T_6(\mathbf{u},\mathbf{z})\,\dd u_k=&   - \mu Q\gamma \sum_{\ell=1}^{k-1}\int_{\C_t}    \frac{1}{(x-u_\ell)^3 } \langle   T(\mathbf{u}^{(\ell,k)})  V_\alpha(0)V_\gamma(x)\prod_{i=1}^nV_{\alpha_i }(z_i) \rangle_t \,\dd x =P_6(\mathbf{u},\mathbf{z})
\\
\frac{1}{2\pi i}\rcircleleftint_{|u_k|=e^{-t}\epsilon_k} T_7(\mathbf{u},\mathbf{z})\,\dd u_k=&  \mu\gamma \sum_{\ell=1}^{k-1} \int_{\C_t} \frac{1}{(x-u_\ell)^2 } \langle   \partial_{z}X(u_{\ell})  T(\mathbf{u}^{(\ell,k)})  V_\alpha(0)V_\gamma(x)\prod_{i=1}^nV_{\alpha_i }(z_i) \rangle_t  \,\dd x=P_7(\mathbf{u},\mathbf{z})
\nonumber
\\
\frac{1}{2\pi i}\rcircleleftint_{|u_k|=e^{-t}\epsilon_k} T_{8}(\mathbf{u},\mathbf{z})\,\dd u_k=& -\frac{\mu\gamma \alpha}{2}   \int_{\C_t} \frac{1}{x} \langle    T(\mathbf{u}^{(k)})  V_\alpha(0)V_\gamma(x)\prod_{i=1}^nV_{\alpha_i }(z_i) \rangle_t  \,\dd x=P_{8}(\mathbf{u},\mathbf{z})
\\ 
\frac{1}{2\pi i}\rcircleleftint_{|u_k|=e^{-t}\epsilon_k} C_4(\mathbf{u},\mathbf{z})\,\dd u_k=&\mu\gamma\sum_{p=1}^n  \frac{\alpha_p }{2}\int_{\C_t}\frac{1}{ (z_p-x)}  \langle  T(\mathbf{u}^{k})  V_\alpha(0)V_\gamma(x)\prod_{i=1}^nV_{\alpha_i }(z_i) \rangle_t     \,\dd x=P_4(\mathbf{u},\mathbf{z}) . 
\end{align*}

\subsection{Analyticity of the vertex operators}\label{app:analytic}
Recall the definition of $U_{\boldsymbol{\alpha}}(\mathbf{z})$ in \eqref{defUward}   for   $\mathbf{z}\in\mathcal{Z}$ and real $\alpha_i$'s such that $\alpha_i<Q$. It is plain to see that $U_{\boldsymbol{\alpha}}(\mathbf{z})$ agrees with the following slightly different regularization
\[\begin{gathered}
U^{{\rm holes}}_{\boldsymbol{\alpha}}(\mathbf{z}) = \underset{ k \to \infty}{\lim} U^{{\rm holes}}_{\boldsymbol{\alpha},k}(\mathbf{z}) 
\\
\textrm{ with }\,\, U^{{\rm holes}}_{\boldsymbol{\alpha},k}(\mathbf{z}) =e^{\sum_i\alpha_i-Qc}e^{\sum_i\alpha_i P\varphi(z_i)}\E_\varphi \Big[\Big(\prod_{i=1}^n\epsilon_k^{\frac{\alpha_i^2}{2}}e^{\alpha_i X_{\D,\epsilon_k}(z_i)}\Big)e^{-\mu e^{\gamma c}M_\gamma(\D_k )} \Big]
\end{gathered}\]
where $\epsilon_k= 2^{-k}$, $X_{\D,\epsilon_k}$ is the circle average of the Dirichlet GFF and $\D_k$ is the unit disk with small holes removed around each insertion, namely $\D_k:=\D\setminus \bigcup_{i=1}^nB(z_i,\epsilon_k)$. Recall that   we get the following explicit expression by using the Girsanov theorem:
\begin{align}
  U^{{\rm holes}}_{\boldsymbol{\alpha}}(\mathbf{z})  
&=   e^{(\sum_{j=1}^n \alpha_j-Q)c} e^{\sum_i \alpha_i P \varphi(z_i)+ \sum_{i<j} \alpha_i \alpha_j G_{\D}(z_i,z_j)  }  \E_\varphi \left [   e^{- \mu e^{\gamma c}\int_{\D}  e^{ \gamma \sum_{i=1}^n \alpha_i G_{\D}(x,z_i)} M_\gamma(\dd x) }    \right ].
\label{explicitexpression}   
\end{align}
We fix $k_0$ such that the open balls $B(z_i, 2^{-k_0})$ are disjoint and included in $\D$. Set $\mathcal{O}^n:=\lbrace  \boldsymbol{\alpha} \in \R^n;  \: \alpha_i<Q, \: \forall i  \rbrace $. Then we have the following analyticity result:
\begin{proposition}
The (random) function $\boldsymbol{\alpha} \to U^{{\rm holes}}_{\boldsymbol{\alpha}}(\mathbf{z})$ admits an analytic extension in a complex neighborhood of $  \mathcal{O}^n $  such that for all real $\boldsymbol{\alpha} \in \mathcal{O}^n$ there exists some $\epsilon>0$ (depending on $\boldsymbol{\alpha} $) and (non random) $C, \tilde{C}>0$ satisfying 
\begin{align*}
& \sup_{\boldsymbol{\beta} \in  [-\epsilon, \epsilon]^n}  |U^{{\rm holes}}_{\boldsymbol{\alpha}+i\boldsymbol{\beta}} (\mathbf{z})- U^{{\rm holes}}_{\boldsymbol{\alpha}+i\boldsymbol{\beta},k_0} (\mathbf{z}) |   \\
& \leq  C (1+e^{\gamma c})e^{(\sum_{j=1}^n \alpha_j-Q)c} e^{\tilde{C} \sup_i\sup_{u \in B(z_i, 2^{-k_0})} P \varphi(u)}  \E_\varphi \left [   e^{- \mu e^{\gamma c}\int_{\D_{k_0}}  e^{ \gamma \sum_{i=1}^n \alpha_i G_{\D}(x,z_i)} M_\gamma(\dd x) }    \right ].
\end{align*}
\end{proposition}

\proof To simplify we will   suppose that $n=1$ and $z_1=0$ and set $G_{\D,k}(0,u)= \E[ X_{\D, \epsilon_k}(0) X_{\D}(u) ] $. This is no restriction as the same analysis can be performed around each insertion in case $n>1$. In this context, we get by using the Markov property of the Dirichlet GFF that:
\begin{align*}
& | U^{{\rm holes}}_{\boldsymbol{\alpha}+i\boldsymbol{\beta},k+1} -U^{{\rm holes}}_{\boldsymbol{\alpha}+i\boldsymbol{\beta},k} (\mathbf{z}) |  \\
&  = e^{(\alpha-Q)c} |  \E_\varphi \left [  \epsilon_{k+1}^{(\alpha+i \beta)^2/2} e^{ (\alpha+i \beta) X_{\D,\epsilon_{k+1}}(0) }  \left (  e^{- \mu e^{\gamma c}M_\gamma(\D_{k+1})}   - e^{- \mu e^{\gamma c}M_\gamma(\D_k) } \right )   \right ]   |  
\\
&  \leq   e^{(\alpha-Q)c}   2^{ (k+1) \beta^2/2 }|  \E_\varphi \left [  \left (   e^{- \mu e^{\gamma c}\int_{\D_{k+1}}  e^{ \gamma   \alpha_1 G_{\D,k+1}(x,0)} M_\gamma(\dd x) }  - e^{- \mu e^{\gamma c}\int_{\D_{k}}  e^{ \gamma  \alpha_1 G_{\D,k+1}(x,0)} M_\gamma(\dd x) }  \right )   \right ]   |  \\
& =   e^{(\alpha-Q)c}2^{ (k+1) \beta^2/2 }|  \E_\varphi \left [  \left (  e^{- \mu e^{\gamma c}  (Y_k+\delta Y_k) } - e^{- \mu e^{\gamma c} Y_k } \right )   \right ]   |.  
\end{align*}
where we have set 
$$Y_k=  \int_{\D_k}  e^{\gamma P\varphi (x)}  e^{  \gamma \alpha_1 G_{\D,\epsilon_{k+1} } (0,x)  } M_\gamma(\dd x)\quad \text{ and }\quad \delta Y_k= \int_{\D_{k}\setminus\D_{k+1}}  e^{\gamma P\varphi (x)}  e^{  \gamma \alpha_1 G_{\D,k+1 } (0,u)  } M_\gamma(\dd x).$$
Now, we consider the cases $\delta Y_k>1$ and $\delta Y_k\leq 1$. By  FKG inequality for the Dirichlet GFF, we have  
\begin{align*}
  \E_\varphi &\left [  \ind_{\delta Y_k>1}  |   e^{- \mu e^{\gamma c} Y_k }- e^{- \mu e^{\gamma c}  (Y_k+\delta Y_k) }  |   \right ]    \\
& \leq 2   \E_\varphi \left [  \ind_{\delta Y_k>1}     e^{- \mu e^{\gamma c} Y_k }   \right ]  \\
& \leq 2 \E_\varphi \left [  \ind_{\delta Y_k>1}  \right ]  \E_\varphi  \left [   e^{- \mu e^{\gamma c} Y_k }   \right ]  \\
& \leq 2   \E_\varphi \left [ ( \delta Y_k)^\eta  \right ]  \E_\varphi  \left [   e^{- \mu e^{\gamma c} Y_k}   \right ] .
\end{align*}
Next, we choose $\beta>0$ and $\eta>0$ such that $2^{ (k+1) \beta^2/2 } \E_\varphi \left [ ( \delta Y_k)^\eta  \right ] \leq C e^{\gamma \eta \sup_{|u| \leq \epsilon_k} P\varphi (u)}2^{-\theta k}$ with $\theta>0$.

In the case $\delta Y_k \leq 1$, we get (using the inequality $x \ind_{\{x \leq 1\}}  \leq  x^\eta$ for $x>0$ and then FKG for the Dirichlet GFF) 
\begin{align*}
  2^{ (k+1) \beta^2/2 } \E_\varphi &\left [  \ind_{\delta Y_k \leq 1}  |   e^{- \mu e^{\gamma c} Y_k }- e^{- \mu e^{\gamma c}  (Y_k+\delta Y_k) }  |   \right ]     \\
& \leq  2^{ (k+1) \beta^2/2 } \mu e^{\gamma c} \E_\varphi \left [  \ind_{\delta Y_k \leq 1}   \delta Y_n    e^{- \mu e^{\gamma c} Y_k }   \right ]  \\
& \leq  2^{ (k+1) \beta^2/2 } \mu e^{\gamma c} \E_\varphi \left [  (\delta Y_k)^\eta     e^{- \mu e^{\gamma c} Y_k }   \right ]  \\
& \leq    2^{ (k+1) \beta^2/2 } \mu e^{\gamma c}  \E_\varphi \left [ ( \delta Y_k)^\eta  \right ]  \E_\varphi  \left [   e^{- \mu e^{\gamma c} Y_k }   \right ]    \\
& \leq    C 2^{-k \theta}  e^{\gamma c}  e^{  \gamma \eta \sup_{|u| \leq \epsilon_k} P\varphi (u)} \E_\varphi  \left [   e^{- \mu e^{\gamma c} Y_k }   \right ] .
\end{align*}
Gathering the above considerations, we get 
\begin{equation*}
 |  U^{{\rm holes}}_{\boldsymbol{\alpha}+i\boldsymbol{\beta},k+1} -U^{{\rm holes}}_{\boldsymbol{\alpha}+i\boldsymbol{\beta},k} (\mathbf{z})  |  \leq  C e^{(\alpha-Q)c} 2^{-k \theta}  (1+ e^{\gamma c})  e^{  \gamma \eta \sup_{|u| \leq \epsilon_n} P\varphi (u)} \E_\varphi  \left [   e^{- \mu e^{\gamma c} Y_n }   \right ]
\end{equation*} 
This shows that the (random) analytic function $U^{{\rm holes}}_{\boldsymbol{\alpha}+i\boldsymbol{\beta},k} (\mathbf{z})$ converges as $k\to\infty$ with probability $1$ and for all $c$ towards an analytic function that satisfies
\[
 |   U^{{\rm holes}}_{\boldsymbol{\alpha}+i\boldsymbol{\beta},k+1} -U^{{\rm holes}}_{\boldsymbol{\alpha}+i\boldsymbol{\beta},k_0} (\mathbf{z})|  \leq   Ce^{(\alpha-Q)c} (1+ e^{\gamma c})  e^{  \gamma \eta \sup_{|u| \leq \frac{1}{2}} P\varphi (u)} \E_\varphi  \left [   e^{- \mu e^{\gamma c} Y_{k_0} }   \right ].\qedhere
\]


\end{document}